\documentclass{birkjour}
\usepackage[english]{babel}
\usepackage{amssymb}
\usepackage{amsmath,esint}
\usepackage[shortlabels]{enumitem}
\usepackage{amsthm}
\usepackage{bbm, mathrsfs,pgf,tikz}
\usepackage[backgroundcolor=white, bordercolor=blue,
linecolor=blue]{todonotes}
\usepackage{url}

\usepackage[pagewise]{lineno} 
\usepackage{dsfont}
\usepackage{cases}
\usetikzlibrary{arrows}
\usepackage[
pagebackref, 
pdfpagelabels, plainpages=false]{hyperref}

\hypersetup{
colorlinks=true,
linkcolor=red,
filecolor=magenta,
urlcolor=cyan,
citecolor=blue,
}

 \newtheorem{theorem}{Theorem}[section]
\newtheorem{corollary}[theorem]{Corollary}
\newtheorem{lemma}[theorem]{Lemma}
\newtheorem{proposition}[theorem]{Proposition}
\theoremstyle{definition}
\newtheorem{definition}[theorem]{Definition}
\newtheorem{remark}[theorem]{Remark}
\newtheorem{example}[theorem]{Example}
\newtheorem{counterexample}[theorem]{Counterexample}
\numberwithin{equation}{section}

\DeclareMathOperator*{\essinf}{ess\,inf}
\DeclareMathOperator{\dist}{dist}
\DeclareMathOperator{\supp}{supp}
\DeclareMathOperator{\diam}{diam}

\DeclareMathOperator{\diag}{diag}

\DeclareMathOperator{\loc}{loc}

\DeclareMathOperator{\pv}{\operatorname{p.\!v.}}

\DeclareMathOperator{\cE}{\mathcal{E}}

\DeclareMathOperator{\cN}{\mathcal{N}}

\DeclareMathOperator{\cJ}{\mathcal{J}}

\DeclareMathOperator{\R}{\mathbb{R}}

\newcommand{\il}{\int\limits}
\newcommand{\iil}{\iint\limits}
\renewcommand{\d}{\mathrm{d}}
\renewcommand{\div}{\operatorname{div}}
\newcommand{\eps}{\varepsilon} 
\newcommand{\classA}[1]{\mathbf{\mathscr{A}_{#1}}}
\newcommand{\nuxminy}{\nu(x\!-\!y)}

\newcommand{\Wnu}{W_{\nu}^{p}(\R^d)}
\newcommand{\WnuOmR}{W^{p}_{\nu}(\Omega|\R^d )}
\newcommand{\WnuOme}{W^{p}_{\nu}(\Omega|\Omega_\nu)}
\newcommand{\WnuOmO}{W^{p}_{\nu,\Omega}(\Omega|\R^d )}
\newcommand{\WnuOmRO}{W_{\nu,0}^{p}(\Omega|\R^d )}
\newcommand{\WnuOm}{W^{p}_{\nu}(\Omega)}

\newcommand{\TnuOm}{T^{p}_{\nu} (\Omega^c)}

\newcommand{\WnuOmRa}{W^p_{\nu_\eps}(\Omega|\R^d )}
\newcommand{\WnuOmRao}{W^p_{\nu_\eps,0}(\Omega|\R^d)}
\newcommand{\WnuOma}{W^{p}_{\nu_\varepsilon}(\Omega)}

\newcommand{\vertiii}[1]{{\left\vert\kern-0.25ex\left\vert\kern-0.25ex\left\vert #1 \right\vert\kern-0.25ex\right\vert\kern-0.25ex\right\vert}}
\addto\extrasenglish{%

}

\begin{document}

\title[Stability of complement value problems for  $p$-L\'{e}vy operators]{Stability of complement value problems for nonlocal $p$-L\'{e}vy operators  
\begin{center}
	\footnotesize{\url{https://doi.org/10.1007/s00030-024-01006-6}}
	\end{center}
	}

\author{Guy Foghem} 
\address{{\tiny Technische Universit\"{a}t Dresden, Fakult\"{a}t f\"{u}r Mathematik, Institut f\"{u}r Wissenschaftliches Rechnen,  Zellescher Weg 23-25 01217, Dresden, Germany.}}
\email{guy.foghem[at]tu-dresden.de}
\thanks{Financial support by the DFG via the Research Group Unit FOR 3013: ``Vector-and Tensor-Valued Surface PDEs''(project number 417223351) is gratefully acknowledged.}

\vspace{-14ex}

\begin{abstract}
We set up a general framework tailor-made to solve complement value problems governed by symmetric nonlinear nonlocal integrodifferential $p$-L\'{e}vy operators. 
A prototypical example of integrodifferential $p$-L\'{e}vy operators is the well-known fractional $p$-Laplace operator. 
Our main focus is on nonlinear Integro-Differential Equations (IDEs) in the presence of Dirichlet, Neumann and Robin conditions and we show well-posedness results. 
Several results are new even for the fractional $p$-Laplace operator but we develop the approach for general translation-invariant nonlocal operators.
We also bridge the gap from nonlocal to local, by showing that solutions to the local Dirichlet and Neumann boundary value problems associated with $p$-Laplacian are strong limits of the nonlocal
ones.
\end{abstract}


\maketitle
\vspace{-5ex}
{\tableofcontents}

\vspace{-5ex}

\section{Introduction}\label{sec:intro}
In this article, we  study certain classes of nonlinear  Integro-Differential Equations (IDEs) associated with symmetric nonlinear nonlocal $p$-L\'{e}vy operators, $1<p<\infty$ which are integro-differential operators of the form
\begin{align}
\begin{split}
L u(x)&= 2\pv \int_{\R^d}|u(x)-u(y)|^{p-2}(u(x)-u(y))\nuxminy \d y, \qquad\quad (x\in \R^d)\\
&=2\pv \int_{\R^d}\psi(u(x)-u(y))\nuxminy \d y,
\end{split}
\end{align}
defined for a measurable function $u:\R^d \to \R$ whenever the right hand side makes sense.
Throughout, we assume that $d\geq 1$, $\psi (t)=|t|^{p-2}t,\, t\in\R$ and  the function $\nu:\R^d\setminus\{0\}\to[0,\infty)$ is the density of a symmetric $p$-L\'{e}vy measure, i.e., 
$\nu$ satisfies 
\begin{align}\tag{L}\label{eq:plevy-integ-cond-intro}
\begin{split}
\nu(h) = \nu(-h) \quad &\text{ \,\, and\,\, }
\qquad \int_{\R^d} (1\land |h|^p )\nu(h) \d h <\infty . 
\end{split}
\end{align}
Moreover, if $\nu$ is radial, we  use the same notation for its radial profile, i.e., $\nu(h) = \nu(|h|)$. The notation $\pv$ stands for the principal value, while  $a\land b$ denotes $\min(a, b)$, $a,b\in \R$. More succinctly, given an open set $\Omega\subset \R^d$,  we are interested in nonlinear nonlocal problems of the forms
\begin{align}\tag{$P_{\nu, \tau}$}\label{eq:main-problem}
L u = f \,\,\text{ in }  \,\,\Omega\quad\text{ and }\quad \tau\mathcal{N}u + (1-\tau)\beta |u|^{p-2}u= g \,\, \text{ on }\,\, \Omega^c,
\end{align}
where for simplicity\footnote{At the risk of being pedantic, the general case $\tau\in (0,1)$ readily reduces to the case $\tau=\frac12$ by renaming the data $\beta$ and $g$.} $\tau\in \{0,\frac{1}{2}, 1\}$ and the data $f: \Omega\to \R$, $g,\beta: \Omega^c\to \R$ are given. The operator $u\mapsto\cN u$ is the nonlocal $p$-normal derivative across the boundary of $\Omega$, 
\begin{align}
\cN u(y)=2\int_{\Omega}\psi(u(y)-u(x))\nuxminy \d x\qquad\quad (y\in \Omega^c). 
\end{align}
The nonlinear operator $u\mapsto \cN u$ appears naturally while deriving the nonlocal Gauss-Green formula 
(see Appendix \ref{sec:appendix-gauss-green}), which was introduced in the linear setting in \cite{DROV17}.  We point out that another type of such an operator appeared earlier in literature see for instance \cite{DGLZ12}. 

In view of problem \eqref{eq:main-problem}, it is necessary and sufficient to prescribe the complement condition on $\Omega_e:=\Omega_\nu\setminus \Omega$ with $\Omega_\nu=\Omega+\supp\nu$ which we name as \emph{the nonlocal boundary} or \emph{the exterior boundary} of $\Omega$ with respect to $\nu$. As well, we also call $\Omega_\nu^{*}=\Omega\cup \Omega_\nu$ to be \emph{the nonlocal hull} of $\Omega$ with respect to $\nu$. Note that if $0\in\supp\nu$ then $\Omega\subset \Omega_\nu$ hence $\Omega_\nu^{*}=\Omega_\nu$; see Section \ref{sec:basicsandnotations} and Section \ref{sec:existence-weak-sol} for more details. Broadly speaking, PDEs oriented folks might regard view the problem \eqref{eq:main-problem}  as the nonlocal counterpart of the nonlinear local problem of the form 
\begin{align}\tag{$P_{\tau}$}\label{eq:main-problem-local}
-\Delta_p u = f \text{ in }  \Omega\quad\text{ and }\quad \tau\partial_{n,p}u + (1-\tau)\beta |u|^{p-2}u= g \text{ on } \partial \Omega,
\end{align}
where $ \Delta_pu=\div(|\nabla u|^{p-2}\nabla u)$ denotes the $p$-Laplace operator and when $\Omega$ is smooth, $\partial_{n,p}u =|\nabla u|^{p-2}\nabla u\cdot n$ is the $p$-normal derivative on the boundary $\partial \Omega$ of $\Omega$, viz., at a point $x\in \partial\Omega$, $n (x)$ is the outward normal derivative on the boundary $\partial\Omega$. Of course, one immediately recognizes that the complement (resp. boundary) condition the problem \eqref{eq:main-problem} (resp.  \eqref{eq:main-problem-local}) is the Neumann condition $\cN u= g$ (resp. $\partial_{n,p} u=g$) for $\tau=1$, the Robin condition $\cN u + \beta|u|^{p-2}u=2g$
(resp. $\partial_{n,p} u+ \beta|u|^{p-2}u=2g$) for $\tau=\frac12$ and  the Dirichlet condition  $|u|^{p-2} u=g$ or equivalently $u= |g|^{p'-2}g$ for $\tau=0$ and $\beta=1$ where we have $pp'=p+p'$.  Actually, our second goal is to show that weak solutions to the problem \eqref{eq:main-problem-local} with $\tau=0,1$ are strong limits of those of  problems of the forms \eqref{eq:main-problem}. Let us  first discuss the assumption in \eqref{eq:plevy-integ-cond-intro} which is utterly  decisive and, afterwards, give a prototypical example. 


\par It remarkably appears that the assumption in  \eqref{eq:plevy-integ-cond-intro} is unavoidable. Indeed, on the one hand, the symmetry condition is natural in the sense that one has
\begin{align*}
\cE_{\R^d}(u,u)&=\iint_{ \R^d\R^d} |u(x)-u(y)|^p\nu(x-y)\d y\d x\\
&= \iint_{ \R^d\R^d} |u(x)-u(y)|^p\nu^{\operatorname{sym}}(x-y)\d y\d x. 
\end{align*} 
where $\nu^{\operatorname{sym}}(h)= \frac12(\nu(h)+\nu(-h))$ is the symmetric part of $\nu$. On the other hand, it appears that $\cE_{\R^d}(u,u)<\infty$ for all $u\in C_c^\infty(\R^d)$ if and only if $\nu$ is $p$-L\'{e}vy integrable, i.e., $\nu\in L^1(\R^d,1\land|h|^p)$; cf. Section \ref{sec:charac-plevy} for the details.  It turns out that the symmetry and the $p$-L\'{e}vy condition in  \eqref{eq:plevy-integ-cond-intro} can be self-generated through the energy form $\cE_{\R^d}(u,u).$ Furthermore, a heuristic computation reveals that the first variation of the operator $u\mapsto Q(u)=\frac1p\cE_{\R^d}(u,u)$ on the Banach space $\Wnu=\{ u\in L^p(\R^d)\,:\, \cE_{\R^d}(u,u)<\infty\}$ with the norm
$\|u\|_{\Wnu}= \big(\|u\|^p_{L^p(\R^d)}+ \cE_{\R^d}(u,u)\big)^{1/p}$ gives rise to the $p$-L\'{e}vy operator $L$, which is also the subdifferential of $Q$. Moreover, we have
$\cE_{\R^d}(u,u)= \langle Lu, u\rangle= Q'(u)(u) $, $u\in C_c^\infty(\R^d)$.  Note  in passing that under mild assumptions, the space $\Wnu$ embeds in an Orlicz space; this is implied by the Sobolev type inequality established in \cite{Fog21b}. The density $\nu$ is somewhat the ``order'' of the operator $L$, which becomes apparent in the case of fractional kernel $\nu(h) = |h|^{-d-sp}$, $ s \in (0,1)$ is fixed. For case $p=2$, one immediately recognizes the so-called L\'{e}vy condition in \eqref{eq:plevy-approx-intro} while the operator $L$ is translation invariant and generates a symmetric L\'{e}vy process \cite{Sat13}. It is important to emphasize that our main assumption here is only the condition \eqref{eq:plevy-integ-cond-intro}, i.e., the symmetry and the $p$-L\'{e}vy integrability of $\nu$. Moreover, we do not compare $\nu$ to the fractional kernel $|h|^{-d-sp}$, this makes the investigations in our framework very general and more challenging. In general, the pointwise evaluation of $Lu(x)$ and  $-\Delta_p u(x)$ for  $u\in C^2_b(\R^d)$ is warranted in the degenerate case (also often called the superquadratic case), i.e. $p\geq 2$. However, in the  singular case (also often called the subquadratic case), i.e.,  $1<p<2$, the pointwise definition of $Lu(x)$ and  $-\Delta_p u(x) $ might not exist even for \emph{a bona fide} function  $u\in C_c^\infty(\R^d)$. In many situations, the operators $L$ and  $-\Delta_p$ do not always act on functions in a reasonable pointwise sense. We refer interested readers to Appendix \ref{sec:appendix-pointwise} for some sketchy examples. In both cases, i.e., $1<p<\infty$, a reasonable alternative is rather to  evaluate  $Lu$ and $\Delta_pu$ in the generalized sense, i.e., in the weak sense or via their respective associated energies forms. For instance by duality (see Section \ref{sec:nonlocal-form}), one finds that the nonlinear operators $L: W^p_\nu(\R^d)\to (W^p_\nu(\R^d))'$ and $\Delta_p:W^{1,p}(\R^d)\to (W^{1,p}(\R^d))'$ are well-defined.  
In particular, $Lu$ and $\Delta_p u$ are distributions. Morally, the  nonlocal operator $L$ may be as good or as bad as the local operator $-\Delta_p$.  The operator $L$ can be seen as a prototype of a nonlinear nonlocal operator of divergence form just as the $-\Delta_p$ is a prototype of a nonlinear local operator of divergence form.

\par For prototypical example, consider $s\in \R$ then an effortless computation reveals  that $ \nu(h) =\frac{C_{d,p,s}}{2}|h|^{-d-sp}$ belongs to $L^1(\R^d, 1\land|h|^p)$, i.e., $\nu$ satisfies \eqref{eq:plevy-integ-cond-intro} if and only if $s\in (0,1)$. In this case, the resulting  space $\Wnu=W^{s,p}(\R^d)$ is the usual fractional Sobolev-Slobodeckij space, while the associated integro-differential operator $L$ is the well-known fractional $p$-Laplace operator $(-\Delta)^s_p$,  
\begin{align*}
(-\Delta)^s_pu(x)&= C_{d,p,s}\pv \int_{\R^d} \frac{|u(x)-u(y)|^{p-2}(u(x)-u(y))}{|x-y|^{d+sp}}\d y, 
\end{align*}
where $C_{d,p,s}$ is a normalizing constant  of the fractional $p$-Laplacian which we define by 
\begin{align*}
C_{d,p,s}=\frac{s(1-2s)\Gamma(\frac{d+sp}{2})}{\pi^{\frac{d-1}{2}}\Gamma(\frac{sp+1}{2})\Gamma(p(1-s))\cos(s\pi)}
=\frac{2^{2s}\,\Gamma(\frac{d+sp}{2})\Gamma(\frac{2s+1}{2})\Gamma(2(1-s))}{\pi^{\frac{d}{2}}\,\, |\Gamma(-s)|\,\Gamma(\frac{sp+1}{2})\Gamma(p(1-s))}. 
\end{align*}
Our choice of the constant $C_{d,p,s}$, (see Section \ref{sec:norming-cste-pfrac} where the asymptotic near $s=0$ is also considered), guaranties the following properties:
\begin{itemize}
\item For $p=2$, by \cite[Proposition 2.21]{guy-thesis},  $C_{d,2,s}$  is the unique normalizing constant of the fractional Laplacian such that $\widehat{(-\Delta)^s u}(\xi)= |\xi|^{2s}\widehat{u}(\xi)$, $\xi\in \R^d$ for all $u\in C_c^\infty(\R^d)$, where $\widehat{u}(\xi)=\int_{\R^d}e^{-ix\cdot\xi} u(x)\d x$ denotes the Fourier transform of $u$. Namely, we have 
\begin{align*}
C_{d,2,s}= \frac{s(1-2s)\Gamma(\frac{d+2s}{2})}{\pi^{\frac{d-1}{2}}\Gamma(\frac{2s+1}{2})\Gamma(2(1-s))\cos(s\pi)}= 
\frac{s2^{2s}\Gamma(\frac{d+2s}{2})}{\pi^{\frac{d}{2}}\Gamma(1-s)}. 
\end{align*}
\item For any $u\in L^\infty(\R^d)\cap C^2(B_1(x))$, $\nabla u(x)\neq0$ we have $(-\Delta)^s_p u(x)\xrightarrow{s\to1}-\Delta_p u(x)$.
\item For all $u\in W^{1,p}(\R^d)$ we have $\|u\|_{W^{s,p}(\R^d)}\xrightarrow{s\to 1} \|u\|_{W^{1,p}(\R^d)}$; see Section \ref{sec:convergence-form}. 
\item Moreover we have the following asymptotic behaviors  
\end{itemize}
\begin{align}\label{eq:asymptotic-norming-intro}
\lim_{s\to 0} \frac{C_{d,p,s}}{s(1-s)}&=\frac{2}{|\mathbb{S}^{d-1}|\Gamma(p)}
\qquad \text{and}\qquad 
\lim_{s\to 1} \frac{C_{d,p,s}}{s(1-s)}=\frac{2p}{|\mathbb{S}^{d-1}| K_{d,p}}. 
\end{align} 
Here we emphasize that the constant $K_{d,p}$, see \cite{Fog23,guy-thesis} for the computation, plays a crucial role in our asymptotic analysis and is given by 
\begin{align}\label{eq:cst-kdp}
K_{d,p}&= \fint_{\mathbb{S}^{d-1}} |w_d|^p\d\sigma_{d-1}(w) = \frac{\Gamma\big(\frac{d}{2}\big)\Gamma\big(\frac{p+1}{2}\big)}{\Gamma\big(\frac{d+p}{2}\big) \Gamma\big(\frac{1}{2}\big)}.
\vspace*{-2ex}\end{align}

The asymptotic $s\to 1$,  highlighting the factor $K_{d,p}$  is already anticipated in \cite[Eq: 2.38]{guy-thesis} for the case $p=2$ since $K_{d,2}=\frac1{d}$. The above asymptotic in \eqref{eq:asymptotic-norming-intro} perfectly aligns with the case $p=2$, as obtained in \cite{Hitchhiker}.
 Despite the amusing fact of this asymptotic, it is important for the reader to keep in mind that $C_{d,p,s}$ is purely artificial and that only the case $p=2$ naturally appears as the unique constant for which $\widehat{(-\Delta)^s u} (\xi)= |\xi|^{2s}\widehat{u}(\xi)$. Other  different normalizing constants for the fractional $p$-Laplacian are proposed in \cite{dTGCV21,War16} wherein, one also finds other representations of the fractional $p$-Laplace  operator. Additional special examples of $p$-L\'evy operators are obtained by considering $p$-L\'evy kernels of the form $ \nu(h) =\frac{C_{d,p,s}}{2}|h|^{-d-sp}+\frac{C_{d,q,t}}{2}|h|^{-d-tq}$ with $s,t\in (0,1)$ and $1<q\leq p$; the particular case $q=p$ yields the mixed nonlocal operator $Lu=(-\Delta)^s_pu+(-\Delta)^t_pu$.

\smallskip
\par The main purpose of this article is twofold. The primary objective is to study the well-posedness of the nonlocal problem \eqref{eq:main-problem} under additional mild assumptions. 
For instance, as an \textit{avant-go\^{u}t}, let us illuminate  what we do with the particular case  of the Neumann problem, i.e., $\tau=1$ and for the particular fractional kernel $\nu(h)=\frac{C_{d,p,s}}{2}|h|^{-d-sp}$.
We refer the reader to Section \ref{sec:existence-weak-sol} for the general setting and  more details.  We point out that $u$ is a weak solution to Neumann problem $(P_{\nu,1})$ if $u\in \WnuOmR$ and satisfies 
\begin{align}\label{eq:illustr-neumann}
\vspace*{-4ex}\cE(u, v)=\int_\Omega f (x)v(x)\d x+ \int_{\Omega^c} g(y) v(y)\d y\qquad\text{for all $u\in \WnuOmR$}. 
\vspace*{-2ex}\end{align}
Here, $\WnuOmR=\big\{ u:\R^d\to\R\,\text{meas.}:\, u|_\Omega\in L^p(\Omega), \,\,\cE(u,u)<\infty\big\}$
where the energy form $\cE(\cdot,\cdot)$ is defined as  
\begin{align*}
\cE(u, v)= \frac{C_{d,p,s}}{2}\iil_{(\Omega^c\times \Omega^c)^c}\frac{ \psi(u(x)-u(y))(v(x)-v(y))}{|x-y|^{d+sp}}\d y\d x.
\end{align*}
In Theorem \ref{thm:nonlocal-Neumann-var}, we establish the well-posedness up to additive constants of the variational problem \eqref{eq:illustr-neumann} and hence of the  Neumann problem $(P_{\nu,1})$  on the space $\WnuOmR$ whenever $f\in L^{p'}(\Omega)$ and $g\in L^{p'}(\Omega^c, \omega^{1-p'})$ with $\omega(h)= 
(1+|h|)^{-d-sp}$ are compatible, i.e., 
$$\int_\Omega f (x)\d x+ \int_{\Omega^c} g(y)\d y=0. 
$$

Let us highlight two important observations regarding the weak formulation \eqref{eq:illustr-neumann} at this stage. 
First of all, one observes in this  particular case  that the weight $\omega^{1-p'}$ is rapidly increasing, hence  the Neumann data $g\in L^{p'}(\Omega^c, \omega^{1-p'})$ is required to decay rapidly at infinity. Although this may seem  restrictive, it is however counterintuitive since the space 
$L^{p'}(\Omega^c, \omega^{1-p'})$ is in fact the largest admissible function space for the Neumann data $g$. Indeed, we establish a non-existence result in Theorem \ref{thm:non-existence-Neumann} where we exhibit some examples of Neumann data $g$ not belonging to $L^{p'}(\Omega^c, \omega^{1-p'})$ and  compatible with $f=0$
and such that  the variational Neumann problem \eqref{eq:illustr-neumann} does not have a  weak solution in $\WnuOmR$. 
As the second observation, one notices that \cite[Definition 3.6]{DROV17} and subsequent definitions like \cite[Definition 2.7]{MP19} look very similar to \eqref{eq:illustr-neumann} at first glance; for the general case, we refer the reader to Definition \ref{def:neumann-var-sol}. However,  the test space defined in \cite[Eq. (3.1)]{DROV17}, \cite[Section 2]{MP19} depends on the Neumann data $g$, which is somewhat unaccustomed. We emphasize that our test space $\WnuOmR$ in the weak formulation \eqref{eq:illustr-neumann} does not depend on the Neumann data $g$. Next, we would like to mention that our strategy of studying the problem \eqref{eq:main-problem} in the general setting, including  Dirichlet and Robin complement problems, requires bringing into play certain crucial tools. These include nonlocal function spaces,
nonlocal Poincar\'{e} types inequalities and nonlocal trace spaces. 
On the one hand, the Poincar\'{e} inequalities encoding the coercivity of the nonlocal energies which we establish in Section \ref{sec:poincare} include the Poincar\'{e} type inequality when $\Omega$ is not necessarily connected but has finitely many connected components and the Poincar\'{e}-Friedrichs type inequality when $\Omega$ is bounded only in one direction. To the best of our knowledge, such generalization of the Poincar\'{e} inequalities do not exist hitherto in the literature. On the other hand, the nonlocal trace spaces truly embody the Dirichlet, Neumann  and/or Robin complement data $g$ of the problem \eqref{eq:main-problem}. There are several new findings on nonlocal trace spaces, we refer interested reader for instance to \cite{GH22,DTCY21,BGPR20,Ru20,DyKa19}.  It is worth mentioning that the result \cite{GH22} provides a robust nonlocal fractional trace space, viz., one is able to recover the classical local trace from the nonlocal one.  

\smallskip 

\par The secondary objective is to prove the $L^p$-convergence of nonlocal  to local weak solutions; cf Section \ref{sec:conv-solution}. Let us explain for which family of problems we are able to prove the convergence. Given a family $(\nu_\eps)_\eps$, $\nu_\eps\geq0$, of radial $p$-L\'{e}vy kernels, we prove in Theorem \ref{thm:charac-plevy-concentrate} that, there holds the formula 
\begin{align}\label{eq:limit-goal-intro}
\lim_{\eps\to0}\iint_{ \R^d\R^d}|u(x)-u(y)|^p\nu_\eps(x-y)\d y \d x=K_{d,p}\int_{\R^d}|\nabla u(x)|^p\d x
\end{align}
for all $u\in W^{1,p}(\R^d)$ if and only if the family $(\nu_\eps)_\eps$ satisfies 
\begin{align}\label{eq:plevy-approx-intro}
\lim_{\eps\to0}\int_{\R^d}(1\land|h|^p) \nu_\eps (h)\d h=1 \,\,\text{and for all $\delta>0,$}\,\,  
 \lim_{\eps\to0} \int_{|h|>\delta} \hspace*{-2ex}\nu_\eps (h)\d h=0. 
\end{align} 

In other words, condition \eqref{eq:plevy-approx-intro} is the sharp condition regarding the validity of the formula \eqref{eq:limit-goal-intro}, a.k.a., the BBM-formula. The aforementioned  family $(\nu_\eps)_\eps$ was introduced in \cite{Fog23,guy-thesis}. Note in passing that a typical example of family $(\nu_\eps)_\eps$  satisfying  \eqref{eq:plevy-approx-intro}, is to consider the normalized fractional kernels $\nu_\eps(h)= a_{\varepsilon, d,p} |h|^{-d-(1-\eps)p}$ with $\eps=1-s$ and $a_{\varepsilon, d,p} = \frac{p\eps(1-\eps)}{|\mathbb{S}^{d-1}|}$; see Example \ref{Ex:stable-class}. Another fascinating example  of $(\nu_{\eps})_{\eps}$ satisfying \eqref{eq:plevy-approx-intro} is obtained from any radial function $\nu: \R^d \setminus\{0\}\to [0, \infty)$ that is $p$-L\'{e}vy  normalized, i.e., $\int_{\R^d}(1\land|h|^p)\,\nu(h)\d h=1$, by defining $\nu_\eps$ the rescaled version of $\nu$, as follows

\begin{align*}
\begin{split}
\nu_\varepsilon(h) = 
\begin{cases}
\varepsilon^{-d-p}\nu\big(h/\varepsilon\big)& 
\text{if}~~|h|\leq \varepsilon,\\
\varepsilon^{-d}|h|^{-p}\nu\big(h/\varepsilon\big)& \text{if}~~\varepsilon<|h|\leq 1,\\
\varepsilon^{-d}\nu\big(h/\varepsilon\big)& \text{if}~~|h|>1.
\end{cases}
\end{split}
\end{align*}
Let $L_\eps$ and $\cN_\eps$ be the nonlocal operators associated with $\nu_\eps$ and put $\mu=K_{d,p}^{-1}$. Consider the problem 
\begin{align}\tag{$P_{\nu_\eps, \tau}$}\label{eq:main-problem-esp}
\mu L_\eps u = f_\eps\,\,\text{ in }  \,\,\Omega\quad\text{ and }\quad \tau\mu\mathcal{N}_\eps u + (1-\tau)|u|^{p-2}u= g_\eps \,\, \text{ on }\,\, \Omega^c. 
\end{align}

Surprisingly, under the condition in \eqref{eq:plevy-approx-intro} and some mild conditions on $\Omega$, $f_\eps$ and $g_\eps$ we prove, cf. Section \ref{sec:conv-solution}, that weak solutions to the problem  $\eqref{eq:main-problem-esp}$ with $\tau\in\{0,1\}$ strongly converge in $L^p(\Omega)$ as $\eps\to0$, to the weak solutions of the corresponding local problem $\eqref{eq:main-problem-local}$. We also prove the convergence of weak associated with regional type operators. Another appealing effect of the approximation family $(\nu_\eps)_\eps$ is that for $u\in  L^\infty(\R^d)\cap C^2(B_1(x))$ with $\nabla u(x)\neq0$ there holds,
\begin{align*}
\lim_{\eps\to0}2\pv\int_{\R^d} |u(x)-u(y)|^{p-2}(u(x)-u(y)) \nu_\eps(x-y)\d y= -K_{d,p}\Delta_p u(x). 
\end{align*}
In particular we find that  $(-\Delta)^s_pu(x)\to -\Delta_p u(x)$ as $s\to1$. A similar pointwise asymptotic from the fractional $p$-Laplacian to the $p$-Laplacian $ -\Delta_p$ can be found in \cite{BS22,dTL21,IN10}. We refer to Section \ref{sec:from-nonlocal-local} for more details. In our strategy of proving the convergence of solutions we need, cf. Section \ref{sec:robust-poincare},  to establish the robust Poincar\'{e} type inequalities \textit{{\`{a}} la} A. Ponce \cite{Pon04} including the situation where $\Omega$ is only bounded in one direction. 

Overall, the main results of the article can be summarized  as follows. 
\begin{enumerate}[(a)]
\item We introduce appropriate nonlocal Sobolev spaces adapted for the problem \eqref{eq:main-problem}. In this framework, we establish completeness and compact embeddings in Theorem \ref{thm:WnuOme-complete}, Theorem \ref{thm:local-compactness} and Theorem \ref{thm:embd-compactness}.
\item In Theorem \ref{thm:w-emb-on-weigh-lp} and Theorem \ref{thm:trace-nonloc-thm} we establish a nonlocal trace theorem.
\item In Theorem \ref{thm:poincare-friedrichs-ext-bis} and Theorem \ref{thm:poincare-friedrichs-finite-bis} we establish Poincar\'{e}-Friedrichs inequalities for a domain $\Omega$ bounded in one direction or $|\Omega|<\infty$. In Theorem \ref{thm:poincare-gene} and Theorem \ref{thm:poincare-gene-bis} we establish Poincar\'{e} type inequalities for a general domain $\Omega$ which is only finite union of connected components. 
\item In Theorem \ref{thm:nonlocal-Neumann-var}, Theorem \ref{thm:non-existence-Neumann}, Theorem \ref{thm:nonlocal-dirichlet-gen} and   Theorem \ref{thm:nonlocal-robin-var} we establish the well-posedness of the problem \eqref{eq:main-problem}. 

\item In Theorem \ref{thm:charac-plevy} and Theorem \ref{thm:charac-plevy-radial} provide an analytic characterization of $p$-L\'{e}vy integrability of the kernel $\nu$. Analogously, in  Theorem \ref{thm:charac-plevy-concentrate} we provide a characterization of a $p$-L\'{e}vy approximation family $(\nu_\eps)_\eps$. 
\item  In Theorem \ref{thm:asymp-plevy-to-plaplace} and  we establish the pointwise convergence of $L_\eps u(x)\to -K_{d,p}\Delta_pu(x)$,  a fortiori the convergence  $(-\Delta)^s_pu(x)\to -\Delta_pu(x)$ in Corollary \ref{cor:asymp-pfrac-to-plaplace}. 
\item In Theorem \ref{thm:robust-poincare}, Theorem \ref{thm:robust-pcare-one-direct} and Theorem \ref{thm:robust-pcare-finite} we establish robust Poincar\'{e} type inequalities; including the situation where $\Omega$ is only bounded in one direction. 
\item Finally, in Theorem \ref{thm:convergence-dirichlet} and Theorem \ref{thm:convergence-neumann} we establish convergence of weak solutions to the problem \eqref{eq:main-problem-esp} to those of \eqref{eq:main-problem-local}, $\tau=0,1$.
\end{enumerate} 
We emphasize that some of our aforementioned findings are also new for the fractional $p$-Laplacian. In general, the results we deem innovative for fractional kernels of the form $\nu(h)= |h|^{-d-sp}\mathds{1}_{B_r(0)}(h),s\in (0,1) $ for some $r\in (0,\infty]$ include the trace theorem in (b), Poincar\'{e}-Friedrichs inequalities in (a) and (g),  Poincar\'{e} type inequalities in (c) and (g) and, the well-posedness and convergence of the Neumann and/or  Robin problem in (d) and (h). 
\smallskip 

\par Let us comment on related and existing works in the literature. For an outstanding reference on basics related to the $p$-Laplace operator, we refer  the reader to the classical lecture notes \cite{Lin19}. The study of nonlocal problems driven by  $p$-L\'{e}vy kernels is becoming popular. We refer the reader to  \cite{ guy-thesis, FK22, Fog23,Fog21b} for studies of  nonlocal functions spaces generated by $p$-L\'{e}vy kernels.  For recent studies of Integro-Differential Equations (IDEs) involving the L\'{e}vy operator of type $L$ (for the case $p=2$) see \cite{FK22,FPS23, DFK22,JW19} and also\cite{ACM23} where the authors compare their interaction kernel to the fractional one; we emphasize that such comparison is not required in our setting. See also \cite{Rut18,Ru20} and \cite{RV16} for the studies of Dirichlet problems associated with  L\'{e}vy operators driven by singular L\'{e}vy measures. For problems related to general nonlocal elliptic operators of L\'{e}vy type see \cite{guy-thesis,Voi17}. We also point out \cite{BGPR23,Ru20} where the nonlinear Douglas identity for $p$-L\'{e}vy type operators are investigated. The fractional $p$-Laplacian is one of the most studied $p$-L\'{e}vy operator in the framework of IDEs. For instance, the study of Dirichlet problems can be found in \cite{FP14,DKP16, ILP16,Pal18,QX16,BP16,FI22,Vaz16}. For the study of Neumann problems see \cite{MP19,MP21}. Note that \cite{MP19} is somewhat the extension of the set-up from \cite{DROV17} to the nonlinear setting.  We emphasize that in both works \cite{DROV17,MP19}, the test spaces in the weak formulation of the Neumann problem depend on the Neumann data, which is not the case in ours. We also refer to \cite{MRT16,AMRT10,War16} for the studies of problems related to regional type operators. For regularity of solutions  associated to the fractional $p$-Laplacian and its alike see \cite{IMS16,IMS16b,BLS18,CK22,CKW22}. Several different other models of nonlocal nonlinear problems such as porous medium equations are listed in \cite{KKK20}. It is noteworthy to emphasize that our setting is sufficiently general and includes kernels with bounded support, integrable kernels and non-integrable kernels. Nonlocal problems involving regional type operators with integrable kernels are addressed for instance in  \cite{AMRT08e, AMRT10,BT10}. Note however, that one of the downsides of integrable kernels is that associated function spaces are not rich enough, as they  do not differ much  from the  underlying $L^p$-space. This contrasts with non-integrable kernels  yielding strict refinement of $L^p$-Lebesgue spaces whose additional structure are of great importance. The class of kernels with bounded support are popular  in peridynamic.  Indeed, nonlocal problems driven by integro-differential operators associated with kernels of  bounded support  appear as core models to several problems, wherein, the nonlocal boundary $\Omega_e=\Omega_\nu\setminus\Omega$ is also called volume constrain or finite horizon. We explain in Section \ref{sec:existence-weak-sol} how to deal with such types of kernels. For recent studies of nonlocal problems  aiming at the application to peridynamics see \cite{ACM23,FK22,FVV22,DTZ22}, several additional references can be found therein. Additional works  in peridynamic involving nonlocal operators of regional type are dressed in \cite{MD15,BM14}. 

\smallskip 
\par Let us comment about the convergence from nonlocal to local.  In fact, under the assumption in \eqref{eq:plevy-approx-intro}, 
the convergence in \eqref{eq:limit-goal-intro}  remains true with $\R^d$ replaced by any extension domain $\Omega\subset\R^d$; see for instance \cite{Fog23}. Interestingly, the latter convergence also holds in the sense of the gamma convergence or Mosco convergence; see for instance \cite{guy-thesis,FGKV20,Pon04-gamma}. This type of limit was originally studied for Lipschitz bounded domains in \cite{BBM01} and several generalizations have recently emerged, see for instance \cite{BMR20,DB22,DD22,Fog23,PS17,Mil05} see also  the variants \cite{NPSV18,LMP19}
along with the references contained there. It is worth mentioning that, in the framework of fractional Sobolev spaces with $p=2$, the earlier work \cite[Section 4]{AAS67} by R. Adams, N. Aronszajn, and T. Smith, which is now largely overlooked in the current literature, also addresses the problem of convergence from nonlocal to local energy forms, as subsequently studied by Bourgain-Brezis-Mironescu \cite{BBM01} and the asymptotic compactness result as studied by A. Ponce \cite{Pon04}. Of course, just as \cite{BBM01}, the work \cite{AAS67} also highlights the necessity of the rectifying factor $(1-s)^{1/p}$ in front of the fractional seminorm when addressing theses asymptotic properties; this is the cornerstone of the question addressed in the work \cite{BBM01}.
To the best of our knowledge, rigorous proofs of convergence  of weak solutions to nonlocal  problems with complement Neumann condition to the local ones appear in \cite{guy-thesis,FK22,GH22}. A heuristic proof of the convergence for fractional Laplacian $(-\Delta)^s$ was investigated in \cite{DROV17}. Note however that the convergence of nonlocal Neumann problems associated with regional type operators can be found in \cite{AMRT08e, AMRT10,DLS15}.  In contrast to the Neumann problems, there is a substantial amount of works treating the convergence from nonlocal to local of weak solutions to homogeneous Dirichlet problems associated with the fractional $p$-Laplacian $(-\Delta)^s_p$; we refer interested reader for instance to  \cite{BPS16,BS20,BO20,BO21,SV22}. The convergence  of weak solutions of Dirichlet problems for fractional  $g$-Laplacian where $g$ is an Orlicz function see \cite{BS19}. For convergence of solutions to elliptic problems, see \cite{guy-thesis,Voi17}. Last, uniform convergence of viscosity solutions to the constrained fractional $p$-Laplace operator is established in \cite{IN10}.  For more on viscosity solutions associated with the $p$-Laplacian and the fractional $p$-Laplacian see for instance the recent results  \cite{Lin16,Lin19,KKL19,dTL21} as well as the references therein. 

The rest of this article is organized as follows. We introduce some basic concepts regarding the support of $\nu$ in Section \ref{sec:basicsandnotations}. In Section \ref{sec:function-spaces} we study nonlocal Sobolev spaces, whereas Section \ref{sec:nonlocal-trace-spaces} is dedicated to the study of the  corresponding nonlocal trace spaces. Meanwhile, we give an analytic characterization of $p$-L\'{e}vy integrability in Section \ref{sec:charac-plevy}. In Section \ref{sec:compactness} we give some compact embeddings.  Nonlocal Poincar\'{e} and nonlocal Poincar\'{e}-Friedrichs inequalities are established in Section \ref{sec:poincare}, while their robust versions are provided in Section \ref{sec:robust-poincare}. Section \ref{sec:from-nonlocal-local}  and Section \ref{sec:conv-solution} focus on convergences from nonlocal to local objects; this includes convergence of forms, weak solutions and nonlinear operators. We achieve the convergence of weak solutions by implicitly establishing the Gamma convergence of related nonlocal forms to the local ones. We also characterize the family $(\nu_\eps)_\eps$ satisfying \eqref{eq:plevy-approx-intro}. In Appendix \ref{sec:appendix-estim} and Appendix \ref{sec:appendix-elment} we provide some basic results such as nonlocal Gauss-Green formula, pointwise evaluations of  $Lu$ and $\cN u$, elementary inequalities that are useful for the study of nonlinear operators under consideration such as $L$ and $-\Delta_p$. 

\vspace{-2mm}
\section{Basic concepts and notations}\label{sec:basicsandnotations}
\subsection{Notations} Throughout this article unless otherwise stated, we assume  that $\Omega\subset \R^d$ is an open set, $1<p',p<\infty$ with $p+p'=pp'$ and $\nu$ is a symmetric kernel satisfying \eqref{eq:plevy-integ-cond}. We frequently use the convex inequality $(a+b)^p\leq 2^{p-1}(a^p+b^p)$ for $a>0, b>0$. If there is no specific mention, all functions and sets are assumed to be at least Borel measurable and are understood in almost everywhere sense. Given two Banach spaces $X$ and $Y$, $Y\subset X$ we write $Y\hookrightarrow X$ to indicate that $Y$ is continuously embedded in $X$.  Given two quantities $F$ and $G$ the relation $F\asymp G$ indicates that there are positive constants $C_1$ and $C_2$ such that $C_1\leq F/G\leq C_2$.  In general, $C>0$ denotes a generic constant depending on the local and  $\eps>0$ is a small quantity tending to $0$. The Euclidean scalar product of $x=(x_1,x_2,\cdots, x_d)\in \R^d $ and $ y=(y_1,y_2,\cdots, y_d)\in \R^d$ is $x\cdot y =x_1x_1+x_2y_2+\cdots+ x_dy_d$ and denote the norm of $x$ by $|x| =\sqrt{x\cdot x}$. For $r>0$ $B_r(x)=B(x,r)=\{y\in \R^d\,:\,|y-x|<r\}$. The distance between two sets $A,B\subset\R^d$ is given by $\dist(A,B)=\inf\{ |x-y|\,:\, x\in A\, y\in B \}$. In addition for $x\in \R^d$ we denote $\dist(x,B)=\dist(\{x\},B)$. Furthermore, $|A|$ denotes the Lebesgue measure of $A$. We also abuse the notation $|\mathbb{S}^{d-1}|$ to denote $d-1$-Hausdorff surface measure of the  $d-1$-dimensional sphere $\mathbb{S}^{d-1}$.

\subsection{Support and nonlocal boundary}\label{sec:basic-suppnu}
We introduce some basic definitions in connection with the measure theoretic support of a measurable function. 
\begin{definition} The support of a measurable function $\omega:\R^d\to [0,\infty)$ is defined as
\begin{align*}
\begin{split}
\supp\omega
&= \R^d\setminus\mathcal{O}\quad\text{with}\quad \mathcal{O}=\bigcup\big\{O\,:\, \text{$O$ is open and $\omega=0$ a.e. on $O$}\big\}\\
&=\big\{x\in \R^d\,:\, \text{$\omega>0$ a.e. on any open set $O$ s.t. $x\in O$}\big\}. 
\end{split}
\end{align*}
\end{definition}

Note that $\supp\omega$ is a closed set and $ \omega=0$ a.e. $\R^d\setminus\supp\omega$. In particular, $\omega=0$ a.e.  if and only if $\supp\omega=\emptyset.$ If $\omega$ is  a continuous function then $\supp\omega= \overline{\big\{h\in \R^d\,:\,\omega(h)\neq 0\big\}}.$ 
The latter is not true in general. For instance, on the real line we put $\omega = \mathds{1}_{\mathbb{Q}}$, where $\mathbb{Q}$ is the set of rational numbers, then $\overline{\big\{h\in \R\,:\,\omega(h)\neq 0\big\}}=\R$ but $\supp\omega=\emptyset$.
\begin{definition}
We say that $\nu: \R^d\setminus\{0\}\to [0,\infty)$ has full support if  $\supp\nu=\R^d$, equivalently $\nu(h)>0$ for almost every $h\in \R^d$ that is $|\{\nu=0\}|=0$. 
\end{definition}

\begin{remark}\label{rem:zero-supp-nu}
It is very tempting to think that the set of zeros of $\nu$ in $\supp\nu$ is a null set, i.e.,  $|\supp\nu \cap \{\nu=0\}|=0$. In other words, is it true that 
$\nu>0$ a.e. on $\supp\nu$?  This is not always true, indeed consider $O\subset \R^d$ a non-empty open set whose boundary has positive measure, i.e., $|\partial O|>0$. The function $g(x)=\dist(x,\R^d\setminus O)$
is continuous. Note that $\{g\neq 0\}=O$ and thus $\supp g= \overline{O}$. However,  we have 
$|\supp g\cap \{g=0\}|=|\partial O|>0$. 
\end{remark} 

\begin{definition}[Nonlocal hull and nonlocal boundary] 
Given a measurable set $S\subset \R^d$:
\begin{enumerate}[$(i)$,wide, labelwidth=1mm, labelindent=6pt]
\item  We define the nonlocal hull  with respect to $\nu$ (or simply the $\nu$-nonlocal hull) of $S$ to be the set $S_\nu^{*}=S\cup S_\nu$ where  $S_\nu= S+ \supp\nu$. Note that if $0\in \supp\nu$ then $S\subset S_\nu$ that is $S_\nu^{*}=S_\nu= S+\supp\nu$. 
\item We define the exterior boundary (or the nonlocal boundary) of $S$ with respect to $\nu$ denoted $S_e$ or  $\partial_\nu S$, to be the set $S_e= S_\nu^{*}\setminus S=S_\nu\setminus S= (S+ \supp\nu) \setminus S$ (or $\partial_\nu S=S_\nu\setminus S$). 
\end{enumerate}
\end{definition}
The terminologies are justified by the following  facts.  
From an analysis point of view, the set $S_\nu= S+ \supp\nu$ is smallest set among the sets $A$ such that 
\begin{align*}
\int_{\R^d}\psi(u(x)-u(y))\nu(x-y)\d y&= \int_{A}\psi(u(x)-u(y))\nu(x-y)\d y\\&=:L_Au(x)\quad \quad\text{for all $x\in S$}.  
\end{align*}
That is $S_\nu\subset A$ and we have $Lu(x)= L_{S_\nu} u(x)$ for all $x\in S$. In this respect, the nonlocal hull $S_\nu^{*}$ is  smallest set containing $S$ such that $Lu(x)= L_{S_\nu} u(x)$ for all $x\in S$. In other words, the values $Lu(x)$ for $x\in S$ solely depend on the values of $u$ in $S_\nu$.  Therefore to solve a nonlocal equation of the form $Lu= f$ in $S$, it is sufficient to prescribe exterior boundary condition on $S_\nu\setminus S$. For instance, if we consider $S=\Omega$ then the Neumann condition $\cN u=g$,  only makes if $g=0$ on $\R^d\setminus \Omega_\nu$; see Section \ref{sec:existence-weak-sol} for additional details. Indeed, for $x\in \Omega$ and $y\in \R^d\setminus \Omega_\nu$ we have $x-y\not\in \supp\nu$ so that $\nu(x-y)=0$. Hence,
\begin{align*}
\cN u(y):= \int_\Omega\psi(u(y)-u(x))\nu(x-y)\d x= 0 \qquad\text{for all $y\in \R^d\setminus \Omega_\nu$.}
\end{align*}
From a probabilistic point of view, when $0\in \supp\nu $ so that $S\subset S_\nu$, the nonlocal hull $S_\nu^{*}=S_\nu$ can be seen as the reachable area of a jump process starting from a point in $S$. In some sense, any jump of the process starting in $S$ cannot reach farther beyond the set $S_\nu$.  Therefore, if a jump process is censored (restricted) at the set $S$ then the whole universe of the process is $S_\nu$. In other words, any censored process on $\Omega$ is never aware of anything happening on $\R^d\setminus\Omega_\nu$.

\begin{example} Let us mention two classes of examples that are well-known in the literature. 
\begin{enumerate}[$(i)$,wide, labelwidth=!, labelindent=10pt]
\item If $\nu$ has full support, i.e. $\supp\nu=\R^d$ then $S_\nu= S+ \R^d=\R^d$ and $S_e= \R^d\setminus S$. That is the exterior boundary of $S$ with respect to $\nu$ coincides with its whole complement. A typical example include $\nu(h)=|h|^{-d-sp}$. Integro-Differential Equation (IDEs) associated with nonlocal operators driven  kernels $\nu$ with $\supp\nu=\R^d$ are often called complement values problems in the modern literature. 

\item If $\supp\nu=\overline{B_\delta(0)}$, $\delta>0$ then $S_\nu= S+ \overline{B_\delta(0)}=\{x\in \R^d: \dist(x, S)\leq\delta\}$, i.e., $S_\nu$ is the $\delta$-tubular thickening neighborhood of $S$. On other hand we have  $S_e= S_\nu\setminus S= \{x\in \R^d\setminus S: \dist(x, S)\leq \delta\}$. 

\noindent For instance, if $S=B_r(a)$, $r>0$ and $a\in \R^d$ then $S_\nu= B_{r+\delta}(a)$ and $S_e= B_{r+\delta}(a)\setminus \overline{B_{r}(a)}$. For a concrete examples we have $\nu(h)=\mathds{1}_{B_\delta(0)}(h)$ or $\nu(h)=|h|^{-d-sp}\mathds{1}_{B_\delta(0)}(h)$. These types of kernels often appear in the area of peridynamics, wherein the exterior boundary $S_e=S_\nu\setminus S$ is also known as the volume constraint or the finite horizon. Thus, Integro-Differential Equation (IDEs) associated with nonlocal operators driven by kernels $\nu$ with $\supp\nu=\overline{B_\delta(0)}$, are called in the area of peridynamic as  volume constrained problems or finite horizon problems. See for instance the recent works \cite{ACM23,FK22,FVV22,DTZ22} and several additional references therein.
\end{enumerate}
\end{example}

\begin{remark} Let us highlight some remarks concerning $\supp\nu$ and $S_\nu$.  
\begin{enumerate}[$(i)$,wide, labelwidth=!, labelindent=10pt]
\item  If $0\in \supp\nu$ then $S\subset S_\nu$. This property is sometimes important to make sense of many IDEs. 
\item  If $S$ and $\supp\nu$ are compact so is $S_\nu$. But if both are closed, and one is  compact  then $S_\nu$ is a closed set. 
\item If  $S$ is open then $S_\nu= S+\supp\nu = \bigcup_{h\in \supp\nu} S+h$ is also open, since each $S+h$ is open. 
\item If each $S$ and $\supp\nu$ is of positive measure then $S_\nu$ has a non-empty interior. 
\end{enumerate}
\end{remark}

\vspace{-5mm}
\section{Nonlocal function spaces} \label{sec:function-spaces}
\vspace{-1.5mm}
In this section we introduce generalized  Sobolev-Slobodeckij-like function spaces with respect to a $p$-L\'{e}vy integrable function $\nu$ and an open subset $\Omega \subset \R^d$, in particular the space $\WnuOmR$ and its nonlocal trace space $\TnuOm$. The function spaces are tailor-made for Integro-Differential Equations (IDEs) with complement  value problems.  
Recall that $\nu:\R^d\setminus \{0\} \to [0, \infty)$  is the density of a symmetric $p$-L\'{e}vy measure, that is, 
\vspace{-1mm}
\begin{align}\tag{L}\label{eq:plevy-integ-cond}
\begin{split}
\nu(h) = \nu(-h) \quad &\text{ and}
\qquad \int_{\R^d} (1\land |h|^p )\nu(h) \d h <\infty . 
\end{split}
\end{align}
In fact $\nu$ is said to be $p$-L\'{e}vy integrable, i.e. $\nu\in L^1(\R^d, 1\land|h|^p)$. 

\vspace{-1.5mm}
\subsection{Nonlocal energy forms} \label{sec:nonlocal-form} 
Recall $\psi(t)= |t|^{p-2} t$. Given a symmetric kernel $k:(\R^d\times \R^d)\setminus \diag \to [0, \infty)$ we consider the energy form $\cE^k(\cdot, \cdot)$ by 
\vspace{-1.5mm}
\begin{align*}
\cE^k(u, v)= \iil_{\R^d\R^d} \psi(u(x)-u(y))(v(x)-v(y))\,
k(x, y)\d y\d x.  
\end{align*}

We begin with the following proposition. 
\begin{proposition}{\ }\label{prop:enrgy-forms} The following assertions hold true. 
\begin{enumerate}[$(i)$]
\item If $\cE^k(u, u)$ and $\cE^k(v, v)$ are finite so is $\cE^k(u, v)$ and we have 
$$|\cE^k(u, v)|\leq \cE^k(u, u)^{1/p'}\cE^k(v, v)^{1/p}.$$
\item Let  the Banach space $W^p_k(\R^d)=\{u\in L^p(\R^d)\,:\, \cE^k(u, u)<\infty \},$ with the norm
$\|u\|^p_{W^p_k(\R^d)}= \|u\|^p_{L^p(\R^d)}+ \cE^k(u, u)$. Then the nonlinear  operator
 $Lu:=\cE^k(u,\cdot): W^p_k(\R^d)\to (W^p_k(\R^d))'$ is bounded,  with 
\begin{align*}
&|\langle Lu, v\rangle|= |\cE^k(u, v)|\leq \cE^k(u, u)^{1/p'}\cE^k(v, v)^{1/p}, \\
&\|Lu\|_{(W^p_k(\R^d))'}\leq \cE^k(u, u)^{1/p'}\leq \|u\|^{p-1}_{W^p_k(\R^d)}.
\end{align*}
\end{enumerate}
\end{proposition}

\begin{proof}
$(ii)$ follows from $(i)$ whereas, $(i)$ is an immediate consequence of the H\"{o}lder inequality.
\end{proof}

\begin{remark}
Formally, the operator $L$ is denoted as 
\begin{align*}
Lu (x):= 2\pv \int_{\R^d} \psi(u(x)-u(y)) k(x,y)\d y,  
\end{align*}
whenever this expression makes sense. For instance if $k(x,y)=\nu(x-y)$ with $\nu\in L^1(\R^d, 1\land |h|^p)$ then under additional mild conditions $Lu(x)$ can be defined for smooth functions (see Appendix \ref{sec:appendix-pointwise}).
Besides this, a curious reader easily verifies that a similar analysis can be achieved with the spaces $W_k^p(\R^d)$ and the nonlocal operator $L$ replaced with the space $W^{1,p}(\R^d)$  and the local $p$-Laplace operator $-\Delta_p$. 
\end{remark}
From now we assume that $\cE^k(u,u)<\infty$ and $\cE^k(v,v)<\infty$. We will consider special cases of the $k(x,y)$ that are of paramount interest. 
First we consider the form $\cE_\Omega(\cdot,\cdot)$ given by
\begin{align*}
\cE_\Omega(u,&v)= \iil_{\Omega\Omega} \psi(u(x)-u(y))(v(x)-v(y))\,\nu(x-y)\d y\d x\\
&= \iil_{\R^d\R^d} \psi(u(x)-u(y))(v(x)-v(y))
\min(\mathds{1}_\Omega(x),\mathds{1}_\Omega(y) )\nu(x-y)\d y\d x.  
\end{align*}

That is, $k(x,y)=\min(\mathds{1}_\Omega(x),\mathds{1}_\Omega(y) )\nu(x-y)$. Next, the Gauss-Green formula (see Appendix \ref{sec:appendix-gauss-green}) motivates us to consider  the special energy form $\cE(\cdot, \cdot)$, define by 
\begin{align*}
\cE(u,v)&= \iil_{(\Omega^c\times \Omega^c)^c} \psi(u(x)-u(y))(v(x)-v(y))\,\nu(x-y)\d y\d x\\
&= \iil_{\R^d\R^d} \psi(u(x)-u(y))(v(x)-v(y))\,\max(\mathds{1}_\Omega(x),\mathds{1}_\Omega(y) )\nu(x-y)\d y\d x.  
\end{align*}
That is, in this case we have $k(x,y)=\max(\mathds{1}_\Omega(x),\mathds{1}_\Omega(y) )\nu(x-y)$. Taking $k(x,y)=\frac12(\mathds{1}_\Omega(x),\mathds{1}_\Omega(y) )\nu(x-y)$ we obtain another related form $\cE_+(\cdot, \cdot)$, that is 
\begin{align*}
\cE_+&(u,v)= \iil_{\Omega\R^d} \psi(u(x)-u(y))(v(x)-v(y))\,\nu(x-y)\d y\d x\\
&= \frac{1}{2}\iil_{\R^d\R^d} \psi(u(x)-u(y))(v(x)-v(y))(\mathds{1}_\Omega(x)+\mathds{1}_\Omega(y) )\nu(x-y)\d y\d x.  
\end{align*}
As shown in the next result, the forms $\cE(\cdot, \cdot)$ and $\cE_+(\cdot, \cdot)$ are equivalents. 
\begin{proposition}\label{prop:equiv-energy}
We have $\cE(u,u)\leq 2\cE_+(u,u)\leq 2\cE(u,u)$ and $\cE_\Omega(u,u)\leq \cE(u,u)$.
\end{proposition}
\begin{proof} The proof is immediate  since we have  $\min(a,b)\leq \max(a,b)\leq a+b\leq 2\max(a,b)$ for all $a,b>0$.
\end{proof}

\begin{remark}\label{rem:form-representation} 
For brevity let us denote $\chi(x,y)=|u(x)-u(y)|^{p-2}(u(x)-u(y))(v(x)-v(y))$.  Recall that $\Omega_e=\Omega_\nu\setminus\Omega$ and  $\Omega_\nu= \Omega+\supp\nu$. The following identities hold

\begin{align*}
\cE(u,v) &= \iil_{(\Omega^c\times \Omega^c)^c} \chi(x,y)\nu(x-y)\d y\,\d x
= \Big(\iil_{\Omega\Omega} 
+2\iil_{\Omega^c \Omega} \Big) \chi(x,y)\nu(x-y)\d y\,\d x\\
&= \Big(\iil_{\Omega\Omega} + 
2 \iil_{\Omega_e \Omega}\Big)\chi(x,y)\nu(x-y)\d y\,\d x
=\hspace*{-2ex}\iil_{\mathcal{D}(\Omega_\nu,\Omega_e)} \hspace*{-2ex}\chi(x,y)\nu(x-y)\d y\d x. 
\end{align*}
where $\mathcal{D}(\Omega_\nu,\Omega_e)=(\Omega_\nu\times \Omega_\nu)\setminus (\Omega_e\times \Omega_e)$. Indeed, since $y\in \Omega$ and $x\in \R^d\setminus\Omega$ necessarily $x\in \Omega_\nu \setminus \Omega$ otherwise $\nu(x-y)=0$.  Analogously we have 
\begin{align*}
\cE_+(u,v) &= \iil_{\Omega\R^d} \chi(x,y)\nu(x-y)\d y\,\d x
= \iil_{\Omega\Omega_\nu} \chi(x,y)\nu(x-y)\d y\,\d x\\
&=  \iil_{\Omega\R^d} \chi(x,x+h)\nu(h)\d h\,\d x=\iil_{\Omega\,\supp\nu} \chi(x,x+h)\nu(h)\d h\,\d x.  
\end{align*}
\end{remark}

\subsection{Nonlocal Sobolev-Slobodeckij-like spaces}
We define the space $\WnuOm$ by
\begin{align*}
\WnuOm= \Big\{u \in L^p(\Omega)\,:\, |u|^p_{\WnuOm}:=\cE_\Omega(u,u) <\infty \Big \},
\end{align*}
equipped with the norm 
\begin{align*}
\|u\|_{\WnuOm} = \big(\|u\|^p_{L^p(\Omega)}+ |u|^p_{\WnuOm}\big)^{1/p}=
\big( \|u\|^p_{L^p(\Omega)}+ \cE_\Omega(u,u)\big) ^{1/p}.
\end{align*}

\noindent For the fractional kernel $\nu(h)=|h|^{-d-sp}$, $s\in (0,1)$, the space $\WnuOm$ corresponds to the classical fractional Sobolev-Slobodeckij space $W^{s,p}(\Omega)$.  
Whereas, if  $\nu\in L^1(\R^d)$ then we find that $\WnuOm=L^p(\Omega)$. Indeed, 
\begin{align*}
\iil_{\Omega\Omega} \big|u(x)-u(y) \big|^p \, \nuxminy\d x\,\d y
&\leq 2^p \iil_{\Omega\R^d}  |u(x)|^p\, \nuxminy\d x\,\d y\\
&= 2^p\|\nu\|_{L^1(\R^d)}\|u\|^p_{L^p(\Omega)}. 
\end{align*}
We point out that a Gagliardo-Nirenberg-Sobolev type inequality for the nonlocal Sobolev space $\Wnu$ is established in  \cite{Fog21b}. For various choices of $\nu$ it appears (see Section \ref{sec:poincare}) that,  $\cE(u, u)<\infty$ implies  $u\in L^p(\Omega)$. 
Motivated by this remark and following \cite{SV13,FKV15, guy-thesis} we introduce the crucial space $ \WnuOmR$,

\begin{align}\label{eq:def-WnuOmR}
\begin{split}
\WnuOmR &= \Big\lbrace u: \R^d \to \R \text{ meas.} \,: \, u|_\Omega\in L^p(\Omega)\,\!\text{ and }\,\! |u|^p_{\WnuOmR}<\infty\Big\rbrace \,\\
&= \Big\lbrace u: \R^d \to \R \text{ meas.} \,: \, u|_\Omega\in L^p(\Omega)\,\!\text{ and }\,\!  \cE(u,u) <\infty \Big\rbrace, 
\end{split}
\end{align}
where we put $|u|^p_{\WnuOmR}:= \cE_+(u,u)$. The equality in \eqref{eq:def-WnuOmR} follows from Proposition \ref{prop:equiv-energy}. The space $\WnuOmR$ is a seminormed space  equipped with the seminorm defined by 
\begin{align*}
\|u\|_{\WnuOmR} = \big(\|u\|^p_{L^p(\Omega)}+ |u|^p_{\WnuOmR}\big)^{1/p}\asymp 
\big( \|u\|^p_{L^p(\Omega)}+ \cE(u,u)\big) ^{1/p}.
\end{align*}
Last, we define $	\WnuOmO$ as the space of functions that vanish on the complement of $\Omega$, 
\begin{align}\label{eq:def-WnuOm-vanish}
\begin{split}
\WnuOmO&= \{ u\in \WnuOmR~: ~u=0~~\text{a.e. on } \R^d\setminus \Omega\}\,\\
&= \{ u\in \Wnu~:~u=0~~\text{a.e. on } \R^d\setminus \Omega\}\,. 
\end{split}
\end{align} 
Naturally, $\WnuOmO$ is endowed with the $\|\cdot \|_{\WnuOmR} $ which is equivalent therein with the norm
$\|\cdot \|_{\Wnu}$, viz.,  for $u\in \WnuOmO$ we have 
\begin{align*}
\|u\|_{\Wnu}=\big( \|u\|^p_{L^p(\Omega)}+ \cE(u,u)\big) ^{1/p}\asymp \|u\|_{\WnuOmR}. 
\end{align*}

It is straightforward to  check that $\WnuOmO$ is a closed subspace of $\Wnu$ and $\WnuOmR$. Another closely related closed subspace of $\WnuOmR$ is defined as 
\begin{align*}
\WnuOmRO= 
\overline{C_c^\infty(\Omega)}^{\WnuOmR}, {\text{ i.e., the closure of $C_c^\infty(\Omega)$ in $\WnuOmR$}}.
\end{align*} 
It is worth emphasizing that if $\Omega$ has a continuous boundary then  $\WnuOmRO$ coincides with $\WnuOmO$ (see  \cite[Theorem 3.76]{guy-thesis} and also Theorem \ref{thm:density}). However, this equality does not hold true in general. For $\nu(h)=|h|^{-d-sp}, s\in (0,1)$ we denote $\WnuOmR=W^{s,p}(\Omega|\R^d)$, $\WnuOmO=W^{s,p}_\Omega(\Omega|\R^d)$ and $\WnuOmRO=W^{s,p}_0(\Omega|\R^d)$. Furthermore, if $\nu\in L^1(\R^d) $ then $\WnuOmR= L^p(\R^d, \mu)$ with $\mu(x)= \mathds{1}_{\Omega}(x)+\mathds{1}_{\Omega^c}(x)\int_{\Omega}\nu(x-y)\d y$. Indeed if $u\in  L^p(\R^d, \mu)$ then $u\in L^p(\Omega)$ and  we have 
\begin{align*}
\cE(u,u)
&= \iil_{ \Omega \Omega} |u(x)-u(y)|^p\nu(x-y)\d y\d x+
2\iil_{ \Omega \Omega^c} |u(x)-u(y)|^p\nu(x-y)\d y\d x \\
&\leq 2^{p+1}\|u\|^p_{L^p(\Omega)}\|\nu\|_{L^1(\R^d)}+  2^{p}\int_{\Omega^c} |u(y)|^p\int_{\Omega}\nu(x-y)\d x\d y\\&\leq 2^{p+1}(1+\|\nu\|_{L^1(\R^d)})\|u\|^p_{L^p(\R^d, \mu)}. 
\end{align*}
Conversely, if $u\in \WnuOmR$ then $u\in L^p(\Omega)=\WnuOm$ and  we have 
\begin{align*}
\int_{\Omega^c}& |u(y)|^p\mu(y)\d y
\leq  2^{p-1}\iil_{\Omega\R^d} (|u(y)-u(x)|^p+ |u(x)|^p)\nu(x-y)\d y\d x\\
&=2^{p-1}\|u\|^p_{L^p(\Omega)}\|\nu\|_{L^1(\R^d)}+ 2^{p-1} \iil_{\Omega\R^d} |u(y)-u(x)|^p\nu(x-y)\d y\d x. 
\end{align*}

\begin{remark} As a direct consequence of the Proposition \ref{prop:enrgy-forms} we have the following.  
\begin{enumerate}[$(i)$]
\item If $u,v\in \WnuOmR$ then $|\cE(u,v)|\leq \cE(u,u)^{1/p'}\cE(v,v)^{1/p}$ and  $\cE(u,\cdot)\in(\WnuOmR)'$. 
\item If $u,v\in \WnuOm$ then $|\cE_\Omega (u,v)|\leq \cE_\Omega(u,u)^{1/p'}\cE_\Omega(v,v)^{1/p}$ and $\cE_\Omega(u,\cdot)\in(\WnuOm)'$.
\end{enumerate}
\end{remark}

A simple proof of the following Theorem can be found in \cite[Theorem 3.46]{guy-thesis}. 
\begin{theorem}The spaces
$\WnuOm$ and $\WnuOmO$ are separable reflexive Banach spaces. 
\end{theorem}

\begin{remark}
It is worthwhile noticing that $  \|\cdot\|_{\WnuOmR}$ is always a norm on  $\WnuOmO$, but not in general a norm on $\WnuOmR$ if $\nu$ is not fully supported. A simple counterexample is given by $\nu(h)=\mathds{1}_{B_1(0)} (h)$ and $\Omega=B_1(0)$. For the function $u(x)= \mathds{1}_{B^c_2(0)}(x)$ we have  $\|u\|_{\WnuOmR}=0$ whereas $u\neq 0$. More generally, assume that $\Omega$ is bounded and $\nu$ has a compact support. Let $S= \R^d\setminus \big(\Omega\cup  \Omega_\nu\big)$ and consider the function $u(x)= \mathds{1}_S(x)$. A routine verification shows that $\|u\|_{\WnuOmR}=0$ but $u\neq0$. This means that $(\WnuOmR, \|\cdot\|_{\WnuOmR})$ cannot be a normed space. 
\end{remark}

However $(\WnuOmR, \|\cdot\|_{\WnuOmR})$ is a Banach space when $\nu$ is of full support; cf. \cite{guy-thesis,FK22} for the proof.  Let us say a few words for the general case where $\nu$ is not necessarily fully supported. In view of the Remark \ref{rem:form-representation} we introduce another version of the space $\WnuOmR$. Namely, recall $\Omega_\nu= \supp\nu+\Omega$ and  $\Omega_\nu^{*}=\Omega\cup \Omega_\nu$ we consider the space 
\begin{align*}
\WnuOme 
&= \big\{u:\Omega_\nu^{*} \to \R \,\text{meas.} :  u|_\Omega\in L^p(\Omega),\, |u|^p_{\WnuOme }<\infty \big\},
\end{align*}
equipped with the seminorm $\|u\|_{W^p_\nu(\Omega|\Omega_\nu)} = \big( \|u\|^p_{L^p(\Omega)}+ |u|^p_{\WnuOme }\big)^{1/p}$ where 
we define 
\begin{align*}
|u|^p_{\WnuOme }= \iil_{\Omega \Omega_\nu} \hspace{-1ex}|u(x)-u(y)|^p  \nuxminy \d y  \d x.
\end{align*}
It is worth recalling that if $0\in\supp\nu$ then $\Omega\subset  \Omega_\nu$ so that $\Omega_\nu^{*}=\Omega\cup \Omega_\nu=\Omega_\nu$. 

\begin{remark}\label{rem:identification}
We have $W^p_\nu(\Omega|\Omega_\nu)\equiv \WnuOmR$ in the sense that both spaces are isometric isomorphic. 
Clearly if $u\in \WnuOmR$ then $u_\nu\in W^p_\nu(\Omega|\Omega_\nu)$ with $u_\nu=u|_{\Omega_\nu^{*}}$ and we have $\|u_\nu\|_{W^p_\nu(\Omega|\Omega_\nu)}=\|u\|_{\WnuOmR}$. Conversely, if $u\in W^p_\nu(\Omega|\Omega_\nu)$ then $\widetilde{u}\in \WnuOmR$ where $\widetilde{u}$ is the zero extension of $u$ off $\Omega_\nu^{*}$ and $\|u\|_{W^p_\nu(\Omega|\Omega_\nu)}=\|\widetilde{u}\|_{\WnuOmR}$. 
Note that $u\in \WnuOmR$ is the null function  if and only if $u=0$ a.e. on $\Omega_\nu^{*}$. The notation $W^p(\Omega|\Omega_\nu )$ instead of $\WnuOmR$ should be 	also appropriate. However, we deliberately keep the  notation $\WnuOmR$ for simplicity.  It is counterintuitive to think that $\|\cdot\|_{\WnuOme}$ always defines a norm on $\WnuOme$. In view of the Remark \ref{rem:zero-supp-nu}, it is legitimate to assume that $|\supp\nu\cap \{\nu=0\}|=0$. Nevertheless, the latter seems not enough to show that $\|\cdot\|_{\WnuOme}$ is a norm.  More precisely, one may consider the following question. 

\medskip 

\textbf{Open Question:} Assume that $\nu(h)>0$ for all $h\in \supp\nu$.  Prove or disprove that is $\|\cdot\|_{\WnuOme}$ defines a norm on $\WnuOme$. 
\end{remark}

Next, we provide a condition under  which we can show that $\WnuOme$ is a Banach space. 
\begin{theorem}\label{thm:WnuOme-complete}
The space $\WnuOmR\equiv \WnuOme$ is a separable and reflexive Banach space if we have
\begin{align}\label{eq:positive-Oe}\tag{$S_1$}
&\omega(x):= \int_\Omega 1\land \nu(x-y)\d y>0\quad \text{for almost all $x\in \Omega_\nu\setminus\Omega$}. 
\end{align}
This claim is true  in the particular case where $\nu$ is of full support. 
\end{theorem}

\begin{proof}
Observe that if $x\not \in \Omega_\nu$ then for all $y\in\Omega$ we have $x-y\not\in \supp\nu$. It follows that $\omega(x)=0$ on $\R^d\setminus\Omega_\nu$. By Theorem \ref{thm:w-emb-on-weigh-lp} we have $\WnuOme\equiv \WnuOmR\hookrightarrow L^p(\R^d,\omega)
\equiv  L^p(\Omega_\nu, \omega).$ Whence, there is $C>0$ such that 
\begin{align}\label{eq:weighted-estx}
\|u\|_{L^p(\Omega_\nu, \omega)}\leq C\|u\|_{\WnuOme}\quad\text{for all $u\in \WnuOme$}. 
\end{align}
Assume that $\eqref{eq:positive-Oe}$ holds, if $\|u\|_{\WnuOme}=0$ then $u=0$ a.e. on $\Omega$ and by \eqref{eq:weighted-estx},  $u\cdot\omega=0$ a.e. on $\Omega_\nu\setminus\Omega$. There holds that $u=0$ a.e. $\Omega_\nu$, since $\omega>0$ a.e. on $\Omega_\nu\setminus\Omega$. Whence the seminorm $\|\cdot\|_{\WnuOme}$ is a norm on $\WnuOme$.
Now, consider $(u_n)_n$ to be a Cauchy sequence in $\WnuOme$, then taking into account the estimate \eqref{eq:weighted-estx}, the sequences $(u_n)_n$ and $(w_n)_n$ with $w_n=u_n\cdot \omega$ are Cauchy sequences in $L^p(\Omega)$ and $L^p(\Omega_\nu)$ respectively. Thus, by passing to a subsequence if necessary, we can assume there exist $u\in L^p(\Omega)$ and $w\in L^p(\Omega_\nu)$ such that 
\begin{enumerate}[$\bullet$]
\item $(u_n)$ converges in $L^p(\Omega)$ and pointwise  a.e.  on $\Omega$ to some $u\in L^p(\Omega)$,  
\item $(w_n)$ converges in $L^p(\Omega_\nu)$ and pointwise  a.e.  in $\Omega_\nu$ to some $w\in L^p(\Omega_\nu)$. 
\end{enumerate}
We extend $u$ for  $x\in \Omega_\nu\setminus\Omega$ by $u(x)= w(x)/\omega(x)$ which is well defined since $\omega>0$ a.e. on $\Omega_\nu\setminus\Omega$. It follows that 
$u:\Omega^*_\nu= \Omega\cup\Omega_\nu\to \R$ is measurable since $u_n\to u$ a.e. on $\Omega$ and for almost all $x\in \Omega_\nu\setminus\Omega$ we have $u_n(x)=w_n(x)/\omega(x)\xrightarrow{n\to\infty}w(x)/\omega(x)= u(x).$ 
It turns out that $(u_n)_n$ is a Cauchy sequence in $\WnuOme$ converging to $u$ on $L^p(\Omega)$ and  a.e. on $\Omega_\nu$. This combined with Fatou's lemma implies that 
\begin{align*}
|u_n-u|^p_{\WnuOme}\leq \liminf_{m\to \infty}|u_n-u_m|^p_{\WnuOme}
\xrightarrow{n\to \infty}0. 
\end{align*} 
We  conclude that $\|u_n-u\|_{\WnuOme}\xrightarrow{n\to\infty}0$ and $u\in \WnuOme$ which show the completeness of $\WnuOme$.
Now, consider the  isometry  $\mathcal{I}: \WnuOme\to L^p(\Omega) \times L^p(\Omega\times \Omega_\nu)$ with $\mathcal{I}(u)= (u(x), (u(x)-u(y))\nu^{1/p}(x-y))$. From its Banach structure, the space  $\big(\WnuOme, \|\cdot\|_{\WnuOme}\big)$, identified with $\mathcal{I}\big(\WnuOme\big)$, is  separable and reflexive as a closed subspace of the separable and reflexive space  $ L^p(\Omega) \times L^p(\Omega\times \Omega_\nu)$.
\end{proof}

\begin{example}
A typical situation where $\WnuOme$ is a Banach space although $\nu>0$ a.e. only on  $\supp\nu$, 
can be obtained as follows. Let $\Omega= B_1(0)$ and consider $\nu(h)= \mathds{1}_{B_1(0)}(h)$ or $\nu(h)= \mathds{1}_{B_1(0)}(h)|h|^{-d-sp}$, $s\in (0,1)$ so that $\supp\nu= \overline{B_1(0)}$. Thus, $\WnuOme= W^p_\nu(B_1(0)|B_2(0))$ is a Banach space with $\Omega_\nu =B_1(0)+\overline{B_1(0)}= B_2(0)$ . 
\end{example}

\subsection{Density of smooth functions} 
For many results, it is crucial that smooth functions with compact support are dense in the function space under consideration. Let us summarize some important results in this direction.
\begin{theorem}\label{thm:density} Let $\nu$ satisfies \eqref{eq:plevy-integ-cond} with full support and let $\Omega \subset \R^d$ be open.
\begin{enumerate}[$(i)$]
\item  $C^\infty(\Omega) \cap\WnuOm$ is dense in $\WnuOm$. 
\item  If  the boundary  $\partial \Omega$ is  continuous then $C_c^\infty(\overline{\Omega})$ is dense in $\WnuOm$. 
\item   If  the boundary  $\partial \Omega$ is  continuous,  then $C_c^\infty(\Omega)$ is dense in $\WnuOmO$.
\item   If  the boundary  $\partial \Omega$ is  Lipschitz, then $C_c^\infty(\R^d)$ is dense in $\WnuOmR$. 
\item $C_c^\infty(\R^d)$ is dense in $\Wnu$. 
\end{enumerate}
\end{theorem}
The proofs of the first and second statements can be found in \cite{guy-thesis} and \cite{DyKi21}. The first statement reminisces of the Meyer-Serrin density type result (see \cite[Theorem 3.67]{guy-thesis}). 
Note that $C_c^\infty(\overline{\Omega})$ consists of functions of the form $v|_{\overline{\Omega}}$ where $v \in C^\infty_c(\R^d)$. The proof of the third statement is given in \cite{FSV15} for a special choice of $\nu$ and in \cite{guy-thesis}, \cite{BGPR20} for the general case. Note however that the third may fail if $\partial\Omega$ is not continuous see for instance  \cite[Remark 7]{FSV15}.  This goes against the claim \cite[Theorem 3.3.9]{CF12}.
The proof of the fourth and the fifth assertions are given in \cite{FGKV20, guy-thesis}.

\subsection{Connection with classical Sobolev spaces}
Let us delve into the connection between nonlocal Sobolev spaces and classical Sobolev spaces by examining the embeddings and related results. Recall that $W^{1,p}(\Omega)$ denotes the classical Sobolev space endowed with the norm 
\[ \|u\|^p_{W^{1,p}(\Omega)} = \|u\|^p_{L^p(\Omega)} +\|\nabla u\|^p_{L^p(\Omega)} \,.\]
While, $W^{1,p}_0(\Omega)$ is the closure of $C_c^\infty(\Omega)$ with respect to the $W^{1,p}(\Omega)$. We emphasize that $W^{1,p}_0(\Omega)$ also coincides with the closure of $C_c^\infty(\Omega)$ in $W^{1,p}(\R^d)$. 
Let $ W^p_{\nu,0}(\Omega) $ be the closure of $C_c^\infty(\Omega)$ in  $\WnuOm$. Note that, the zero extension to $\R^d$ of any function in $ W^{1,p}_0(\Omega)$ belongs to $W^{1,p}(\R^d)$. Recall that  $\Omega$ is an $W^{1,p}$-extension domain if there is an operator $E:W^{1,p}(\Omega)\to W^{1,p}(\R^d)$ and $C>0$ such that $Eu|_\Omega= u$  and $\|Eu\|_{W^{1,p}(\R^d)} \leq C \|u\|_{W^{1,p}(\Omega)}$ for all  $u\in W^{1,p}(\Omega)$. 

\begin{proposition} Let $\Omega \subset \R^d$ be open. The following embeddings hold true. 
\begin{enumerate}[$(i)$]
\item $W^{1,p}(\R^d) \hookrightarrow W^p_{\nu}(\R^d).$
\item If $\Omega\subset \R^d$ is an $W^{1,p}$-extension domain then $W^{1,p}(\Omega)\hookrightarrow \WnuOm$.
\item$W^p_{\nu}(\R^d) \hookrightarrow \WnuOmR \hookrightarrow \WnuOm \hookrightarrow L^p(\Omega)\,.$
\item $W^{1,p}_0(\Omega)\hookrightarrow \WnuOmO \hookrightarrow \WnuOm.$
\item If $\partial\Omega$ continuous then $\WnuOmO\hookrightarrow W^p_{\nu, 0}(\Omega) \hookrightarrow L^p(\Omega)$. 
\end{enumerate}
\end{proposition}
\begin{proof}
Note that $(ii)$ is implied by $(i)$. Whereas $(v)$ follows from the fact that $C_c^\infty(\Omega)$ is dense in $\WnuOmO$; see Theorem \ref{thm:density}. We only prove $(i)$, i.e., $W^{1,p}(\R^d) \hookrightarrow W^p_{\nu}(\R^d)$  since the remaining ones are trivial. A routine argument yields that, for $u \in W^{1,p}(\R^d)$ and $h\in \R^d$ 
\begin{align*}
\int_{\R^d} |u(x+ h)-u(x)|^p\d x
\leq |h|^p\| \nabla u \|^p_{L^p(\R^d)}, 
\end{align*} 
whereas $\|u(\cdot+ h)-u\|^p_{L^p(\R^d)} \leq 2^p\|u\|^p_{L^p(\R^d)} $. Therefore we  get 
\begin{align*}
\int_{\R^d} |u(x+ h)-u(x)|^p\d x
\leq 2^p(1\land |h|^p)\| u \|^p_{W^{1,p}(\R^d)}\,.
\end{align*} 
Integrating both side over $\R^d$ with respect to the measure $\nu(h)\d h$ yields 
\begin{align}\label{eq:levy-p-estimate}
\iint\limits_{\R^d\R^d} |u(x)-u(y)|^p \nuxminy\d y\,\d x \leq
2^p \| \nu\|_{L^1(\R^d, 1\land |h|^p\d h)} \| u \|^p_{W^{1,p}(\R^d)}.
\end{align} 
The desired embedding readily follows.   
\end{proof}

\begin{remark}
The embedding $W^{1,p}(\Omega)\hookrightarrow \WnuOm$ may fail  if $\Omega$ is not an extension domain 
(see \cite[Counterexample 3.10]{Fog23} or 
\cite[Example 9.1]{Hitchhiker}). 
More importantly, $W^{1,p}(\Omega)$ can be viewed as limiting space of the nonlocal space of type $\WnuOm$ and $\WnuOmR$; see for instance  \cite{Fog23,guy-thesis}.
\end{remark}

\section{Characterization of the p-L\'{e}vy integrability}\label{sec:charac-plevy}
We will now see that  the  $p$-L\'{e}vy integrability condition \eqref{eq:plevy-integ-cond} is optimal and can be self-generated from the associated energy form. In fact the  $p$-L\'{e}vy integrability condition \eqref{eq:plevy-integ-cond} renders the space $\Wnu$ more consistent, in a sense that it warrants the space $\Wnu$ to contain smooth functions.  

\begin{theorem}\label{thm:charac-plevy}
Let $\nu\geq0$ be symmetric. The following assertions are equivalent. 
\begin{enumerate}[$(i)$]
\item The $p$-L\'{e}vy condition \eqref{eq:plevy-integ-cond} holds, i.e. $\nu\in L^1(\R^d,1\land|h|^p)$.
\item The embedding $W^{1,p}(\R^d)\hookrightarrow\Wnu$ is continuous. 
\item $\cE_{\R^d}(u,u)<\infty$ for all $u\in W^{1,p}(\R^d)$. 
\item $\cE_{\R^d}(u,u)<\infty$ for all $u\in C_c^\infty(\R^d)$. 
\end{enumerate}
Moreover, this remains true when $p=1$ with $BV(\R^d)$ in place of $W^{1,1}(\R^d)$.
\end{theorem}
This characterization in Theorem \ref{thm:charac-plevy} highlights the robustness of the $p$-L\'{e}vy integrability condition across different function spaces, and  ensuring the finiteness of the associated energy forms.

\begin{proof} $(i)\implies (ii)$ The estimate \eqref{eq:levy-p-estimate} readily yields  the continuity of the embedding  $W^{1,p}(\R^d)\hookrightarrow\Wnu$. The implications $(ii)\implies (iii)$ and $(iii)\implies (iv)$ are straightforward. It remains to prove that $(iv)\implies (i)$. To this end, assume that $\cE_{\R^d}(u,u)<\infty$ for all $u\in C_c^\infty(\R^d)$. Let $\zeta\in C_c^\infty(\R)$ with $\zeta\not\equiv 0$ and  $\supp \zeta\subset (0,1)$ and $\vartheta\in C_c^\infty(\R^{d-1})$ such that $\vartheta=1$ on $Q'$ with $Q'=[0,1]^{d-1}$. Consider $u_i (x)=\zeta\otimes \vartheta(x_i ,x'_i)= \zeta(x_i) \vartheta(x'_i)$ where $x'_i= (x_1,\cdots x_{i-1}, x_{i+1}\cdots, x_d)$. Clearly we have   $u_i\in C_c^\infty(\R^d)$ and $\nabla u_i (x)= \zeta'(x_i)e_i$ for $x\in Q= [0,1]^{d}$. By the continuity of the shift, for each $\varepsilon>0,$ there is $\delta>0$ such that 
\begin{align*}
\text{$\|\nabla u(\cdot+h)-\nabla u\|^p_{L^p(Q)}< \varepsilon$\quad  for all $|h|\leq \delta.$ }
\end{align*}
Using this and the inequality $b^p\geq 2^{1-p}(a+b)^p-a^p$, $a,b\geq0$ we find that 
\begin{align*}
\begin{split}
\cE_{\R^d}(u_i,u_i)&\geq \int_{Q} \int_{B_\delta(0)} \Big| \int_0^1\nabla u_i(x+th)\cdot h\d t\Big|^p \nu(h) \d h\,\d x\\
&\geq 2^{1-p}\int_{Q} \int_{B_\delta(0)} |\nabla u_i(x)\cdot h|^p\nu(h) \d h\,\d x- \varepsilon \int_{B_\delta(0)} |h|^p\nu(h)\d h\\
&= 2^{1-p}\int_{Q}  |\zeta'(x_i)|^p\d x \int_{B_\delta(0)}|h_i|^p \nu(h) \d h- \varepsilon \int_{B_\delta(0)} |h|^p\nu(h)\d h.
\end{split}
\end{align*}
Note  that $\int_{Q}  |\zeta'(x_i)|^p\d x = \|\zeta'\|^p_{L^p(0,1)}
$.  We have  $|h_1|^p+\cdots+ |h_d|^p\geq c_{d,p}|h|^p$ for some $c_{d,p}>0$, by equivalence of Euclidean norms. Taking $\eps=\frac{C}{2d}$ with $C= 2^{1-p}c_{d,p} \|\zeta'\|^p_{L^p(0,1)}$ we get 
\begin{align}\label{eq:xaround-zero}
\sum_{i=1}^d \cE_{\R^d}(u_i,u_i)
&\geq 
(C- d\varepsilon) \int_{B_\delta(0)} |h|^p\nu(h)\d h= 
\frac{C}{2} \int_{B_\delta(0)} |h|^p\nu(h)\d h . 
\end{align}
Now, consider $u\in C_c^\infty( B_{\delta/2}(0)), $ $u\not\equiv 0$.  Since  $B_{\delta/2}(0)\subset B_\delta(x)$ for all $x\in B_{\delta/2}(0)$ we get 
\begin{align}\label{eq:xtail}
\begin{split}
\cE_{\R^d}(u,u)
&\geq 2\int_{B_{\delta/2}(0)} |u(x)|^p \int_{B^c_{\delta/2}(0)} \hspace{-2ex}\nu(x-y)\d y\,\d x\\
&\geq 2\|u\|^p_{L^p( B_{\delta/2}(0))} \int_{ B^c_\delta(0)} \nu(h)\d h.
\end{split}
\end{align} 
Combing \eqref{eq:xaround-zero} and \eqref{eq:xtail} yields $\nu\in L^1(\R^d, 1\land|h|^p)$  that is $\nu$ satisfies \eqref{eq:plevy-integ-cond}. The case $p=1$ follows analogously.
\end{proof}
The next result underscores that for radial $\nu$, the condition $\Wnu \neq \{0\}$ is both necessary and sufficient for the $p$-L\'{e}vy integrability condition to hold. We first need the following lemma. 
\begin{lemma}\label{lem:convolu-convex}
Let $\phi\in L^1(\R^d)$  such that $\phi\geq0$ and $\int_{\R^d}\phi(z)\d z=1$. Then for all $u\in L^1_{\loc}(\R^d)$, we have 
\begin{align*}
\cE_{\R^d} (u*\phi,u*\phi)\leq \cE_{\R^d} (u,u). 
\end{align*}
\end{lemma}
\begin{proof}
From Jensen's inequality and simple change of variables we find that
\begin{alignat*}{2}
\cE_{\R^d} (u,u)
&=\int_{\R^d} \phi(z) \iil_{ \R^d\R^d}|u(x-z)-u(y-z  )|^p\nu(x-y)\d y\d x\, \d z\\
&\geq \iil_{ \R^d\R^d}\Big|\int_{\R^d} \phi(z)(u(x-z)-u(y-z))\d z\Big|^p \nu(x-y)\d y\d x\\
&= 	\iil_{ \R^d\R^d}|u*\phi(x)-u*\phi(y)|^p\nu(x-y)\d y\d x=\cE_{\R^d} (u*\phi,u*\phi).
\end{alignat*}
\end{proof}

\begin{theorem}\label{thm:charac-plevy-radial}
If  $\nu$ is radial, then $\nu$ satisfies \eqref{eq:plevy-integ-cond} if and only if $\Wnu\neq \{0\}$.
\end{theorem}

\begin{proof}
By Theorem \ref{thm:charac-plevy} we know  $C_c^\infty(\R^d)\subset \Wnu$ if $\nu$ satisfies \eqref{eq:plevy-integ-cond}. 
Conversely, assume there is $v\in \Wnu$ such that $v\not \equiv0$. Then $|v|\in \Wnu$ with $\cE(|v|,|v|)\leq \cE(v, v)$, since $||v(x)|-|v(y)||\leq |v(x)-v(y)|$.  Let $\phi(x)= \kappa e^{-|x|^2}$ satisfies $\|\phi\|_{L^1(\R^d)}=1$ for some $\kappa>0$. By  Young's inequality and Lemma \ref{lem:convolu-convex}  we get $u=|v|*\phi\in \Wnu$. Moreover,  we have 
\begin{align*}
\cE(u,u)\leq \cE(|v|,|v|)\leq \cE(v, v)<\infty.
\end{align*}
Clearly $u\in C^\infty(\R^d)\cap L^\infty(\R^d)$ and $u(x)>0$ for each $x\in \R^d$, i.e., $\supp u=\R^d$. On the other hand, Young's inequality implies $u\in W^{1,p}(\R^d)$. In particular, $\|\nabla u\|_{L^p(\R^d)}\neq0$. 
For each $\eps>0$, there is $\delta= \delta(\eps)>0$ such that $\|\nabla u(\cdot+h)-\nabla u\|_{L^p(\R^d)}<\eps$ for all $|h|\leq \delta$. 
Minkowski's inequality implies 
\begin{align*}
\Big(\hspace{-1ex}\iil_{\R^dB_\delta(0)}\hspace{-2ex}|\nabla u(x)\cdot h|^p\nu(h)\d h\d x\Big)^{1/p} \hspace*{-2ex}
&\leq \Big(\hspace{-1ex}\iil_{\R^dB_\delta(0)}	\hspace{-1ex} \Big|\int_0^1 \hspace*{-1ex}\nabla u(x+th) \cdot h\d t\Big|^p\nu(h) \d h\d x\Big)^{1/p}\\
&+\eps \Big(\int_{B_\delta(0)}|h|^p\nu(h)\d h\Big)^{1/p}.
\end{align*} 
By the rotation invariance of the Lebesgue measure, for $z\in \R^d$ and $e\in \mathbb{S}^{d-1}$,
\begin{align}\label{eq:rotation-inv}
\fint_{\mathbb{S}^{d-1}}  |w\cdot z|^p \d\sigma_{d-1}(w) = |z|^p\fint_{\mathbb{S}^{d-1}} |w\cdot e|^p\d\sigma_{d-1}(w)  = K_{d,p}|z|^p.
\end{align} 

The fundamental theorem of calculus and a passage to polar coordinates yield
\begin{align*}
\begin{split}
\cE(u,u)^{1/p}&\geq\Big( \iil_{\R^d\, B_\delta(0)} \Big| \int_0^1\nabla u(x+th)\cdot h\d t\Big|^p \nu(h) \d h\,\d x\Big)^{1/p}\\
&\geq \Big( \int_{\R^d} \int_{\mathbb{S}^{d-1}} \hspace{-2ex} |\nabla u(x)\cdot w|^p \mathrm{d}\sigma_{d-1}(w) \int_{0}^{\delta}\hspace{-1ex} r^{p+d-1} \nu(r)\d r
\Big)^{1/p}\\&- \eps \Big( \int_{B_\delta(0)} \hspace{-2ex}|h|^p\nu(h)\d h\Big)^{1/p}\\
&= \Big(K^{1/p}_{d,p}\|\nabla u\|_{L^p(\R^d)}-\eps\Big)\Big( \int_{B_\delta(0)} \hspace{-2ex}|h|^p\nu(h)\d h\Big)^{1/p}. 
\end{split}
\end{align*}
 Since $u\in W^{1,p}(\R^d)\setminus \{0\}$, we can take  $\eps=\frac{K^{1/p}_{d,p}}{2} \|\nabla u\|_{L^p(\R^d)}\neq0$  so that  
\begin{align*}
\begin{split}
\cE(u,u)\geq \frac{K_{d,p}}{2} \|\nabla u\|^p_{L^p(\R^d)} \int_{B_\delta(0)} \hspace{-2ex}|h|^p\nu(h)\d h. 
\end{split}
\end{align*}
Necessarily, we get  $\int_{B_\delta(0)} |h|^p\nu(h)\d h<\infty$.  
Next, it is not difficult to  show that $u(x)\to 0 $ as $|x|\to \infty$. Thus for $r>0$ there is $R>\delta>0$ such that $0<u(x)\leq r$ whenever $|x|>\frac{R}{2}$. Now consider the 1-Lipschitz function 
\begin{align*}
\zeta_r(s)=\max(-r, \min(r, s))=\begin{cases}
s &|s|\leq r \\
r\operatorname{sgn}(s)&|s|\geq  r.
\end{cases} 
\end{align*}
It is readily seen that  $|\zeta_r(u(x))-\zeta_r(u(y))|\leq |u(x)-u(y)|$ and $0\leq \zeta_r(u(x))\leq u(x)$. In particular, $\cE(u,u)\geq\cE(\zeta_r(u), \zeta_r(u))$. If we set  $w_r(x)= u(x)-\zeta_r(u(x))$ then $w_r(x)= 0$ whenever $|x|>R/2$.
Since $B_{R/2}(x)\subset B_{R}(0)$ for all $x\in B_{R/2}(0)$, we find that 
\begin{align*}
\cE(u,u)\geq\cE(w_r, w_r)&\geq 2\int_{B_{R/2}(0)}|w_r(x)|^p\int_{\R^d\setminus B_{R/2}(0)}\nu(x-y)\d y \d x\\
&\geq \|w_r\|^p_{L^p(\R^d)}\int_{|h|\geq R}\nu(h)\d h. 
\end{align*}
Since $u$ is not constant we get  $\|w_r\|^p_{L^p(\R^d)}\neq 0$. Therefore we deduce that 
\begin{align*}
\int_{|h|\geq R}\nu(h)\d h<\infty. 
\end{align*}
Last, note that for $\delta\leq |h|\leq R$ and $x\in B_{\frac{\delta}{4}}(0)$ we have $|x|<\frac{\delta}{4}<\frac{\delta}{2}\leq |x+h|\leq R+ \frac{\delta}{2}$. Since $u$ is smooth with full support we put 
\begin{align*}
M=\min_{\frac{\delta}{2} \leq z\leq R+\frac{\delta}{2}}\int_{B_{\frac{\delta}{4}}(0)} | u(x)-u(z)|^p\d x >0.
\end{align*}
From this we find that 
\begin{align*}
\cE(u,u)&\geq \hspace*{-1ex}\int_{B_{R}(0)\setminus B_\delta(0)}\hspace*{-0.5ex}\int_{B_{\frac{\delta}{4}}(0)} \hspace*{-2ex}|u(x)-u(x+h)|^p\d x\, \nu(h)\d h
\geq M  \int_{B_{R}(0)\setminus B_\delta(0)}\hspace*{-4ex}\nu(h)\d h.
\end{align*}
Hence, it follows that $
\int_{B_{R}(0)\setminus B_\delta(0)}\nu(h)\d h<\infty.$
Finally, we find that $\nu\in L^1(\R^d, 1\land|h|^p)$.
\end{proof}
The following result is a straightforwards consequence of Theorem \ref{thm:charac-plevy} and Theorem \ref{thm:charac-plevy-radial}. 

\begin{theorem}\label{thm:charac-plevy-radial-bis} 
Assume $\nu\geq0$ is radial. The following assertions are equivalent. 
\begin{enumerate}[$(i)$]
\item The $p$-L\'{e}vy condition \eqref{eq:plevy-integ-cond} holds, i.e. $\nu\in L^1(\R^d,1\land|h|^p)$.
\item The embedding $W^{1,p}(\R^d)\hookrightarrow\Wnu$ is continuous. 
\item $\cE_{\R^d}(u,u)<\infty$ for all $u\in W^{1,p}(\R^d)$. 
\item $\cE_{\R^d}(u,u)<\infty$ for all $u\in C_c^\infty(\R^d)$. 
\item The space $\Wnu$ is nontrivial, i.e.,  $\Wnu\neq \{0\}$. 
\end{enumerate}
Moreover, this  remains true when $p=1$ with $BV(\R^d)$ in place of $W^{1,1}(\R^d)$.
\end{theorem}

As illustrated in the next result, the L\'{e}vy integrability \eqref{eq:plevy-integ-cond}  draws the  borderline for which a space of type $\Wnu$ is trivial or not, see \cite[Proposition 2.15]{Fog23} or \cite[Proposition 3.46]{guy-thesis} for a general setting. 

\begin{proposition}\label{prop:boderline-plevy} Let  $\nu\geq0$ be symmetric and put $C_\delta=\int_{B_\delta(0)} |h|^p\nu(h)\d h$, $\delta>0$. The following hold true. 
\begin{enumerate}[$(i)$]
\item If $\nu\in L^1(\R^d)$ then $\Wnu= L^p(\R^d)$ with equivalence in norms.
\item If $\nu \in L^1(\R^d, 1\land |h|^p\d h)$ and $\Omega$ is bounded, then bounded Lipschitz functions are in  $\WnuOm$ and $\WnuOmR$.  
\item 
If  $\nu$ is radial, $\Omega$ is connected and  $C_\delta=\infty$ for all $\delta>0$, in particular  $\nu \not\in L^1(\R^d, 1\land |h|^p\d h)$, then $u\in W^{1,p}(\Omega)\cap \WnuOm$ or  $u\in C^1(\Omega)\cap \WnuOm$ if and only if  $u=c$ is a constant function. 
\item If  $\nu$ is radial and $u\in W^{1,p} (\R^d)$ then there is $\delta=\delta(u)>0$, such that 
\begin{align*}
\frac{K_{d,p}}{2} C_\delta\|\nabla u\|^p_{L^p (\R^d)} \leq |u|^p_{\Wnu} \leq 2^p\|\nu\|_{L^1(\R^d, 1\land |h|^p\d h)}\|u\|^p_{W^{1,p} (\R^d)}.
\end{align*}
\end{enumerate}
\end{proposition}

If $\nu$ is radial, then for any $u\in W^{1,p} (\R^d)$ there is $\delta=\delta(u)>0$, such that 
\begin{align*}
	\frac{K_{d,p}}{2} C_\delta\|\nabla u\|^p_{L^p (\R^d)} \leq |u|^p_{\Wnu} \leq 2^p\|\nu\|_{L^1(\R^d, 1\land |h|^p\d h)}\|u\|^p_{W^{1,p} (\R^d)}.
\end{align*}

\section{Nonlocal Trace Theorem}\label{sec:nonlocal-trace-spaces}
\subsection{Weighted $L^p$-spaces} 
In order to set up IDEs in $L^p$-spaces over $\R^d$, we introduce Borel measures on $\R^d$ that capture the behavior of $\nu$ at infinity. We will establish the  embedding of $\WnuOmR$ into  weighted spaces $L^p(\R^d, \omega)$ for different measures $\omega$. We will need the notion of unimodality.

\begin{definition}\label{def:unimodality}
$\nu$ is called \emph{unimodal} if $\nu$ is radial with an almost decreasing profile, i.e., there is $c>0$ such that $\nu(|x|) \geq c\, \nu(|y|)$ for all whenever $ |x| \leq |y|$. 
\end{definition}

\begin{definition}\label{def:different-nus}
Let $\nu$ satisfies  the $p$-L\'{e}vy integrability condition \eqref{eq:plevy-integ-cond} and $B\subset \R^d$ be an open set with positive measure, i.e., $|B|>0$. Define the weights  $\overline{\nu}_B, \widetilde{\nu}_B: \R^d \to [0,\infty)$ by  
\begin{align*}
\widetilde{\nu}_B (x) &= \int_B \left( 1\wedge \nu(x-y) \right) \d y, \\
\overline{\nu}_B(x) &= \operatorname{ess}\inf\limits_{y\in B}\nu(x-y) \,.
\intertext{If $\nu$ is a unimodal, then for an arbitrary fixed number $R>1$, we define  the weight $\widehat{\nu}_{R} : \R^d \to [0,\infty)$  by} 
\widehat{\nu}_{R} (x) &= \nu(R(1+|x|)). 
\end{align*} 
We emphasize that $\nu$ needs not be unimodal for the definition of  $\widetilde{\nu}_B$ and $\overline{\nu}_B$. 
\end{definition}

\noindent It is worthwhile noticing that $\widetilde{\nu}_B$, $\overline{\nu}_B$ and $\widehat{\nu}_{R}$ somehow rule out the eventual singularity of the original kernel $\nu$. This is illustrated in the next proposition on important properties of the three measures $\widetilde{\nu}_B$, $\overline{\nu}_B$ and $\widehat{\nu}_{R} $.

\begin{proposition}
\label{prop:weigth-prop} The following assertion are true.
\begin{enumerate}[$(i)$]
\item  $\widehat{\nu}_{R} ,\, 1\land\nu\in L^1(\R^d)\cap L^\infty(\R^d)$. 
\item $\widetilde{\nu}_B\in L^1_{\loc}(\R^d)\cap L^\infty(\R^d)$. If $|B|<\infty$, then $\widetilde{\nu}_B\in L^1(\R^d)\cap L^\infty(\R^d)$.
\item  $\overline{\nu}_B\in L^1 (\R^d)$. If  $\nu$ unimodal then $\overline{\nu}_B\in L^1(\R^d)\cap L^\infty(\R^d) $. 
\item $\overline{\nu}_B(x)= \widetilde{\nu}_B(x)=0$ for all $x\not\in B_\nu:= B+\supp\nu$.
\item  For $B_1\subset B_2$ we have $\overline{\nu}_{B_2}\leq \overline{\nu}_{B_1}$ and $ \widetilde{\nu}_{B_1}\leq \widetilde{\nu}_{B_2}$.
\end{enumerate}
\end{proposition}

\begin{proof}
The proof of $(i)-(iii)$ can be found in \cite[Chapter 3]{guy-thesis} or adapted from \cite[Section 2.3]{FK22}.
To prove $(iv)$, consider  $x\not \in B+\supp\nu$. Then $x-y\not\in \supp\nu$  for all $y\in B$. In other words, $\nu(x-y)=0$ for almost all $y\in B.$ Whence,  $\overline{\nu}_B(x)= \widetilde{\nu}_B(x)=0$. $(v)$ is blatantly obvious.
\end{proof}

The weights $\widehat{\nu}_{R} $, $\widetilde{\nu}_B$ and $\overline{\nu}_B$ turn out to be comparable when $\nu$ is unimodal and satisfies the doubling scaling condition at infinity \eqref{eq:doubling-scaling_infinity}. 

\begin{definition}\label{def:doublingatinfinity}
We say that  a radial kernel $\nu$  satisfies a doubling condition at infinity if: 
\begin{align}\label{eq:global-scaling_infinity}
\forall\,  \, \theta\geq 1 \text{ there exist } c_1, c_2 > 0 \text{ with } c_1\nu(r)\leq \nu(\theta r)\leq c_2\nu(r)\,\, \forall\, \,  r\geq 1\,.
\end{align}
Not that the property \eqref{eq:global-scaling_infinity} is indeed equivalent  to say that 
\begin{align}\label{eq:doubling-scaling_infinity}
\text{ There exist } c_1, c_2 > 0 \text{ with } c_1\nu(r)\leq \nu(2 r)\leq c_2\nu(r) \text{ for all } r\geq 1. 
\end{align}
The doubling condition is  only relevant when the support of $\nu$ is large enough. 
\end{definition}

\begin{theorem}\label{thm:compare-nu-weight}
If $\nu$ is unimodal and $B$ is bounded, e.g.  say  $B\subset B_R(0)$ then  
\begin{align*}
&\sigma^*_R C^{-1}(1\land\nu(2Rx))\leq \widehat{\nu}_{R} (x)\leq C(1\land\nu(x)),\\
& C^{-1}\widehat{\nu}_{R} (x)\leq \overline{\nu}_B(x) \leq C\widetilde{\nu}_B(x),\\
&\widetilde{\nu}_B(x) \leq  \sigma_R C(1\land \nu(\frac{x}{2})), 
\end{align*}
where $\sigma_R= (1+\nu^{-1}(2R))$ and $\sigma^*_R= 1\land\nu(2R)$.
\noindent  In addition, if $\nu$ is of full support and  satisfies the doubling condition \eqref{eq:global-scaling_infinity} then
\begin{align*}
\widetilde{\nu}_B (x) \asymp \overline{\nu}_B(x)\asymp \widehat{\nu}_{R} (x)\asymp 1\land \nu(x). 
\end{align*}
Here the constant $C>0$ and the constants behind the relation $\asymp$ are generic and depend on $B$, $R$ and $\nu$. Moreover, one notes that $\sigma^*_R=\sigma^{-1}_R=0$ if $\nu(2R)=0$. 
\end{theorem}

\begin{proof}
Let us fix $x\in \R^d$.  Since $R\leq R(1+|x|)$ and $|x|\leq R(1+|x|)$  by unimodality,  
\begin{align*}
\text{$\widehat{\nu}_{R} (x)\leq C1\land \nu(x)\quad $ with $\quad C =c(1+\nu(R))$.}
\end{align*}
For $|x|\leq 1$ we have $ R(1+|x|)\leq 2R$, the unimodality implies 
$$c\widehat{\nu}_{R} (x)\geq \nu(2R)\geq \nu(2R)(1\land \nu(2Rx))\geq \sigma^*_R(1\land \nu(2Rx)).$$
For $|x|\geq 1$ we have  $R(1+|x|)\leq 2R|x|$  the unimodality implies $c\widehat{\nu}_{R} (x)\geq \nu(2Rx)\geq \sigma^*_R(1\land \nu(2Rx)).$
In any case, taking  $C =c(1+\nu(R))$,  we find that
\begin{align*}
\sigma^*_RC^{-1}(1\land \nu(2Rx))\leq  \widehat{\nu}_{R} (x)\leq C(1\land \nu(x)).
\end{align*}

Next, according to Proposition \ref{prop:weigth-prop} $\overline{\nu}_B$ is bounded, thus by definition of $\overline{\nu}_B$ 
\begin{align*}
\overline{\nu}_B(x)\leq (1+\|\overline{\nu}_B\|_{L^\infty(\R^d)} ) (1\land \nu(x-y))\quad\text{for almost all $\, y\in B$}
\end{align*}
whence letting $C=|B|^{-1}(1+\|\overline{\nu}_B\|_{L^\infty(\R^d)})$ we find that 
\begin{align*}
\overline{\nu}_B(x)\leq C \int_B 1\land \nu(x-y)\d y= C\widetilde{\nu}_B (x).
\end{align*}
We  find that $|x-y|\leq R(1+|x|)$ for $y\in B$, the unimodality implies $\widehat{\nu}_{R} (x)\leq c\nu(x-y)$. Hence we get 
\begin{align*}
c^{-1}\widehat{\nu}_{R} (x)\leq \essinf_{y\in B}\nu(x-y)=\overline{\nu}_B(x). 
\end{align*}
The second claim is proved. Now, for $\frac{|x|}{2}\leq 2R$ the unimodality implies $1\leq c\nu^{-1}(2R)\nu(\frac{x}{2})\leq c\sigma_R\nu(\frac{x}{2})$. Since by definition of  $\widetilde{\nu}_B(x)$ we have  $\widetilde{\nu}_B(x)\leq\|1\land \nu\|_{L^1(\R^d)}$, we deduce that 
\begin{align*}
\widetilde{\nu}_B(x)\leq\sigma_R (c+1)\|1\land\nu\|_{L^1(\R^d)}(1\land \nu(\frac{x}{2})). 
\end{align*}
If $\frac{|x|}{2}\geq 2R$ then  $\frac{|x|}{2}\leq |x-y|$ for all $y\in B\subset B_R(0)$. By unimodality we get $ \nu(x-y)\leq c\nu(\frac{x}{2})$ which implies $ 1\land \nu(x-y)\leq (c+1)1\land \nu(\frac{x}{2})$ for all $y\in B$. Thence we get 
\begin{align*}
\widetilde{\nu}_B(x)= \int_B1\land \nu(x-y)\d y
\leq\sigma_R (c+1) |B|(1\land \nu(\frac{x}{2})).
\end{align*}
In any case, taking $C= (c+1)(|B|+\|1\land \nu\|_{L^1(\R^d)}) $ we find that 
\begin{align*}
\widetilde{\nu}_B(x)\leq \sigma_RC (1\land \nu(\frac{x}{2})).
\end{align*}

Last, assume that $\nu$ has full support and  satisfies the doubling condition \eqref{eq:global-scaling_infinity}. Given the previous estimates, it is sufficient to show that $1\land \nu(\frac{x}{2}) \asymp1\land\nu(x)\asymp1\land\nu(2Rx)$. For $|x|\leq 2$ then unimodality implies  
$c\nu(2Rx)\geq \nu(2R)\geq \nu(2R)(1\land \nu(x))$. Likewise, $c\nu(x)\geq \nu(2)(1\land \nu(\frac{x}{2}))$.  We deduce that
\begin{align*}
& (1+ c\nu^{-1}(2))^{-1}1\land \nu(\frac{x}{2})\leq 1\land\nu(x)\leq (1+ c\nu^{-1}(2R))1\land\nu(2Rx).
\end{align*}
For $|x|\geq 2$, the doubling condition \eqref{eq:global-scaling_infinity} implies $c_1\nu(\frac{x}{2})\leq \nu(x)\leq c_2\nu(2Rx)$, so that
\begin{align*}
&(c_1^{-1}+1)^{-1}1\land \nu(\frac{x}{2})\leq1\land\nu(x)\leq (c_2+1)1\land \nu(2Rx).
\end{align*}
Finally we put  $C= \max(c_1^{-1}+1, c_2+1,1+ c\nu^{-1}(2R), 1+ c\nu^{-1}(2))$ so that $0<C<\infty$ since  $\nu$ has full support and hence  we get 
\begin{align*}
C^{-1}(1\land \nu(\frac{x}{2}))\leq 1\land\nu(x)\leq C(1\land\nu(2Rx). 
\end{align*} 
This achieves the proof. 
\end{proof}

\begin{example}
Consider $\nu(h) = |h|^{-d-sp}$, $s\in(0,1)$ then  $\widetilde{\nu}\asymp 1\land \nu $. In this case one can take $ \widetilde{\nu}(h)= (1+|h|)^{-d-sp}$.
In this case, the space $\WnuOm$ equals the classical Sobolev-Slobodeckij space $W^{s,p}(\Omega)$. For the same choice of $\nu$ we define $W^{s,p}(\Omega|\R^d)$ as the space $\WnuOmR$. 
\end{example}

It is possible to define other interesting weighted norms on $\WnuOmR$ which are equivalent to the norm $\|\cdot\|_{\WnuOmR}$.
\begin{lemma}\label{lem:natural-norm-on-w}
Assume that $B\subset \Omega$ where $B$ is bounded, say $B\subset B_R(0)$. 
Let $\omega\in \{ \widetilde{\nu}_B, \overline{\nu}_B, \widehat{\nu}_{R}  \}$. If $\nu$ is unimodal  then the following  seminorms $\|\cdot\|^{\#}_{\WnuOmR}$ and $\|\cdot\|^{*}_{\WnuOmR}$ are equivalent.
\begin{align*}
\|u\|^{*p}_{\WnuOmR}&=\int_{\R^d}  |u(x)|^p \omega(x)\d x+\hspace*{-1ex} \iil\limits_{\Omega\R^d}\hspace*{-1ex} |u(x)-u(y)|^p \nuxminy\d y\d x,\\
\|u\|^{\#p}_{\WnuOmR}&=\int_{\Omega}  |u(x)|^p \omega(x)\d x+ \hspace*{-1ex}\iil\limits_{\Omega \R^d}\hspace*{-1ex} |u(x)-u(y)|^p \nuxminy\d y\d x.
\end{align*} 

Furthermore, the seminorms $\|\cdot\|_{\WnuOmR}$ and $\|\cdot\|^{*}_{\WnuOmR}$ are also equivalent provided that  $\Omega$ is bounded. 
\end{lemma}

\begin{remark}
It is noteworthy to emphasize that the seminorms $\|\cdot\|^{\#}_{\WnuOmR}$,  $\|\cdot\|^{*}_{\WnuOmR}$ and $\|\cdot\|_{\WnuOmR}$ are norms on $\WnuOme$ whenever the condition \eqref{eq:positive-Oe} holds true; cf Theorem \ref{thm:WnuOme-complete}. In particular, if $\nu$ is of full support. 
\end{remark}

\begin{proof} 
In virtue of Theorem \ref{thm:compare-nu-weight} we have 
$C^{-1}\widehat{\nu}_{R} (x)\leq \omega(x)\leq C\widetilde{\nu}_B(x)$ for some $C>0$.  
Whence, since $B$ is bounded and $\nu$ is of full support, the unimodality implies that there is $c'>0$ such that $\omega(x)\geq c'$ for almost all $x\in B$. This together with the fact that
$B\subset \Omega$ and $\int_{\Omega^c}(1\land\nu(x-y))\d y\leq \|1\land\nu\|_{L^1(\R^d)}$ implies 

\begin{align*}
\int_{\Omega}& |u(x)|^p\omega(x)\d x
+\iil_{\Omega\Omega^c} |u(x)-u(y)|^p \nu(x-y)\d y\,\d x
\\
&\geq c'\int_{B} |u(x)|^p \d x+\iil_{\Omega\Omega^c} |u(x)-u(y)|^p \nu(x-y)\d y\,\d x
\\
&\geq c'\|1\land\nu\|_{L^1(\R^d)}^{-1}\iil_{B\Omega^c}  |u(x)|^p (1\land \nu(x-y))\d y\,\d x
\\&+\iil_{B\Omega^c} |u(x)-u(y)|^p \nu(x-y)\d y\,\d x\\
&\geq ( 1\land c'\|1\land\nu\|_{L^1(\R^d)}^{-1}) \iil_{\Omega^cB}\Big[  |u(x)|^p+ |u(x)-u(y)|^p  \Big] 1\land \nu(x-y)\d x\,\d y. 
\end{align*}
Using $b^p\geq 2^{1-p}(a+b)^p-a^p$, $a,b\geq0$ and  $\omega(y)\leq C\widetilde{\nu}_B(y)$ we get
\begin{align*}
\int_{\Omega} |u(x)|^p&\omega(x)\d x
+\iil_{\Omega\Omega^c} |u(x)-u(y)|^p \nu(x-y)\d y\,\d x
\geq C_1\int_{\Omega^c}|u(y)|^p \omega(y)\d y\, 
\end{align*}
with $C_1= \frac{C^{-1}}{2^{p-1}} ( 1\land c'\|1\land\nu\|_{L^1(\R^d)}^{-1})$. This clearly implies that $\| u\|^{*}_{\WnuOmR}\leq  C\| u\|^{\#}_{\WnuOmR}$ for some $C>0$. Thus, $\| \cdot \|^{*}_{\WnuOmR}$ and $\| \cdot \|^{\#}_{\WnuOmR}$ are equivalent.
\noindent Now if $\Omega$ is bounded, the unimodality and the boundedness of $\omega$ (see Proposition \ref{prop:weigth-prop}) yield $\|\omega\|_{L^\infty(\R^d)}\geq \omega(x)\geq c'$, $x\in \Omega$ for some $c'>0$.  The equivalence of $\|\cdot\|_{\WnuOmR}$ and $\| \cdot\|^{\#}_{\WnuOmR}$ readily follows.
\end{proof}

The next result shows that functions in $\WnuOmR$  possess certain weighted integrability. 
\begin{theorem}\label{thm:w-emb-on-weigh-lp}
Assume $B\subset \Omega$ and $R\geq1$.  Let $\omega\in \{ \widetilde{\nu}_B, \overline{\nu}_B, \widehat{\nu}_{R}  \}$. 
\begin{enumerate}[$(i)$]
\item  
For $\omega\in \{ \widetilde{\nu}_B, \overline{\nu}_B\}$ we have the continuous embedding 
$$\WnuOmR\hookrightarrow L^p(\R^d,\omega).$$ 
The same holds for $\omega=\widehat{\nu}_{R} $ if in addition $|B_R(0) \cap \Omega| > 0$.
\item For $\omega\in \{\overline{\nu}_B, \widehat{\nu}_{R} \}$ we have  the continuous embeddings  
\begin{align}
&L^p(\R^d,\omega) \hookrightarrow L^{p-1}(\R^d,\omega)\cap L^1(\R^d,\omega)\\
\intertext{equivalently for any $q\in [\min(1, p-1), p]$ we have  }
&L^p(\R^d,\omega) \hookrightarrow L^q(\R^d,\omega)\hookrightarrow L^{\min(1,p-1)}(\R^d,\omega).
\end{align}
The same holds for $\omega=\widetilde{\nu}_B$ if in addition  $|B|<\infty$. 
\end{enumerate}
\end{theorem}

\begin{proof}
$(i)$ The embedding 
$ \WnuOmR\hookrightarrow L^p(\R^d, \widetilde{\nu}_B)$ is obtained as follows
\begin{align*}
\int_{\R^d} |u(x)|^p\widetilde{\nu}_B(x)\d x
& \leq 2^{p-1}\int_{B} |u(y)|^p \int_{\R^d} 1\land \nu(x-y)\d x\d y \\&+ 2^{p-1}\iil_{B\R^d}  |u(x)-u(y)|^p  1\land \nu(x-y)\d y\,\d x\\
&\leq C\int_{\Omega} |u(y)|^p \d y+ C\iil_{\Omega\R^d}  |u(x)-u(y)|^p  \nu(x-y)\d y\,\d x
\\&= C\|u\|^p_{\WnuOmR}.
\end{align*}
Here $C=2^{p-1}(1+\|1\land\nu\|_{L^1(\R^d)})$. To  prove 
$ \WnuOmR\hookrightarrow L^p(\R^d, \overline{\nu}_B)$, let us consider $K'\subset B\subset \Omega$ be a measurable such that  $0<|K'|<\infty$. For all $y\in K'$ we have  
\begin{align*}
\int_{\R^d}\hspace*{-1ex}|u(x)|^p \overline{\nu}_B (x)\d x
&\leq 2^{p-1}\big(\|\overline{\nu}_B\|_{L^1(\R^d)} |u(y)|^p+ \hspace*{-1ex} \int_{\R^d}\hspace*{-2ex}|u(x)-u(y)|^p \nu(x-y) \d x\big).
\end{align*}

Integrating over $K'$ we find that 
\begin{align*}
\int_{\R^d}|u(x)|^p \overline{\nu}_B (x)\d x
&\leq C\int_{K'}|u(y)|^p\d y+ C\iil_{K'\R^d}|u(x)-u(y)|^p \nuxminy\d y \d x\\&\leq C\|u\|^p_{\WnuOmR}.
\end{align*}
where $C=2^{p-1}|K'|^{-1}( \|\overline{\nu}\|_{L^1(\R^d) }+1).$
Finally, if $|B_R(0)\cap \Omega|>0$, consider $B'=B_R(0)\cap \Omega$ so that $B'\subset B_R(0)$. Since for the definition of $\widehat{\nu}_{R} $, $\nu$ is unimodal, by Theorem \ref{thm:compare-nu-weight}
there is $C=C(B', R)>0$ such that $\widehat{\nu}_{R} (x)\leq C\overline{\nu}_{B'}(x)$. 
Thus, we deduce from the previous case that 
$\WnuOmR\hookrightarrow L^p(\R^d,\overline{\nu}_{B'})\hookrightarrow L^p(\R^d,\widehat{\nu}_{R} ).$

$(ii)$ For $\omega\in \{\widetilde{\nu}_B, \overline{\nu}_B,\widehat{\nu}_{R} \}$,  Proposition \ref{prop:weigth-prop} implies $\omega\in L^1(\R^d)$. Thus H\"{o}lder's inequality implies $\|u\|_{L^r(\R^d, \omega)}\leq \|\omega\|^{1/r-1/q}_{L^1(\R^d)}\|u\|_{L^q(\R^d, \omega)}$  whenever $q\geq r>0$. The desired embeddings immediately follow. 
\end{proof}

\begin{corollary}\label{cor:w-emb-on-weigh-lp}
Let $\omega\in \{ \widetilde{\nu}_B, \overline{\nu}_B, \widehat{\nu}_{R}  \}$ where $B\subset \Omega$.  Assume $|B|<\infty$ and $|B_R(0)\cap \Omega|>0$, $R\geq1$.  
The following embeddings  are continuous 
\begin{align*}
\WnuOmR\hookrightarrow L^p(\R^d,\omega)\hookrightarrow L^{p-1}(\R^d,\omega)\cap L^1(\R^d,\omega). 
\end{align*}
\end{corollary}
\begin{proof}
This is	a straightforward consequence of Theorem \ref{thm:w-emb-on-weigh-lp}. 
\end{proof}

\begin{corollary}\label{cor:w-emb-on-weigh-lp-bis}
Assume $\nu$ is unimodal with full support and satisfies the doubling condition \eqref{eq:global-scaling_infinity}. The following embeddings  are continuous 
\begin{align*}
\WnuOmR\hookrightarrow L^p(\R^d,1\land\nu)\hookrightarrow L^{p-1}(\R^d,1\land\nu)\cap L^1(\R^d,1\land\nu). 
\end{align*}
\end{corollary}

\begin{proof}
This is a direct  consequence of Corollary \ref{cor:w-emb-on-weigh-lp} and Theorem \ref{thm:compare-nu-weight}. 
\end{proof}
In the fractional setting $L^{p-1}(\R^d, \omega)$ is  sometimes  called the tail space. 
\begin{corollary}\label{cor:w-emb-on-weigh-lp-frac}
Let $\nu(h)=|h|^{-d-sp}$ with  $s\in (0,1)$, we put $ \WnuOmR= W^{s,p}(\Omega|\R^d)$. The following embeddings  are continuous 
\begin{align*}
W^{s,p}(\Omega|\R^d)&\hookrightarrow L^p(\R^d,(1+|h|)^{-d-sp}))\\&\hookrightarrow L^{p-1}(\R^d,(1+|h|)^{-d-sp}))\cap L^1(\R^d,(1+|h|)^{-d-sp})). 
\end{align*}
\end{corollary}

\begin{proof}
This follows from Corollary \ref{cor:w-emb-on-weigh-lp-bis} since  $1\land\nu(h) \asymp(1+|h|)^{-d-sp}$. 
\end{proof}
\subsection{Nonlocal , trace space of $\WnuOmR$}
The main goal of this part is to introduce an abstract notion trace space of $\WnuOmR$ similarly as one does for the space $W^{1,p}(\Omega)$.  Due to its nonlocal character, the nonlocal  trace space of $\WnuOmR$ assumes functions defined on $\R^d\setminus\Omega$, or strictly speaking on the nonlocal boundary $\Omega_e=\Omega_\nu\setminus\Omega$ of $\Omega$ with respect to $\nu$.  The main reason is that elements of $\WnuOmR$ are essentially defined on $\R^d$, or strictly speaking  on the nonlocal hull 
$\Omega_\nu^{*} =\Omega\cup\Omega_\nu$,  $\Omega_\nu=\Omega+\supp\nu$ (see Theorem \ref{thm:WnuOme-complete}). This contrasts with the local situation, where the trace space of $W^{1,p}(\Omega)$ (with $\Omega$ smooth enough) are elements defined on the boundary $\partial\Omega$. 
\begin{definition}[Trace space of $\WnuOmR$]
The trace space of $\WnuOmR$ denoted  $\TnuOm$ is the space of restrictions to $\R^d\setminus \Omega$ of functions of $\WnuOmR$. More precisely, 
\begin{align*}
\TnuOm = \{v: \Omega^c\to \mathbb{R}~\text{meas.} ~\hbox{ such that }~ v = u|_{\Omega^c} ~~\hbox{with }~~ u \in \WnuOmR\}.
\end{align*}
We endow $\TnuOm $ with its natural norm, 
\begin{align*}
\|v\|_{\TnuOm } = \inf\{ \|u\|_{\WnuOmR }: ~~ u \in \WnuOmR ~~ \hbox{ with }~~ v = u|_{\Omega^c} \}. 
\end{align*}
If we identify $\WnuOmR\equiv W^p_\nu(\Omega|\Omega_\nu)$ then we also identify $\TnuOm\equiv T^p_\nu(\Omega_e)$ where $\Omega_e=\Omega_\nu\setminus\Omega$ and 
\begin{align*}
T^p_\nu(\Omega_e) 
= \{v: \Omega_e\to \mathbb{R}~\text{meas.} ~~\hbox{ such that }~~ v = u|_{\Omega_e} ~~\hbox{with }~~ u \in W^p_\nu(\Omega|\Omega_\nu)\}.
\end{align*}
\end{definition}

\begin{theorem}
If $(\WnuOmR, \|\cdot\|_{\WnuOmR})$	is a Banach space then  $(\TnuOm,  \|\cdot\|_{\TnuOm})$ is also a Banach space. 
\end{theorem}

\begin{proof}
Noting that $\TnuOm$ and the quotient space $\WnuOmR/\WnuOmO$ are identical with equal norm in space  and that $\WnuOmO$ is a closed subspace of $\WnuOmR,$ one concludes that $\TnuOm$ is complete. 
Alternatively, for the reader's convenience, we provide a detailed proof regarding the completeness of $\TnuOm$ under the norm $\|\cdot \|_{\TnuOm }$. Let $(u_n)_n$ be a Cauchy sequence in $\TnuOm$ then up to the  extraction of a subsequence we may assume that
\begin{align*}	
\|u_n-u_{n+1}\|_{\TnuOm } <\frac{1}{2^{n+1}}\quad\hbox{for all $n\geq 1$}. 
\end{align*}
Fix $\overline{u}_1\in \WnuOmR $ such that $u_1= \overline{u}_1|_{\Omega^c}$. 
By definition of $\|u_1-u_2\|_{\TnuOm}$ there exists $v \in \WnuOmR$ such that $u_1-u_2 = v|_{\Omega^c}$ and 
$\|v\|_{\WnuOmR } <\|u_1-u_2\|_{\TnuOm } +\frac{1}{4} \,. $
Letting $\overline{u}_2= v+ \overline{u}_1$ i.e. $v= \overline{u}_1-\overline{u}_2$ yields $\overline{u}_2\in \WnuOmR$, $u_2= \overline{u}_2|_{\Omega^c}$ and hence
$\|\overline{u}_1-\overline{u}_2\|_{\WnuOmR } <\|u_1-u_2\|_{\TnuOm } +\frac{1}{4} <\frac{1}{2}\,.$ Repeating this process one constructs a sequence  $\overline{u}_n\in \WnuOmR $ such that
\begin{align*}
\|\overline{u}_n-\overline{u}_{n+1}\|_{\WnuOmR } <\frac{1}{2^{n}}\qquad\hbox{for all}~~~ n\geq 1
\end{align*}
which turns out to be a Cauchy sequence in the complete space $\WnuOmR$. 
Let $\overline{u}\in \WnuOmR$ be the limit of $(\overline{u}_n)_n$. 
Clearly, setting $u = \overline{u}\mid_{\Omega^c} $ we have 
$\|u_n-u \|_{\TnuOm } \leq \|\overline{u}_n-\overline{u}\|_{\WnuOmR } \to 0$ as  $n \to \infty$. Finally, the original Cauchy sequence $(u_n)_n$ converges up to the extraction of subsequence to $u$ and hence converges itself to $u$. 
\end{proof}

Let us state the nonlocal trace theorem for the trace space $\TnuOm$.
\begin{theorem}[Nonlocal trace theorem]\label{thm:trace-nonloc-thm}
Let $\omega\in \{\widetilde{\nu}_{\Omega}, \overline{\nu}_{\Omega}, \widehat{\nu}_{R}\}$ where $|B_R(0)\cap \Omega|>0$ $($see Definition \ref{def:different-nus}$)$. 
Define the trace operator $u \mapsto \operatorname{Tr}(u) = u\mid_{\Omega^c}$. The following assertions hold. 
\begin{enumerate}[$(i)$]
\item $ \ker(\operatorname{Tr} ) = \WnuOmO$ and $\operatorname{Tr}(\WnuOmR)= \TnuOm$. 
\item The mappings $\WnuOmR \xrightarrow{\operatorname{Tr}}\TnuOm\xrightarrow{\operatorname{Id}} L^p(\Omega^c,\omega)$ 
are continuous.
\item If $\Omega$ is bounded in one direction (see Section \ref{sec:poincare}), then there is a bounded and continuous mapping $\operatorname{Ext}: \TnuOm\to \WnuOmR$ with  $\operatorname{Ext}\circ \operatorname{Tr}= Id$. 
\end{enumerate}
\end{theorem}
\begin{proof}
It is easy to check that $\operatorname{Tr}(\WnuOmR)= \TnuOm$ and $ \ker(\operatorname{Tr}) = \WnuOmO$. 
It immediately follows from the definition of $\|\cdot\|_{\TnuOm}$ that 
$\|\operatorname{Tr}(u)\|_{\TnuOm}\leq \|u\|_{\WnuOmR}$ for all $u\in \WnuOmR$ whereas,  by Theorem \ref{thm:w-emb-on-weigh-lp} there exists $C>0$ such that
\begin{align*}
\|\operatorname{Tr}(u)\|_{L^p(\Omega^c, \omega)} \leq \| u\|_{L^p(\R^d, \omega)} 
\leq C \| u\|_{\WnuOmR} \qquad \text{for all }~~~u \in \WnuOmR. 
\end{align*}
Inasmuch as the above estimate is true for all $u\in \WnuOmR$, we deduce $\|u\|_{L^p(\Omega^c, \omega)} \leq C \| u\|_{\TnuOm}$, 
$u \in \TnuOm$.  Last, for $g\in \TnuOm$ consider $\operatorname{Ext}g=u_g$ where by Theorem \ref{thm:nonlocal-dirichlet-gen} $ u_g\in \WnuOmR$ is the unique weak solution to the Dirichlet problem $Lu=0$ in $\Omega$ and $u=g$ in $\R^d\setminus\Omega$.  The operator $\operatorname{Ext}$ is well-defined and $\operatorname{Ext}\circ \operatorname{Tr}= Id$. The
continuity and the boundedness follow from Theorem \ref{thm:nonlocal-dirichlet-gen}. 
\end{proof}
\noindent One may view the objects $L^p(\Omega^c, \omega)$, $\TnuOm$, $\WnuOmR$ and $\WnuOmO$ respectively as the nonlocal counterpart of the objects $L^p(\partial\Omega)$, $W^{1-1/p,p}(\partial\Omega)$, $W^{1,p}(\Omega)$ and $W^{1,p}_0(\Omega)$.  Let us recall the classical trace theorem. 

\begin{theorem}[{Classical trace theorem, see \cite[Chap III]{BF13}}]\label{thm:trace-loc-thm}
Assume $\Omega\subset \R^d$ is bounded Lipschitz. There exists a linear and continuous trace operator $\gamma_0: W^{1,p}(\Omega)\to L^p(\partial\Omega)$ such that 
$\gamma_0 u= u|_{\partial\Omega}$ for all $u\in C^1(\overline{\Omega})\cap W^{1,p}(\Omega)$. Moreover, $ \ker(\gamma_0 ) = W^{1,p}_0(\Omega)$ and, by definition, $\gamma_0 (W^{1,p}(\Omega))= W^{1,1-1/p}(\partial \Omega)$.  
\end{theorem}
\smallskip

\begin{remark}
Let us emphasize that  our nonlocal trace operator $\operatorname{Tr}$ does not need any special construction via 
functional analysis and density argument. Since $\Omega^c$ is still a d-dimensional manifold. Then it makes sense to consider hardcore restriction of functions on $\Omega^c.$ 
Moreover, no regularity on $\Omega$ is required nor on $u$. Whereas in the local situation (see Theorem \ref{thm:trace-loc-thm}), the trace of a Sobolev function $u$ on the boundary $\partial\Omega$ requires the smoothness of both $u$ and $\partial\Omega$.
\end{remark}

\noindent
It is natural to ask the following question: Can the space $\TnuOm$ be self-defined with an intrinsic norm preserving its initial Banach structure in a way that its trivial connection to $\WnuOmR$ is less seeable? 
In the local situation, it is possible to define a scalar product on the space $H^{1/2}(\partial \Omega)$ when $\Omega$ is a special Lipschitz domain (see\cite{Din96}). This question  is discussed in  \cite{FK22,GH22} when $p=2$, using the main result from \cite{BGPR20}. 
Another treatment of nonlocal trace operator for the special fractional kernel  $\nu(h)=|h|^{-d-sp}$ is encapsulated in  \cite[Theorem 3]{DyKa19}. 

\section{Compact embeddings}\label{sec:compactness}

In this section we prove compact embeddings of the spaces $\WnuOm$, $\WnuOmR$ and $W_{\nu,\Omega}(\Omega|\R^d)$ into $L^p(\Omega)$. Our result on global compactness Theorem \ref{thm:embd-compactness} requires some extra regularity assumptions on $\Omega$ compatible  with $\nu$. We exploit  some recent ideas  from \cite{JW20,DMT18}. However, extended  details and discussions for the global compactness can be found in \cite{guy-thesis,FK22}. 

\subsection{Local and global compactness results}\label{sec:compact-poincare}

Given a Banach space  $X$,  we say that an operator $T: X\to L^p_{\loc}(\Omega)$ is compact if for any  compact set $K\subset \Omega$, the operator  
$R_KT: X\to L^p_{\loc}(K)$, $R_K Tu= Tu|_K$ is compact.

\begin{lemma}[{\hspace*{-1ex} \cite[Corollary 4.28]{Bre10}}] \label{lem:compactness-convolution} 
Let $w \in L^1(\R^d)$. The convolution operator $T_w : L^p(\R^d )\to L^p(\R^d)$, $ T_wu=w*u $ is continuous with $\|T_w\|_{\mathcal{L}(L^p(\R^d), L^p(\R^d))} \leq \|w\|_{L^1(\R^d)}$. Moreover, $R_KT_w : L^p(\R^d) \to L^p(K)$ is compact for any measurable set $K\subset \R^d$ with $|K|<\infty$. 
\end{lemma}

We now present the characterization of the local compactness result that will be used in the sequel. 
\begin{theorem}\label{thm:local-compactness}
If $\nu$ satisfies \eqref{eq:plevy-integ-cond}, then the following assertions are equivalent. 
\begin{enumerate}[$(i)$]
\item $\nu$ is not integrable, i.e., $\nu\not\in L^1(\R^d)$. 
\item The embedding $\Wnu\hookrightarrow L^p(K)$ is compact for any measurable set with $|K|<\infty$. 
\item The embedding $\Wnu\hookrightarrow L^p_{\loc}(\R^d)$ is compact. 
\item The embedding $\Wnu\hookrightarrow L^p(B(0,1))$ is compact.
\end{enumerate}
\end{theorem}

\begin{proof} We only  prove $(iv)\implies (i)$ and $(i)\implies (ii)$ as the other implications are straightforward. Assume $\nu\in L^1(\R^d)$ and put $B=B(0,1)$ then $\Wnu= L^2(\R^d)$ and hence the mappings $L^p(B)\xrightarrow{E}\Wnu\xrightarrow{R_B} L^p(B)$, where  $Eu= \overline{u}$ is the zero extension of $u$ and  $R_{B}u=u|_{B}$,  are continuous. Since   is the identity map $I=R_{B}\circ E:L^p(B)\to L^p(B)$ is not compact, necessarily, $R$ is not compact. 
\noindent Now assume $\nu\not \in L^1(\R^d)$ and  $\delta>0$ sufficiently  small such that $0<\|\nu_\delta\|_{L^1(\mathbb{R}^d)}$. Consider $w_\delta(h) =\nu_\delta(h) \|\nu_\delta\|^{-1}_{L^1(\mathbb{R}^d)}$ so that $\|w_\delta\|_{L^1(\mathbb{R}^d)} =1$. Define $T_{w_\delta}= w_\delta*u$,  $u \in L^p(\mathbb{R}^d)$. By the symmetry of $\nu $ we have, 
\begin{align*}
T_{w_\delta} u(x)= \int_{\mathbb{R}^d} w_\delta(y) u(x-y)\d y = \int_{\mathbb{R}^d} w_\delta(y) u(x+y)\d y\,.
\end{align*} 

 Meanwhile, Jensen's inequality implies
\begin{align*}
\|u-& T_{w_\delta} u\|^p_{L^p(\mathbb{R}^d)} 
= \int_{\mathbb{R}^d} \Big| \int_{\mathbb{R}^d} [u(x)-u(x+h) ]w_\delta (h)\d\, h\Big|^p\d x\\
& \leq \|\nu_\delta\|^{-1}_{L^1(\mathbb{R}^d)} \iil_{\mathbb{R}^d\mathbb{R}^d} |u(x)-u(x+h) |^p \nu (h)\d\, h\d x \leq \|\nu_\delta\|^{-1}_{L^1(\mathbb{R}^d)} \|u\|^p_{\Wnu} \,.
\end{align*}
\noindent Thus, since $\nu\not\in L^1(\R^d)$, for a subset $K\subset \mathbb{R}^d$ with $|K|<\infty$ we find that  
%
\begin{align*}
\|R_K -R_K T_{w_\delta} \|_{\mathcal{L}\big(\Wnu,\, L^p(K)\big)}\leq \|\nu_\delta\|^{-1/p}_{L^1(\mathbb{R}^d)}\xrightarrow{\delta\to 0}0\,.
\end{align*}
\noindent It follows that the operator $R_K: \Wnu\to L^p(K)$ with $R_Ku= u|_K$ is compact since each operator $R_K T_{w_\delta}$ is compact (by Lemma \ref{lem:compactness-convolution}) and the set of compact operator is closed.
\end{proof}

Another version of Theorem \ref{thm:local-compactness} is proved in \cite[Theorem 1.1]{JW20} for the case $p=2$ and also in \cite[Theorem 2.1]{CP18} under restrictive assumptions on the kernel $\nu$, using the Pego criterion for compact compactness in $L^2(\R^d)$. It is worth mentioning that earlier analogous results are provided in \cite[Proposition 6]{PZ17}  and \cite[Proposition 1]{BJ16} for periodic functions on the  torus. This technique of killing the singularity is also used in \cite[Lemma 3.1]{BJ13}. 

\begin{corollary}
Assume  $\Omega\subset \R^d$ is open and $\nu\not\in L^1(\R^d)$.  The embedding $\WnuOm \hookrightarrow L^p_{\loc}(\Omega)$ is compact.
Moreover, if  $|\Omega|<\infty$ then the embedding $\WnuOmO \hookrightarrow L^p(\Omega)$, is compact. 
\end{corollary}
\begin{proof}
For fixed $\varphi\in C_c^\infty(\Omega)$  the operator $J_\varphi:\WnuOm\to\Wnu $, $J_\varphi u= u\varphi$ is continuous.  By Theorem \ref{thm:local-compactness}
the embedding $\Wnu \hookrightarrow L^p(K)$, $K=\supp\varphi\subset \Omega$, is compact. 
It follows that the embedding $\WnuOm\hookrightarrow L^p(K)$ is compact. Hence the  first claim is proved. 
The embeddings $ \WnuOmO\hookrightarrow\Wnu \hookrightarrow L^p(\Omega)$ are continuous and, by Theorem \ref{thm:local-compactness}, the last one is compact when $|\Omega|<\infty$. Whence the second claim follows. 
\end{proof}
\noindent Another consequence of Theorem \ref{thm:local-compactness} is the following  local compactness. 

\begin{corollary}\label{cor:local-compatcness} 
Assume $\Omega\subset \R^d$ be open and $\nu\not\in L^1(\R^d)$. 
The following assertions hold.
\begin{enumerate}[$(i)$]
\item Any bounded sequence $(u_n)_n$ in $\WnuOm$ admits  a subsequence $(u_{n_j})_j$ that converges  in $L^p_{\loc}(\Omega)$ to some  $u\in \WnuOm$. Moreover, 
\begin{align*}
|u|_{\WnuOm}\leq \liminf_{n \to \infty}|u_n|_{\WnuOm}.
\end{align*}
\item Assume $|\Omega|<\infty$ then for a bounded sequence $(u_n)_n$ in $\WnuOmO$ there exists $u\in \WnuOmO$ and subsequence $(u_{n_j})_j$ converging to $u$ in $L^p(\Omega)$. Moreover, 
\begin{align*}
|u|_{\Wnu}\leq \liminf_{n \to \infty}|u_n|_{\Wnu}. 
\end{align*}
\end{enumerate}
\end{corollary}

\subsection{Global compactness.}
The global compactness requires some compatibility between the singularity of the kernel $\nu$ and the regularity of the
boundary  $\partial\Omega$. Let us introduce conditions on $\Omega$ and $\nu$ yielding the global compactness result. 
\begin{definition}\label{def:classesAi} Assume $\Omega\subset \R^d$ is open bounded, $\nu\in L^1(\R^d,1\land|h|^p)$ and $\nu\not\in L^1(\R^d)$. 
We say that $(\nu, \Omega)$ is in the class $\mathscr{A}_i$, $i\in \{1,2,3\}$, if in addition 
\begin{itemize}
\item [($\classA{1}$)] $\ldots$ there exists an $\WnuOm$-extension operator $E: \WnuOm\to \Wnu$, i.e., there is $C(\nu, \Omega,d)>0$ such that for every $u\in \WnuOm$, $\|u\|_{\Wnu}\leq C\|u\|_{\WnuOm}$ and $Eu|_\Omega =u$. 
\item[($\classA{2}$)] $\ldots$  $\Omega$ has Lipschitz boundary, $\nu$ is radial and $q(\delta)\xrightarrow[]{\delta \to 0}\infty $ where
\begin{align}\label{eq:class-lipschitz}
q(\delta):= \frac{1}{\delta^p}\int\limits_{B_\delta(0)} |h|^p\nu(h)\d h\,.
\end{align}
\item [($\classA{3}$)] $\ldots$ the following condition holds true: $\widetilde{q}(\delta)\xrightarrow[]{\delta \to 0}\infty $ where
\begin{align}\label{eq:class-sing-boundary}
\widetilde{q}(\delta): = \inf_{a\in \partial\Omega}\int_{\Omega_\delta}\nu(h-a)\d h
\end{align}
with $\Omega_\delta= \{x\in \Omega: \operatorname{dist}(x,\partial\Omega)>\delta\}$.
\end{itemize}
\end{definition}
Note that the assumption  $\nu\not\in L^1(\R^d)$ here is  redundant. Indeed, having $q(\delta)\xrightarrow{\delta\to0}\infty$ 
or $\widetilde{q}(\delta)\xrightarrow{\delta\to0}\infty$ implies $\nu\not\in L^1(\R^d)$. The converse is not true, e.g., for $\nu(h)=|h|^{-d}\mathds{1}_{B_1(0)}$ we have $\nu\not\in L^1(\R^d)$ but $q(\delta)=c_d.$ For the same example  and $\Omega=B_1(0)$ we have also $\widetilde{q}(\delta)=c_d.$
Here is our global compactness result; see  \cite{FK22} for the case $p=2$ or \cite[Theorem 3.89]{guy-thesis}. 
\begin{theorem}\label{thm:embd-compactness}
If the couple $(\nu, \Omega)$ belongs to one of the class $\mathscr{A}_i,~ i=1,2,3$ then the embedding $\WnuOm \hookrightarrow L^p(\Omega)$ is compact. In particular, the embedding $\WnuOmR \hookrightarrow L^p(\Omega)$ is compact.
\end{theorem}

\begin{remark}
The Rellich-Kondrachov's compact embeddings $W^{1,p}_0(\Omega)\hookrightarrow L^p(\Omega) $ and $W^{1,p}(\Omega)\hookrightarrow L^p(\Omega)$ when $\Omega$ is Lipschitz, respectively can be derived from Theorem \ref{thm:local-compactness} combined with the embedding $W^{1,p}_0(\Omega)\hookrightarrow \WnuOmO$ and from Theorem \ref{thm:embd-compactness} combined with the embedding $W^{1,p}(\Omega)\hookrightarrow \WnuOm$ when $\Omega$ is Lipschitz. 
\end{remark}

\section{Nonlocal Poincar\'{e} type inequalities} \label{sec:poincare} 
\subsection{Nonlocal Poincar\'{e}-Friedrichs inequality} 
In this section  $1\leq p<\infty$, $\Omega\subset \R^d$ is open and that $\nu:\R^d\setminus\{0\}\to [0, \infty)$  is symmetric satisfying $\nu\in L^1(\R^d\setminus B(0,r))$ for every $r>0$.  In particular, the latter holds true if $\nu\in L^1(\R^d, 1\land|h|^\gamma\d h)$ for any $ \gamma\geq 0$. 
\noindent Let us define  the space $L^p_\Omega(\R^d)$ by 
\begin{align*}
L^p_\Omega(\R^d)= \{u\in L^p(\R^d): \,  \text{$u= 0$ a.e. on $\Omega^c$}\}.  
\end{align*}
Next, we use a more refined argument to prove the above Poincar\'{e}-Friedrichs type inequalities with strictly positive constants, in a general setting only by assuming that $\nu$ is nontrivial that is, $\nu\not\equiv0$. We follow the strategy from\cite{JW20} using the $2^m$-folded convolutions which is also used in \cite[proof of Lemma 2.7]{FKV15}. 
\begin{lemma}\label{lem:folded-convo}
Let $q\in L^1(\R^d)\cap L^\infty(\R^d)$ with  $q\geq0$ be symmetric, i.e., $q(h)=q(-h)$ and nontrivial,  i.e., $\big|\{q>0\}\big|>0$.  Define $2^m$-fold convolution of $q$ as follows $q_0=q$, $q_m= q_{m-1}* q_{m-1}= \underbrace{q*\cdots*q}_{2^m-\text{times}}$. For every $m\geq1$, the following assertions are true.  
\vspace{-1ex}

\begin{enumerate}[$(i)$]
\item  Each $q_m$  belongs to $L^1(\R^d)\cap L^\infty(\R^d)$ and is uniformly continuous.
\item 	There exists $\delta>0$,  such that $\inf_{B_{\delta_m }(0)} q_m >0$ with 
$\delta_m= (\frac{7}{4})^m\delta $ so that  $\delta_m\xrightarrow{m\to\infty}\infty$. 
\item For $u:\R^d\to \R$ measurable,  we have 
\begin{align}
\mathcal{Q}_{m}(u)\leq B_m\mathcal{Q}(u), \qquad B_m=2^{pm} \|q\|^{2^m-1}_{L^1(\R^d)}, 
\end{align} 
where  we set $\mathcal{Q}(u)= \mathcal{Q}_0(u)$ and 
\begin{align*}
\mathcal{Q}_{m}(u):=\iil_{\R^d \R^d} |u(x) -u(y)|^p q_m(x-y)\d y\,\d x.
\end{align*} 
\end{enumerate}
\end{lemma}

\begin{proof}
(i)  Since  $q_m =q_{m-1}*q_{m-1}$, by the Young's inequalities $q_m\in L^1(\R^d)\cap L^\infty(\R^d)$ and
\begin{align}\label{eq:mfold-young-inequality}
&\|q_m\|_{L^1(\R^d)}\leq \|q_{m-1}\|^2_{L^1(\R^d)}\leq 	\|q\|^{2^m}_{L^1(\R^d)},\\
&\|q_m\|_{L^\infty(\R^d)}\leq  \|q_{m-1}\|_{L^\infty(\R^d)}\|q_{m-1}\|_{L^1(\R^d)}\leq \|q\|_{L^\infty(\R^d)}	\|q\|^{2^m-1}_{L^1(\R^d)}.
\end{align}
The uniform continuity $q_m$  follows from the continuity of the shift in $L^1(\R^d)$.

\noindent (ii) Since $q$ is not identically vanishing, we have 
\begin{align*}
q_1(0)= q*q(0)= \int_{\R^d} |q(h)|^2\d h>0
\end{align*}
\noindent wherefrom, the continuity of $q_1$ at $0$ implies that $\theta=\inf_{B_\delta(0)} q_1>0$ for some  $\delta>0$. We claim that $q_2(x)=q_1*q_1(x)>0$ whenever $|x|\leq \frac{7\delta}{4}$.  Indeed, note that   $B_{\frac{\delta}{16}} (\frac{x}{2})\subset  B_{\delta}(0) \cap B_{\delta}(x)$ since for $z\in  B_{\frac{\delta}{16}} (\frac{x}{2})$  we have 
\begin{align*}
\big|z-\frac{x}{2} \pm \frac{x}{2} \big|\leq \big|z-\frac{x}{2} \big| + \big|\frac{x}{2}\big|\leq \frac{15\delta}{16}<\delta.
\end{align*}
Given that $q_1(h)q_1(x-h)>\theta^2$ for all $h\in  B_{\delta}(0) \cap B_{\delta}(x)$ we  finally get 
\begin{align*}
q_2(x)=q_1*q_1(x)\geq \theta^2|B_{\frac{\delta}{16}} (0)|>0\quad\text{for all $|x|\leq \frac{7\delta}{4}$}.
\end{align*}
Repeating this process, one finds that for  $\theta_{m-1}=\inf_{B_{\delta_{m-1} }(0)} q_{m-1} (x)>0$ 
\begin{align*}
q_m(x)\geq \theta^2_{m-1}|B_{\frac{\delta_{m-1}}{16}} (0)|>0\quad\text{for all $|x|\leq \delta_m.$}
\end{align*}
\noindent (iii) The inequality $|a+b|^p\leq 2^{p-1}(|a|^p+|b|^p)$ and Fubini's theorem  yield
\begin{align*}
\mathcal{Q}_{m}(u)& = \iil_{\R^d \R^d} |u(x) -u(x+h)|^p \int_{\R^d} q_{m-1}(z)q_{m-1}(h-z)\d z\d h \d x\\
& \leq 2^{p-1}\iil_{\R^d \R^d} |u(x) -u(x+z)|^p\int_{\R^d} q_{m-1}(z)q_{m-1}(h-z)\d z\d h \d x\\
&+2^{p-1}\iil_{\R^d \R^d} |u(x+z) -u(x+h)|^p\int_{\R^d} q_{m-1}(z)q_{m-1}(h-z)\d z\d h \d x\\
&= 2^{p-1}\iil_{\R^d \R^d} |u(x) -u(x+z)|^p\Big(\int_{\R^d} q_{m-1}(\xi)\d \xi \Big) q_{m-1}(z)\d z \d x\\
&\overset{\xi=h-z}{+} 2^{p-1}\int_{\R^d} q_{m-1}(z)\d z \iil_{\R^d \R^d}  |u(x+z) -u(x+z+\xi)|^p q_{m-1}(\xi)\d \xi \d x
\\&=2^p \|q_{m-1}\|_{L^1(\R^d)}\mathcal{Q}_{m-1}(u). 
\end{align*}

\noindent In short, this combined with inequality \eqref{eq:mfold-young-inequality} implies 
\begin{align*}
\mathcal{Q}_{m}(u) &\leq 2^p \|q_{m-1}\|_{L^1(\R^d)}\mathcal{Q}_{m-1}(u)\\
&\leq \Big(\prod_{k=1}^{m}2^p \|q\|^{2^{k-1}}_{L^1(\R^d)}\Big) \mathcal{Q}(u)= 2^{pm} \|q\|^{2^m-1}_{L^1(\R^d)}\mathcal{Q}(u).
\end{align*}
\end{proof}

In the sequel, it is decisive to keep in mind that for $u\in L^p_\Omega(\R^d)$ we have 
\begin{align*}
\mathcal{E}(u,u)&=\hspace*{-3ex}\iil_{(\Omega^c\times\Omega^c)^c}\hspace*{-2ex}|u(x)-u(y)|^p\nuxminy\d y\,\d x=\hspace*{-1ex}\iil_{\R^d \R^d}\hspace*{-1ex}|u(x)-u(y)|^p\nuxminy\d y\,\d x.
\end{align*}
\begin{theorem}[Poincar\'{e}-Friedrichs inequality I]
\label{thm:poincare-friedrichs-ext-bis}
Assume that $\nu\not \equiv0$, i.e., $\big|\{\nu >0\}\big|>0$
and  that $\Omega$ is bounded one direction, i.e., there exist $R>0$ and  $e\in \R^d$, $|e|=1$ such that $\Omega\subset H_R$ with $H_R=\{ z\in \R^d: |z\cdot e|\leq R \}$. Then for $m\in \mathbb{N}$ large, there is $0<C_{R, m}:= C(d,p,R,m,\nu)<\infty$ such that 
\begin{align}\label{eq:poincare-friedrichs-ext}
\|u\|^p_{L^p(\Omega)}\leq C_{R, m} \mathcal{E}(u,u)\,\qquad\text{for all $u\in L^p_\Omega(\R^d)$}. 
\end{align}
Moreover,  letting  $q_m = q_{m-1}* q_{m-1}$, $q_0=q$, $q:=1\land \nu$ we can choose 
\begin{align*}
C_{R, m} = 2^{pm} \|q\|^{2^m-1}_{L^1(\R^d)}\Big(\int_{H^c_{2R} }q_m(h)\d h\Big)^{-1}.
\end{align*}
\end{theorem}

\begin{proof}
We have  $H_R\subset H_{2R}(x) =\{ z\in \R^d: |(z-x)\cdot e|\leq 2R\}$ for $x\in H_R$.  
Accordingly, for $u\in L^p_\Omega(\R^d)$ there holds
\begin{align*}
\mathcal{Q}_{m}(u) 
&\geq 2\int_{\Omega}|u(x)|^p \d x \int_{ H_R^c} q_m(x-y)\d y  \\
&\geq \int_{\Omega}|u(x)|^p \d x \int_{H_{2R}^c(x)} q_m(x-y)\d y= \|u\|^p_{L^p(\Omega)}\int_{H_{2R}^c} q_m(h)\d h. 
\end{align*}
The fact that $\nu\in L^1(\R^d\setminus B(0,r))$ implies $q=1\land \nu\in L^1(\R^d)\cap L^\infty(\R^d) $. In addition, $\nu\equiv0$ if and only if $q\equiv0$.  Thus, $q= 1\land \nu\leq \nu$, the estimate above and  Lemma \ref{lem:folded-convo} (iii)  yield
\begin{align*}
\cE(u,u)\geq \mathcal{Q}(u) \geq B^{-1}_m\mathcal{Q}_{m}(u) 
\geq  C^{-1}_{R,m}\|u\|^p_{L^p(\Omega)}.
\end{align*}
\noindent According to Lemma \ref{lem:folded-convo} (i), $ q_m\in L^1(\R^d)\cap L^\infty(\R^d)$ and by Lemma \ref{lem:folded-convo} (ii) for $m$ large we have  $\inf_{B_{\delta_m }(0)} q_m>0$ and $|B_{\delta_m }(0) \setminus H^c_{2R}|>0$ where we recall $\delta_m\xrightarrow{m\to \infty}\infty$. Hence 
\begin{align*}
0<	|B_{\delta_m }(0) \setminus H^c_{2R}|\big(\inf_{B_{\delta_m }(0)} q_m\big)\leq \int_{H^c_{2R} }q_m(h)\d h\leq \|q_m\|_{L^1(\R^d)}<\infty, 
\end{align*}
wherefrom we deduce that $0<C_{R,m}<\infty$.
\end{proof}

Next, we want to consider the case where $|\Omega|<\infty$. We need the following Lemma from \cite{Fog21b}. 
\begin{lemma}\label{lem:inequality-indicator}
Let $E\subset \R^d$ be measurable with $|E|<\infty$. Then we have 
\begin{align}\label{eq:nu-hash}
&\essinf_{x\in\R^d}\int_{E^c} \nu(x-y)\d y \geq \nu^\# (|E|), \quad\text{with} \quad \nu^\# (|E|)= \int_{ \{\nu <\nu^*(r_E)\} }\nu(h)\d h,
\end{align}	
where  $ r_E= \big(\frac{|E|}{|B_1(0)|}\big)^{1/d}$ and  $\nu^*$ is the  symmetric rearrangement of $\nu$ 
\begin{align}\label{eq:sym-rearran}
\nu^*(x)=\int_0^\infty \mathds{1}_{\{ |\nu| >s\}^* }(x)\d s= \inf\{s>0: ~|\{ |\nu| >s\}|\leq c_d|x|^d\}. 
\end{align}
\end{lemma}

It is worth noting that when  $\nu(h)= |h|^{-d-sp},$ $h\neq 0$ we have $\nu^*=\nu$ and 
\begin{align*}
\nu^\# (|E|)= \int_{ \{\nu <\nu^*(r_E)\} } \hspace*{-3ex} \nu(h)\d h
= \int_{|h|>r_E } \hspace*{-3ex}|h|^{-d-sp} \d h= \frac{1}{sp} \Big(\frac{|E|}{|B_1(0)|} \Big)^{-sp}|\mathbb{S}^{d-1}|. 
\end{align*}
Accordingly, Lemma \ref{lem:inequality-indicator} implies the following; see also \cite[Lemma 6.1]{Hitchhiker}, 
\begin{align*}
&\essinf_{x\in\R^d}\int_{E^c} \frac{\d y}{|x-y|^{d+sp}} \geq \frac{1}{sp} \Big(\frac{|E|}{|B_1(0)|} \Big)^{-sp}|\mathbb{S}^{d-1}|. 
\end{align*}	

\begin{theorem}[Poincar\'{e}-Friedrichs inequality II]\label{thm:poincare-friedrichs-finite-bis}
Assume $\nu\not \equiv0$, i.e., $\big|\{\nu >0\}\big|>0$ and  $|\Omega|<\infty$.  Then for every $m\in \mathbb{N}$ the following inequality holds 
\begin{align}\label{eq:poincare-friedrichs-ext-bis}
\|u\|^p_{L^p(\Omega)}\leq C^\#_{m} \mathcal{E}(u,u)\,\qquad\text{for all $u\in 	L^p_\Omega(\R^d)$},
\end{align}
where, for  sufficiently large $m$, we have  $0<C^\#_{m}= C(d,p,|\Omega|,m,\nu)<\infty$ with 	
\begin{align*}
C^\#_{m}  = 2^{pm} \|q\|^{2^m-1}_{L^1(\R^d)}\Big(q^\#_m(|\Omega|)\Big)^{-1}.
\end{align*}
Here we recall $q_m = q_{m-1}* q_{m-1}$, $q_0=q$ with $q:=1\land \nu= \min(1, \nu)$ and $q_m^\#$ is defined as in \eqref{eq:nu-hash}. 	 
\end{theorem}

\begin{proof}
Accordingly, for $u\in L^p_\Omega(\R^d)$ there holds that 
\begin{align*}
\mathcal{Q}_{m}(u) &\geq 2\int_{\Omega}|u(x)|^p \d x \int_{\Omega^c} q_m(x-y)\d y \geq  \|u\|^p_{L^p(\Omega)}q^\#_m(|\Omega|). 
\end{align*}
The assumptions on $\nu$ imply that $q=1\land \nu\in L^1(\R^d)\cap L^\infty(\R^d) $ and $\nu\equiv0$ if and only if $q\equiv0$.  Moreover, $q= 1\land \nu\leq \nu$, the above estimate and Lemma \ref{lem:folded-convo} (iii)  yield 
\begin{align*}
\cE(u,u)\geq \mathcal{Q}(u) \geq B^{-1}_m\mathcal{Q}_{m}(u) 
\geq  (C^\#_m)^{-1}\|u\|^p_{L^p(\Omega)}.
\end{align*}
\noindent By Lemma \ref{lem:folded-convo} (i), $ q_m\in L^1(\R^d)\cap L^\infty(\R^d)$ and hence $q^\#_m(|\Omega|)<\infty$. By Lemma \ref{lem:folded-convo} (ii) we have  $\delta_m>|\Omega|$ for $m$ large so that $q^\#_m(|\Omega|)\geq\inf_{B_{\delta_m }(0)} q_m>0$. Therefore, $0<	C^\#_m <\infty$.
\end{proof}

\subsection{Nonlocal Poincar\'{e} inequality}
In this section $\Omega\subset \R^d$ is open and bounded, $1\leq p<\infty$ and $\nu:\R^d\setminus\{0\}\to [0, \infty)$  is symmetric. Our goal in this section is to find some conditions on $\Omega$ and $\nu$ under which the following Poincar\'{e} inequality holds true, i.e., we can find a constant $C= C(d,p, \Omega,\nu)>0$ such that 
\begin{align}\tag{$P$}\label{eq:poincare-semi-gagliardo}
\|u-\mbox{$\fint_{\Omega}$}u\|^p_{L^p(\Omega)}\leq C\cE_\Omega(u,u) \,\quad\text{for all $u\in L^p(\Omega)$}. 
\end{align}  
Here and in what follows the notation $\mbox{$\fint_E$} u= \frac{1}{|E|}\int_E u(x)\d x$. So that since  $\cE_\Omega(u,u) \leq \cE(u,u)$, \eqref{eq:poincare-semi-gagliardo} would imply
\begin{align}
\label{eq:poincare-big-semi}
\|u-\mbox{$\fint_{\Omega}$}u\|^p_{L^p(\Omega)}\leq C\cE(u,u) \,\quad\text{for all $u\in L^p(\Omega)$}. 
\end{align}

We were unable to find a suitable reference of the following Lemma \ref{lem:chain-covering}, which is probably known in the literature. We however give a simple proof for the convenience of the reader.
\begin{lemma}[Finite chain covering]\label{lem:chain-covering}
Assume $\Omega\subset \R^d$ is connected. For every $r>0$ there is  a finite family of balls $(B_i)_{1\leq i\leq n}$, $B_i=B(x_i, r)$ covering $\Omega$ with $x_i\in \Omega$ such that $B_{i-1}\cap B_i\neq \emptyset$, $i=2,3,\cdots,n$.
\end{lemma}

\begin{proof}
It is readily seen that the balls $B(x,r)$, $x\in \Omega$ covers the compact set $\overline{\Omega}$ (the closure of $\Omega$). Thus, there is a finite sub-cover $\mathcal{B}= \{B_i\,:\, i\in \{1,2,\cdots,n\} \}$ of $\Omega$ with $B_i=B(x_i, r), x_i\in \Omega$. Next, let us write $B_{i-1}\sim B_i$ to indicate that $B_{i-1}\cap B_i \neq \emptyset$. Since $\Omega$ is connected, there are $B, B'\in \mathcal{B}$ such that $B\cap B'\neq \emptyset$. Up to a relabeling, denote the chain $C_2= B_1\sim B_2$ with $B_1=B, B_2=B'$. Assume, up to relabeling the indices, there is  a chain 
\begin{align*}
C_i= B_1\sim B_2\sim\cdots\sim B_i.
\end{align*} 
Given that $\Omega$ is connected, there is $B\in \mathcal{B}\setminus \bigcup_{j=1}^i B_j$ such that $B\cap \bigcup _{j=1}^i B_j\neq \emptyset$. Thus, we can consider $j_0=\max\{j\in \{1,2,\cdots,i\}\,: B_j\sim B\}$ and define a chain  $C_{i+1}$ as follows
\begin{alignat*}{2}
C_{i+1}&= B_1\sim\cdots \sim B_{j_0-1}\sim B\sim B_{j_0}\sim\cdots\sim B_i&&\quad\text{if $j_0\neq1$},\\
C_{i+1}&= B\sim B_1\sim B_2\sim\cdots\cdots\sim B_i&&\quad\text{if $j_0=1$}.
\end{alignat*}
In this manner, up to relabeling the indices, one gets a chain $C_{n}$  containing the whole $ \mathcal{B}$. 
\end{proof}
\smallskip 

We also need to consider the following generalization of connected sets.   
\begin{definition} 
We say that 
$\Omega\subset \R^d$ is $\rho-$connected, $\rho\geq 0$, if $\Omega= \Omega_1\cup \Omega_2\cup\cdots \cup\Omega_m$, $m\geq1,$ where each $\Omega_i$ is open and connected such that $\dist(\Omega_i, \Omega_{i+1})<\rho$, $i=1,2, \cdots,n-1$. 
\end{definition}

It is worthwhile to observe that  every connected set is $0$-connected and  the converse is not true. Consider $B_\pm= B_1(0)\cap \{x=(x', x_d): \pm x_d>0\}$ which are connected. Therefore, since $\dist(B_-,B_+)=0$ we see that $\Omega= B_-\cup B_+$ is $0$-connected  
but not connected. Another observation is that if  $\dist(\Omega_i, \Omega_{i+1})<\rho $ one finds $a_i\in \Omega_i, b_i\in \Omega_{i+1}$ such that $|a_i-b_i|<\rho$.

\begin{lemma}\label{lem:zero-enegery-cst} 
Assume $\nu> 0$ a.e. on $B_r(0)$ for some $r \in(0,\infty]$. Assume  one of the following  conditions holds. 
\begin{enumerate}[$(i)$]
\item $r\geq \diam(\Omega)$. 
\item $\Omega$ is connected.  
\item $\Omega$ is $\rho$-connected with $0\leq \rho< r$.
\end{enumerate}
There holds $\cE_\Omega(u,u)=0$ if and only if $u=\lambda$ a.e. in $\Omega$, $\lambda\in \R$.
\end{lemma}

\begin{proof}
Clearly, if $u=\lambda$ a.e. in $\Omega$ then $\cE_\Omega(u,u)=0$. Conversely, if $\cE_\Omega(u,u)=0$, there is a null set $A\subset \Omega$ such that 
\begin{align*}
\int_{\Omega} |u(x)-u(y)|^p\nu(x-y)\d y=0\qquad\text{for all $x\in \Omega\setminus A$}.
\end{align*}
For each $x\in \Omega\setminus A$ there is a null set $A_x\subset\Omega$ such that 
\begin{align*}
|u(x)-u(y)|^p\nu(x-y)=0\qquad\text{for all $y\in \Omega\setminus A_x$}.
\end{align*}
Since $\nu>0$ a.e. in $B_r(0)$, we can assume  up to renaming the null set $A_x$ that
\begin{align}\label{eq:u-constant-ae}
u(y)=u(x)\qquad\text{for $y\in B_r(x)\cap (\Omega\setminus A_x),\,\, x\in \Omega\setminus A$}.
\end{align}

$(i)$ If $r\geq \diam(\Omega)$ then for fixed $x_0\in \Omega\setminus A$ we have $\Omega\subset B_r(x_0)$ so that $u(x)=u(x_0)$ for all $x\in \Omega\setminus A_{x_0}$. Hence, $u= u(x_0)$ almost everywhere in $\Omega$. 

\smallskip 

$(ii)$ Assume $\Omega$ is connected. By Lemma \ref{lem:chain-covering}, we can cover $\Omega$ with a family of balls $(B_i)_{1\leq i\leq n}$, $B_i=B(x_i,r/2)$ such that $B_{i-1}\cap B_i \neq \emptyset,\,\,  i=2,3,\cdots, n.$ Thus, there is $z_i\in (B_{i-1}\cap B_i)\setminus A$. Consider the null set $A'=  A_{z_2}\cup A_{z_3}\cup\cdots\cup A_{z_n}$. 
Note that $B_{i-1}\cup B_i\subset B_r(z_i)$ and according to \eqref{eq:u-constant-ae} we get $u(x)=u(z_i)$ for all $x\in (B_r(z_i)\cap \Omega)\setminus A_{z_i}$. In particular, we have  $u(x)=u(z_i)$ for all $x\in(B_{i-1}\cup B_{i})\cap (\Omega\setminus A')$. By the same token,  $u(x)=u(z_{i+1})$ for all $x\in(B_{i}\cup B_{i+1})\cap (\Omega\setminus A')$. It turns out that $u(x)=u(z_i)= u(z_{i+1})$ for all $x\in B_{i}\cap (\Omega\setminus A')$. Finally, we deduce that $u(z_2)=u(z_3)=\cdots u(z_n) =\lambda$. Since $\Omega=\bigcup_{i=1}^n B_i\cap\Omega$,  we can conclude that  $u(x)= \lambda$ for all $x\in \Omega\setminus A'$ where we recall that $A'$ is a null set.

\smallskip 

$(iii)$ Now assume $\Omega$ is $\rho$-connected with $0\leq \rho<r$, i.e.  $\Omega=\Omega_1\cup\Omega_2\cdots\Omega_m$ where each $\Omega_i$ is open bounded, connected and $\dist(\Omega_i, \Omega_{i+1})<\rho$. The previous step implies that $u=\lambda_i$ a.e. on $\Omega_i$, $i=1,2,\cdots,m$. On the other hand, considering $0<\eps <\frac{r-\rho}{2}$ so that $0\leq \rho<\rho+2\eps <r$, we can find $a_i\in \Omega_i$ and $b_i\in \Omega_{i+1}$ such that $$\dist(\Omega_i, \Omega_{i+1})\leq|a_i-b_i|<\rho+\eps<r.$$  
It follows that $B(a_i, \eps)\cap \Omega_i\neq \emptyset$ and $B(a_i, \rho+\eps)\cap \Omega_{i+1}\neq \emptyset$. Given that $A$ is a null set, we can find $\xi_i\in (\Omega_{i+1}\cap B(a_i, \rho+\eps)) \setminus A$. Since $u=\lambda_{i+1}$ a.e. in $\Omega_{i+1}$, by \eqref{eq:u-constant-ae} we get
\begin{align*}
\text{$u(x)= u(\xi_i)= \lambda_{i+1}$ for all $x\in B(\xi_i,r)\cap(\Omega\setminus A_{\xi_i}) .$}
\end{align*}
Note that for $z\in B(a_i, \eps)$, since $\xi_i\in B(a_i, \rho+\eps)$, we have 
$$|z-\xi_i|\leq |z-a_i|+|a_i-\xi_i|< \rho+2\eps<r.$$ 
This implies $B(a_i, \eps)\subset B(\xi_i, r)$. Since $u=\lambda_i$ a.e. in $\Omega_i$, in particular we have 
\begin{align*}
\text{$\lambda_i= u(x)= u(\xi_i)= \lambda_{i+1}$\,\,  for \,\,  $x\in B_\eps(a_i)\cap(\Omega_{i}\setminus A_{\xi_i}) \subset  B_r(\xi_i)\cap(\Omega\setminus A_{\xi_i}) $}. 
\end{align*}
This shows that $\lambda_1= \lambda_2=\cdots=\lambda_n,$ and hence  we find $u=\lambda_1$ a.e. in $\Omega$. 
\end{proof}
 
The following result is some  general nonlocal Poincar\'{e} inequality. 
\begin{theorem}
\label{thm:poincare-gene}
Assume $\Omega$ is bounded and there is $r>0$ such that 
\begin{align}\label{eq:kappa}
\kappa:=\essinf_{h\in B_r(0)}\nu(h)>0.
\end{align}
If  $r\geq \diam(\Omega)$ or else $\Omega$ is $\rho$-connected with  $r>\rho\geq0$ then there is $C=C(d,p,\Omega,\nu)>0$ such that 
\begin{align}
\|u-\mbox{$\fint_{\Omega}$}u\|^p_{L^p(\Omega)} \leq C\cE_\Omega(u,u)\,\qquad\text{for all $u\in L^p(\Omega)$}.
\end{align} 
\end{theorem}

\begin{proof}
If $r\geq \diam(\Omega)$ then $\nu(x-y)\geq \kappa$ for all $x,y\in\Omega$. Hence Jensen's inequality yields 
\begin{align*}
\cE_\Omega(u,u)\geq\kappa \int_{\Omega}\int_{\Omega}|u(x)-u(y)|^p\d y\d x\geq \kappa|\Omega|\|u-\mbox{$\fint_{\Omega}u$}\|^p_{L^p(\Omega)}. 
\end{align*}
Now assume that $0<r<\diam(\Omega)$ and $\Omega$ is $\rho$-connected with $0<\rho\leq r$.  Assume such $C$ does not exist. We can find  a sequence $(u_n)_n\subset L^p(\Omega)$ such that  $\fint_{\Omega}u_n =0$, $\|u_n\|_{L^p(\Omega)}=1$ and $\cE_\Omega(u_n,u_n)\leq \frac{1}{2^n}$ and hence
\begin{align*}
\kappa \iint\limits_{\Omega \Omega}|u_n(x)-u_n(y)|^p\mathds{1}_{B_r(0)}(x-y)\d y\,\d x\leq \cE_\Omega(u_n,u_n)
\leq \frac{1}{2^n}\,.
\end{align*}
\noindent Since $(u_n)_n$ is bounded in $L^p(\Omega)$, passing through a subsequence if necessary we can assume that $(u_n)_n$ weakly converges in $L^p(\Omega)$ to some $u$, where $u\in L^p(\Omega)$ $p>1$ and $u$ is a  signed Radon measure on $\Omega$ for $p=1$.   In particular, $\fint_E u_n(x)\d x\to \fint_E u(x)\d x$ for every Borel set $E\subset \Omega$. In addition, we can assume there is a null set $A\subset \Omega$ such that 
\begin{align}\label{eq:conv-average}
\int_{B_r(x)\cap\Omega}|u_n(x)-u_n(y)|^p\d y\xrightarrow{n\to \infty}0\quad \text{for all $x\in \Omega\setminus A$}. 
\end{align}
Moreover, for each $x\in \Omega\setminus A$ there is another null set $A_x\subset \Omega$ such that 
\begin{align}\label{eq:conv-point}
u_n(x)-u_n(y) \xrightarrow{n\to \infty}0\quad \text{for all $y\in (B_r(x)\cap\Omega) \setminus A_x$}. 
\end{align}
Consider $(B_{r/2}(x_i))_{1\leq i\leq n }$ a finite cover of $\Omega$ with $x_i\in \Omega$. 
Choose $\xi_i\in (B_{r/2}(x_i)\cap\Omega) \setminus A$. Hence $(E_{\xi_i})_{1\leq i\leq n}$ with $E_{\xi_i}= B_{r}(\xi_i)\cap\Omega$ is also a finite cover of  $\Omega$.  Note that
\begin{align*}
\Big|u_n(\xi_i)- \fint_{E_{\xi_i}} u(y)\d y \Big|^p
&\leq  2^{p-1} \fint_{E_{\xi_i}} \big|u_n(\xi_i)-u_n(y)\big|^p\d y 
\\&+ 2^{p-1}\Big| \fint_{E_{\xi_i}} u_n(y)\d y - \fint_{E_{\xi_i}} u(y)\d y \Big|^p. 
\end{align*}
This together with  the weak convergence and the convergence in \eqref{eq:conv-average} imply 
\begin{align*}
u_n(\xi_i)\xrightarrow{n\to \infty} \lambda_{\xi_i}= \fint_{E_{\xi_i}} u(y)\d y.
\end{align*}
Whence by \eqref{eq:conv-point} we get that
\begin{align*}
u_n(y)\xrightarrow{n\to \infty} \lambda_{\xi_i} \quad \text{for all $y\in E_{\xi_i} \setminus A_{\xi_i}$}, \quad E_{\xi_i}= B_r(\xi_i)\cap\Omega.
\end{align*}
In other words, $u_n\to v$ a.e. in $\Omega$ where $v(x)= \lambda_{\xi_i} $ for $x\in E_{\xi_i}$. The Fatou's Lemma yields
\begin{align*}
\iint\limits_{\Omega\Omega}\hspace*{-1ex} |v (x)-v (y)|^p\nu(x-y)\d y\,\d x\leq 
\liminf_{n\to\infty} \hspace*{-1ex}\iint\limits_{\Omega\Omega}\hspace*{-1ex}|u_n(x)-u_n(y)|^p\nu(x-y)\d y\,\d x=0.
\end{align*}
It follows from Lemma \ref{lem:zero-enegery-cst}  that $v=\lambda$ in $\Omega$, that is,  $\lambda=\lambda_{\xi_i}= \cdots= \lambda_{\xi_n}$. On the other hand, since $\Omega=\bigcup_{i=1}^n E_{\xi_i}$, this combined with \eqref{eq:conv-average} gives  
\begin{align*}
\|u_n-v\|_{L^p(\Omega)}
&\leq \sum_{i=1}^{n}  \Big(\int_{E_{\xi_i}} |u_n(\xi_i)-u_n(y)|^p\d y\Big)^{1/p}
\\&+ |E_{\xi_i}|^{1/p} |u_n(\xi_i)-\lambda_{\xi_i}|\xrightarrow{n\to \infty}0. 
\end{align*}
That is, $(u_n)_n$ strongly converges to $v=\lambda$ in $L^p(\Omega)$. Taking into account the weak convergence we deduce that $u=v= \lambda$. Since $\fint_{\Omega}u_n =0$ and $\|u_n\|_{L^p(\Omega)}=1$,  it follows that $u=\fint_{\Omega}u =0$ and $\|u\|_{L^p(\Omega)}=1$ which is a contradiction. 
\end{proof}

\begin{theorem}\label{thm:poincare-uni}
Assume $\Omega\subset \R^d$ is connected or $0$-connected. Assume $\nu$ is unimodal and $|\{\nu>0\}|>0$, i.e., $\nu\not\equiv 0$. Then there is  $C=C(d,p,\Omega,\nu)>0$ such that 
\begin{align}
\|u-\mbox{$\fint_{\Omega}$}u\|^p_{L^p(\Omega)} \leq C\cE_\Omega(u,u)\,\qquad\text{for all $u\in L^p(\Omega)$}.
\end{align}
\end{theorem}
\begin{proof}
Since $\nu\not\equiv 0$  there is $\xi\in \R^d$, $\xi\neq 0$ such that $\nu(\xi)>0$. Let $r=|\xi|>0$, since $\nu$ is unimodal we have  $\nu(h)\geq c\nu(\xi)$ for all $h\in B_r(0)$.  Therefore,  we have $\kappa= \essinf_{h\in B_r(0)}\nu(h)>0$ where $\kappa= c\nu(\xi)>0$.  The desired result follows from Theorem \ref{thm:poincare-gene}.  
\end{proof}

\begin{corollary}[Fractional Poincar\'{e} inequality]\label{cor:poincare-frac-bis}
Let $r\in (0,\infty]$ and $\Omega\subset \R^d$ be bounded. Assume either $r\geq \diam(\Omega)$ or  $\Omega$ is $\rho$-connected with $0\leq \rho<r$. Then there is  $C=C(d,p,\Omega, s,r)>0$ such that 
\begin{align*}
\|u-\mbox{$\fint_{\Omega}$}u\|^p_{L^p(\Omega)} \leq C\iil_{\Omega \Omega}\frac{|u(x)-u(y)|^p}{|x-y|^{d+sp}}\mathds{1}_{B_r (0)}(x-y)\d y\d x\,\, \text{for all $u\in L^p(\Omega)$}.
\end{align*}
\end{corollary}
\begin{proof}
	The claim is a direct consequence of Theorem \ref{thm:poincare-gene} applied to $\nu(h)= |h|^{-d-sp}\mathds{1}_{B_r(0)}(h)$.  
\end{proof}
\begin{remark}
One readily recovers the usual fractional Poincar{\'e} inequality holding true for any arbitrary bounded $\Omega$. Indeed, taking $r\geq \diam(\Omega)$ implies  $\mathds{1}_{B_r (0)}(x-y)=1$ for all $x,y\in\Omega$ so that 
\begin{align*}
\|u-\mbox{$\fint_{\Omega}$}u\|^p_{L^p(\Omega)} \leq C\iil_{\Omega \Omega}\frac{|u(x)-u(y)|^p}{|x-y|^{d+sp}}\d y\,\d x\,\qquad\text{for all $u\in L^p(\Omega)$}.
\end{align*} 
\end{remark}

If the assumption \eqref{eq:kappa} fails i.e., $\kappa=0$ for any $r$, we can balance this deficiency with the compactness. 
\begin{theorem}
\label{thm:poincare-gene-bis}Let $r>\rho\geq0$ and assume $\Omega$ is bounded. Assume that 
\begin{itemize}
\item $\nu> 0$ a.e. on $B_r(0)$,
\item  $r\geq \diam(\Omega)$ or that $\Omega\subset \R^d$ is $\rho$-connected,
\item the embedding $\WnuOm\hookrightarrow L^p(\Omega)$ is compact. 
\end{itemize}
Then there is $C=C(d,p,\Omega,\nu)>0$ such that 
\begin{align*}
\|u-\mbox{$\fint_{\Omega}$}u\|^p_{L^p(\Omega)} \leq C\iil_{\Omega \Omega}|u(x)-u(y)|^p\nu(x-y)\d y\,\d x\,\qquad\text{for all $u\in L^p(\Omega)$}.
\end{align*} 
\end{theorem}

\begin{proof}
Assume $C$ does not exist, then we can find a sequence $(u_n)_n\subset L^p(\Omega)$ such that  $\fint_{\Omega}u_n =0$, $\|u_n\|_{L^p(\Omega)}=1$ and $\cE_{\Omega}(u_n,u_n)\leq \frac{1}{2^n}.$ The sequence $(u_n)_n$ is thus bounded in $\WnuOm$. Since the embedding  $\WnuOm\hookrightarrow L^p(\Omega)$ is compact, passing through a subsequence if necessary we can assume that $(u_n)_n$ converges in $L^p(\Omega)$ and a.e. in $\Omega$ to some  $u\in L^p(\Omega)$. It clearly follows that $\fint_{\Omega}u =0$ and $\|u\|_{L^p(\Omega)}=1$. Moreover, by Fatou's lemma, we have 
\begin{align*}
\cE_{\Omega}(u, u)\leq
\liminf_{n\to\infty}\cE_{\Omega}(u_n,u_n)=0.
\end{align*}
It follows from Lemma \ref{lem:zero-enegery-cst} that  $u$ is constant and hence $u= \fint_{\Omega}u =0$ a.e. in $\Omega$. This goes against the fact that $\|u\|_{L^p(\Omega)}=1$ and hence  our initial assumption is wrong.
\end{proof}

The Poincar\'{e} inequality fails if $\Omega$ is $\rho$-connected with $\rho$ too large.  
\begin{counterexample} 
Assume $d\geq2$ and consider $\Omega=\Omega_1\cup \Omega_2$ with $\Omega_1= B_1(0)$ and $\Omega_2=B_4(0)\setminus B_3(0)$ and $\nu(h)= \mathds{1}_{B_1(0)} (h)$. 
The sets $\Omega_1$ and $\Omega_2$ are connected thus $\Omega$ is $\rho$-connected with $\rho>2$. The Poincar\'{e} inequality fails for the couple $(\Omega, \nu)$. Indeed, for $u(x)=\mathds{1}_{\Omega_1}(x)- \mathds{1}_{\Omega_2}(x) $. 
As $u$ is not constant on $\Omega$, it is readily seen that  $\|u-\mbox{$\fint_{\Omega} u$}\|^p_{L^p(\Omega)}\neq 0$.  On the other hand, $u=1$  on $\Omega_1$ and $u=-1$ on $\Omega_2$, moreover $\mathds{1}_{B_1(0)} (x-y)=0$ if $x\in \Omega_1$ and $y\in \Omega_2$.  Therefore, one easily verifies that $\cE_\Omega(u,u)=0.$

\smallskip 
In general, the Poincar\'{e} inequality fails if $\Omega=\Omega_1\cup\Omega_2$, $\dist(\Omega_1, \Omega_2)> r$ and $\supp\nu \subset  \overline{B_r(0)}$. Indeed as above, for $u(x)=\mathds{1}_{\Omega_1}(x)- \mathds{1}_{\Omega_2}(x)$ we have  $\|u-\mbox{$\fint_{\Omega} u$}\|^p_{L^p(\Omega)}\neq 0$ and $\cE_\Omega(u,u)=0.$
\end{counterexample}

\section{Existence of weak solutions}\label{sec:existence-weak-sol}
In this section, we deal with the well-posedness of nonlocal problems \eqref{eq:main-problem}. It is worth mentioning that our setting easily extends to the general class of symmetric operators $L$ and  $\cN$, of the form
\begin{align*}
Lu(x) &=2\pv\int_{\R^d}\psi(u(x)-u(y))k(x,y)\d y,\\
\cN u(x)&=\int_{\R^d}\psi(u(y)-u(x))k(x,y)\d x
\end{align*}
where, for some $\Lambda\geq1$, $k:\R^d\times\R^d\setminus\diag\to (0,\infty)$ is a symmetric kernel satisfying the elliptic condition 
\begin{align*}
\Lambda^{-1}\nu(x-y)\leq k(x,y)\leq \Lambda\nu(x-y)\qquad\text{for all $(x,y)\in \R^d\times\R^d\setminus\diag$}.
\end{align*}
We only consider case $k(x,y)=\nu(x-y)$ and leave this generalization to the good hand of the readers.

\subsection{Basics on direct method}
Our strategy for the existence results   relies on the direct method of calculus of variations. Let us evoke some fundamental results of calculus of variation mainly collected from \cite{BC17,ET76,Ri18}.
\begin{definition} Let $(X,\tau)$ be a topological space and  $\mathcal{J}:X\to \R\cup \{\infty\}$ be a functional. 
\begin{itemize}
\item
$\mathcal{J}$ is called sequentially $\tau$-coercive (or simply coercive) if lower level sets of $\mathcal{J}$ are $\tau$-sequentially  precompact, i.e., every  sequence $(u_n)_n\subset \{ u\in X: \mathcal{J}(u)\leq a\}$,  $a\in \R$, has a $\tau$-converging subsequence in $X$. 
\item
$\mathcal{J}$ is called sequentially $\tau$-lower semicontinuous (or simply lower semicontinuous) if for any sequence $(u_n)_n$ $\tau$-converging to $u$ in $X$, it holds that
\begin{align*}
\mathcal{J}(u)\leq \liminf_{n\to\infty} \mathcal{J}(u_n).
\end{align*}
\end{itemize}
\end{definition}

\noindent The direct method for the minimization problem 
is encapsulated in the following result. 
\begin{theorem}\label{thm:direct-method}
Let $(X,\tau)$ be a topological space. If the functional $\mathcal{J}:X\to \R\cup \{\infty\}$ is both coercive and lower semicontinuous, then there is $u_*\in X$ and 
\begin{align*}
\mathcal{J}(u_*)= \min_{u\in X} \mathcal{J}(u).
\end{align*} 
\end{theorem}
\noindent Next we see characterizations of the weak lower semicontinuity and the weak coercivity for the particular  situation where $(X,\|\cdot\|_X)$  is a normed space. Recall on $X$ there is the strong topology induced by $\|\cdot\|_X$  and the weak topology induced by the weak convergence. Let us recall  Mazur's Lemma beforehand. 
\begin{theorem}[{Mazur's Lemma, \cite[page 6]{ET76}}]\label{thm:Banach-Mazur}
If a sequence  $(u_n)_n$ is weakly converge in $X$ to some $u\in X$, then for each $n$ there is $(\theta_k^n)_{k=n,\cdots, N_n}$, $N_n\geq n$, such that $\|v_n-u\|_X\xrightarrow{n\to\infty}0$ where for each $n,$
\begin{align*}
v_n=\sum_{k=n}^{N_n}\theta_k^nu_k \quad \text{with} \quad \sum_{k=n}^{N_n} \theta_k^n=1,\quad \theta_k^n \geq 0.   
\end{align*}
\end{theorem}

In Mazur's Lemma, the coefficients $(\theta_k^j)'s$ are not explicitly given.  However, if $X$ is a $L^p$-Lebesgue space with $1\leq p<\infty$, then the Banach-Saks theorem as stated in \cite{Fog23BS} infers that for a suitable subsequence $(u_{n_k})_k$, one can explicitly choose $v_j =\frac{1}{j}\sum_{k=1}^j u_{n_k}$, that is we have  $\theta_k^j\in\{0, \frac{1}{j}\}$. The Mazur's Lemma implies the following characterization of the weak lower semicontinuity.

\begin{theorem}
\label{thm:charac-weak-lsc}
A  weakly lower semicontinuous  functional $\mathcal{J}:X\to \R\cup \{\infty\}$ is lower semicontinuous. The converse is true if in addition $\mathcal{J}$ is convex. 
\end{theorem}

\begin{theorem}
\label{thm:charac-weak-coercive}
If $\mathcal{J}:X\to \R\cup \{\infty\}$ is weakly coercive then $\mathcal{J}(u)\to \infty$ as $\|u\|_X\to \infty.$ The converse holds if in addition $X$ is a reflexive Banach space. 
\end{theorem}

\subsection{Neumann problem}
The Neumann problem for the operator $L$ associated with the data $f:\Omega\to \R$ and $g: \Omega^c \to \mathbb{R}$, is to find $u:\R^d\to \R$ such that 
\begin{align}\label{eq:nonlocal-Neumann}\tag{$N$}
L u = f \quad\text{in}~~~ \Omega \quad\quad\text{ and } \quad\quad \mathcal{N} u= g ~~~ \text{ on }~~~ \R^d\setminus\Omega.
\end{align} 
It is worth to emphasize that problem \eqref{eq:nonlocal-Neumann} makes sense only if have $g= 0$ on ${\R^d\setminus \Omega_\nu}=0$. Recall that $\Omega_\nu=\Omega+\supp\nu$ and  $\Omega_\nu^{*}=\Omega\cup \Omega_\nu$ is the nonlocal hull of $\Omega$. Keep in mind that $ \Omega\subset \Omega_\nu=\Omega_\nu^{*}$ if $0\in\supp\nu$. 
This is obviously due to the  fact that by nonlocality we have 
\begin{align*}
\cN u(x)=\int_{\Omega} \psi(u(x)- u(y))\nu(x-y)\d y =0\qquad\text{for all $x\in \R^d\setminus\Omega_\nu$},
\end{align*} 
we recall $\psi(t)= |t|^{p-2} t$. By nonlocality, cf. Section \ref{sec:basicsandnotations}, it turns out that $\Omega_\nu$ is the smallest set such that
\begin{align}\label{eq:nonlocal-hullnu}
L_{\Omega_\nu} u(x) :=\int_{\Omega_\nu} \psi(u(x)- u(y))\nu(x-y)\d y = Lu(x)\quad \text{for all $x\in \Omega$}. 
\end{align}
Therefore, it is sufficient to prescribe the complement data on $\Omega_e=\Omega_\nu\setminus\Omega$. It is rather counterintuitive to see that prescribing the nonlocal boundary data on $\Omega_e$ is not all restrictive. Indeed, if we put $g_e= g|_{\Omega_e}$, the restriction of $g$ on $\Omega_e$, the problem \eqref{eq:nonlocal-Neumann} is the same as finding $u:\Omega_\nu^{*}= \Omega\cup \Omega_\nu\to \R$  such that 
\begin{align}\label{eq:nonlocal-Neumann-ex}\tag{$N_\nu$}
L_{\Omega_\nu^{*}} u = f \quad\text{in}~~~ \Omega \quad\quad\text{ and } \quad\quad \mathcal{N} u= g_e
~~~ \text{ on }~~~ \Omega_e=\Omega_\nu\setminus\Omega,
\end{align} 
where we recall  $\Omega_e=\Omega_\nu^{*}\setminus\Omega= \Omega_\nu\setminus\Omega$ is nonlocal boundary of $\Omega$ with respect to $\nu$. Actually, both  problems \eqref{eq:nonlocal-Neumann}  and \eqref{eq:nonlocal-Neumann-ex} are equivalent.
Indeed, if $u$ solves \eqref{eq:nonlocal-Neumann} then  clearly $u_\nu=u|_{\Omega_\nu^{*}}$ solves \eqref{eq:nonlocal-Neumann-ex}. Conversely, if $u_\nu$ solves \eqref{eq:nonlocal-Neumann-ex}, one verifies that $u= \widetilde{u}_\nu$ solves \eqref{eq:nonlocal-Neumann} where $\widetilde{u}_\nu$ is the zero the extension of $u_\nu$ off $\Omega_\nu^{*}$, i.e.,  $\widetilde{u}_\nu= u_\nu$ on $\Omega_\nu^{*}$ and $\widetilde{u}_\nu=0$ on $\R^d\setminus\Omega_\nu^{*}$.  From now on, the problem \eqref{eq:nonlocal-Neumann} is understood in the sense of \eqref{eq:nonlocal-Neumann-ex}. Motivated by the Gauss-Green formula \eqref{eq:gauss-green-nonlocal}, in  Appendix \ref{sec:appendix-gauss-green},
\begin{align*}
	\int_{\Omega} Lu(x)v(x)\d x= \mathcal{E}(u,v) -\int_{\Omega^c}\mathcal{N}u(y)v(y)\d y, 
\end{align*}
 we define weak solutions of the Neumann problem as follows.

\begin{definition}\label{def:neumann-var-sol} Let $ f \in \WnuOmR'$ and $g \in \TnuOm'$. We call $u \in \WnuOmR$  a weak solution or a variational solution of Neumann problem \eqref{eq:nonlocal-Neumann} if 
\begin{align}\label{eq:var-nonlocal-Neumann-gen}\tag{$V'$}
\mathcal{E}(u,v) = \langle f , v \rangle  + \langle g , v \rangle  \quad \mbox{for all}~~v \in \WnuOmR\,. 
\end{align}
Note that if $\Omega$ is bounded, the existence of a solution implies the compatibility condition $\langle f , 1 \rangle  + \langle g , 1 \rangle = 0.$ 
\end{definition}
It is worth emphasizing that, the choice of $f$ and $g$ legitimately  follows from the natural embeddings  $\WnuOmR \hookrightarrow \TnuOm\hookrightarrow L^p(\Omega^c,  \omega)$. In particular, we have the following. 
\begin{definition}
Let   $\omega\in\{\widetilde{\nu}_B, \overline{\nu}_B, \widehat{\nu}_R\}$, $B\subset\Omega$ (see Definition \ref{def:different-nus}). Let $ f \in L^{p'}(\Omega)$ and $g \in L^{p'}(\Omega^c, \omega^{1-p'})$.  We say that $u \in \WnuOmR$ is a weak solution to the Neumann problem  \eqref{eq:nonlocal-Neumann} if 
\begin{align}\label{eq:var-nonlocal-Neumann}\tag{$V$}
\mathcal{E}(u,v) = \int_{\Omega} f(x)v(x)\d x +\int_{\Omega^c} g(y)v(y)\d y,\quad \mbox{for all}~~v \in \WnuOmR\,.
\end{align}
In this case,  taking $v=1,$ the compatibility condition reads
\begin{align}\label{eq:compatible-nonlocal}\tag{$C$}
\int_{\Omega} f(x)\d x +\int_{\Omega^c} g(y)\d y=0 \,.
\end{align}
\end{definition} 

\noindent Some comments and remarks about the Neumann problem may be helpful at this stage. 
\begin{remark}
Note that \cite[Def. 3.6]{DROV17} and subsequent definitions like \cite[Definition 2.7]{MP19} look very similar to \eqref{eq:var-nonlocal-Neumann} at first glance. However,  the test space defined in \cite[Eq. (3.1)]{DROV17}, \cite[Section 2]{MP19} depends on the Neumann data $g$, which is not natural. Our test space $\WnuOmR$ in the weak formulation \eqref{eq:var-nonlocal-Neumann} does not depend on the Neumann data $g$. 
\end{remark}

\begin{remark}
The compatibility condition $\langle f, 1 \rangle  + \langle g, 1 \rangle = 0$ or \eqref{eq:compatible-nonlocal} is an implicit requirement that the data $f$ and $g$ must fulfill before any attempt of solving the problems \eqref{eq:nonlocal-Neumann},                             
\eqref{eq:var-nonlocal-Neumann} or \eqref{eq:var-nonlocal-Neumann-gen}. This is essentially due to the fact that the operators $L$ and $\mathcal{N}$ annihilate additive constants. An immediate effect is that, if $f$ and $g$ are compatible, as long as $u$ is a solution to the problem \eqref{eq:nonlocal-Neumann}, \eqref{eq:var-nonlocal-Neumann} or \eqref{eq:var-nonlocal-Neumann-gen} so is any function $u+c$ with $c\in \mathbb{R}$. Accordingly, problems \eqref{eq:nonlocal-Neumann}, \eqref{eq:var-nonlocal-Neumann} or \eqref{eq:var-nonlocal-Neumann-gen} are ill-posed in the sense of Hadamard. One of the finest strategies to overcome this issue is to introduce the reduced problem 
\begin{align}\label{eq:var-nonlocal-Neumann-bis}\tag{$V'^\perp$}
\mathcal{E}(u,v) = \langle f , v \rangle  +\langle g , v \rangle ,\quad \mbox{for all}~~v \in \WnuOmR^{\perp}\,
\end{align}
where the space $\WnuOmR^{\perp}$ is given by  

\begin{align*}
\WnuOmR^{\perp}:= \big\{ u\in \WnuOmR: \int_{\Omega}u(x)\d x=0\big\}.
\end{align*}
Note that $\WnuOmR^{\perp}$  only discards non-zero constants from $\WnuOmR$ and  is a closed subspace of $\WnuOmR$. 
\noindent It is worth to emphasize that, if $f$ and $g$ are compatible and  $u$ is a solution to reduced problem \eqref{eq:var-nonlocal-Neumann-bis} then  each $u+c,\, c\in \R$ is also solution to \eqref{eq:var-nonlocal-Neumann-gen} and vice versa.  
\end{remark}

\begin{remark}
The nonlocal formulation of the Neumann problem should be contrasted with the local one, where the operators $L$ and $\mathcal{N}$ respectively replaced by the operators $-\Delta_p$ and $\partial_{n,p}u(x)= |\nabla u(x)|^{p-2} \nabla u(x)\cdot n(x)$. Recall that in passing that  $u\in W^{1,p}(\Omega)$  is a  weak  to the Neumann problem 
\begin{align}\label{eq:local-Neumann}
-\Delta_p u = f \quad\text{in}~~~ \Omega \quad\quad\text{ and } \quad\quad \partial_{n,p}u= g ~~~ \text{ on }~~~ \partial \Omega.
\end{align}
if $u$ satisfies the variational problem
\begin{align}
\cE^0(u,v)
= \int_{\Omega} f(x)v(x)\d x +\int_{\partial\Omega} g(y)v(y)\d \sigma(y),\quad \mbox{for all}~~v \in W^{1,p}(\Omega)\,, 
\end{align} 
here we define
\begin{align*}
\cE^0(u,v)	=\int_{\Omega} |\nabla u(x)|^{p-2} \nabla u(x)\cdot \nabla v(x) \d x. 
\end{align*}
The local counterpart of the compatibility condition \eqref{eq:compatible-nonlocal} is given by 
\begin{align}
\int_{\Omega} f(x)\d x +\int_{\partial\Omega} g(y)\d \sigma(y)=0.
\end{align} 
Define $W^{1,p}(\Omega)^{\perp}:= \big\{ u\in W^{1,p}(\Omega): \int_{\Omega}u(x)\d x=0\big\}.$ The local counterpart of \eqref{eq:var-nonlocal-Neumann-bis} reads
\begin{align}
\cE^0(u,v)
= \int_{\Omega} f(x)v(x)\d x +\int_{\partial\Omega} g(y)v(y)\d \sigma(y),\quad \mbox{for all}~~v \in W^{1,p}(\Omega)^\perp\,. 
\end{align} 
\end{remark}

For $\Omega$ and $u$ sufficiently regular, $p\geq2$, it  can be shown that (see for instance \cite{FK22, guy-thesis}), $u$ solves the problem  \eqref{eq:nonlocal-Neumann}  if and only if it solves the problem  \eqref{eq:var-nonlocal-Neumann}. This is in particular possible under the condition that the Gauss-Green formula \eqref{eq:gauss-green-nonlocal} holds. To put it simply, under mild conditions on $\Omega$ and $\nu$ both problems \eqref{eq:nonlocal-Neumann} and  \eqref{eq:var-nonlocal-Neumann}  are equivalent when the solutions are  sufficiently regular enough. Actually, solutions to the variational problem \eqref{eq:var-nonlocal-Neumann-gen}  are critical points of the functional 
\begin{align*}
\mathcal{J}(v) &= \frac{1}{p} \mathcal{E}(v,v) -\langle f , v \rangle  -\langle g , v \rangle .
\end{align*}
It is important to keep in mind that the  Fr\'echet derivative of  $\mathcal{J}$ on $\WnuOmR$ is given by
\begin{align*}
\langle	\mathcal{J}'(u), v\rangle= \cE(u,v) -\langle f , v \rangle  - \langle g , v \rangle. 
\end{align*}

\begin{proposition}\label{prop:min-equiv-var}
The variational problem \eqref{eq:var-nonlocal-Neumann-bis} is equivalent to the minimization problem
\begin{align}\label{eq:nonlocal-Neumann-min}\tag{$M^\perp$}
\mathcal{J}(u) =\min_{v\in \WnuOmR^{\perp}}\mathcal{J}(v). 
\end{align}
Analogously, the problem \eqref{eq:var-nonlocal-Neumann-gen} is equivalent to the minimization problem 
\begin{align}\label{eq:nonlocal-Neumann-min-const}\tag{$M$}
\mathcal{J}(u) =\min_{v\in \WnuOmR}\mathcal{J}(v).
\end{align}
\end{proposition} 
\begin{proof}
Let $ u,v\in\WnuOmR^{\perp}$. Assume \eqref{eq:var-nonlocal-Neumann-bis} holds. H\"{o}lder's and Young's inequalities imply 
\begin{align*}
\mathcal{E}(u,v) 
\leq \mathcal{E}(u,u)^{1/p'}\mathcal{E}(v,v)^{1/p}
&\leq \frac{1}{p'} \mathcal{E}(u,u)+ \frac{1}{p} \mathcal{E}(v,v)
\\&= \mathcal{E}(u,u)- \frac{1}{p} \mathcal{E}(u,u)+ \frac{1}{p} \mathcal{E}(v,v).
\end{align*} 
By virtue of \eqref{eq:var-nonlocal-Neumann-bis} which holds for $u$ and $v$ we get $\mathcal{J}(u)\leq \mathcal{J}(v)$ and thus $u $ solves \eqref{eq:nonlocal-Neumann-min}.
Conversely let $ u $ satisfies \eqref{eq:nonlocal-Neumann-min}, i.e.,  $\mathcal{J}(u)\leq \mathcal{J}(v)$ for  $ v\in\WnuOmR^{\perp}$. In particular, $\mathcal{J}(u) \leq \mathcal{J}(u+tv)$ for all $t\in \R$. Hence, the mapping  $t\mapsto \mathcal{J}(u+t v)$ is differentiable and  has a critical point at $t=0$. It follows that  $u$ satisfies  \eqref{eq:var-nonlocal-Neumann-bis} since
\begin{align*}
\cE(u,v) -\langle f , v \rangle  - \langle g , v \rangle= \langle	\mathcal{J}'(u), v\rangle= \lim_{t\to 0}\frac{\mathcal{J}(u+tv)-\mathcal{J}(u)}{t} =0. 
\end{align*}
The equivalence between \eqref{eq:var-nonlocal-Neumann} and  \eqref{eq:nonlocal-Neumann-min-const} can be proved analogously. 
\end{proof}

\noindent We are now in position to state the well-posedness of the problems  \eqref{eq:var-nonlocal-Neumann-bis},  \eqref{eq:var-nonlocal-Neumann-gen} and  \eqref{eq:var-nonlocal-Neumann}. Given a weight $\omega$, we opt for the convention that $\omega^{-1}(x)=0$ whenever $\omega(x)=0$. 
\begin{theorem} \label{thm:nonlocal-Neumann-var}
Let $\omega\in \{\widetilde{\nu}_{\Omega}, \overline{\nu}_{\Omega}, \widehat{\nu}_{R}\}$ where $|B_R(0)\cap \Omega|>0$ $($see Definition \ref{def:different-nus}$)$. 
Assume that \eqref{eq:positive-Oe} holds so that $\WnuOmR\equiv W^p_\nu(\Omega|\Omega_\nu)$ is a reflexive Banach space and that Poincar\'{e} inequality \eqref{eq:poincare-semi-gagliardo} holds (see page \pageref{eq:poincare-semi-gagliardo}).  Assume  $f\in \WnuOmR'$ $($or $ f \in L^{p'}(\Omega))$  and  $g\in \TnuOm'$ $($or $g\in L^{p'}(\Omega^c, \omega^{1-p'}))$. The following assertions hold.
\begin{enumerate}[$(i)$]
\item  
\textbf{Existence}. 
There is a unique  $u\in \WnuOmR^\perp$ satisfying \eqref{eq:var-nonlocal-Neumann-bis}. 
The problem \eqref{eq:var-nonlocal-Neumann-gen} has a solution $w\in\WnuOmR$ if and only if  the compatibility  $\langle f , 1 \rangle  + \langle g , 1 \rangle=0$ holds  and $w=u+c$ for some $c\in \R.$
\item \textbf{Boundedness}. 
There is $C= C(d,p,\Omega,\nu)>0$ such that any solution  $w$ to the problem \eqref{eq:var-nonlocal-Neumann-gen} satisfies 
\begin{align}\label{eq:weak-sol-bounded}
\|w-\hbox{$\fint_{\Omega}w$} \|_{\WnuOmR}\leq C \Big(\|f\|_{\WnuOmR'}+\|g\|_{\TnuOm'}\Big)^{1/(p-1)}.
\end{align}
\item \textbf{Continuity}.  
Moreover, if $u_i= w_i-\hbox{$\fint_{\Omega}w_i$}$, $i= 1,2$ where $w_i$ satisfies \eqref{eq:var-nonlocal-Neumann-gen} with $f=f_i$ then we have 
\begin{align}\label{eq:weak-sol-conti}
\|u_1-u_2\|_{\WnuOmR}\leq 
\begin{cases}
C \big(\|f_1-f_2\|_{\WnuOmR'}+\|g_1-g_2\|_{\TnuOm'}\big)^{1/(p-1)} & \hspace*{-2ex} \text{$p\geq2$},\\
CM\big( \|f_1-f_2\|_{\WnuOmR'}+  \|g_1-g_2\|_{\TnuOm'}\big) & \hspace*{-2ex}\text{$p<2$},
\end{cases}
\end{align}
where 
$M=M(f_1,f_2,g_1,g_2)= \big(\sum_{i=1}^2 \|f_i\|_{\WnuOmR'}+\|g_i\|_{\TnuOm'}\big)^{\frac{2-p}{p-1}}.$
\item \textbf{Problem \eqref{eq:var-nonlocal-Neumann}.} In the case $f\in L^{p'}(\Omega)$ and $g\in L^{p'}(\Omega^c, \omega^{1-p'})$ identified with the forms $v\mapsto\langle f,v\rangle= \int_\Omega f(x)v(x)\d x$ and  $v\mapsto \langle g,v\rangle= \int_{\Omega^c} g(y)v(y)\d y$ respectively. Then $f\in \WnuOmR'$, $ g\in \TnuOm'$ and there is $C_1=C_1( d, p,\Omega,\nu)$ such that 
\begin{align*}
\|f\|_{\WnuOmR'}\leq \|f\|_{L^{p'}(\Omega)}\quad\text{and}\quad 
\|g\|_{\TnuOm'}	\leq C_1\|g\|_{L^{p'}(\Omega^c,\omega^{1-p'})}.
\end{align*}
In other words, both problems \eqref{eq:var-nonlocal-Neumann} and \eqref{eq:var-nonlocal-Neumann-gen} are identical. 
\end{enumerate}
\end{theorem}

\begin{proof}We emphasize that throughout the proof,  $C>0$ denotes a generic constant only depending on the constant from the Poincar\'{e} inequality and $p$. 
\smallskip 

$(i)$ Since $ \WnuOmR\hookrightarrow\TnuOm$ we have $g\in \TnuOm' \hookrightarrow \WnuOmR'$, i.e., $g\in \WnuOmR'$. Thus, the functional 
$$v\mapsto \mathcal{J}(v) = \frac{1}{p} \mathcal{E}(v,v) -\langle f , v \rangle  -\langle g , v \rangle$$
is clearly continuous (hence lower semicontinuous)  and convex on $\WnuOmR$ a fortiori on $\WnuOmR^\perp$. According to Theorem \ref{thm:charac-weak-lsc}, $\mathcal{J}$ is weakly lower semicontinuous. On the other hand, in virtue of the Poincar\'{e} inequality \eqref{eq:poincare-semi-gagliardo}  one readily finds a constant $C>0$
\begin{align}\label{eq:coercivity-est}
C\|v-\mbox{$\fint_{\Omega} v$}\|^p_{\WnuOmR}\leq \cE(v,v)\quad\text{for all $v\in \WnuOmR$}.
\end{align}
Therefore  if $v\in \WnuOmR^\perp$ then we have 
\begin{align*}
\mathcal{J}(v)\geq C\|v\|^p_{\WnuOmR}- \|f\|_{\WnuOmR'}\|v\|_{\WnuOmR}-\|g\|_{\TnuOm'}\|v\|_{\WnuOmR}. 
\end{align*}
Since $p>1$ it follows that $\mathcal{J}(v)\to \infty$ as $ \|v\|_{\WnuOmR}\to\infty$. In fact, we have 
\begin{align*}
\frac{\mathcal{J}(v)}{\|v\|_{\WnuOmR}}
\xrightarrow{}\infty, \quad\text{as $\|v\|_{\WnuOmR}\to \infty$}. 
\end{align*}
Since $\WnuOmR$ is a Banach space,  one deduces in view of Theorem \ref{thm:charac-weak-coercive} that $\mathcal{J}$ is  weakly coercive on $\WnuOmR^\perp$. 
Hence, by Theorem \ref{thm:direct-method} $\mathcal{J}$  possesses a minimizer
$u\in \WnuOmR^\perp$, i.e., $u$ solves \eqref{eq:nonlocal-Neumann-min}. By Proposition \ref{prop:min-equiv-var}, $u$ is also a solution to \eqref{eq:var-nonlocal-Neumann-bis}. The uniqueness follows from the strict convexity of $\mathcal{J}$ or merely from $(iii)$; see the estimate \eqref{eq:weak-sol-conti}.

\noindent Now if $f$ and $g$ are compatible then we easily observe that $\mathcal{J}(v+c)=\mathcal{J}(v)$ for every $v\in \WnuOmR$ and every $c\in \R$. For $ v\in \WnuOmR$ we have $ v-\mbox{$\fint_{\Omega}v$}\in \WnuOmR^\perp$. Hence 
$\mathcal{J}(u+c)=\mathcal{J}(u)\leq \mathcal{J}(v-\mbox{$\fint_{\Omega}v$})=\mathcal{J}(v)$.
Thus, $u+c$ also solves \eqref{eq:nonlocal-Neumann-min-const} which is equivalent to \eqref{eq:var-nonlocal-Neumann-gen}. Conversely if $w$ solves  \eqref{eq:var-nonlocal-Neumann-gen} then one verifies that $ w-\mbox{$\fint_{\Omega}w$}\in \WnuOmR^\perp$ solves \eqref{eq:var-nonlocal-Neumann-bis}. By uniqueness we get $u= w- \mbox{$\fint_{\Omega}w$}$ that is $w=u+c$ with  $c=\mbox{$\fint_{\Omega}w$}$. Moreover, by taking $v=1$ in  \eqref{eq:var-nonlocal-Neumann-gen} yields the compatibility condition. 

\smallskip 

$(ii)$ If $w\in \WnuOmR$ is a solution to \eqref{eq:var-nonlocal-Neumann-gen} we have  
\begin{align*}
\cE(w,w)&= \cE(w,w-\mbox{$\fint_{\Omega} w$}) 
= \langle f , w-\mbox{$\fint_{\Omega} w$} \rangle  +\langle g , w -\mbox{$\fint_{\Omega} w$} \rangle \\
&\leq \|w-\mbox{$\fint_{\Omega} w$}\|_{\WnuOmR} ( \|f\|_{\WnuOmR'}+ \|g\|_{\TnuOm'}).
\end{align*}
Combining this with  the estimate  \eqref{eq:coercivity-est} we find that 
\begin{align*}
\|w-\mbox{$\fint_{\Omega} w$}\|_{\WnuOmR}
\leq 
C( \|f\|_{\WnuOmR'}+  \|g\|_{\TnuOm'})^{1/(p-1)}.
\end{align*}

$(iii)$ \textbf{Case $p\geq 2$.} The  estimate $(|b|^{p-2}b-|a|^{p-2}a|)(b-a)\geq A'_p|b-a|^p$ from \eqref{eq:under-elem-degen} implies 
\begin{align}\label{eq:under-form-degen}
|\cE(v,v-v')-\cE(v',v-v') |&\geq A'_p\cE(v-v',v-v'). 
\end{align}
Put $w=w_1-w_2$ so that $u_1-u_2=w-\mbox{$\fint_{\Omega} w$} \in \WnuOmR$. By \eqref{eq:under-form-degen} we have  
\begin{align*}
A'_p\cE(w,w)
&\leq |\cE(w_1,w_1-w_2)-\cE(w_2,w_1-w_2)|\\
&= | \cE(w_1,w-\mbox{$\fint_{\Omega} w$})-\cE(w_2,w-\mbox{$\fint_{\Omega} w$})|\\
&= |\langle f_1-f_2 , w-\mbox{$\fint_{\Omega} w$} \rangle  +\langle g_1-g_2, w -\mbox{$\fint_{\Omega} w$} \rangle |\\
&\leq \|w-\mbox{$\fint_{\Omega} w$}\|_{\WnuOmR} ( \|f_1-f_2\|_{\WnuOmR'}+  \|g_1-g_2\|_{\TnuOm'}).
\end{align*}
Inserting this into  \eqref{eq:coercivity-est} implies  
\begin{align*}
\|w-\mbox{$\fint_{\Omega} w$}\|_{\WnuOmR}
\leq 
C(\|f_1-f_2\|_{\WnuOmR'}+  \|g_1-g_2\|_{\TnuOm'})^{1/(p-1)}.
\end{align*}
\textbf{ Case $1<p<2$.} Let rewrite the inequality  
$(|b|^{p-2}b-|a|^{p-2}a)(b-a)\geq A'_p |b-a|^2(|a|^p+|b|^p)^{\frac{p-2}{p}}$ (see \eqref{eq:under-elem-sing}) as
\begin{align*}
c_p |b-a|^p\leq
\big((|b|^{p-2}b-|a|^{p-2}a)(b-a)\big)^{1/q} (|a|^p+|b|^p)^{1/q'}, 
\end{align*}
where $q=\frac{2}{p}$, $q'=\frac{2}{2-p}$ and $c_p=2^{\frac{2-p}{2}}{A'_p}^{\frac{p}{2}} $.
This, together with H\"{o}lder inequality  yields, with $ v_1=v-v'$, 
\begin{align}\label{eq:under-form-sing}
\big(\cE(v,v_1)-\cE(v',v_1)\big)^{\frac{p}{2}}
\big( \cE(v,v)+\cE(v',v')\big)^{\frac{2-p}{2}} \geq c_p \cE(v_1, v_1).
\vspace*{-4ex}
\end{align}
The same reasoning as above yields 
\begin{align*}
c_p\cE(w,w)
\leq \|w-\mbox{$\fint_{\Omega} w$}\|_{\WnuOmR}^{\frac{p}{2}}&\big( \|f_1-f_2\|_{\WnuOmR'}+  \|g_1-g_2\|_{\TnuOm'}\big)^{\frac{p}{2}}\\&
\times \big(\cE(w_1,w_1)+\cE(w_2,w_2)\big)^{\frac{2-p}{2}} . 
\end{align*}
On the other hand, by the estimate \eqref{eq:weak-sol-bounded} we have 
\begin{align*}
\big( \cE(w_1,w_1)+\cE(w_2,w_2)\big)^{\frac{2-p}{2}}
\hspace*{-1ex}\leq \hspace*{-0.5ex}
C\big( \hspace*{-0.2ex}\sum_{i=1}^2 \|f_i\|_{\WnuOmR'}+\|g_i\|_{\TnuOm'}\hspace*{-0.3ex}\big)^{\frac{2-p}{p-1}} 
\hspace*{-1ex}= \hspace*{-0.5ex}CM. 
\end{align*}
Altogether with the estimate  \eqref{eq:coercivity-est} implies 
\begin{align*}
\|w-\mbox{$\fint_{\Omega} w$}\|_{\WnuOmR}
&\leq CM\big( \|f_1-f_2\|_{\WnuOmR'}+  \|g_1-g_2\|_{\TnuOm'}\big). 
\end{align*}
$(iv)$ Clearly, H\"{o}lder inequality implies 
\begin{align*}
\Big|\int_{\Omega} f(x)v(x)\d x \Big|\leq \|f\|_{L^{p'}(\Omega)}\|v\|_{L^p(\Omega)} \leq \|f\|_{L^{p'}(\Omega)}\|v\|_{\WnuOmR}\,.
\end{align*}
By H\"{o}lder inequality and the continuity of the embedding $ \operatorname{Tr}: \TnuOm\hookrightarrow L^p(\Omega^c, \omega)$ (see  Theorem \ref{thm:trace-nonloc-thm}) we get
\begin{align*}
\Big|\int_{\Omega^c} g(y) v(y)\d y \Big|
&\leq \|g\|_{ L^{p'}(\Omega^c, \omega^{1-p'})}\|v\|_{L^p(\Omega^c, \omega)} \\&
\leq C_1\|g\|_{L^{p'}(\Omega^c, \omega^{1-p'})} \|v\|_{\TnuOm}.
\end{align*}
The remaining follows since  in this case problems \eqref{eq:var-nonlocal-Neumann} and \eqref{eq:var-nonlocal-Neumann-gen} are identical. 
\end{proof}

\begin{remark}\label{rem:variant-Neumann}
A modification of the Neumann data up to a multiplicative weight; such as the substitution $g(y)= g_{*}(y)\omega^\beta(y)$, $\beta\in \R$ results in another variant of the  Neumann problem.  
Naturally, one retrieves the following configuration: 
\begin{enumerate}[$\bullet$]
\item The Neumann problem \eqref{eq:nonlocal-Neumann} becomes 
\begin{align}\label{eq:nonlocal-Neumann-weight}\tag{$N_{*}$}
L u = f \quad\text{in}~~~ \Omega \quad\quad\text{ and } \quad\quad \mathcal{N} u= g_{*}\,\omega^{\beta} ~~~ \text{ on }~~~ \R^d\setminus\Omega.
\end{align} 
\item The weak formulation \eqref{eq:var-nonlocal-Neumann} becomes 
\begin{align}\label{eq:var-nonlocal-Neumann-weigthed}\tag{$V_{*}$}
\mathcal{E}(u,v) = \int_{\Omega} f(x)v(x)\d x +\int_{\Omega^c} g_{*}(y)v(y)\omega^{\beta}(y)\d y,\quad \mbox{for all}~~v \in \WnuOmR\,, 
\end{align}
whereas the compatibility condition \eqref{eq:compatible-nonlocal} becomes  \begin{align}\tag{$C_{*}$}\label{eq:compatible-nonlocal-weighted}
\int_{\Omega} f(x)\d x +\int_{\Omega^c} g_{*}(y)\omega^{\beta}(y)\d y=0. 
\end{align}
\item 
Last, if $g_*\in L^{p'}(\Omega^c, \omega^{\gamma})$ with $\beta=\frac1p+\frac{\gamma}{p'}$, the  map  $v\mapsto \int_{\Omega^c} g_{*}(y)\omega^{\beta}(y)\d y$ is linear and  belongs to $\TnuOm'$. Some special couples are given by $(\gamma, \beta)\in \{ (1-p', 0), (1,1), (0, \frac1p)\}.$ 
\end{enumerate}
\end{remark}
The next result concerns the non-existence of weak solutions when the Neumann data $g$ is not in the weighted nonlocal trace space $L^{p'}(\Omega^c,\omega^{1-p'})$. In other words,  $L^{p'}(\Omega^c,\omega^{1-p'})$ is a sufficiently large function space as the data space for the Neumann problem. 
\begin{theorem}[\textbf{Non-existence of weak solution}]\label{thm:non-existence-Neumann}
Let $\Omega= B_1(0)$ and   $\nu(h)=|h|^{-d-sp}$, $s\in (0,1)$ so that $\widetilde{\nu}(h)\asymp (1+|h|)^{-d-sp}$. 
There exists $g \in L^{1}(\Omega^c)\setminus L^{p'}(\Omega^c, \widetilde{\nu}^{1-p'})$ compatible with $f=0$ i.e. $\int_{\Omega^c} g (y)\d y=0$, such that the  Neumann problem $Lu=0$ on $\Omega$ and $\mathcal{N} u=g$ on $\R^d\setminus\Omega$ does not have a weak solution in $\WnuOmR$. 
\end{theorem}
\begin{proof}
Let us define $ g(x)= g_\gamma(x)\widetilde{\nu}(x)$ with $g_\gamma(x)=\frac{x_1}{|x|} (|x|-1)^{\gamma}\mathds{1}_{B^c_1(0)} (x)$, $\gamma\in \R$. 
For $x\in B^c_1(0)$ we have  $\dist(x, \partial \Omega)=\dist(x, \mathbb{S}^{d-1})= (|x|-1)$ and 
\begin{align*}
\int_{B_1(0)}\frac{\d y}{|x-y|^{d+sp}}\asymp (|x|-1)^{-sp}\land (|x|-1)^{-d-sp}. 
\end{align*}
Using  polar coordinates yields that   
\begin{align*}
\|g_\gamma\|^p_{\WnuOmR}
&= 2\int_{B^c_1(0) }\frac{|x_1|^p}{|x|^p} (|x|-1)^{p\gamma}
\int_{B_1(0)}|x-y|^{-d-sp}\d y\,\d x\\
&\asymp 2\int_{B^c_1(0) }\frac{|x_1|^p}{|x|^p} (|x|-1)^{p\gamma-sp}(1\land (|x|-1)^{-d})\d x  \\
&=2|\mathbb{S}^{d-1}|K_{d,p}
\Big( \int_0^1 r^{p\gamma-sp} \d r+ \int_1^\infty r^{p\gamma-sp-1} \d r\Big).
\end{align*}
Therefore we deduce that $g_\gamma\in \WnuOmR$ if and only if $\gamma\in (\frac{sp-1}{p},s)$. Analogously, $g_{\gamma}g_{\beta}\in L^1(\Omega^c,\widetilde{\nu})$ if and only if $\gamma+\beta\in (-1,sp)$. Since $g_{p'\gamma+0}=g^{p'}_{\gamma}$, it follows that  $g_{\gamma}\in L^{p'}(\Omega^c,\widetilde{\nu})$ if and only if $\gamma\in (-\frac{1}{p'}, \frac{sp}{p'})$. Indeed, 
\begin{align*}
\|g_\gamma g_\beta\|_{L^1(\Omega^c, \widetilde{\nu})}
&= \int_{B^c_1(0) }\frac{|x_1|^2}{|x|^2}  (|x|-1)^{\gamma+\beta} (1+|x|)^{-d-sp} \d x\\
&\asymp \Big( \int_0^1 r^{\gamma+\beta} \d r+ \int_1^\infty r^{\gamma+\beta-sp-1} \d r\Big).
\end{align*} 
By duality $g= g_\gamma\widetilde{\nu}\in L^1(\Omega^c)\setminus L^{p'}(\Omega^c, \widetilde{\nu}^{1-p'})$ if and only if $g_\gamma\in L^1(\Omega^c, \widetilde{\nu})\setminus L^{p'}(\Omega^c, \widetilde{\nu})$, i.e., if and only if $\gamma\in (-1, -\frac{1}{p'}]\cup [\frac{sp}{p'}, sp)$. 
Moreover by symmetry of $g_\gamma$, we have that $g= g_\gamma\widetilde{\nu}$ satisfies the compatibility condition
\begin{align*}
\int_{\Omega^c} g(y)\d y=\int_{\Omega^c} g_\gamma(y)\widetilde{\nu}(y)\d y=0. 
\end{align*}
Assume the Neumann problem has a weak solution $u$. That is 
\begin{align*}
\mathcal{E}(u,v)=\int_{\Omega^c} g_\gamma(y)v(y)\widetilde{\nu}(y)\d y\quad \text{for all}\quad v\in \WnuOmR.
\end{align*}
Consider $C= 1+\|u\|_{\WnuOmR}>0$, we find that  
\begin{align*}
\Big|\int_{\Omega^c} g_\gamma(y)v(y)\widetilde{\nu}(y)\d y\Big|\leq C \|v\|_{\WnuOmR} 
\quad \text{ for all}\quad  v\in \WnuOmR.
\end{align*}
In particular taking  $v=g_\beta\in \WnuOmR$ amounts the above estimate to 
\begin{align*}
\|g_{\gamma}g_{\beta}\|_{L^1(\Omega^c, \widetilde{\nu})}= \Big|\int_{\Omega^c} g_\gamma(y) g_\beta(y)\widetilde{\nu}(y)\d y\Big|\leq C \|g_\beta\|_{\WnuOmR}, 
\end{align*}
for all $\beta\in (\tfrac{sp-1}{p}, s)$. Now if $sp\geq1$, consider $\gamma\in (\frac{sp}{p'}, \frac{1}{p}+\frac{sp}{p'})\subset  [\frac{sp}{p'}, sp)$ and take $\beta=sp-\gamma$.
If $sp\leq1$ consider $\gamma\in (-1, -(s+\tfrac{1}{p'}))\subset  (-1, -\frac{1}{p'}]$ and take $\beta= -\gamma-1$. 
In both cases, $\gamma\in (-1, -\frac{1}{p'}]\cup [\frac{sp}{p'}, sp)$ and 
$\beta\in  (\tfrac{sp-1}{p}, s)$ and $\gamma+\beta\in \{-1,sp\}$. 
In other words, $g_\gamma\in L^1(\Omega^c, \widetilde{\nu})\setminus L^{p'}(\Omega^c, \widetilde{\nu})$ and $g_\beta\in \WnuOmR$. Whence $\|g_{\gamma}g_{\beta}\|_{L^1(\Omega^c, \widetilde{\nu})}=\infty$ and $\|g_{\beta}\|_{\WnuOmR}<\infty$, which contradicts the above inequality. 
\end{proof}

The well-posedness of the Neumann problem for the regional operator $L_\Omega$ can be derived analogously.
\begin{align*}
L_\Omega u(x):= 2\pv\int_\Omega\psi(u(x)-u(y)) \nu(x-y)\d y. 
\end{align*}
\begin{theorem} \label{thm:nonlocal-Neumann-regional}
Assume that the Poincar\'{e} inequality \eqref{eq:poincare-semi-gagliardo} holds (see page \pageref{eq:poincare-semi-gagliardo}).  Let $ f \in \WnuOm'$. The following assertions are true. 
\begin{enumerate}[$(i)$]
\item \textbf{Existence}. There is a unique  $u\in \WnuOm^\perp$ satisfying $\cE_\Omega(u,v)=\langle f, v\rangle$ for all $v\in \WnuOm^\perp$.  Furthermore a function $w\in \WnuOm$ is weak solution to $L_\Omega u= f$ in $\Omega$, i.e.,  satisfies 
\begin{align}\label{eq:var-nonlocal-Neumann-regional}
\cE_\Omega(u,v)=\langle f, v\rangle\qquad\text{for all $v\in \WnuOm$}, 
\end{align}
if and only if $w=u+c$ form some $c\in\R$  and $\langle f,1\rangle=0$. 
\item \textbf{Boundedness}. 
Any solution  $w$ to \eqref{eq:var-nonlocal-Neumann-regional} satisfies 
\begin{align*}
\|w-\hbox{$\fint_{\Omega}w$} \|_{\WnuOm}\leq C \|f\|^{1/(p-1)}_{\WnuOm'}.
\end{align*}
\item \textbf{Continuity}.  
If $u_i= w_i-\hbox{$\fint_{\Omega}w_i$}$, $i= 1,2$ where $w_i$ satisfies \eqref{eq:var-nonlocal-Neumann-regional} with $f=f_i$ then 
\begin{align*}
\|u_1-u_2\|_{\WnuOm}\leq 
\begin{cases}
C \|f_1-f_2\|^{1/(p-1)}_{\WnuOm'} &  \hspace*{-2ex}\text{$p\geq2$},\\
C\big(\|f_1\|_{\WnuOm'} + \|f_2\|_{\WnuOm'}\big)^{\frac{2-p}{p-1}}\|f_1-f_2\|_{\WnuOm'} & \hspace*{-2ex}\text{$p<2$}. 
\end{cases}
\end{align*}

\end{enumerate}
\end{theorem}
\begin{proof}
The proof is analogous to that of Theorem \ref{thm:nonlocal-Neumann-var}. 
\end{proof}

\subsection{Dirichlet problem}\label{sec:dirichlet-problem}

The Dirichlet  problem associated with the data $f:\Omega\to \mathbb{R}$ and $g: \Omega^c \to \R$ is to find $u:\R^d\to \R$ such that 
\begin{align}\label{eq:nonlocal-Dirichlet}\tag{$D$}
L u = f \quad\text{in}~~~ \Omega \quad\quad\text{ and } \quad\quad  u=g ~~~ \text{ on }~~~ \R^d\setminus\Omega.
\end{align}
\noindent In contrast to the Neumann condition ($\cN u =g$ on $\R^d\setminus\Omega$), the Dirichlet condition 
($u=g$ on $\R^d\setminus\Omega$) does not impose any constraint on $g$. Note however that the evaluation of $g$ on $\R^d\setminus\Omega_\nu$ with $\Omega_\nu=\Omega+\supp\nu$, does not influence the values of $u$ in $\Omega$ since $Lu(x)=L_{\Omega_\nu} u(x)$ for all $x\in\Omega$. This is merely due to the fact that $L u(x)= L_{\Omega_\nu} u(x)$ for all $x\in \Omega$; see \eqref{eq:nonlocal-hullnu}. 
It is therefore enough to prescribe the Dirichlet data only the exterior domain (the nonlocal boundary) $\Omega_e=\Omega_\nu^{*}\setminus\Omega= \Omega_\nu\setminus\Omega$ where $\Omega_\nu=\Omega+\supp\nu$ and $\Omega_\nu^{*}=\Omega\cup \Omega_\nu$ is the nonlocal hull of $\Omega$. Accordingly, the problem \eqref{eq:nonlocal-Dirichlet} is the same as finding $u:\Omega_\nu^{*}=\Omega\cup \Omega_\nu\to \R$ such that 
\begin{align}\label{eq:nonlocal-Dirichlet-ex}\tag{$D_\nu$}
L u = f \quad\text{in}~~~ \Omega \quad\quad\text{ and } \quad\quad  u= g_e ~~~ \text{ on }~~~ \Omega_e.
\end{align} 
where $g_e= g|_{\Omega_e}$ is the restriction of $g$ on $\Omega_e$. 
Actually, both  problems \eqref{eq:nonlocal-Dirichlet} and \eqref{eq:nonlocal-Dirichlet-ex} are equivalent. 
Indeed, if $u$ solves  \eqref{eq:nonlocal-Dirichlet} then  clearly $u_\nu=u|_{\Omega_\nu^{*}}$ solves  \eqref{eq:nonlocal-Dirichlet-ex}. Conversely, if $u_\nu$ solves  \eqref{eq:nonlocal-Dirichlet-ex} then the function defined $u(x)= \widetilde{u}_\nu(x)$ for $x\in \Omega_\nu^{*}$ and $u(x)= g(x)$ for $x\in \R^d\setminus\Omega_\nu^{*}$ solves \eqref{eq:nonlocal-Dirichlet}.
From now on, the problem \eqref{eq:nonlocal-Dirichlet} is understood in the sense of \eqref{eq:nonlocal-Dirichlet-ex}. Motivated by the Gauss-Green formula \eqref{eq:gauss-green-nonlocal}, in  Appendix \ref{sec:appendix-gauss-green},
\begin{align*}
	\int_{\Omega} Lu(x)v(x)\d x= \mathcal{E}(u,v) -\int_{\Omega^c}\mathcal{N}u(y)v(y)\d y, 
\end{align*}
we define weak solutions of the Dirichlet problem as follows.

\begin{definition}\label{def:dirichlet-var-sol} Let $ f \in \WnuOmO'$ and $g \in \TnuOm$. We say that $u \in \WnuOmR$ is a weak solution or the variational solution of the Dirichlet problem \eqref{eq:nonlocal-Dirichlet} if 
\begin{align}\label{eq:var-nonlocal-Dirichlet-gen}\tag{$V_0$}
u-g\in \WnuOmO\quad \text{and}\quad \mathcal{E}(u,v) = \langle f , v \rangle   \quad \mbox{for all}~~v \in \WnuOmR\,. 
\end{align}
\end{definition} 
Actually, for any extension $\overline{g}\in \WnuOmR$ of $g$, i.e., $g=\overline{g}$ a.e. on $\Omega^c$, solution to the variational problem \eqref{eq:var-nonlocal-Dirichlet-gen}  are critical points of the functional 
\begin{align}
\mathcal{J}_0(v) &= \frac{1}{p} \mathcal{E}(v,v) -\langle f , v-\overline{g}\rangle,\qquad\text{$v\in g+\WnuOmO$}.
\end{align}
It is decisive to keep in mind that the Fr\'echet derivative of $\mathcal{J}_0$ is given as  
\begin{align*}
\langle	\mathcal{J}'_0(u), v\rangle= \cE(u,v) -\langle f , v \rangle   
\qquad\text{for all $u,v\in \WnuOmO$}.
\end{align*}

\begin{proposition}\label{prop:min-equiv-var-D}
Let $\overline{g}\in \WnuOmR$ be an extension of $g$, i.e., $g=\overline{g}$ a.e. on $\Omega^c$.  The variational problem \eqref{eq:var-nonlocal-Dirichlet-gen} is equivalent to the minimization problem
\begin{align}\label{eq:nonlocal-Dirichlet-min}\tag{$M_0$}
\begin{split}
\mathcal{J}_0(u) =\min_{v-\overline{g}\in \WnuOmO}\mathcal{J}_0(v).
\end{split}
\end{align}
Moreover, if $\Omega\subset \R^d$ is bounded in one direction or $|\Omega|<\infty$ and   $\nu\not\equiv0$ then any solution to\eqref{eq:var-nonlocal-Dirichlet-gen} or  \eqref{eq:nonlocal-Dirichlet-min} is independent on the choice of $\overline{g}$. 
\end{proposition} 
\begin{proof}
Let $u, v\in \overline{g}+\WnuOmO$ then  $v-u\in \WnuOmO$. Assume $u$ solves  \eqref{eq:var-nonlocal-Dirichlet-gen} then using H\"{o}lder's and Young's inequalities we get 
\begin{align*}
\langle f , v-u \rangle= \mathcal{E}(u,v-u) 
&\leq  \mathcal{E}(u,u)^{1/p'}\mathcal{E}(v,v)^{1/p} -\mathcal{E}(u,u)\\&
\leq  \frac{1}{p} \mathcal{E}(v,v)- \frac{1}{p} \mathcal{E}(u,u).
\end{align*} 
Since $\langle f , v-u \rangle=\langle f , v-\overline{g} \rangle- \langle f , u-\overline{g} \rangle$, it follows that $\mathcal{J}_0(u)\leq \mathcal{J}_0(v)$ hence $u$ solves \eqref{eq:nonlocal-Dirichlet-min}.

\noindent Conversely let $u $ satisfies \eqref{eq:nonlocal-Dirichlet-min}, i.e.,  $u-g\in\WnuOmO$ and $\mathcal{J}_0(u)\leq \mathcal{J}_0(v)$ for  $ v\in g+\WnuOmO$. 
In particular, since $u+tv\in g+ \WnuOmO$ for all $v\in \WnuOmO$ and $t\in \R$ it follows that $\mathcal{J}_0(u) \leq \mathcal{J}_0(u+tv)$. Thus, the mapping  $t\mapsto \mathcal{J}_0(u+t v)$ is differentiable and  has a critical point at $t=0$. It follows that  $u$ satisfies  \eqref{eq:var-nonlocal-Dirichlet-gen} since
\begin{align*}
\cE(u,v) -\langle f , v \rangle=  \langle	\mathcal{J}'_0(u), v\rangle= \lim_{t\to 0}\frac{\mathcal{J}_0(u+tv)-\mathcal{J}_0(u)}{t} =0.
\end{align*}
Now, assume $g_i\in \WnuOmR$, $i=1,2$ are different extensions of $g$, i.e., $g_1=g_2=g$ a.e. on $\Omega^c$. Let $u_i$ be the associated solution to \eqref{eq:var-nonlocal-Dirichlet-gen} (or  minimizer of $\mapsto \mathcal{J}_0(v+ g_i)$). In particular  $u_1=u_2=g$ a.e. on $\Omega^c$ and $u_i-g_i\in \WnuOmO$. Hence  testing with $u_1-u_2\in\WnuOmO$, by definition of $u_1$ and $u_2$ we have $\cE(u_1, u_1-u_2)=\cE(u_1, u_1-u_2) \langle f, u_1-u_2\rangle$.  In virtue of the estimates \eqref{eq:under-form-degen} and \eqref{eq:under-form-sing} we deduce that $\cE(u_1-u_2, u_1-u_2)=0$. According to  Theorem \ref{thm:poincare-friedrichs-ext-bis} or Theorem \ref{thm:poincare-friedrichs-finite-bis} the Poincar\'{e}-Friedrichs inequality \eqref{eq:poincare-friedrichs-ext} or \eqref{eq:poincare-friedrichs-ext-bis}  holds, i.e., 
\begin{align*}
\|u_1-u_2\|^p_{L^p(\Omega)}\leq C\cE(u_1-u_2, u_1-u_2)=0. 
\end{align*}
Thus, $u_1=u_2$ a.e. on $\Omega$ and $u_1=u_2=g$ a.e. on $\Omega^c$, that is we get  $u_1=u_2$ a.e. on $\R^d$. 
\end{proof}

\noindent We are now in position to state the well-posedness of the problem \eqref{eq:var-nonlocal-Dirichlet-gen}. 

\begin{theorem}\label{thm:nonlocal-dirichlet-gen}
Assume $\nu\not\equiv 0$, i.e., $|\{\nu>0\}|>0$ and $\Omega\subset \R^d$ is bounded in one direction or $|\Omega|<\infty$. Let $f\in \WnuOmO'$ and $g\in \TnuOm$.   The following assertions hold.
\begin{enumerate}[$(i)$]
\item \textbf{Existence}. The variational problem \eqref{eq:var-nonlocal-Dirichlet-gen} has a unique  solution $u\in \WnuOmR$. 
\item \textbf{Boundedness}. Moreover, there is $C= C(d,p,\Omega,\nu)>0$ such that for any $\overline{g}\in \WnuOmR$ with $\overline{g}|_{\Omega^c}=g$, 
\begin{align}\label{eq:energy-bound-D}
\cE(u,u)&\leq C(\|f\|_{\WnuOmO'} ^{p'} + \cE(\overline{g},\overline{g})), 
\end{align}
\begin{align}\label{eq:weak-sol-bounded-D}
\|u \|_{\WnuOmR}\leq C \big(\|f\|^{p'}_{\WnuOmO'}+\|g\|^p_{\TnuOm}\big)^{1/p}.
\end{align}
\item \textbf{Continuity}.  
Let $u_i$ be the solution associated with $f=f_i$ and $g=g_i$, $i=1,2$. Let us put  $R=	(D^{\frac{1}{p-1}} + D)^{1/2}$ with
$$D=D(f_1,f_2,g_1,g_2)=\sum_{i=1}^2  \big(\|f_i\|^{p'}_{\WnuOmO'}+\|g_i\|^p_{\TnuOm}\big)^{1/p'}.$$
The following estimates hold true.  If $p\geq 2$ we have 
\end{enumerate}

\begin{align*}
\|u_1-u_2\|_{\WnuOmR}
&\leq C(\|f_1-f_2\|^{1/(p-1)}_{\WnuOmR}\\&+\|g_1-g_2\|_{\TnuOm}+  D^{1/p}\|g_1-g_2\|^{1/p}_{\TnuOm}). 
\end{align*}

If $1<p<2$ we have 
\begin{align*}
\begin{split}
\|u_1-u_2\|_{\WnuOmR} &\leq C( D^{\frac{2-p}{p-1}} \|f_1-f_2\|_{\WnuOmO'}\\ &+ \|g_1-g_2\|_{\TnuOm}+
R\|g_1-g_2\|_{\TnuOm}^{1/2}).
\end{split}
\end{align*}
\end{theorem}

\begin{remark}
By definition of the trace space, for $g\in \TnuOm$ there is $\overline{g}\in \WnuOmR$ such that $g=\overline{g}$ a.e. on $\Omega^c$. Recall that by Proposition \ref{prop:min-equiv-var-D}, any solution $u$ to the Dirichlet problem \eqref{eq:var-nonlocal-Dirichlet-gen} does not depend on the choice of $\overline{g}$. Furthermore we recall that 
\begin{align*}
\|g\|_{\TnuOm}=\inf\{\|\overline{g}\|_{\WnuOmR}\,: \, \overline{g}\in \WnuOmR, \,\overline{g}= g \,\text{ a.e. on $\Omega^c$} \}.
\end{align*}
Without loss of generality, it is sufficient to assume that all Dirichlet data $g\in \WnuOmR$. 
\end{remark}

\begin{proof}
We emphasize that throughout the proof,  $C>0$ denotes a generic constant only depending on the constant from the Poincar\'{e} inequality and $p$. 

$(i)$ The functional $v\mapsto \mathcal{J}_0(v+g) = \frac{1}{p} \mathcal{E}(v+g,v+g) -\langle f , v \rangle $ is clearly convex and continuous (hence lower semicontinuous) on $\WnuOmO$.  According to Theorem \ref{thm:charac-weak-lsc} $\mathcal{J}_0$ is weakly lower semicontinuous. 
On the other hand, in virtue of Theorem \ref{thm:poincare-friedrichs-ext-bis} or Theorem \ref{thm:poincare-friedrichs-finite-bis} the Poincar\'{e}-Friedrichs inequality \eqref{eq:poincare-friedrichs-ext} or \eqref{eq:poincare-friedrichs-ext-bis}  holds.  In any case, one readily finds a constant $C>0$  
\begin{align*}
C\|v\|^p_{\WnuOmR}\leq \cE(v,v)\quad\text{for all $v\in \WnuOmO$}.
\end{align*}
In particular, we have 
\begin{align}\label{eq:coercivity-est-D}
C\|v-g\|^p_{\WnuOmR}\leq C\cE(v-g,v-g)\quad\text{for all $v\in g+\WnuOmO$}.
\end{align}
Therefore, since $\cE(v,v)\leq 2^p\cE(v+g,v+g)+2^p\cE(g,g)$ we have 
\begin{align*}
\mathcal{J}_0(v+g)\geq 2^{-p}C\|v\|^p_{\WnuOmR}-\cE(g,g)- \|f\|_{\WnuOmO'}\|v\|_{\WnuOmR}.  
\end{align*}
Since $p>1$ we have  $\mathcal{J}(v+g)\to \infty$ as $ \|v\|_{\WnuOmR}\to\infty$. In fact, we have 
\begin{align*}
\frac{\mathcal{J}_0(v+g)}{\|v\|_{\WnuOmR}}
\xrightarrow{}\infty, \quad\text{as $\|v\|_{\WnuOmR}\to \infty$, $v\in \WnuOmO$}. 
\end{align*}
\noindent Since  $\WnuOmO$ is always a reflexive Banach space as $p>1$,  Theorem \ref{thm:charac-weak-coercive} implies that  $v\mapsto \mathcal{J}_0(v+g)$ is  weakly coercive on $\WnuOmO$. Hence, by Theorem \ref{thm:direct-method} $\mathcal{J}_0(\cdot+g)$  possesses a minimizer $u_0\in \WnuOmO$, that is,  
\begin{align*}
\mathcal{J}_0(u_0+g)= \min_{w\in \WnuOmO}\mathcal{J}_0(w+g)= \min_{v\in g+ \WnuOmO}\mathcal{J}_0(v). 	
\end{align*}
In other words $u= u_0+g$ solves \eqref{eq:nonlocal-Dirichlet-min} and, by Proposition \ref{prop:min-equiv-var-D}, $u$ is also a solution to \eqref{eq:var-nonlocal-Dirichlet-gen}. 
It is worth emphasizing that $u$ is independent of the choice of the extension $g$; see Proposition \ref{prop:min-equiv-var-D}. The uniqueness follows from the strict convexity of $\mathcal{J}_0(\cdot+g)$ or merely from the estimates in $(iii)$. 

\smallskip 

$(ii)$ 
Since $u$ is a solution to \eqref{eq:var-nonlocal-Dirichlet-gen} we have  
\begin{align*}
\cE(u,u)&= \cE(u,u-g) +\cE(u,g)= \langle f , u-g \rangle  +\cE(u,g)\\
&\leq \| u-g \|_{\WnuOmR} \|f\|_{\WnuOmO'}+ \cE(u,u)^{1/p'}\cE(g,g)^{1/p}.
\end{align*}
Since $u-g\in \WnuOmO$ the coercivity estimate \eqref{eq:coercivity-est-D} yields 
\begin{align*}
\| u-g \|_{\WnuOmR}\leq C \cE(u-g,u-g)^{1/p}
\leq C\cE(u,u)^{1/p}+ C\cE(g,g)^{1/p}. 
\end{align*}	

Next, for $a,b\in \R$ $\delta>0$, applying the Young's inequality on $a\delta$ and $\frac{b}{\delta}$ implies 
\begin{align}\label{eq:young-delta}
|ab|\leq \frac{\delta^p a^p}{p}+ \frac{b^{p'}}{p'\delta^{p'}}. 
\end{align}
Accordingly, by exploiting the Young inequality \eqref{eq:young-delta} we get 
\begin{align*}
\cE(u,u)^{1/p'}\cE(g,g)^{1/p}
&\leq  \frac{\delta^{p}\cE(u,u)}{p'}+ \frac{\cE(g,g)}{p\delta^{\frac{p^2}{p'}}},\\
C\|f\|_{\WnuOmO'} \cE(u,u)^{1/p} 
&\leq  \frac{\delta^p\cE(u,u)}{p}+ \frac{C^{p'}\|f\|_{\WnuOmO'} ^{p'}}{p'\delta^{p'}},\\
C\|f\|_{\WnuOmO'}\cE(g,g)^{1/p}
&\leq \frac{\delta^{p}\cE(g,g)}{p} + \frac{C^{p'}\|f\|_{\WnuOmO'} ^{p'}}{p'\delta^{p'}}.
\end{align*}
Inserting altogether in the previous estimate we obtain 

\begin{align*}
\cE(u,u)&\leq C\|f\|_{\WnuOmO'} \cE(u,u)^{1/p}
\\&+ C\|f\|_{\WnuOmO'} \cE(g,g)^{1/p} + \cE(u,u)^{1/p'}\cE(g,g)^{1/p}\\
&\leq \delta^{p}\cE(u,u) + ( \frac{1}{p\delta^{\frac{p^2}{p'}}}+ \frac{\delta^{p}}{p}) \cE(g,g) 
+ \frac{2C^{p'}}{p'\delta^{p'}} \|f\|_{\WnuOmO'} ^{p'}. 
\end{align*}

Taking in particular $\delta^p=\frac12$ yields the desired estimate \eqref{eq:energy-bound-D} 
\begin{align*}
\cE(u,u)&\leq C(\|f\|_{\WnuOmO'} ^{p'} + \cE(g,g)) . 
\end{align*}

The coercivity estimate \eqref{eq:coercivity-est-D} implies 
\begin{align}\label{eq:xlp-bound-dirich}
\begin{split}
\|u\|_{L^p(\Omega)}
&\leq  \|g\|_{L^p(\Omega)}+ C\cE(u-g, u-g)^{1/p}\\
&\leq C\|g\|_{\WnuOmR}+ C\cE(u,u)^{1/p}. 
\end{split}
\end{align}
This together with  the penultimate estimate yields the  inequality \eqref{eq:weak-sol-bounded-D} 
\begin{align*}
\|u\|_{\WnuOmR}
&\leq C(\|f\|^{p'}_{\WnuOmR}+\|g\|^p_{\WnuOmR})^{1/p}.
\end{align*}
$(iii)$ Put $u=u_1-u_2$, $f=f_1-f_2$ and $g=g_1-g_2$. We clearly have  $u-g\in \WnuOmO$.

\smallskip 

\textbf{Case $p\geq 2$.} 
By exploiting the estimate  \eqref{eq:coercivity-est-D} and Young's inequality \eqref{eq:young-delta} we find that 
\begin{align}\label{eq:bounded-dirich1}
\begin{split}
\|u-&g\|_{\WnuOmR}  \|f\|_{\WnuOmO'}
\leq C\cE(u-g, u-g)^{1/p} \|f\|_{\WnuOmO'}\\
&\leq  C\|f\|_{\WnuOmO'} (\cE(u,u)^{1/p}+ \cE(g,g)^{1/p}) \\
&\leq \frac{\delta^p\cE(u,u)}{p} + \frac{C}{p'\delta^{p'}}\|f\|^{p'}_{\WnuOmO'}+ CD\cE(g,g)^{1/p}, 
\end{split}
\end{align}
where we used $\|f\|_{\WnuOmO'}\leq D$.  Analogously as for  \eqref{eq:energy-bound-D}, we find that 
\begin{align}\label{eq:bounded-dirich2}
\begin{split}
	|\cE(u_1,g)-\cE(u_2,g)|
&\leq (\cE(u_1,u_1)^{1/p'}+\cE(u_2,u_2)^{1/p'})\cE(g,g)^{1/p}
\\
&\leq CD \cE(g,g)^{1/p}. 
\end{split}
\end{align}
Using the definition of $u_i$ and the estimates  \eqref{eq:under-form-degen}, \eqref{eq:bounded-dirich1}  and \eqref{eq:bounded-dirich2} we obtain  
\begin{align*}
A'_p\cE(u,u)
&\leq  |\cE(u_1,u-g)-\cE(u_2,u-g)+ \cE(u_1,g)-\cE(u_2,g)|\\
&= |\langle f,u-g\rangle + \cE(u_1,g)-\cE(u_2,g)|\\
&\leq \|u-g\|_{\WnuOmR} \|f\|_{\WnuOmO'}+ |\cE(u_1,g)-\cE(u_2,g)|\\
&\leq \frac{\delta^p\cE(u,u)}{p} + \frac{C}{p'\delta^{p'}}\|f\|^{p'}_{\WnuOmO'}+ CD\cE(g,g)^{1/p}. 
\end{align*}
Accordingly, taking $\delta^p= A'_p$ yields 
\begin{align*}
\cE(u,u)\leq 
C\|f\|^{p'}_{\WnuOmO'}+ C D\cE(g,g)^{1/p}. 
\end{align*}
Since $u-g\in \WnuOmO$, combining this with the estimate \eqref{eq:xlp-bound-dirich} gives 
\begin{align*}
\|u\|_{\WnuOmR}
&\leq C(\|f\|^{p'}_{\WnuOmR}+\|g\|^p_{\WnuOmR}+  D\|g\|_{\WnuOmR})^{1/p}. 
\\&\leq C(\|f\|^{1/(p-1)}_{\WnuOmR}+\|g\|_{\WnuOmR}+  D^{1/p}\|g\|^{1/p}_{\WnuOmR}). 
\end{align*}
\textbf{Case $1<p<2$.} Using the definition of $u_i$ and the estimates  \eqref{eq:under-form-sing} yields 
\begin{align*}
c_p\cE(u, u)&\leq \big(\cE(u_1,u-g)-\cE(u_2,u-g) + \cE(u_1,g)-\cE(u_2,g)\big)^{\frac{p}{2}} H\\
&=\big(\langle f,u-g\rangle+ \cE(u_1,g)-\cE(u_2,g)\big)^{\frac{p}{2}}H\\
&\leq \big(\|u-g\|_{\WnuOmR}\|f\|_{\WnuOmO'}+ |\cE(u_1,g)-\cE(u_2,g)|\big)^{\frac{p}{2}}H,
\end{align*}
with $H=\big( \cE(u_1,u_1)+\cE(u_2,u_2)\big)^{\frac{2-p}{2}}.$ By exploiting once more, the estimates in \eqref{eq:bounded-dirich1} and \eqref{eq:bounded-dirich2} one readily arrives at the following 
\begin{align*}
c_p\cE(u, u)
&\leq \hspace*{-0.5ex}CD^{\frac{p'(2-p)}{2}}
(\|f\|_{\WnuOmO'} (\cE(u,u)^{\frac{1}{p}}+ \cE(g,g)^{1/p}))^{\frac{p}{2}} \hspace*{-0.4ex}+ \hspace*{-0.5ex}D^{p/2}\,\cE(g,g)^{1/2}) \\
&\leq CD^{\frac{p(2-p)}{2(p-1)}}
(\|f\|^{\frac{p}{2}}_{\WnuOmO'} (\cE(u,u)^{1/2}+ \cE(g,g)^{1/2}))+ 
D^{p/2}\,\cE(g,g)^{1/2}\\
&\leq \frac{\delta^2\cE(u,u)}{2} + \frac{C}{2\delta^{2}}\|f\|^{p}_{\WnuOmO'}D^{^{\frac{p(2-p)}{p-1}}}\hspace*{-0.5ex}+
C(D^{p'/2}+ D^{p/2}) \cE(g,g)^{1/2}, 
\end{align*}

we used $\|f\|_{\WnuOmO'}\leq D $ and the  inequality \eqref{eq:young-delta}. For $\delta^2= c_p$ we get 

\begin{align*}
\cE(u, u)\leq  CD^{^{\frac{p(2-p)}{p-1}}}\|f\|^{p}_{\WnuOmO'}+
C(D^{p/2(p-1)}+ D^{p/2}) \cE(g,g)^{1/2}. 
\end{align*}
Since $u-g\in \WnuOmO$, combining this with the estimate \eqref{eq:xlp-bound-dirich} implies
\begin{align*}
\|u\|_{\WnuOmR}
&\leq C( D^{\frac{2-p}{p-1}} \|f\|_{\WnuOmO'}+ \|g\|_{\WnuOmR}
\\
&+ 
(D^{\frac{1}{p-1}} + D)^{1/2}\|g\|_{\WnuOmR}^{1/2}).
\end{align*}
\end{proof}

\begin{theorem}[Weak comparison principle]\label{thm:comparison-principle}
Assume that $\Omega\subset \R^d$ is bounded in one direction or that  $|\Omega|<\infty$ and $\nu\not\equiv0$.  Let  $u,v\in \WnuOmR$. Assume that  $v\leq u$ a.e. on $\Omega^c$ and $Lv\leq Lu$ in $\Omega$ in the weak sense, i.e., 
\begin{align*}
\cE(v, w)\leq \cE(u,w) \qquad\text{for all $w\in \WnuOmRO$, $w\geq0$}. 
\end{align*}
Then we have $v\leq u$ a.e. in $\R^d$. 
\end{theorem}
\begin{proof}
Recall that for $t\in \R$ we put $\psi(t)=|t|^{p-2}t$ and $t_{+}=\max(t, 0)$ and $t_{-}=\max(-t, 0)$ so that $t=t_{+}-t_{-}$. Note that by Corollary \ref{cor:elm-est-cutplus} we have 
\begin{align*}
(\psi(b)- \psi(a)) &((b_1-a_1)_{+}-(b_2-a_2)_{+})\\
&\geq
\begin{cases} 
A'_p |(b_1-a_1)_{+}-(b_2-a_2)_{+}|^p& \hspace{-1ex} p\geq2, 
\\
A'_p |(b_1-a_1)_{+}-(b_2-a_2)_{+}|^2(|b|+|a|)^{p-2} &\hspace{-1ex}  1< p<2. 
\end{cases}
\end{align*}
Consider $w= (v-u)_+$ so that $w=0$ on $\Omega^c$ since $v-u\leq0$ on $\Omega^c$. Hence $w\in \WnuOmRO$ and $w\geq0$. 
Taking $b_1= v(x),$ $  b_2= v(y),$ $ a_1= u(x),\, a_2=u(y)$ and proceeding as for the estimates \eqref{eq:under-form-degen} and \eqref{eq:under-form-sing} we get
\begin{align*}
0\geq \cE(v, w)-\cE(u, w) &\geq A'_p\cE(w,w)\qquad \text{$p\geq2$},\\
0\geq \big(\cE(v,w)-\cE(u,w)\big)
\big( \cE(v,v)+\cE(u,u)\big)^{\frac{2-p}{p}} &\geq c_p \cE(w,w)^{\frac{2}{p}}\qquad\text{$1<p<2$}.
\end{align*}
In any case we deduce that $\cE(w,w)=0$. In view of the Poincar\'{e}-Friedrichs inequality (see Theorem \ref{thm:poincare-friedrichs-ext-bis} and Theorem \ref{thm:poincare-friedrichs-finite-bis}) we also have $\|w\|_{L^p(\Omega)}=0$ and hence $\|w\|_{\WnuOmR}=0$. It follows that $w=(v-u)_+=0$ a.e. on $\R^d$ equivalently  $v\leq u$ a.e. on $\R^d$.
\end{proof}

\begin{theorem}[Weak  and strong maximum principle]
Let the assumptions of Theorem \ref{thm:comparison-principle} be in force. Assume that  $u=g\geq 0$ a.e. on $\Omega^c$ and $Lu=f\geq 0$ in $\Omega$ in the weak sense. 
Then we have $u\geq 0$ a.e. in $\R^d$. If in addition $u$ is continuous and $\nu>0$ then either $u\equiv0$ in $\R^d$ or $u>0$ in $\Omega$.  
\end{theorem}

\begin{proof} 
The fact that $u\geq 0$ a.e. in $\R^d$ follows from Theorem \ref{thm:comparison-principle}. Assume $u$ is continuous and $\nu>0$. If $u(x_0)\leq 0$ for some $x_0\in \Omega$ then  since $\psi$ is increasing we have 
\begin{align*}
0\leq Lu(x_0)&=2\pv \int_{\R^d} \psi(u(x_0)-u(y))\nu(x_0-y)\d y\\
&\leq -2\pv\int_{\R^d} (u(y))^{p-1}\nu(x_0-y)\d y\leq 0,
\end{align*}
which is only possible if $u (y)=0$ for all $y\in\R^d$ since $u\geq0$ and $\nu>0$. 
\end{proof}
\subsection{Robin problem}
In the classical setting for the $p$-Laplace operator, the Robin boundary problem\footnote{
According to the over 20 years of survey work \cite{GuAb98}, there is no historical evidence why the problem \eqref{eq:robin-local} is termed  after Robin's name in the setting $p=2$. The survey \cite[p.69]{GuAb98} also points out that the first mathematical appearance of the problem \eqref{eq:robin-local} goes back at least to the works on cooling law by Fourier(1822) and/or Newton (1701, but a mathematical contribution  by Newton is uncertain).}
-- also known as Fourier boundary problem or third boundary problem --combines the Dirichlet and Neumann boundary problem as follows
\begin{align}\label{eq:robin-local}
-\Delta_p u = f \,\,\text{ in }  \Omega\quad \text{ and } \quad \partial_{n,p}u + \beta |u|^{p-2}u= g \,\,\text{ on } \partial\Omega.
\end{align}
Here $f\in L^{p'}(\Omega)$ and $\beta, g: \partial\Omega\to \mathbb{R}$ are given. Analogously, in the nonlocal set-up, the Robin problem for $L$ with data  $\beta, g: \Omega^c\to \mathbb{R}$ and $f\in L^{p'}(\Omega)$ is to find $u:\R^d\to \mathbb{R}$ such that 
\begin{align}\label{eq:robin-problem}
L u = f \text{ in }  \Omega\quad\text{ and }\quad \mathcal{N}u +
\beta |u|^{p-2}u= g \text{ on } \Omega^c.
\end{align}
\noindent Note that, for  $\beta=0$ one recovers the inhomogeneous Neumann problem. For $\beta\to \infty $ it leads to the homogeneous Dirichlet problem. Let define the form 
\begin{align*}
Q_\beta(u,v)= \mathcal{E}( u,v)+ \int_{\Omega^c} |u(y)|^{p-2}u(y)v(y)\beta(y)\d y. 
\end{align*} 
As for the Neumann problem, we define a weak solution of \eqref{eq:robin-problem} as follows. 

\begin{definition}
We say that $u\in \WnuOmR$ is a weak solution (or variational solution) of the Robin problem \eqref{eq:robin-problem} if 
\begin{align}\tag{$V_\beta$}\label{eq:var-nonlocal-robin}
Q_\beta(u,v)=\langle f,v\rangle +\langle g,v \rangle ~~\text{for all }~v\in \WnuOmR.
\end{align}
In fact, the problem \eqref{eq:var-nonlocal-robin} is equivalent to the minimization problem
\begin{align}\label{eq:nonlocal-Robin-min}\tag{$M_\beta$}
\begin{split}
\mathcal{J}_\beta(u) &=\min_{v\in \WnuOmR}\mathcal{J}_\beta(v),
\\
\mathcal{J}_\beta(v) &= \frac{1}{p} Q_\beta(v,v) -\langle f,v\rangle  -\langle g,v \rangle. 
\end{split}
\end{align}
\end{definition}

\begin{theorem} \label{thm:nonlocal-robin-var}
Let $\omega\in \{\widetilde{\nu}_{\Omega}, \overline{\nu}_{\Omega}, \widehat{\nu}_{R}\}$ where $|B_R(0)\cap \Omega|>0$ $($see Definition \ref{def:different-nus}$)$. 
Assume that $\nu$ have full support, $ \beta\omega^{-1}\in L^\infty(\Omega^c)$, $\beta$ is nontrivial, i.e., $|\Omega^c\cap \{\beta>0\}|>0$ and the embedding $\WnuOmR\hookrightarrow L^p(\Omega)$ is compact.  Assume  $f\in \WnuOmR'$ $($or $ f \in L^{p'}(\Omega))$  and  $g\in \TnuOm'$ $($or $g\in L^{p'}(\Omega^c, \omega^{1-p'}))$. 
\begin{enumerate}[$(i)$]
\item \textbf{Existence}. There exists a unique  $u\in \WnuOmR$ satisfying \eqref{eq:var-nonlocal-robin}. 
\item \textbf{Boundedness}. Moreover, there exists $C = C(d,p, \Omega, \nu, \beta)>0$ such that
\begin{align}\label{eq:weak-sol-bounded-robin}
\|u\|_{\WnuOmR}\leq C \Big(\|f\|_{\WnuOmR'}+\|g\|_{\TnuOm'}\Big)^{1/(p-1)}.
\end{align}
\item \textbf{Continuity}.   
If $u_i$, $i= 1,2$ satisfies \eqref{eq:var-nonlocal-robin} with $f=f_i$  and $g=g_i$ then 
\begin{align*}
\|u_1-u_2\|_{\WnuOmR}\leq 
\begin{cases}
C \big(\|f_1-f_2\|_{\WnuOmR'}\hspace*{-1ex}+\|g_1-g_2\|_{\TnuOm'}\big)^{\frac{1}{p-1}} &\hspace*{-1ex} \text{$p\geq2$},\\
CM\big( \|f_1-f_2\|_{\WnuOmR'} \hspace*{-1ex}+  \|g_1-g_2\|_{\TnuOm'}\big) &\hspace*{-1ex} \text{$p<2$},
\end{cases}
\end{align*}
where  
$M=M(f_1,f_2,g_1,g_2)= \big(\sum_{i=1}^2 \|f_i\|_{\WnuOmR'}+\|g_i\|_{\TnuOm'}\big)^{\frac{2-p}{p-1}}.$
\end{enumerate}
\end{theorem}

\begin{proof}
First of all, we claim that the form $Q_\beta(\cdot, \cdot)$ is coercive on $\WnuOmR$. To prove this, it is sufficient to prove that there exists a constant $C=C(d, p, \Omega, \nu,\beta)>0$ such that
\begin{align}\label{eq:coercive-robin}
Q_\beta(u,u)\geq C\|u\|^p_{\WnuOmR}\quad\text{for all }~u\in W_\nu(\Omega|\mathbb{R}^d).
\end{align}
Assume $C$ does not exist, then one finds a sequence $u_n \in \WnuOmR $ preferably  $\|u_n\|_{\WnuOmR}=1$ such that
\begin{align*}
\mathcal{E}( u_n,u_n)+ \int_{\Omega^c} |u_n(y)|^p\beta(y)\d y= Q_\beta(u_n , u_n )
<\frac{1}{2^n}\,. 
\end{align*}
\noindent  In virtue of the compactness of the embedding
$\WnuOmR \hookrightarrow L^p(\Omega)$, the sequence $(u_n)_n$ converges up to a subsequence in $L^p(\Omega)$ to some $u\in L^p(\Omega)$. It turns out that $\|u\|_{L^p(\Omega)}= 1$, since 
$\mathcal{E}(u_n,u_n )\xrightarrow{n\to\infty}0$ and for all $n\geq 1$, $\|u_n\|_{\WnuOmR}=1$. Moreover, Fatou's lemma implies 
\begin{align*}
	\cE_\Omega(u,u)\leq\liminf_{n \to \infty}\cE(u_n,u_n)=0.
\end{align*}
Given that $\nu$ is of full support and $\mathcal{E}_\Omega(u,u)=0$ it follows that $u$ is constant almost everywhere on $\Omega$. Since $\mathcal{E}(u_n,u_n )\xrightarrow{n \to \infty} 0$, we can extend $u$ as a function that is constant almost everywhere on $\R^d$ so that $u\in \WnuOmR$. That $\mathcal{E}(u_n,u_n )\xrightarrow{n \to \infty} 0$ and $\|u_n-u\|_{L^p(\Omega)} \xrightarrow{n \to \infty}0$ imply that $\|u_n-u\|_{\WnuOmR} \xrightarrow{n\to\infty}0$ as we have  $\mathcal{E}(u,u)=0$.  On the other hand, since $\beta\omega^{-1}\in L^{\infty}(\Omega^c)$ and we have  the continuous embedding $\WnuOmR \hookrightarrow L^p(\Omega^c, \omega)$, by the nonlocal trace Theorem \ref{thm:trace-nonloc-thm}, we have 
\begin{align*}
\int_{\Omega^c} |u(y)|^p\beta(y)\d y&\leq 2^{p-1} \int_{\Omega^c} |u_n(y)|^p\beta(y)\d y
\\&+2^{p-1}\|\beta\omega^{-1}\|_\infty \int_{\Omega^c} |u_n(y)-u(y)|^p\omega(y)\d y\\
&\leq 2^{p-1}Q_\beta(u_n, u_n)+ C\|u_n-u\|^p_{\WnuOmR}\xrightarrow[]{n \to \infty}0\,.
\end{align*}
\noindent It follows that $u=0$, since $u$ is constant a.e and $|\Omega^c\cap \{\beta>0\}|>0$. This negates the fact that $\|u\|_{L^p(\Omega)}=1$ and hence our initial assumption was wrong. 
The other details follow analogously as for the Neumann problem \eqref{eq:var-nonlocal-Neumann-gen}, by replacing the form $\cE(\cdot,\cdot)$ with $Q_\beta(\cdot, \cdot)$. 
\end{proof}

\section{Transition from nonlocal to local}\label{sec:from-nonlocal-local}
In this  section  we introduce  and characterize what we name as $p$-L\'{e}vy approximation family; which serves as the main tool for the convergence of nonlocal objects to local ones. For instance, we show the convergence of nonlocal  energy forms to local ones, as well as the pointwise convergence of nonlocal $p$-L\'{e}vy operators to the $p$-Laplacian.  

\subsection{Basics on $p$-L\'{e}vy approximation family} A family of  radial $p$-L\'{e}vy integrable kernels $(\nu_\eps)_{\eps>0}$, $\nu_\eps : \R^d\setminus\{0\}\to [0,\infty)$, is said to be a  $p$-L\'{e}vy approximation family if it satisfies 
\begin{align}\label{eq:plevy-approx}
\int_{\R^d}\hspace*{-1ex}(1\land|h|^p) \nu_\eps (h)\d h=1\,\,\text{and for all $\delta>0,$}\,\,\lim_{\eps\to0} \int_{|h|>\delta} \hspace*{-2ex}(1\land|h|^p) \nu_\eps (h)\d h=0. 
\end{align}
\begin{proposition}\label{prop:asymp-nu}
Let $\beta\in \R$ and  $R>0$. Assume $(\nu_\eps)_\eps$ satisfies \eqref{eq:plevy-approx} then 
\begin{align*}
\lim_{\varepsilon\to 0}\int_{|h|\leq R} (1\land |h|^{p+\beta})\nu_\eps (h)\d h =
\begin{cases}
0\quad \text{ if}\quad \beta >0\\
1\quad \text{ if}\quad \beta =0\\
\infty \quad \text{if}\quad \beta <0.
\end{cases}
\end{align*}
\end{proposition}

\begin{proof}
Fix $\delta\in (0,1\land R)$ sufficiently small. By \eqref{eq:plevy-approx} we get 
\begin{align*}
&\lim_{\varepsilon \to 0} \il_{ \delta<|h|\leq R} (1\land |h|^p) \nu_\varepsilon(h) \d h
\leq \lim_{\eps \to 0} \int_{ |h|>\delta} (1\land |h|^p) \nu_\eps (h) \d h=0,
\\
&\lim_{\eps \to 0}\int_{ |h|<\delta} (1\land |h|^p) \nu_\varepsilon(h) \d h
=1- \lim_{\varepsilon \to 0} \int_{ |h|\geq \delta} (1\land |h|^p) \nu_\eps (h) \d h=1. 
\end{align*}
The case $\beta=0$ follows immediately.
If $\beta>0$ then we have 
\begin{align*}
&\lim_{\varepsilon \to 0} \hspace*{-1ex}\il_{|h|\leq R} \hspace{-2ex}(1\land |h|^{p+\beta})\nu_\varepsilon(h)\mathrm{d}h
\leq \lim_{\varepsilon \to 0} \Big( R^{\beta} \hspace{-3ex}\il_{\delta< |h|\leq R} \hspace{-3ex}\nu_\varepsilon(h) \mathrm{d}h + \delta^{\beta} \hspace{-2ex}\il_{|h|\leq\delta}\hspace{-2ex} (1\land |h|^p) \nu_\varepsilon(h) \mathrm{d}h\Big) \hspace{-0.5ex}= \delta^{\beta} \hspace{-1ex}.
\end{align*}
Analogously, since $\delta<1$, if $\beta<0$ then  we have 
\begin{align*}
\lim_{\varepsilon \to 0}\hspace{-1ex}\il_{|h|\leq R} \hspace{-2ex}(1\land |h|^{p+\beta})\nu_\varepsilon(h)\mathrm{d}h
&\geq \lim_{\varepsilon \to 0} \Big( \delta^pR^{\beta} \hspace{-3ex}\il_{\delta< |h|\leq R} \hspace{-3ex}\nu_\varepsilon(h) \d h + \delta^{\beta} \hspace{-2ex}\il_{|h|\leq\delta}\hspace{-2ex} (1\land |h|^p)\nu_\varepsilon(h) \mathrm{d}h\Big) \hspace{-0.5ex}= \hspace{-0.5ex} \delta^{\beta}\hspace{-1ex}.
\end{align*}
In either case, letting $\delta \to 0$ provides the claim.
\end{proof}

\begin{remark} Assume the family $(\nu_\eps)_\eps$ satisfies \eqref{eq:plevy-approx}. Note that the relation 
\begin{align}\label{eq:concentration-prop}
\lim_{\varepsilon\to 0}\int_{|h|> \delta }(1\land |h|^p)\nu_\varepsilon(h)\,\d h=0,
\end{align}
is often known as the concentration property and is merely equivalent to 
\begin{align*}
\lim_{\varepsilon\to 0}\int_{|h|>\delta }\nu_\varepsilon(h)\,\d h=0\quad\text{for all}\quad \delta >0.
\end{align*}
\noindent Indeed, for all $\delta >0$ we have 
\begin{align*}
\int_{|h|>\delta }(1\land |h|^p)\nu_\varepsilon(h)\,\d h
\leq \int_{|h|> \delta }\hspace{-2ex}\nu_\varepsilon(h)\,\d h\leq (1\land \delta^p)^{-1}\,\int_{|h|> \delta }\hspace{-2ex}(1\land |h|^p)\nu_\varepsilon(h)\,\d h.
\end{align*}
Consequently, for all $\delta >0$ we also have 
\begin{align}\label{eq:concentration-bis}
\lim_{\varepsilon\to 0}\int_{|h|\leq \delta }(1\land |h|^p)\nu_\varepsilon(h)\,\d h= \lim_{\varepsilon\to 0}\int_{|h|\leq \delta }|h|^p\nu_\varepsilon(h)\,\d h=1.\quad
\end{align}
\end{remark}

\noindent Let us mention some prototypical examples  of sequences $(\nu_\eps)_\eps$ satisfying  \eqref{eq:plevy-approx} of interest here. For more  examples we refer the reader to \cite{Fog23, guy-thesis,FGKV20}. 

\begin{example}
Assume   $\nu$ is radial and $p$-L\'{e}vy  normalized, i.e.,  $$\int_{\R^d}(1\land|h|^p)\,\nu(h)\d h=1.$$	
Consider the family $(\nu_{\eps})_{\eps}$ defined as the rescaled version of $\nu$  with
\begin{align*}
\begin{split}
\nu_\varepsilon(h) = 
\begin{cases}
\varepsilon^{-d-p}\nu\big(h/\varepsilon\big)& 
\text{if}~~|h|\leq \varepsilon,\\
\varepsilon^{-d}|h|^{-p}\nu\big(h/\varepsilon\big)& \text{if}~~\varepsilon<|h|\leq 1,\\
\varepsilon^{-d}\nu\big(h/\varepsilon\big)& \text{if}~~|h|>1.
\end{cases}
\end{split}
\end{align*}

\end{example}

\begin{example}\label{Ex:stable-class}
Consider the sequence $(\nu_\varepsilon)_\varepsilon$ of fractional kernels defined by
$$\nu_\varepsilon(h) = a_{\varepsilon, d,p} |h|^{-d-(1-\eps)p}\quad\text{with }\quad a_{\varepsilon, d,p} = \frac{p\eps(1-\eps)}{|\mathbb{S}^{d-1}|}.$$ 
Indeed, passing through polar coordinates yields 
\begin{align*}
& \int_{\R^d} \frac{(1\land |h|^p)}{|h|^{-d-(1-\eps)p} }\d h
= |\mathbb{S}^{d-1}| \Big(\int_{0}^{1}\hspace{-1ex} r^{\eps p-1}\,\d r+ \hspace{-1ex} \int_{1}^{\infty} \hspace{-1ex}r^{-1-(1-\eps)p}\,\d r\Big) 
\hspace{-0.5ex}=	a^{-1}_{\eps,d,p}. 
\end{align*}
For $\delta>0,$ a similar computation gives
\begin{align*}
&a_{\eps, d,p}
\il_{|h|\geq \delta} (1\land |h|^p) |h|^{-d-(1-\eps)p}\,\d h
\leq \eps \delta^{-(1-\eps)p} \xrightarrow{\eps \to 0}0.
\end{align*}
The choice of $\nu_\eps(h) = a_{\eps,d,p}  |h|^{-d-(1-\eps)p}$ gives rise to a multiple of fractional $p$-Laplace operator, namely, we have $L_\eps u= \frac{2a_{\eps,d,p}}{C_{d,1-\eps,p}} (-\Delta)^s_pu$, $s=1-\eps$. We emphasize that  $C_{d,p,s}$ (cf. Section \ref{sec:norming-cste-pfrac}) is  the normalizing constant of $(-\Delta)^s_p$ and  that $ \frac{ 2a_{\eps,d,p} }{C_{d,1-\eps,p}}\to K_{d,p}$  as $\eps\to 0$. 
\end{example}

\begin{example} \label{Ex:example-poincre1} Let $0<\varepsilon < 1$ and $\beta >-d$. Set 
\begin{align*}
\nu_\varepsilon(h) = \frac{d+\beta}{ |\mathbb{S}^{d-1}|\varepsilon^{d+\beta}} |h|^{\beta-p}\mathds{1}_{B_{\varepsilon}}(h). 
\end{align*} 
Some special cases are obtained with $ \beta = 0$, $\beta =p$ and $\beta = (1-s)p-d$ for $s\in (0,1)$. For the limiting case $\beta=-d$ consider $0<\varepsilon <\varepsilon_0<1$ and put
\begin{align*}
\nu_\varepsilon(h) = \frac{1}{|\mathbb{S}^{d-1}|\log(\varepsilon_0/\varepsilon )}|h|^{-d-p}\mathds{1}_{ B_{\varepsilon_0}\setminus B_\varepsilon}(h). 
\end{align*} 

\end{example}

\subsection{Characterization of $p$-L\'{e}vy approximation family} Now we characterize the class $(\nu_\eps)_\eps$ such that for all $u\in W^{1,p}(\R^d)$, the following formula holds
\begin{align}\label{eq:limit-goal}
\lim_{\eps\to0} \iil_{\R^d\R^d} |u(x)- u(y)|^p\nu_\eps(x-y)\d y\,\d x
=K_{d,p}\int_{\R^d}|\nabla u(x)|^p\d x.
\end{align}
In fact, this is equivalent to saying that 	
\begin{align}\label{eq:plevy-approx-bis}
\lim_{\eps\to0}\int_{\R^d}\hspace{-2ex}(1\land|h|^p) \nu_\eps (h)\d h=1\,\, \text{and for all $\delta>0,$}\,\, \lim_{\eps\to0} \int_{|h|>\delta} \hspace{-2ex}\nu_\eps (h)\d h=0. 
\end{align} 

To be more accurate, we have the following. 
\begin{theorem}\label{thm:charac-plevy-concentrate}
Let $(\nu_\eps)_\eps$  be a family of radial functions.  
The following are equivalent.
\begin{enumerate}[$(i)$]
\item  
The family  $(\nu_\eps)_\eps$ satisfies \eqref{eq:plevy-approx-bis}.  

\item  The relation \eqref{eq:limit-goal} holds for all $u\in W^{1,p}(\R^d)$ .

\item   The relation \eqref{eq:limit-goal} holds for all $u\in  C_c^\infty(\R^d)$.

\item  There exists $u\in C_c^\infty(\R^d)\setminus\{0\}$ such that the relation \eqref{eq:limit-goal} holds for each $u_\tau(x) =\tau^d u(\tau x)$, $\tau>0$. 
\end{enumerate}
\end{theorem}
\begin{remark}
It is worthwhile noticing that properties in \eqref{eq:plevy-approx} and \eqref{eq:plevy-approx-bis}  are equivalent in the sense that, using the normalizing factor $c_\eps^{-1}=\int_{\R^d}(1\land|h|^p)\nu_\eps(h)\d h$ one readily recovers \eqref{eq:plevy-approx}  from \eqref{eq:plevy-approx-bis} vice-versa. Whence Theorem \ref{thm:charac-plevy-concentrate} also characterizes families satisfying \eqref{eq:plevy-approx}. 
\end{remark}
\begin{proof}
Up to replacing $\nu_\eps$ with $c_\eps\nu_\eps$ with $c_\eps^{-1}=\int_{\R^d}(1\land|h|^p)\nu_\eps(h)\d h$, the implication $(i)\implies (ii)$ follows from \cite{Fog23} or \cite[Theorem 5.23]{guy-thesis}. We only prove that  $(iv)\implies (i)$, as the remaining implications are trivial.  
Note in passing that, since $u\in C_c^\infty(\R^d)\setminus\{0\}$ we have $\|u_\tau\|_{L^p(\R^d)} = \|u\|_{L^p(\R^d)}\neq0$ and $\|\nabla u_\tau\|_{L^p(\R^d)} = \tau^{1/p}\|\nabla u\|_{L^p(\R^d)}\neq0$. By continuity of the shift operator, for every $\eta\in (0,1)$ there  is $0<\delta_\eta<\eta$ such that $\|\nabla u(\cdot+h)-\nabla u\|_{L^p(\R^d)}<\eta$ for all $|h|\leq \delta_\eta $. Thus for $\tau>0$ we find that 
\begin{align*}
\|\nabla u_\tau(\cdot+h)-\nabla u_\tau\|_{L^p(\R^d)}<\tau^{1/p}\eta\quad \text{for all $|h|\leq \delta_\eta/\tau $.} 
\end{align*} 
Minkowski's inequality implies 
\begin{align*}
\Big(\int_{\R^d}\int_{B_{\delta_\eta/\tau}(0)}
&	|\nabla u_\tau(x)\cdot h|^p\nu_\eps (h)\d h\d x\Big)^{1/p} 
\leq \tau^{1/p}\eta\Big( \int_{B_{\delta_\eta/\tau}(0)} \hspace{-2ex}|h|^p\nu_\eps(h)\d h\Big)^{1/p}\\ \\&+\Big(\int_{\R^d}\int_{B_{\delta_\eta/\tau}(0)}\hspace{-0.2ex} \Big|\int_0^1 \nabla u_\tau(x+th) \cdot h\d t\Big|^p\nu_\eps(h) \d h\d x\Big)^{1/p}.
\end{align*} 
Observe  that, since $\nu_\eps$ is radial, using polar coordinates yields, 
\begin{align}\label{eq:radial-decoupling}
\begin{split}
\int_{B_{\delta}(0)}\hspace*{-1ex}|\nabla u(x)\cdot h|^p\nu_\eps (h)\d h
&=K_{d,p}|\nabla u(x)|^p\int_{B_{\delta}(0)} \hspace{-2ex}|h|^p\nu_\eps(h)\d h.
\end{split}
\end{align}
\noindent Accordingly, using the fundamental theorem of calculus, the formula   \eqref{eq:rotation-inv} and the foregoing  yields
\begin{align*}
\begin{split}
&\iil_{\R^d\,B_{\delta_\eta/\tau}(0)}\hspace*{-1ex}|u_\tau(x)-u_\tau(x+h)|^p\nu_\eps(h)\,\d h\,\d x\\
&=\int_{\R^d} \int_{B_{\delta_\eta/\tau}(0)} \Big| \int_0^1\nabla u_\tau(x+th)\cdot h\d t\Big|^p \nu_\eps(h) \d h\,\d x\\
&\geq \Big[\Big(\int_{\R^d}\int_{B_{\delta_\eta/\tau}(0)}\hspace*{-2ex}	|\nabla u_\tau(x)\cdot h|^p\nu_\eps (h)\d h\d x\Big)^{\frac1p}\hspace*{-1ex}-\eta\Big( \int_{B_{\delta_\eta/\tau}(0)} \hspace*{-0.1ex}|h|^p\nu_\eps(h)\d h\Big)^{1/p} \Big]^p\\
&= \tau \Big(K^{1/p}_{d,p}\|\nabla u\|_{L^p(\R^d)}-\eta\Big)^p \int_{B_{\delta_\eta/\tau}(0)} |h|^p\nu_\eps (h)\d h. 
\end{split}
\end{align*}
 For the sake of brevity, let us put 
\begin{align*}
\cE_{\R^d}^\eps(u_\tau,u_\tau)
&=\iil_{\R^d\,\R^d}|u_\tau(x)-u_\tau(x+h)|^p\nu_\eps(h)\,\d h\,\d x,\\  
R_{\tau} (\eta,\eps)&= \iil_{\R^d\,B^c_{\delta_\eta/\tau}(0)}\hspace*{-1ex}|u_\tau(x)-u_\tau(x+h)|^p\nu_\eps(h)\,\d h\,\d x. 
\end{align*}
Therefore from the above, we find that 
\begin{align}\label{eq:low-eta-eps-estimate}
\begin{split}
\cE_{\R^d}^\eps(u_\tau,u_\tau)
&\geq \tau \Big(K^{1/p}_{d,p}\|\nabla u\|_{L^p(\R^d)}-\eta\Big)^p \hspace*{-1ex}\int_{B_{\delta_\eta/\tau}(0)} \hspace{-3ex}|h|^p\nu_\eps (h)\d h+ R_{\tau} (\eta,\eps). 
\end{split}
\end{align}
Using once again the fundamental theorem of calculus and  \eqref{eq:radial-decoupling} we easily get
\begin{align}\label{eq:upper-eta-eps-estimate}
\begin{split}
\cE_{\R^d}^\eps(u_\tau,u_\tau)&\leq \tau K_{d,p}\|\nabla u\|^p_{L^p(\R^d)} \int_{B_{\delta_\eta/\tau}(0)} \hspace{-1ex}|h|^p\nu_\eps (h)\d h+ R_{\tau} (\eta,\eps). 
\end{split}
\end{align} 
Now we consider the following quantities 
\begin{align*}
\rho_\tau^+&=  \limsup_{\eta\to 0}  \limsup_{\eps\to 0} \int_{B_{\delta_\eta/\tau}(0)} \hspace{-4ex}|h|^p\nu_\eps (h)\d h, \,\,
\rho_\tau^-=  \liminf_{\eta\to 0}  \liminf_{\eps\to 0} \int_{B_{\delta_\eta/\tau}(0)} \hspace{-4ex}|h|^p\nu_\eps (h)\d h, 
\\
R_\tau^+&=  \limsup_{\eta\to 0}  \limsup_{\eps\to 0} R_{\tau} (\eta,\eps) 
\quad \text{and}\quad 
R_\tau^-=  \liminf_{\eta\to 0}  \liminf_{\eps\to 0}  R_{\tau} (\eta,\eps). 
\end{align*}

Passing to the  $\lim\hspace{-0.5ex}\sup$ and $\lim\hspace{-0.5ex}\inf$ in \eqref{eq:low-eta-eps-estimate} and \eqref{eq:upper-eta-eps-estimate} respectively, we get the following  
\begin{align*}
\tau K_{d,p}\|\nabla u\|^p_{L^p(\R^d)}&\geq  \tau\rho_\tau^+ K_{d,p}\|\nabla u\|^p_{L^p(\R^d)} + R_\tau^+
\\
\tau K_{d,p}\|\nabla u\|^p_{L^p(\R^d)}&\leq  \tau\rho_\tau^- K_{d,p}\|\nabla u\|^p_{L^p(\R^d)} + R_\tau^-. 
\end{align*}
It follows that $\rho_\tau^+\leq 1$. To show that $\rho_\tau^-\geq1$ we observe that analogously to \eqref{eq:upper-eta-eps-estimate},  for all $\delta>0$ and $\theta>0$ we have
\begin{align*}
\begin{split}
\cE_{\R^d}^\eps(u_\theta, u_\theta)&\leq \theta K_{d,p}\|\nabla u\|^p_{L^p(\R^d)} \hspace*{-1ex}\int_{B_{\delta}(0)} \hspace{-2ex}|h|^p\nu_\eps (h)\d h+ 2^p\|u_\theta\|^p_{L^p(\R^d)} \hspace*{-1ex}\int_{|h|\geq \delta}\hspace*{-2ex}\nu_\eps(h)\d h. 
\end{split}
\end{align*}
Then passing to the $\lim\hspace{-0.5ex}\inf$  like previously also yields that 
\begin{align*}
1\leq \liminf_{\eps\to 0} 
\int_{B_{\delta}(0)} \hspace{-1ex}|h|^p\nu_\eps (h)\d h + \frac{1}{\theta}  \frac{2^p\|u\|^p_{L^p(\R^d)} }{K_{d,p}\|\nabla u\|^p_{L^p(\R^d)}} \liminf_{\eps\to 0}\int_{|h|\geq \delta}\hspace*{-1ex}\nu_\eps(h)\d h.
\end{align*}
Letting $\theta\to \infty$ implies that 
\begin{align*}
1\leq \liminf_{\eps\to 0} 
\int_{B_{\delta}(0)} \hspace{-1ex}|h|^p\nu_\eps (h)\d h  \quad\text{for all $\delta>0$}. 
\end{align*}
In particular taking $\delta=\delta_\eta/\tau$ we obtain
\begin{align*}
\rho_\tau^-=  \liminf_{\eta\to 0}  \liminf_{\eps\to 0} \int_{B_{\delta_\eta/\tau}(0)} \hspace{-1ex}|h|^p\nu_\eps (h)\d h\geq1. 
\end{align*}

From the foregoing we get $\rho_\tau^+\leq 1\leq \rho_\tau^-$ , that is, we have  
\begin{align}\label{eq:low-upper-lim-rho}
\rho_\tau^+= \rho_\tau^-= \lim_{\eta\to 0}  \lim_{\eps\to0} \int_{B_{\delta_\eta/\tau}(0)}\hspace{-1ex}|h|^p\nu_\eps (h)\d h=1. 
\end{align}
Therefore, we also deduce that 
\begin{align*}
R_\tau^+&=  \limsup_{\eta\to 0}  \limsup_{\eps\to 0} \iil_{\R^d\,B^c_{\delta_\eta/\tau}(0)}\hspace*{-1ex}|u_\tau(x)-u_\tau(x+h)|^p\, \d h\,\d x=0. 
\end{align*}
For fixed $\delta>0$ and $\eta<\delta\tau$ we have 
\begin{align*}
\iil_{\R^d\,B^c_{\delta}(0)}\hspace*{-2ex}|u_\tau(x)-u_\tau(x+h)|^p\nu_\eps(h)\,\d h\,\d x\leq \hspace*{-3ex} \iil_{\R^d\,B^c_{\delta_\eta/\tau}(0)}\hspace*{-4ex}|u_\tau(x)-u_\tau(x+h)|^p\nu_\eps(h)\,\d h\,\d x. 
\end{align*}
This implies that,  for all $\delta>0$ we have  
\begin{align*}
\limsup_{\eps\to 0} \iil_{\R^d\,B^c_{\delta}(0)}\hspace*{-1ex}|u_\tau(x)-u_\tau(x+h)|^p\nu_\eps(h)\,\d h\,\d x=0.
\end{align*}
For fixed $\delta>0$ and $\tau< \frac{\delta}{2}$ so that $\supp u_\tau\subset  B_{\delta/2}(0)$, if $x\in B_{\delta/2}(0)$ and $h\in B^c_{\delta}(0)$ we have $x+h\in B^c_{\delta/2}(0)$  and hence $u_\tau(x+h)=0$. Therefore, it follows that 

\begin{align*}
0&= \limsup_{\eps\to 0} \iil_{\R^d\,B^c_{\delta}(0)}\hspace*{-1ex}|u_\tau(x)-u_\tau(x+h)|^p\nu_\eps(h)\,\d h\,\d x\\
&\geq  \|u\|^p_{L^p(\R^d)}  \limsup_{\eps\to 0} \int_{|h|\geq\delta} (1\land |h|^p)\nu_\eps(h)\,\d h.
\end{align*}
Therefore we have, 
\begin{align*}
\lim_{\eps\to0} \int_{|h|\geq\delta} (1\land |h|^p)\nu_\eps(h)\,\d h=0\quad\text{for all $\delta>0$}. 
\end{align*}
Finally, combining this and \eqref{eq:low-upper-lim-rho}, with $\tau=1$, it follows that 
\begin{align*}
\lim_{\eps\to0} \int_{\R^d} (1\land |h|^p)\nu_\eps(h)\,\d h 
= \lim_{\eta\to 0}  \lim_{\eps\to0} \int_{|h|<\delta_\eta} \hspace*{-1.2ex}|h|^p\nu_\eps(h)\,\d h
= 1. 
\end{align*}
The proof is now complete. 
\end{proof}
\subsection{Pointwise asymptotic}
In this section, we wish to establish the asymptotics of the nonlinear nonlocal operators $L_{\varepsilon} u $ and $\cN_{\varepsilon} u $ as $\eps\to0$, with 
\begin{align*}
L_{\varepsilon} u (x)&= 2\pv \int_{\R^d}\!\! \psi(u(x)-u(y))\nu_\varepsilon(x-y)\d y,\qquad\hbox{$(x\in \R^d)$}
\\
\cN_\varepsilon u (x)&= 2 \int_{\Omega}\psi(u(x)-u(y))\nu_\varepsilon(x-y)\d y
\qquad\hbox{$(x\in \Omega^c)$}.
\end{align*}
Typically we show that $L_{\varepsilon} u$ converges to $K_{d,p}\Delta_p u$ pointwise and in the weak sense, while  $\cN_{\varepsilon} u $ converges to $K_{d,p}\partial_{n,p} u= K_{d,p}|\nabla u|^{p-2}\nabla u\cdot n$ in a (sort of) weak sense, where $n$ is normal vector on $\partial\Omega$. In fact, here we extend the linear results \cite[Proposition 2.5]{Fog23} and \cite[Lemma 5.3]{FK22} which only deals with the particular case $p=2$; see also \cite[Lemma 5.75 \& Proposition 2.38]{guy-thesis} wherein the asymptotic for general nonlocal elliptic operators is treated. To begin with, let us establish the following spherical mean representation for the $p$-Laplacian. 
\smallskip 

\begin{lemma}[Spherical mean representation of $\Delta_p$]\label{lem:mean-value-plaplace}
Assume $u\in C^2(B_1(x))$ and $\nabla u(x)\neq 0$ when $1<p<2$ then for $d\geq2$ we have 
\begin{align*}
\fint_{\mathbb{S}^{d-1}}|\nabla u(x)\cdot w|^{p-2} D^2u(x)w\cdot w \, \d \sigma_{d-1}(w)= \frac{K_{d,p}}{p-1}\Delta_p u(x).
\end{align*}
\end{lemma}

\begin{proof}
If $p\geq 2$ and $\nabla u(x)=0$, there is nothing to prove. Let us put $e(x)= \frac{\nabla u(x)}{|\nabla u(x)|}$ and consider $O:\R^d\to \R^d$ be a rotation, i.e., $O^t O=I$ such that $ e(x)= Oe_d$ with $e_d=(0,0,\cdots,1)$. Note that $Oz\cdot Oy= z\cdot y$. By rotational invariance of the Lebesgue measure, the change of variables $w=O\xi$ yields $\d\sigma_{d-1}(w)= \d\sigma_{d-1}(\xi)$  and 
\begin{align*}
Iu(x)&:=\fint_{\mathbb{S}^{d-1}}|\nabla u(x)\cdot w|^{p-2} D^2u(x)w\cdot w \, \d \sigma_{d-1}(w)\\
&= |\nabla u(x)|^{p-2}\fint_{\mathbb{S}^{d-1}}|e(x)\cdot w|^{p-2} D^2u(x)w\cdot w \, \d \sigma_{d-1}(w)\\
&= |\nabla u(x)|^{p-2}\fint_{\mathbb{S}^{d-1}}|\xi_d|^{p-2} [O^tD^2u(x) O] \xi\cdot \xi\, \d \sigma_{d-1}(\xi)\\
&= |\nabla u(x)|^{p-2} \sum_{i,j=1}^d q_{ij}(x)\fint_{\mathbb{S}^{d-1}}|w_d|^{p-2} w_iw_j\, \d \sigma_{d-1}(w). 
\end{align*}
Here $Q(x)= (q_{ij}(x))_{1\leq i,j\leq d}$ is the $d\times d$-matrix $Q(x)= O^tD^2u(x) O$. By symmetry, we get 
\begin{align}\label{eq:xzero-symmetry}
\fint_{\mathbb{S}^{d-1}}|w_d|^{p-2} w_iw_j\, \d \sigma_{d-1}(w)=0 \quad\text{if $i\neq j$}. 
\end{align}
It is known from\cite{Fog23} that for $i=j=d$ we have  the following formula, 
\begin{align*}
K_{d,p}&= \fint_{\mathbb{S}^{d-1}} |w_d|^p\d\sigma_{d-1}(w) = \frac{\Gamma\big(\frac{d}{2}\big)\Gamma\big(\frac{p+1}{2}\big)}{\Gamma\big(\frac{d+p}{2}\big) \Gamma\big(\frac{1}{2}\big)}.
\end{align*}
Furthermore, the quantity $\fint_{\mathbb{S}^{d-1}}|w_d|^{p-2} w^2_j\, \d \sigma_{d-1}(w)$ is oblivious to the choice of $j=1,2,\cdots, d-1$. Indeed, let  $O':\R^{d-1}\to \R^{d-1}$ be a rotation such that $O' e_j'= e_1'$ where we put $e_i= (e_i', 0)$ and $x=(x', x_d)$. Then $x\mapsto Ox= (O'x', x_d)$ is a $x_d$-invariant rotation such that $Oe_j= e_1$.  Enforcing the change of variables $\xi=O w$, $ \d\sigma_{d-1}(w)= \d\sigma_{d-1}(\xi)$ we obtain 
\begin{align*}
\fint_{\mathbb{S}^{d-1}} |w_d|^{p-2} w_j^2\d\sigma_{d-1}(w)=\fint_{\mathbb{S}^{d-1}} |\xi_d|^{p-2} \xi_1^2\d\sigma_{d-1}(\xi). 
\end{align*}
Since $|w'|^2= w_1^2+w_2^2+\cdots+w_{d-1}^2= 1-|w_d|^2$ for $w\in\mathbb{S}^{d-1},$ this implies that 
\begin{align*}
\fint_{\mathbb{S}^{d-1}} |w_d|^{p-2} w_j^2\d\sigma_{d-1}(w)
&= \frac{1}{d-1}\fint_{\mathbb{S}^{d-1}} |w_d|^{p-2}(1-|w_d|^2)\d\sigma_{d-1}(w)
\\&= \frac{1}{d-1}\big(K_{d,p-2} -K_{d,p}\big).
\end{align*}
Applying the formula $\Gamma(x+1)= x\Gamma(x)$ one readily obtains
\begin{align*}
K_{d,p-2}=\frac{\Gamma\big(\frac{d}{2}\big)\Gamma\big(\frac{p-1}{2}\big)}{\Gamma\big(\frac{d+p-2}{2}\big) \Gamma\big(\frac{1}{2}\big)}
=\frac{d+p-2}{p-1}  \frac{\Gamma\big(\frac{d}{2}\big)\Gamma\big(\frac{p+1}{2}\big)}{\Gamma\big(\frac{d+p}{2}\big) \Gamma\big(\frac{1}{2}\big)} = \frac{d+p-2}{p-1} K_{d,p}. 
\end{align*}
Inserting this into the previous expression yields
\begin{align}\label{eq:xKdp}
\fint_{\mathbb{S}^{d-1}} |w_d|^{p-2} w_j^2\d\sigma_{d-1}(w)
=\begin{cases}
\frac{K_{d,p}}{p-1} &\qquad j\neq d,\\
K_{d,p}&\qquad j=d. 
\end{cases}
\end{align}
Beside this, using $Tr(AB)= Tr(BA)$ we observe that 
\begin{align}\label{eq:xtraceqQ}
\sum_{j=1}^d q_{jj} (x)& = Tr(Q(x))= Tr(O^tD^2u(x) O)= Tr( D^2u(x))= \Delta u(x). 
\end{align}
In addition,  since  $Q(x) e_d\cdot e_d= D^2u(x) Oe_d\cdot Oe_d$ and  $\frac{\nabla u(x)}{ |\nabla u(x)|}= e(x)=Oe_d$ we obtain  
\begin{align}\label{eq:xqdd-expression}
\begin{split}
	q_{dd}(x) = Q(x) e_d\cdot e_d
	&= |\nabla u(x)|^{-2} D^2u(x) \nabla u(x)\cdot \nabla u(x)\\
	&= |\nabla u(x)|^{-2} \Delta_\infty u(x). 
\end{split}
\end{align}
Combing \eqref{eq:xzero-symmetry}, \eqref{eq:xKdp}, \eqref{eq:xtraceqQ} and \eqref{eq:xqdd-expression} we find that 
\begin{align*}
&\sum_{i,j=1}^d q_{ij}(x)\fint_{\mathbb{S}^{d-1}} \hspace*{-2ex} |w_d|^{p-2} w_iw_j\, \d \sigma_{d-1}(w)
= \sum_{j=1}^d q_{jj}(x)\fint_{\mathbb{S}^{d-1}} \hspace*{-2ex}|w_d|^{p-2} w^2_j\, \d \sigma_{d-1}(w)\\
& = \frac{K_{d,p}}{p-1}  \sum_{j=1}^{d-1} q_{jj}(x)+q_{dd}(x)K_{d,p} 
= \frac{K_{d,p}}{p-1}  \sum_{j=1}^d q_{jj}(x)+\frac{p-2}{p-1}q_{dd}(x)K_{d,p}\\
&= \frac{K_{d,p}}{p-1}\big( Tr(Q(x))+(p-2)q_{dd}(x)\big)
=\frac{K_{d,p}}{p-1} \big(\Delta u(x) + (p-2)\frac{\Delta_\infty u(x)}{|\nabla u(x)|^{2}} \big)
\end{align*}
Finally, inserting  this in the foregoing expression yields
\begin{align*}
Iu(x)&=
|\nabla u(x)|^{p-2} \sum_{i,j=1}^d q_{ij}(x)\fint_{\mathbb{S}^{d-1}}|w_d|^{p-2} w_iw_j\, \d \sigma_{d-1}(w)\\
&= \frac{K_{d,p}}{p-1} |\nabla u(x)|^{p-2} \big(\Delta u(x) + (p-2)|\nabla u(x)|^{-2} \Delta_\infty u(x)\big)\\
&= \frac{K_{d,p}}{p-1}\div( |\nabla u(x)|^{p-2} \nabla u(x))=  \frac{K_{d,p}}{p-1}\Delta_p u(x). 
\end{align*}
\end{proof}
It is worth mentioning that the computations yielding  the identities \eqref{eq:xtraceqQ} and \eqref{eq:xqdd-expression} are essentially adapted from the computations in \cite[Section 7]{IN10}. 

\begin{theorem}\label{thm:asymp-plevy-to-plaplace}
Let  $u\in L^\infty(\R^d)\cap C^2(B_\tau(x))$ for some $\tau>0$ with $\nabla u(x)\neq 0$ when $1<p<2$. Then we have 
\begin{align*}
\lim_{\varepsilon\to 0} L_{\varepsilon} u(x)= -K_{d,p} \Delta_p u(x).
\end{align*}
\end{theorem}

\begin{proof}
First of all, for every $\delta>0$ by boundedness of $u$ we have 
\begin{align}\label{eq:xfar-from-zero}
2\int_{|h|\geq \delta } \hspace*{-2ex} \big|\psi(u(x+h)-u(x))\big|\,\nu_\eps(h)\d h\leq 2^p\|u\|^{p-1}_{L^\infty(\R^d)} \int_{|h|\geq \delta } \hspace*{-2ex}\nu_\eps(h)\d h\xrightarrow{\eps\to0}0.
\end{align}
Since $\pv \int_{B_1(0)} \psi(\nabla u(x)\cdot h)\big]\,\nu_\eps(h)\,\d h=0$, the claim reduces to the following

\begin{align*}
\begin{split}
-\lim_{\eps\to0} L_\eps u(x) &= \lim_{\eps\to0} 2\int_{|h|<\delta } \psi(u(x+h)-u(x))\,\nu_\eps(h)\,\d h\\
&=	\lim_{\eps\to0} 2\int_{|h|<\delta } \big[\psi(\int_0^1\nabla u(x+th)\cdot h)-\psi(\nabla u(x)\cdot h)\big]\,\nu_\eps(h)\,\d h\\
&=	\lim_{\eps\to0} 2\int_{|h|<\delta } \psi'(a)(b-a) +R(a,b)\,\nu_\eps(h)\,\d h,
\end{split}
\end{align*}
where, using  the fundamental theorem $u(x+h)-u(x)= \int_0^1\nabla u(x+th)\cdot h\d t$, we put
\begin{align*}
a=\nabla u(x)\cdot h \qquad\text{and}\qquad  b=\int_0^1\nabla u(x+th)\cdot h\, \d t. 
\end{align*}
Furthermore, $\psi'(t)= (p-1)|t|^{p-2}$ and the remainder $R(a,b)$ is given by 
\begin{align*}
R(a,b)= \psi(b)-\psi(a) -\psi'(a)(b-a)= (b-a)\hspace*{-0.5ex}\int_0^1\hspace*{-1ex}\psi'(a+ t(b-a))-\psi'(a)\d t. 
\end{align*}
Without loss of generality, assume $\tau=1$, i.e., $u\in C^2(B_1 (x))$. For fixed $0<\eta<1$ there is   $0<\delta<1$ such that 
\begin{align}\label{eq:xmatrix-diff}
|D^2 u(x+h)-D^2u(x)|<\eta\qquad \text{whenever  $|h|<\delta$}. 
\end{align}
The  fundamental theorem of calculus yields
\begin{align}\label{eq:xdifference-b-a}
\begin{split}
b-a&= \int_0^1\nabla u(x+th)\cdot h-\nabla u(x)\cdot h\d t
= \frac{1}{2}\big[D^2u(x)\cdot h\big]\cdot h\\
&+  \int_0^1 t\int_0^1 \big[D^2u(x+sth)\cdot h-D^2 u(x)\cdot h\big]\cdot h\, \d s\,  \d t. 
\end{split}
\end{align}
In particular, the above implies 
\begin{align*}
|b-a|\leq C_0 |h|^2, \qquad C_0=\sup_{z\in B_1(x)} |D^2 u(z)|.
\end{align*}
\noindent Next we estimate the remainder $R(a,b)$, by distinguishing 3 cases: $1<p<2$, $2< p<3$ and $p\geq 3$. Observe in passing that $R(a,b)=0$ when $p=2$. To this end, let $h\in B_\delta(0)$. 


\textbf{Case: $1<p<2$.}  In this case, $\nabla u(x)\neq 0$ and $\psi(t)= |t|^{p-2}t $ is $C^1$ on $\R\setminus\{0\}$  so that  
\begin{align*}
&\lim_{b\to a}\frac{\psi(b)-\psi(a) -\psi'(a)(b-a)}{b-a}=0. 
\end{align*}
Set $\sigma=2-p>0$. Since $b\to a$ as $h\to 0$,  for $\nabla u(x)\cdot h\neq 0$ and  $0<\delta<1$ small, we get 
\begin{align*}
|R(a,b)|\leq |b-a|\leq C_0|h|^2= C_0|h|^{p+\sigma}. 
\end{align*}


\textbf{Case: $2< p<3$.} Since $1<p-1<2$, the estimate \eqref{eq:upper-elem-sing}, from Appendix \ref{sec:appendix-estim} infers that
$$  \big||b|^{p-2}-|a|^{p-2}\big|\leq 2^{4-p} |b-a|^{p-2}.$$  
Set $\sigma=p-2>0$. Recall that $\psi'(t)=(p-1)|t|^{p-2}$,  for all $|h|<\delta$, we find  
\begin{align*}
\begin{split}
|R(a,b)| &\leq|b-a|\int_0^1|\psi'(a+ t(b-a))-\psi'(a)|\d t\leq (p-1) 2^{4-p} |b-a|^{p-1}\\
&\leq C_1|h|^{2(p-1)}= C_1|h|^{p+\sigma}, \quad\text{$C_1= (p-1)2^{4-p} C^{p-1}_0$}. 
\end{split}
\end{align*}
\textbf{Case: $p\geq 3$.} Since $p-1\geq 2$,   the estimate \eqref{eq:upper-elem-degen} (see Appendix \ref{sec:appendix-estim}) implies 
\begin{align*}
||b|^{p-2}- |a|^{p-2}|\leq (p-2) |b-a|(|b|^{p-3}+ |a|^{p-3})\leq  C_3|b-a||h|^{p-3}. 
\end{align*}
with $C_3= 2(p-2)\sup_{z\in B_1(x)}|\nabla u(z)|^{p-3}.$ It follows that
\begin{align*}
\begin{split}
|R(a,b)| &\leq|b-a|\int_0^1|\psi'(a+ t(b-a))-\psi'(a)|\d t\\
&\leq C_3(p-1)|b-a|^2 |h|^{p-3}\leq C_4|h|^{p+1},\quad C_4=C^2_0C_3(p-1). 
\end{split}
\end{align*}
Altogether for $\sigma= |p-2|$ if $1<p<3$ and $\sigma=1$ if $p\geq3$ and for $0<\delta<1$ small, we have shown that 
\begin{align*}
|R(a,b)|\leq \begin{cases}  C|h|^{p+\sigma} &\quad 0<|p-2|<1,\\
0 &\quad p=2,\\
C|h|^{p+1} &\quad p\geq 3.
\end{cases}
\end{align*}
\noindent The case $1<p<2$ is understood in almost everywhere sense,  since $\nabla u(x)\neq0$ so that $|B_\delta(0)\cap \{\nabla u(x)\cdot z=0\}|=0$. Applying Proposition \ref{prop:asymp-nu} then in any case,  we get 

\begin{align}\label{eq:remainder-go-zero}
\lim_{\eps\to0} \int_{|h|<\delta} R(a,b)\,\nu_\eps(h)\,\d h\leq C\lim_{\eps\to0} \int_{|h|<\delta}|h|^{p+\sigma}\,\nu_\eps(h)\d h=0.
\end{align}
From the foregoing, combining \eqref{eq:xdifference-b-a} and \eqref{eq:remainder-go-zero}  we find that 
\begin{align*}
-\lim_{\eps\to0} &L_\eps u(x) =  \lim_{\eps\to0} 2\int_{|h|<\delta } \psi'(a)(b-a) \,\nu_\eps(h)\,\d h\\
&=	\lim_{\eps\to0} 2\int_{|h|<\delta } \hspace*{-2ex}\Big( \psi'(\nabla u(x)\cdot h)\int_0^1\nabla u(x+th)\cdot h-\nabla u(x)\cdot h\,\d t\Big)\,\nu_\eps(h)\,\d h\\
&=\lim_{\eps\to0} \int_{|h|<\delta } \big(\psi'(\nabla u(x)\cdot h) D^2u(x)h\cdot h\\
&+  \psi'(\nabla u(x)\cdot h) R_1(x,h)h\cdot h\big) \,\nu_\eps(h)\,\d h. 
\end{align*}
Here the remainder $R_1(x,h)$ is the matrix defined by 
\begin{align*}
R_1(x,h)h\cdot h= 2\int_0^1 t\int_0^1 \big[D^2u(x+sth)\cdot h-D^2 u(x)\cdot h\big]\cdot h\, \d s\,  \d t. 
\end{align*} 
It clearly occurs that $|\nabla u(x)\cdot h|^{p-2}h|\leq |\nabla u(x)|^{p-2}|h|^{p-1}$ with $\nabla u(x)\neq 0$ for $1<p<2$. Hence for $|h|<\delta$, \eqref{eq:xmatrix-diff} yields 
\begin{align*}
&\lim_{\eps\to0} 2\int_{|h|<\delta }  \psi'(\nabla u(x)\cdot h) R_1(x,h)h\cdot h \,\nu_\eps(h)\,\d h
\\&\leq 
\eta (p-1)|\nabla u(x)|^{p-2} \lim_{\eps\to0} \int_{|h|<\delta }
\hspace*{-2ex} |h|^p \,\nu_\eps(h)\,\d h\leq \eta (p-1)|\nabla u(x)|^{p-2} \xrightarrow{\eta\to0}0. 
\end{align*}
Since $\eta>0$ is arbitrarily chosen, we deduce that  
\begin{align}\label{eq:xgeneral-lim}
- \lim_{\eps\to0} L_\eps u(x) &=\lim_{\eps\to0} \int_{|h|<\delta }\psi'(\nabla u(x)\cdot h) D^2u(x)h\cdot h \,\nu_\eps(h)\,\d h
\end{align}
The case $d=1$ follows immediately by exploiting \eqref{eq:xgeneral-lim}. Now assume $d\geq2$. Since  $\nu_\eps$ is radial, using the polar coordinates in \eqref{eq:xgeneral-lim} and  Proposition \ref{prop:asymp-nu} yields
\begin{align*}
&-\lim_{\eps\to0} L_\eps u(x) =\lim_{\eps\to0}  (p-1)\int_{|h|<\delta }|\nabla u(x)\cdot h|^{p-2} D^2u(x)h\cdot h \,\nu_\eps(h)\,\d h\\
&=  \lim_{\eps\to0}  (p-1)\int_0^\delta \int_{\mathbb{S}^{d-1}}|\nabla u(x)\cdot w|^{p-2} D^2u(x)w\cdot w \d \sigma_{d-1}(w) \, r^{d+p-1}\nu_\eps(r)\,\d r\\
&= \frac{ (p-1)}{|\mathbb{S}^{d-1}|}\int_{\mathbb{S}^{d-1}}|\nabla u(x)\cdot w|^{p-2} D^2u(x)w\cdot w \, \d \sigma_{d-1}(w) 
\lim_{\eps\to0} \int_{|h|<\delta }|h|^p\nu_\eps(h)\,\d h\\
&= \frac{ (p-1)}{|\mathbb{S}^{d-1}|} \int_{\mathbb{S}^{d-1}}|\nabla u(x)\cdot w|^{p-2} D^2u(x)w\cdot w \, \d \sigma_{d-1}(w) .
\end{align*}
The desired result follows from  the spherical  representation of $\Delta_p$ (see Lemma \ref{lem:mean-value-plaplace}), viz., 
\begin{align*}
\lim_{\eps\to0} L_\eps u(x) &= -\frac{ (p-1)}{|\mathbb{S}^{d-1}|} \int_{\mathbb{S}^{d-1}}|\nabla u(x)\cdot w|^{p-2} D^2u(x)w\cdot w \, \d \sigma_{d-1}(w)\\&
=- K_{d,p}\Delta_p u(x) .
\end{align*}
\end{proof}
Using \eqref{eq:xfar-from-zero} and \eqref{eq:xdifference-b-a}, one obtains the following straightforward variants of  Theorem \ref{thm:asymp-plevy-to-plaplace}. 
\begin{theorem}
Let $\Omega\subset \R^d$ be open and $\delta>0$. If $u\in L^\infty(\Omega)\cap C^2(B_\tau(x))$, $x\in \Omega$,   $0<\tau<\dist(x,\partial\Omega)$, with $\nabla u(x)\neq 0$ when $1<p<2$, then  we have 
\begin{align*}
\lim_{\varepsilon\to 0}L_{\Omega,\varepsilon} u(x)
&= \lim_{\varepsilon\to 0} L_{\Omega,\delta, \varepsilon}u(x) 
=-K_{d,p} \Delta_p u(x)
\end{align*}
where $L_{\Omega,\varepsilon} $ is the regional operator and $L_{\Omega,\delta, \varepsilon}$  is the constrained operator on $\Omega$, respectively defined by
\begin{align*}
L_{\Omega,\varepsilon}u(x) 
&= 2\pv \int_{\Omega}\psi(u(x)-u(y))\nu_\varepsilon(x-y)\d y,\\
L_{\Omega,\delta, \varepsilon}u(x) 
&= 2\pv \int_{\Omega\cap B(x,\delta)}\psi(u(x)-u(y))\nu_\varepsilon(x-y)\d y.
\end{align*}

Note that if  $\delta=\delta_x:=\dist(x,\partial\Omega)$, the constrained operator merely becomes
\begin{align*}
L_{\Omega,\delta_x, \varepsilon}u(x) 
&= 2\pv \int_{B(x,\delta_x)}\psi(u(x)-u(y))\nu_\varepsilon(x-y)\d y.
\end{align*}
\end{theorem}
Note that by definition, $ \cN_{\varepsilon} u(x)= L_{\Omega, \eps} u(x)$ for all $x\in \R^d\setminus \overline{\Omega}$, hence  \eqref{eq:xfar-from-zero} readily implies the following. 
\begin{theorem} Let $\Omega\subset \R^d$ be open. If $u:\R^d\to \R$ is measurable and $u|_\Omega\in L^\infty(\Omega)$, then we have 
\begin{align*}
\lim_{\varepsilon\to 0}\cN_{\varepsilon} u(x)=0\qquad\text{for  $x\in\R^d\setminus\overline{\Omega}$}. 
\end{align*}  
\end{theorem}

Let us point out some particular cases of Theorem \ref{thm:asymp-plevy-to-plaplace} already appeared in the literature, viz., \cite[Theorem 2.8]{BS22}, \cite[Corollary 6.2]{dTL21} and the variant in\cite[Section 7]{IN10}. 

\begin{corollary}[{\hspace{-0.1ex} \cite[Section 7]{IN10}, \cite[Theorem 2.8]{BS22}}]  
\label{cor:asymp-pfrac-to-plaplace}
Under the assumptions of Theorem \ref{thm:asymp-plevy-to-plaplace} we have 
\begin{align*}
\lim_{s\to 1}2s(1-s) \pv\int_{\R^d}\frac{\psi(u(x)-u(y)) }{|x-y|^{d+sp}}\d y 
=  -\frac{|\mathbb{S}^{d-1}|}{p} K_{d,p}\Delta_p u(x).
\end{align*}
\end{corollary}

\begin{proof}
	Apply Theorem \ref{thm:asymp-plevy-to-plaplace} with $\eps =1-s$ and $\nu_\eps(h)= \frac{p\eps(1-\eps) }{|\mathbb{S}^{d-1}|}|h|^{-d-(1-\eps)p}$. 
\end{proof}

\begin{corollary}[{\hspace{-0.1ex}\cite[Corollary 6.3]{dTL21}}]
If $u\in C^2(B_\tau(x))$, $\tau>0$ then we have 
\begin{align*}
\lim_{\eps\to0} \frac{1}{\eps^{d+p}} \int_{B_\eps(x)}|u(x)-u(y)|^{p-2} (u(x)-u(y)) \d y =  -\frac{|\mathbb{S}^{d-1}|}{d+p} K_{d,p}\Delta_p u(x).
\end{align*}
\end{corollary}
\begin{proof}
Apply Theorem \ref{thm:asymp-plevy-to-plaplace} with $\nu_\eps(h)= \frac{(d+p)  \eps^{-d-p}}{|\mathbb{S}^{d-1}|}\mathds{1}_{B_\eps(0)}(h)$. 
\end{proof}
\subsection{A normalization constant for the fractional $p$-Laplacian}\label{sec:norming-cste-pfrac}
In light of Corollary \ref{cor:asymp-pfrac-to-plaplace} we define a suitable  normalizing constant $C_{d,p,s}$  for the fractional $p$-Laplacian $(-\Delta)^s_p$ such that $ C_{d,2,s}$ is the normalizing constant of the fractional Laplacian $(-\Delta)^s$  and that for $u\in L^\infty(\R^d)\cap C^2(B_1(x))$ we have $(-\Delta)^s_p u(x)\xrightarrow{s\to 1}-\Delta_p u(x) $, i.e., 
\begin{align}\label{eq:norma-frac-p-lap}
(-\Delta)^s_p u(x):= C_{d,p,s}\pv\int_{\R^d}\frac{\psi(u(x)-u(y)) }{|x-y|^{d+sp}}\d y \xrightarrow{s\to 1} -\Delta_p u(x).
\end{align}
In view of Corollary \ref{cor:asymp-pfrac-to-plaplace} we find that 
\begin{align*}
\lim_{s\to 1}2s(1-s)B_p\pv\int_{\R^d}\frac{\psi(u(x)-u(y)) }{|x-y|^{d+sp}}\d y =- \Delta_p u(x). 
\end{align*}
where, since $|\mathbb{S}^{d-1}|= \frac{2\pi^{\frac{d}{2}}}{\Gamma(d/2)}$ and $\Gamma (1/2)=\pi^{1/2},$ the constant $B_p$ is given by 
\begin{align*}
\frac{1}{B_p}:=\frac{|\mathbb{S}^{d-1}|}{p} K_{d,p}=\frac{2\pi^{\frac{d}{2}}}{p\Gamma\big(\frac{d}{2}\big)} \frac{\Gamma\big(\frac{d}{2}\big)\Gamma\big(\frac{p+1}{2}\big)}{\Gamma\big(\frac{d+p}{2}\big) \Gamma\big(\frac{1}{2}\big)}  =\frac{2\pi^{\frac{d-1}{2}} \Gamma\big(\frac{p+1}{2}\big)}{p\Gamma\big(\frac{d+p}{2}\big)}.
\end{align*}
On the other hand, by \cite[Proposition 2.21]{guy-thesis}  the normalizing constant of the fractional Laplacian $(-\Delta)^s$ is
\begin{align*}
\frac{1}{C_{d,2,s}}=\frac{\pi^{\frac{d-1}{2}}\Gamma(\frac{2s+1}{2})\Gamma(2(1-s))}{s(1-2s)\Gamma(\frac{d+2s}{2})}\cos(s\pi). 
\end{align*}
Alternatively, be aware that a common representation formula of $C_{d,2,s}$ is as follows
\begin{align*}
\frac{1}{C_{d,2,s}}=
\frac{\pi^{\frac{d}{2}}\Gamma(1-s)}{s2^{2s}\Gamma(\frac{d+2s}{2})}= \frac{\pi^{\frac{d}{2}}|\Gamma(-s)|}{2^{2s}\Gamma(\frac{d+2s}{2})}. 
\end{align*}
Our normalizing constant, mimics the first expression of $\frac{1}{C_{d,2,s}}$ and $\frac{1}{B_p}$. To wit, we define
\begin{align*}
\frac{1}{C_{d,p,s}}=\frac{\pi^{\frac{d-1}{2}}\Gamma(\frac{sp+1}{2})\Gamma(p(1-s))}{s(1-2s)
\Gamma(\frac{d+sp}{2})}\cos(s\pi).
\end{align*}
Namely, taking also into account the two expressions of $C_{d,2,s}$ yielding  a formula for $\cos(s\pi)$ we have 
\begin{align*}
C_{d,p,s}\hspace*{-0.5ex}=\hspace*{-0.5ex}\frac{sp(1-s)(1-2s)\Gamma(\frac{d+sp}{2})}{\pi^{\frac{d-1}{2}}
\Gamma(\frac{sp+1}{2})\Gamma(p(1-s)+1)\hspace*{-0.5ex}\cos(s\pi)}
\hspace*{-0.5ex}=\hspace*{-1ex} \frac{2^{2s}\,\Gamma(\frac{d+sp}{2})\Gamma(\frac{2s+1}{2})\Gamma(2(1-s))}{\pi^{\frac{d}{2}} |\Gamma(-s)|\,\Gamma(\frac{sp+1}{2})\Gamma(p(1-s))} . 
\end{align*}
Clearly the relation \eqref{eq:norma-frac-p-lap}  holds true since the asymptotic $s\to 1$ can be rewritten with help of 
\begin{align*}
\lim_{s\to 1} \frac{C_{d,p,s}}{2s(1-s)B_p}=1. 
\end{align*}
We immediately get the  following asymptotic behaviors  
\begin{align}\label{eq:asymptotic-norming}
\lim_{s\to 0} \frac{C_{d,p,s}}{s(1-s)}&=\frac{2}{|\mathbb{S}^{d-1}|\Gamma(p)}\qquad \text{and}\qquad \lim_{s\to 1} \frac{C_{d,p,s}}{s(1-s)}=\frac{2p}{|\mathbb{S}^{d-1}| K_{d,p}}. 
\end{align}
The above asymptotic perfectly aligns with the case $p=2$, as shown in \cite{Hitchhiker}. We emphasize that the other normalizing constants for the fractional $p$-Laplacian are proposed in \cite{dTGCV21,War16} and  the constant $K_{d,p}$ plays a crucial role in our asymptotic analysis. For the sake of completeness, it is natural to study the asymptotic near $s=0$. To this end, one possibility is to modify $C_{d,p,s}$ near $s=0$  replacing it with  $C_{d,2,\frac{sp}{2}}$. The constant $C_{d,2,\frac{sp}{2}} $ behaves nicely near $s=0$ and might diverge near $s=1$\footnote{Using the Euler's reflection formula $\Gamma(\eps-m)\Gamma(m+1-\eps)= (-1)^m\Gamma(\eps)\Gamma(1-\eps)$, $\eps\in(0,1)$ and $m\in\mathbb{N}$, one derives that 
$(1-s)|\Gamma(-\frac{sp}{2})|\to \frac{2}{ p\Gamma(\frac{p+2}{2})}$ if $\frac{p}{2} \in \mathbb{N}$;  otherwise, $(1-s)|\Gamma(-\frac{sp}{2})|\to 0$.} and we have 
\begin{align*}
\lim_{s\to0} \frac{C_{d,2,\frac{sp}{2}} }{s(1-s)}= \frac{p}{2|\mathbb{S}^{d-1}|}\,\, \text{and}\,\,
\lim_{s\to1} \frac{C_{d,2,\frac{sp}{2}} }{s(1-s)}=
\begin{cases}
\infty& \frac{p}{2}\notin \mathbb{N}\\
\frac{2^{p-1} p\Gamma(\frac{p+2}{2})  \Gamma(\frac{d+p}{2})}{\pi^{d/2}}  &\frac{p}{2} \in \mathbb{N}. 
\end{cases}	
\end{align*}
It is straightforward to verify that,  granted with the  constant $\widetilde{C}_{d,p,s}$  
\begin{align*}
	\widetilde{C}_{d,p,s}=\begin{cases}
		C_{d,p,s} & \text{if $sp\geq 1$}\\
		C_{d,2,\frac{sp}{2}} &\text{if $sp<1$}, 
	\end{cases}
\end{align*}
hence a fortiori  $C_{d,p,s}$ and $C_{d,2,\frac{sp}{2}}$,  the following properties are guaranteed.  
\begin{enumerate}[$\bullet$]
\item For $p=2$, we find that $\widetilde{C}_{d,p,s}=C_{d,p,s}= C_{d,2,\frac{sp}{2}}= C_{d,2,s}$  is the normalizing constant of the fractional Laplacian $(-\Delta)^s$. 
	\item For any $u\in L^\infty(\R^d)\cap C^2(B_1(x))$
	we have $(-\Delta)^s_p u(x)\xrightarrow{s\to1}-\Delta_p u(x)$,  in fact 
\begin{align*}
&\lim_{s\to 1} \frac{\widetilde{C}_{d,p,s}}{2}\int_{\R^d}\frac{\psi(u(x)-u(y))}{|x-y|^{d+sp}}\d y\d x
	= -\Delta_p u(x).
	\end{align*}
\item For any $u\in L^\infty(\R^d)\cap C^\eps(B_1(x))$, $\eps>0$ we have $(-\Delta)^s_p u(x)\xrightarrow{s\to0}|u(x)|^{p-2} u(x)$,  
	\begin{align*}
&\lim_{s\to 0} \frac{\widetilde{C}_{d,p,s}}{2}\int_{\R^d}\frac{\psi(u(x)-u(y))}{|x-y|^{d+sp}}\d y\d x
= |u(x)|^{p-2} u(x).
\end{align*}
	\item For any $u\in L^p(\R^d)$ we have
	\begin{align*}
&\lim_{s\to 1} \frac{\widetilde{C}_{d,p,s}}{2}\iil_{\R^d\R^d}\frac{|u(x)-u(y)|^p}{|x-y|^{d+sp}}\d y\d x
= \int_{\R^d}|\nabla u(x)|^p\d x.
	\end{align*}
	\item For any $u\in \bigcup_{s\in (0,1)} W^{s,p}_0(\R^d)$, see \cite{MS02}, we have
\begin{align*}
	&\lim_{s\to 0} \frac{\widetilde{C}_{d,p,s}}{2}\iil_{\R^d\R^d}\frac{|u(x)-u(y)|^p}{|x-y|^{d+sp}}\d y\d x
	= \int_{\R^d}|u(x)|^p\d x.
\end{align*}
\item Moreover, we have the following asymptotic behaviors  
\end{enumerate} 
\begin{align}\label{eq:asymptotic-norming-bis}
\lim_{s\to 0} \frac{\widetilde{C}_{d,p,s}}{s(1-s)}&=\frac{p}{2|\mathbb{S}^{d-1}|}
\qquad \text{and}\qquad 
\lim_{s\to 1} \frac{\widetilde{C}_{d,p,s}}{s(1-s)}=\frac{2p}{|\mathbb{S}^{d-1}| K_{d,p}}. 
\end{align}
The asymptotic $s\to 1$,  highlighting the factor $K_{d,p}$  is already anticipated in \cite[Eq: 2.38]{guy-thesis} in the case $p=2$.  Despite the amusing fact of this asymptotic, it is important for the reader must keep in mind that  $C_{d,p,s}$ is purely artificial and that only the case $p=2$ naturally appears as the unique normalizing constant for which $\widehat{(-\Delta)^s u} (\xi)= |\xi|^{2s}\widehat{u}(\xi)$ for all $u\in C_c^\infty(\R^d)$.

\subsection{Convergence of forms}\label{sec:convergence-form}
We  are interested in the asymptotic of the energy forms
\begin{align*}
\mathcal{E}^{\eps}_{\Omega}(u,v) &=  \iil_{\Omega \Omega} |u(y)-u(x)|^{p-2}(u(y)-u(x)) (v(y)-v(x)) \nu_\eps(x-y)\d y \, \d x, \\
\mathcal{E}^{\eps}(u,v) &=\hspace*{-3ex} \iil_{(\Omega^c\times \Omega^c)^c} |u(y)-u(x)|^{p-2}(u(y)-u(x)) (v(y)-v(x) )\nu_\eps(x-y) \d y \, \d x, 
\\ \mathcal{E}^{\eps}_+(u,v) &= \iil_{\Omega \R^d} |u(y)-u(x)|^{p-2}(u(y)-u(x))(v(x)-v(y))\nu_\eps(x-y)\d y \, \d x,
\\ \mathcal{E}^{\eps}_{cr}(u,v) &= \iil_{\Omega \Omega^c} |u(y)-u(x)|^{p-2}(u(y)-u(x))(v(x)-v(y))\nu_\eps(x-y)\d y \, \d x,
\\\mathcal{E}^{0}(u,v)&= \int_\Omega |\nabla u(x)|^{p-2}\nabla u(x)\cdot \nabla v(x)\d x. 
\end{align*}
Note  in passing that $\mathcal{E}^{\eps}(u,v)= \mathcal{E}^{\eps}_+(u,v) =\mathcal{E}^{\eps}_{\Omega}(u,v)$ when $\Omega=\R^d$. Moreover, by \cite[Theorem 5.23]{guy-thesis}, \cite[Theorem 1.3]{Fog23}  see also the variant in \cite{Brezis-const-function, BBM01} we have 
\begin{align}
&\lim_{\eps\to0}\mathcal{E}^{\eps}_{\R^d}(u,u) =K_{d,p}\,\int_{\R^d}|\nabla u(x)|^p\d x.
\label{eq:form-conv-full} 
\end{align}
Note that since $\frac{C_{d,p,s}}{2s(1-s)}\xrightarrow{s\to1}
\frac{p}{|\mathbb{S}^{d-1}|} $,  see the asymptotic in \eqref{eq:asymptotic-norming}, for the standard fractional  case 
$\nu_\eps(h)=\frac{C_{d,p,s}}{2}|h|^{-d-sp}$, $s=1-\eps$  we have 
\begin{align*}
&\lim_{s\to 1} \frac{C_{d,p,s}}{2}\iil_{\R^d\R^d}\frac{|u(x)-u(y)|^p}{|x-y|^{d+sp}}\d y\d x
= \int_{\R^d}|\nabla u(x)|^p\d x.
\end{align*}
In short, up to a multiplicative factor, we have $|u|_{W^{s,p}(\R^d)}\xrightarrow{s\to1}\|\nabla u\|_{L^p(\R^d)}$. Next, we need the following result involving the collapse across the boundary.  
\begin{lemma}\label{lem:collap-bdary}
Assume $\Omega\subset \R^d$ satisfies $|\partial\Omega|=0$. For any $u,v\in W^{1,p}(\R^d)$ we have 
\begin{align*}
&\lim_{\eps\to0}\iil_{\Omega \Omega^c} \psi(u(x)-u(y)) (v(x)-v(y))\nu_\eps(x-y)\d y\,\d x= 0.
\end{align*}
\end{lemma}

\begin{proof}
Consider $U_\delta=\{x\in \R^d: \dist(x,\Omega)>\delta \}$, $\delta>0$, so that  $U_\delta\subset \Omega^c$. Since $\Omega$ and $U_\delta$ are open, by \cite[Theorem 3.3]{Fog23} or \cite[Theorem 5.16]{guy-thesis}  we find that 
\begin{align*}
&K_{d,p}\int_{\Omega}|\nabla u(x)|^p\d x\leq \liminf_{\eps\to 0} \cE_\Omega^{\eps}(u,u),\\
& K_{d,p}\int_{U_\delta}|\nabla u(x)|^p\d x\leq \liminf_{\eps\to 0}  \cE^{\eps}_{U_\delta}(u,u)\leq \liminf_{\eps\to 0}  \cE^{\eps}_{\Omega^c}(u,u). 
\end{align*}
By the convergence dominated theorem we obtain 
\begin{align*}
K_{d,p}\int_{\Omega^c}|\nabla u(x)|^p\d x
\leq \liminf_{\eps\to 0}  \cE^{\eps}_{\Omega^c}(u,u).
\end{align*}
Accordingly, since $|\partial\Omega|=0$, together with  \eqref{eq:form-conv-full},  we get the claim as follows 
\begin{align*}
&\limsup_{\eps\to 0}\iil_{\Omega\Omega^c} |u(x)-u(y)|^p\nu_\eps(x-y)\d y\,\d x\\
&=\frac12\limsup_{\eps\to 0}\big(\cE^\eps_{\R^d}(u,u) -\cE^\eps_{\Omega}(u,u)-\cE^\eps_{\Omega^c}(u,u)\Big)\\
&\leq \limsup_{\eps\to 0}\cE^\eps_{\R^d}(u,u) -\liminf_{\eps\to 0}\cE^\eps_{\Omega}(u,u)-\liminf_{\eps\to 0}\cE^\eps_{\Omega^c}(u,u)\\
&\leq K_{d,p} \Big(\|\nabla u\|^p_{L^p(\R^d)}-\|\nabla u\|^p_{L^p(\Omega)}-\|\nabla u\|^p_{L^p(\Omega^c)}\Big)=0. 
\end{align*}
The case $u\neq v$ follows this by applying the H\"{o}lder inequality.
\end{proof}

Next  we combine the ideas of \cite[Theorem 1.5]{Fog23} and \cite[Lemma 2.8]{BS20}. 
\begin{theorem}\label{thm:BBM-dual-limit}
Assume that $\Omega\subset \R^d$ is open satisfying $(i)$ or $(ii)$,  
\begin{enumerate}[$(i)$]	
\item $\Omega$ is an $W^{1,p}$-extension domain, 
\item $\partial \Omega=\partial\overline{\Omega}$ and  $\R^d\setminus \overline{\Omega}$ is an $W^{1,p}$-extension domain.
\end{enumerate}
Then for $u,v\in W^{1,p}(\R^d)$, we have 
\begin{align}
&\lim_{\eps\to0}\mathcal{E}^{\eps}_{\Omega}(u,v) =K_{d,p}\,\mathcal{E}^{0}(u,v),
\label{eq:form-conv-min}\\
&\lim_{\eps\to0}\mathcal{E}^{\eps}(u,v) = K_{d,p}\,\mathcal{E}^{0}(u,v), \label{eq:form-conv-max}\\
&\lim_{\eps\to0}\mathcal{E}^{\eps}_+(u,v) = K_{d,p}\,\mathcal{E}^{0}(u,v). \label{eq:form-conv-plus} 
\end{align}
\end{theorem}

\begin{proof}
\noindent 
The elementary inequality $|b|^p- |a|^p- p|a|^{p-2}a(b-a)\geq 0$ (see Corollary \ref{cor:elm-est-taylor}) yields, for $t>0$ and $\sigma\in\R$, 
\begin{align*}
&\cE_{\Omega}^\eps(u+t\sigma v, u+t\sigma v)- \cE_{\Omega}^\eps(u, u)- p t\cE_{\Omega}^\eps(u,\sigma v)
\geq0.
\end{align*}
If  $\Omega$ is a $W^{1,p}$-extension domain then $\cE^\eps_\Omega(w, w) \xrightarrow{\eps\to 0} K_{d,p}\cE^0(w, w)$ for $w\in W^{1,p}(\Omega)$; see \cite[Theorem 5.23]{guy-thesis}, \cite[Theorem 1.3]{Fog23} or  the variant in \cite{BBM01}. Accordingly, passing to the $\lim\hspace{-0.5ex}\sup$  yields
\begin{align*}
&K_{d,p}\frac{\|\nabla(u+t\sigma v)\|^p_{L^p(\Omega)}-\|\nabla u\|^p_{L^p(\Omega)}}{t}
\geq p  \limsup_{\eps\to 0} \cE_{\Omega}^\eps(u,\sigma v).
\end{align*}
Letting $t\to 0$, with $\sigma=\pm 1$ yields $\sigma pK_{d,p}\cE^0(u, v)\geq p\limsup_{\eps\to 0} \sigma\cE_{\Omega}^\eps(u, v)$ and hence 
\begin{align*}
&\lim_{\eps\to0} \cE_{\Omega}^\eps(u, v)
= K_{d,p}\cE^0(u, v). 
\end{align*}
This remains true with $\Omega$ replaced by $\R^d$ or by $\R^d\setminus\overline{\Omega}$ when $\R^d\setminus\overline{\Omega}$ is a $W^{1,p}$-extension domain. Note that in either  case,  we have $|\partial\Omega|=|\partial\overline{\Omega}|=0$ since the boundary of an extension domain is a null set.  Thus, all claims follow by combing Lemma \ref{lem:collap-bdary} and the situation where $\R^d$, $\Omega$ and/or $\R^d\setminus\overline{\Omega}$ is a $W^{1,p}$-extension domain. Indeed, $(i)$ follows since $ (\Omega^c\times \Omega^c)^c =\Omega\times\Omega\cup \Omega\times \Omega^c\cup \Omega^c\times \Omega $  and $\Omega\times \R^d= \Omega\times \Omega\cup \Omega\times \Omega^c$.  The case $(ii)$  follows analogously since   $\Omega\times \Omega=(\R^d\times \R^d)\setminus[\Omega^c\times \Omega\cup\Omega\times \Omega^c\cup\Omega^c\times \Omega^c]$ and   $(\Omega^c\times \Omega^c)^c =(\R^d\times \R^d)\setminus(\Omega^c\times \Omega^c)$.
\end{proof}
As illustrated in the next result,  $\partial\Omega$  need not be regular if $u$ or $v$ vanishes on  $\partial\Omega$.
\begin{theorem}\label{thm:BBM-dual-limit-zero-bdry}
Let $u, v\in W^{1,p}(\R^d)$ and $\Omega\subset \R^d$ be any open set. If $u\in W^{1,p}_0(\Omega)$ or $v\in W^{1,p}_0(\Omega)$ then the convergences \eqref{eq:form-conv-min}, \eqref{eq:form-conv-max} and \eqref{eq:form-conv-plus} hold. 
\end{theorem}

\begin{proof}
By density it is sufficient to assume $u\in C_c^\infty(\Omega)$ or $v\in C_c^\infty(\Omega)$. In any case  $\nabla u(x)\cdot \nabla v(x)=0$ a.e. 
on $\Omega^c$ and $(u(x)-u(y)) (v(x)-v(y))=0,$ a.e. on 
$\Omega^c\times \Omega^c$. The result follows by combining  Lemma \ref{lem:collap-bdary} and  Theorem \ref{thm:BBM-dual-limit} for $\Omega=\R^d$, i.e., $\cE^\eps_{\R^d}(u, v) \xrightarrow{\eps\to 0}K_{d,p}\cE^0(u, v).$
\end{proof}

\begin{theorem}[\hspace*{-0.6ex}{\cite[Theorem 3.37]{guy-thesis}}]\label{thm:gamma-liminf}
Let  $\Omega\subset \R^d$ is open. Let  $ (u_\eps)_\eps\subset L^p(\Omega)$ and $u\in L^p(\Omega)$ such that  
\begin{align*}
\sup_{\eps>0} \Big(\|u_\eps\|^p_{L^p(\Omega)}+ \iil_{ \Omega \Omega} |u_\eps(x)-u_\eps(y)|^p\nu_\eps(x-y)\d y\,\d x\Big)<\infty.
\end{align*}
\noindent If $\|u_\eps-u\|_{L^p(\Omega)}\xrightarrow{\eps\to0}0$ then  $u\in W^{1,p}(\Omega)$ and we have 
\begin{align*}
K_{d,p}\|\nabla u\|^p_{L^p(\Omega)}\leq \liminf_{\eps\to 0}\iil_{ \Omega \Omega} |u_\eps(x)-u_\eps(y)|^p\nu_\eps(x-y)\d y\,\d x. 
\end{align*}
\end{theorem}

\noindent  A refinement of \cite[Theorem 5.35 \& 5.40]{guy-thesis} and also \cite{Pon04} or  \cite[Section 4]{AAS67} yields the following result. 
\begin{theorem}[Asymptotic compactness]
\label{thm:asymp-compactness}
Assume $\Omega\subset \R^d$ is open. Let  $(u_\eps)_\eps\subset L^p(\Omega)$ be such that 
\begin{align*}
\sup_{\eps>0} \Big(\|u_\eps\|^p_{L^p(\Omega)}+ \iil_{ \Omega \Omega} |u_\eps(x)-u_\eps(y)|^p\nu_\eps(x-y)\d y\,\d x\Big)<\infty.
\end{align*}
\noindent There exist $u\in W^{1,p}(\Omega)$ and a subsequence $(\eps_n)_n$ with $\eps_n\to0^+$ as $n\to \infty$ such that $(u_{\eps_n})_n$ converges  to $u$ in $L^p_{\loc}(\Omega)$.  
Moreover, we have 
\begin{align*}
K_{d,p}\|\nabla u\|^p_{L^p(\Omega)}\leq \liminf_{\eps\to 0}\iil_{ \Omega \Omega} |u_\eps(x)-u_\eps(y)|^p\nu_\eps(x-y)\d y\,\d x. 
\end{align*}
 
\noindent 	In addition, we have the following strong convergences.  
\begin{enumerate}[$(i)$]
\item If $\Omega$ bounded and Lipschitz  then we have $\|u_{\eps_n}-u\|_{L^p(\Omega)}\xrightarrow{n\to\infty}0$, 
\item If $\Omega=\R^d$ then we have $\|u_{\eps_n}-u\|_{L^p(\Omega')}\xrightarrow{n\to\infty}0$ whenever $|\Omega'|<\infty$.  
\end{enumerate}
\end{theorem}

\section{Robust Poincar\'{e} and Poincar\'{e}-Friedrichs inequalities}\label{sec:robust-poincare}

In this section, we establish \emph{robust Poincar\'{e}  type inequalities}. The robustness should be understood in the sense that within such inequalities, one can recover the corresponding  classical local Poincar\'{e} inequalities.  In this section $(\nu_\eps)_\eps$ is as in Section \ref{sec:from-nonlocal-local}, i.e.,  satisfies \eqref{eq:plevy-approx}. 

\subsection{Robust Poincar\'{e} inequality}
Note that  a set $G$ is called homogeneous, if $\lambda u\in G$ for $u\in G$ and $\lambda \in\R$. 
\begin{theorem}[Robust Poincar\'{e} inequality]\label{thm:robust-poincare}
Let  $\Omega\subset \R^d$ be bounded Lipschitz and connected.  Let $G\subset L^p(\Omega)$ be a nonempty and homogeneous satisfying: 
\begin{align}\label{eq:homogenenous-set}
\hspace{-1ex} \text{$G$ is closed in $L^p(\Omega)$, $1\not\in G$\, and  $\lambda u\in G$ for $\lambda\in \R, u\in G$.}
\end{align}
There exist $\eps_0=\eps_0 (d, p, \Omega, G)>0$ and  $B=B(d, p, \Omega, G)>0$  such  that 
\begin{align}\label{eq:robust-poincare}
\|u\|^p_{L^p(\Omega)}\leq B\cE_{\Omega}^\eps(u,u) \quad \text{for all  $\eps\in (0,\eps_0)$ and  $u\in G$.}
\end{align}
\end{theorem}

 For $0<\gamma\leq |\Omega|$ and $E\subset\Omega$  measurable such that $|E|>0$, obvious examples of sets $G$ satisfying \eqref{eq:homogenenous-set} include:

\begin{itemize}
\item $G_1= \{u\in L^p(\Omega)\,:\, |\{u=0\}|\geq \gamma\}$, 
\item  $G_2= \{u\in L^p(\Omega)\,:\, \mbox{$\fint_E$} u=0\}$, 
\item  $G_3= \{u\in L^p(\Omega)\,:\,  u=0\, \mbox{ a.e. on }\,  E\}$.
\end{itemize}

\begin{proof} 
Assume no such $\eps_0$ and $B$ exist. For each $n\geq1$ taking $\eps_0= \frac{1}{2^n}$ and $B=2^n$ there exist $\eps_n\in (0,\frac{1}{2^n})$ and $u_n\in G$ for which   \eqref{eq:robust-poincare} fails, i.e., $\|u_n\|^p_{L^p(\Omega)}> 2^n\cE_{\Omega}^{\eps_n}(u_n,u_n)$. 
By the homogeneity condition \eqref{eq:homogenenous-set}, we can assume without loss of generality that $u_n\in G$,  $\|u_n\|_{L^p(\Omega)}=1$ so that $\cE_{\Omega}^{\eps_n}(u_n,u_n)\leq \frac{1}{2^n}$.
According to Theorem \ref{thm:asymp-compactness} there is  $u\in W^{1,p}(\Omega)$ and a  subsequence still denoted $(u_n)_n$  converging  to $u$ in $L^p(\Omega)$.  Moreover, we have 
\begin{align*}
K_{d,p}\|\nabla u\|^p_{L^p(\Omega)}\leq 
\liminf_{n\to\infty}\cE_{\Omega}^{\eps_n}(u_n,u_n)=0.
\end{align*}
This implies that $\nabla u=0$ almost everywhere  on $\Omega$, which is a connected set. Necessarily, $u=c$ is a constant function.  We find that $\|u\|_{L^p(\Omega)}=1$, hence $u=c\neq0$ and $u\in G$ since $G$ is closed in $L^p(\Omega)$ and  $\|u_n-u\|_{L^p(\Omega)}\xrightarrow{n\to \infty}0$. By homogeneity of G, we have $c^{-1}u=1\in G$. But by assumption,  we know that  $1\not \in G.$ We have reached a contradiction.
\end{proof}

\vspace{1mm}

Here is a direct consequence of Theorem \ref{thm:robust-poincare}; see also \cite[Theorem 1.1]{Pon04}. 
\begin{corollary}\label{cor:robust-poincare}
There exist $\eps_0=\eps_0 (d, p, \Omega)>0$ and  $B=B(d, p, \Omega)>0$  such  that, for all $\eps\in (0,\eps_0)$ and $u\in L^p(\Omega)$
\begin{align*}
\|u-\mbox{$\fint_\Omega$}\|^p_{L^p(\Omega)}&\leq B
\iil_{ \Omega\Omega}|u(x)-u(y)|^p\nu_\eps(x-y)\d y\, \d x.
\end{align*}
\end{corollary}
\begin{proof}
It suffices to take $G=\{u\in L^p(\Omega)\;:\, \int_\Omega u=0\}$ in Theorem \ref{thm:robust-poincare}.
\end{proof}

\noindent Noting that the constant  $B=B(d,p, \Omega, G)$ is  independent of $\eps>0$, a noteworthy consequence of Theorem \ref{thm:robust-poincare} is obtained letting $\eps\to 0$; using Theorem \ref{thm:BBM-dual-limit}, \cite[Theorem 1.3]{Fog23} or \cite[Theorem 5.23]{guy-thesis}, we recover the classical Poincar\'{e} type inequality.

\begin{corollary}\label{cor:clasical-poincare}
Under the conditions and notations of Theorem \ref{thm:robust-poincare} we have 
\begin{align}\label{eq:poincare-local-bis} 
\|u\|^p_{L^p(\Omega)}\leq B K_{d,p}\int_{\Omega}|\nabla u(x)|^p\d x\quad \text{for every $u\in G$.}
\end{align}
\end{corollary}

\begin{corollary}\label{cor:robust-pcare-fract} Assume $\Omega\subset \R^d$ is open, bounded, Lipschitz and connected. There exists $C=C(p,d,\Omega)>0$ such that, 
for all $s\in(0 ,1)$, $u\in L^p(\Omega)$, 
\begin{align}\label{eq:robust-pcare-fract}
\big\|u-\hbox{ $\fint_{\Omega}$}u\big\|^p_{L^p(\Omega)}\leq C(1-s)\iint\limits_{\Omega\Omega} \frac{|u(x)-u(y)|^p}{|x-y|^{d+sp}}\, dy\,dx. 
\end{align}
\end{corollary}

When $\Omega=Q$ is a cube, the robust inequality \eqref{eq:robust-pcare-fract} is also proved in  \cite{BBM02} and improved in \cite{MS02,HMPV22}. The approaches therein use techniques from harmonic analysis. 
\begin{proof}
Take, $\nu_\eps(h)=\frac{p\eps(1-\eps)}{|\mathbb{S}^{d-1}|}|h|^{-d-(1-\eps)p}$ where we put $\eps=1-s$. By Theorem \ref{thm:robust-poincare}, there exist $s_0\in (0,1)$ and $C=B>0$ such that the inequality \eqref{eq:robust-pcare-fract} holds for all $s\in (s_0, 1)$ and $u\in L^p(\Omega)$.  Now, let  $R-1=\operatorname{diam}(\Omega)$ be the diameter of $\Omega$ so that $R\geq1$. For $s\in (0, s_0)$ we have $1-s_0<1-s$ and   $|x-y|^{-d-sp} \geq R^{-d-s_0p}$ for
all $x,y\in\Omega$.  This, together  with Jensen's inequality  yield
\begin{align*}
(1-s)\iil_{\Omega \Omega}\frac{ |u(x)-u(y)|^p}{|x-y|^{d+sp}} \d y\, \d x  
&\geq (1-s_0) R^{-d-s_0p}   \iint\limits_{\Omega \Omega} |u(x)-u(y)|^p \d y\, d x \\
&\geq  (1-s_0)R^{-d-s_0p} |\Omega| \int_{\Omega } \big |u(x)-\hbox{$\fint_{\Omega}$}u\big|^p  \, dx.  
\end{align*}
Up to a relabeling , it suffices to take $C=\max(B, \frac{R^{d+s_0p}}{|\Omega|(1-s_0)}).$
\end{proof}

The following variant of Theorem \ref{thm:robust-poincare}, encapsulates a sort of double robustness (bi-robustness) in parameters for the fractional type Poincar\'{e} inequality    
\begin{theorem}[Double robustness for fractional Poincar\'{e} inequality]
\label{thm:double-robust-poincare}
Under the conditions  of Theorem \ref{thm:robust-poincare}, there exist $B=B(d,p,\Omega)>0$, $r_0=r_0(d,p,\Omega)>0$ and $s_0=s_0(d,p,\Omega)>0$ such that, for every $r\in (0, r_0)$, $s\in (s_0,1)$ and $u\in G$ we have 
\begin{align}\label{eq:double-robust-poincare}
\|u-\hbox{$\fint_{\Omega}$}u\|^p_{L^p(\Omega)}	
\leq B\frac{(1-s)}{r^{p(1-s)}}\iil_{\Omega \Omega} \frac{|u(x)-u(y)|^p}{|x-y|^{d+sp}} \mathds{1}_{B_r(0)}(x-y)\d y\,\d x.
\end{align}
\end{theorem}

\begin{remark}
The double robustness (bi-robustness) in \eqref{eq:double-robust-poincare} is well understood, since letting $s\to 1$ and/or $r\to0$ in \eqref{eq:double-robust-poincare} one recovers the local Poincar\'{e} inequality \eqref{eq:poincare-local-bis}. Indeed, consider 
\begin{align}\label{eq:double-plevy}
\rho_{r,s} (h)=p(1-s)r^{-p(1-s)}|\mathbb{S}^{d-1}|^{-1} |h|^{-d-sp}\mathds{1}_{B_r(0)}(h).
\end{align}
Then $\nu_\eps(h)=\rho_{\eps,s}(h)$ for fixed $s\in (0,1)$ or $\nu_\eps(h)=\rho_{r,1-\eps}(h)$  for fixed $r>0$ satisfies the condition \eqref{eq:plevy-approx}. More importantly,  $\nu_\eps(h)=\rho_{\eps,1-\eps}(h)$ also satisfies the condition \eqref{eq:plevy-approx}. 
\end{remark}

\smallskip 
\begin{proof}
The proof is analogous to that of Theorem \ref{thm:robust-poincare} with the slight
difference that $\nu_{\eps_n}(h)$ is replaced by $\rho_{r_n, 1-\eps_n}$, $r_n ,\eps_n\in (0,\frac{1}{2^n})$, with $\rho_{r,s}$ given in  \eqref{eq:double-plevy}. 
Indeed, just as for $ \nu_{\eps_n}$, it is easy to check that, $\rho_{r_n, 1-\eps_n}$ is Dirac p-L\'{e}vy approximation sequence in the sense of  \eqref{eq:plevy-approx}.
\end{proof}

\subsection{Robust Poincar\'{e}-Friedrichs inequality}
The analogous robust Poincar\'{e}-Friedrichs inequality, is delicate and requires  a different slightly formulation. Here we identify $C_c^\infty(\Omega)$ as a natural subspace of $C_c^\infty(\R^d)$. Next, we deal with the situation where $\Omega$ is  bounded in  only one direction. 
\begin{theorem}[Robust Poincar\'{e}-Friedrichs inequality]\label{thm:robust-pcare-one-direct}
Assume $\Omega\subset \R^d$ is bounded in one direction, say $\Omega\subset H_R$ with $H_R=\{ x\in \R^d\,:\, |x\cdot e|\leq R \}$ with $R>0$ and $|e|=1$.
There exist $\eps_0=\eps_0(d, p, \Omega)$ and $B=B(d, p, \Omega)$  such  that,  for all $\eps\in (0,\eps_0)$ and $u\in C_c^\infty(\Omega)$, 
\begin{align}\label{eq:robust-pcare-fried-one-direct}
\|u\|^p_{L^p(\Omega)}\leq B\iil_{ \R^d\R^d}|u(x) -u(y)|^p\nu_\eps(x-y)\d y\d x. 
\end{align}
\end{theorem}

\begin{proof} 
\textbf{First proof}. 
Assume no such $\eps_0$ and $B$ exist. For $\eps_0= \frac{1}{2^n}$ and $B=2^n$, $n\geq1$, there exist $\eps_n\in (0,\frac{1}{2^n})$ and $u_n\in C_c^\infty(\Omega)$ for which   \eqref{eq:robust-poincare} fails, i.e., $\|u_n\|^p_{L^p(\Omega)}> 2^n\cE_{\R^d}^{\eps_n}(u_n,u_n)$. 
We can assume without loss of generality that $\|u_n\|_{L^p(\Omega)}=1$ so that $\cE_{\R^d}^{\eps_n}(u_n,u_n)\leq \frac{1}{2^n}$. 
Since $u_n\in C_c^\infty(\Omega)$, we find $x_n\in \Omega$ and $r_n>0$ such that $\supp u_n\subset \overline{B_{2r_n}(x_n)} \subset \Omega$. Next, we have $z_n= x_n+2r_ne\in \overline{B_{2r_n}(x_n)} \subset \Omega \subset H_R$. We find that 
\begin{align*}
R\geq z_n\cdot e= 2r_n+  x_n\cdot e \geq  2r_n-R\quad \text{hence,}\quad r_n \leq R.
\end{align*}

That is, we have $\supp u_n\subset \overline{B_{R}(x_n)}$. Set $\tau_{x_n}u_n(x):=u_n(x+x_n)$ then $\tau_{x_n} u_n\in C_c^\infty(B_{2R} (0))$ and $\tau_{x_n} u_n=0$ on $B_{2R}(0)\setminus B_{R}(0)$. Since $u_n\in C_c^\infty(\Omega)$ we have  
\begin{align*}
\|\tau_{x_n} u_n\|_{L^p(B_{2R}(0))} &= 	\|u_n\|_{L^p(B_{2r_n}(x_n))} =\|u_n\|_{L^p(\Omega)}=1,\\
\cE_{B_{2R}(0)}^{\eps_n}(\tau_{x_n} u_n,\tau_{x_n} u_n)
&\leq\cE_{\R^d}^{\eps_n}(\tau_{x_n} u_n,\tau_{x_n} u_n)= \cE_{\R^d}^{\eps_n}(u_n,u_n)\leq \frac{1}{2^n}.
\end{align*} 
According to Theorem \ref{thm:asymp-compactness} there is  $u\in W^{1,p}(B_{2R}(0))$ and a subsequence still denoted $(\tau_{x_n} u_n)_n$  converging  to $u$ in $L^p(B_{2R}(0))$.  Moreover, we have 
\begin{align*}
K_{d,p}\|\nabla u\|^p_{L^p(B_{2R}(0))}\leq 
\liminf_{n\to\infty}\cE_{B_{2R}(0)}^{\eps_n}(u_n,u_n)=0.
\end{align*}
This implies that $\nabla u=0$ almost everywhere  on $B_{2R}(0)$, which is a connected set. Necessarily, $u=c$ is a constant function on $B_{R}(0)$. However, since $\tau_{x_n} u_n=0$ on $B_{2R}(0)\setminus B_{R}(0)$ we deduce, via the convergence  $\|\tau_{x_n} u_n-u\|_{L^p(B_{2R})}\xrightarrow{n\to \infty}0$, that $u=c=0$. We also have $\|u\|_{L^p(B_{2R}(0))}=1$ since $\|\tau_{x_n} u_n\|_{L^p(B_{2R}(0))}=1$. We have reached a
contradiction.

\textbf{Second proof}. \textbf{Claim.} First, we prove the following claim.
\begin{align}\label{eq:trick-one-direct}
\text{For each $u\in C_c^\infty(\Omega)$ there exists $x_0\in \Omega$ s.t. $\supp\, u\subset B_R(x_0)$}.
\end{align}
Indeed, if $u\in C_c^\infty(\Omega)$, we can find $x_0= x_0(u)\in \Omega$ and $r= r(u)>0$ so that  $\supp u\subset B_R(x_0)  $ and $ \overline{B_{2r}(x_0)} \subset \Omega$. Clearly, we have $z:= x_0+2re\in \overline{B_{2r}(x_0)} \subset \Omega \subset H_R$. We find that
\begin{align*}
R\geq z\cdot e= 2r+  x_0\cdot e \geq  2r-R\quad \text{hence,}\quad r \leq R.
\end{align*}
It follows that $\supp u\subset B_{R}(x_0)$ and hence  $ \tau_{x_0}u \in C_c^\infty( B_{R}(0))$ where we consider the shift $\tau_{x_0}u(x):=u(x+x_0)$. Next, the set 
$G=\big\{u\in L^p(B_{2R}(0)):\text{$u=0$ a.e. on $B_{2R}(0) \setminus B_{R}(0)$} \big\}$ is a closed subset of $L^p(B_{2R}(0))$ such that $1\not\in G$ and $G$ is homogeneous, i.e., $\lambda u\in G$ whenever $u\in G$ and $\lambda\in \R$. 
Accordingly, by Theorem \ref{thm:robust-poincare}  we find $\eps_0=\eps_0(d,p,R)$ and $B=B(d,p,R)$ so that 
\begin{align*}
\|v\|^p_{L^p(B_{2R}(0))} \leq B\cE_{B_{2R}(0)}^{\eps_n}(v,v) \quad \text{for all $v\in G$,  $\eps\in (0,\eps_0)$}. 
\end{align*}
In particular, taking $v=\tau_{x_0}u\in C_c^\infty(B_{R}(0))\subset G$, the result follows  since 
\begin{align*}
\|\tau_{x_0}u\|_{L^p(B_{2R}(0))} &= 	\|u\|_{L^p(\R^d)} =\|u\|_{L^p(\Omega)},\\
\cE_{B_{2R}(0)}^{\eps_n}(\tau_{x_0}u, \tau_{x_0}u)
&\leq\cE_{\R^d}^{\eps_n}(\tau_{x_0}u, \tau_{x_0}u)= \cE_{\R^d}^{\eps_n}(u,u).
\end{align*}
\end{proof}

\begin{theorem}[Robust Poincar\'{e}-Friedrichs inequality]\label{thm:robust-pcare-finite}
Assume $\Omega\subset \R^d$ has finite measure, i.e.,  $|\Omega|<\infty$. 	There exist $\eps_0=\eps_0(d, p, \Omega)>0$ and $B=B(d, p, \Omega)$  such  that, for all $\eps\in (0,\eps_0)$ and $u\in C_c^\infty(\Omega)$, 
\begin{align}\label{eq:robust-pcare-finite}
\|u\|^p_{L^p(\Omega)}\leq B\iil_{ \R^d\R^d}|u(x) -u(y)|^p\nu_\eps(x-y)\d y\d x.
\end{align}
\end{theorem}

\begin{proof} 
Assume $\eps_0$ and $B$ do not exist. For $\eps_0= \frac{1}{2^n}$ and $B=2^n$, $n\geq1$, there exist $\eps_n\in (0,\frac{1}{2^n})$ and $u_n\in C_c^\infty(\Omega)$ for which   \eqref{eq:robust-poincare} fails, i.e., $\|u_n\|^p_{L^p(\Omega)}> 2^n\cE_{\R^d}^{\eps_n}(u_n,u_n)$. 
We can assume without loss of generality that $\|u_n\|_{L^p(\R^d)}=\|u_n\|_{L^p(\Omega)}=1$ so that $\cE_{\R^d}^{\eps_n}(u_n,u_n)\leq \frac{1}{2^n}$. 

Consider $\Omega'=\Omega\cup B
$ where  $B$ is, for instance,  an arbitrary nonempty ball such that $B\subset \R^d\setminus \Omega$. Observe that $|\Omega'|<\infty$ and thus by Theorem \ref{thm:asymp-compactness} we find  $u\in W^{1,p}(\R^d)$
and a subsequence still denoted $(u_n)_n$  strongly converging  to $u$ in $L^p(\Omega')$ and hence $\|u\|_{L^p(\Omega')}=1$ since $\|u_n\|_{L^p(\Omega')}=\|u_n\|_{L^p(\Omega)}=1$. Moreover,  we have
\begin{align*}
K_{d,p}\|\nabla u\|^p_{L^p(\R^d)}\leq 
\liminf_{n\to\infty}\cE_{\R^d}^{\eps_n}(u_n,u_n)=0.
\end{align*}
This implies that $u=c$  is constant a.e.  on $\R^d$. Since $u_n=0$ on $B$  as $u_n\in C_c^\infty(\Omega)$, we find $\|u\|_{L^p(B)}\leq \|u_n-u\|_{L^p(\Omega')}\xrightarrow{n\to\infty}0.$ Wherefrom, we deduce that 
$u=c=0$ a.e.  on $\R^d$. This contradict $\|u\|_{L^p(\Omega)}=1$. 
\end{proof}

\smallskip

\noindent By  letting $\eps\to 0$ in  Theorem \ref{thm:robust-pcare-finite} and/or Theorem \ref{thm:robust-pcare-one-direct} one recovers classical Poincar\'{e}-Friedrichs inequality.

\begin{corollary}[Classical Poincar\'{e}-Friedrichs inequality] \label{cor:clasical-pcare-fried}  Assume  $\Omega\subset \R^d$ is bounded in one direction or  has a finite measure, i.e., $|\Omega|<\infty$ then  
\begin{align}\label{eq:poincare-local-friedrichs}
\|u\|^p_{L^p(\Omega)}\leq B K_{d,p}\int_{\Omega}|\nabla u(x)|^p\d x\quad \text{for all $u\in W_0^{1,p}(\Omega)$}.
\end{align}
\end{corollary}

\noindent Next results is combined consequence of Theorem \ref{thm:robust-pcare-one-direct} and Theorem \ref{thm:robust-pcare-finite}. 

\begin{corollary}\label{cor:robust-pcare-fract-fried} 
Assume that   $\Omega\subset \R^d$ is bounded in one direction or that
$|\Omega|<\infty$. There are  $C=C(d,p,\Omega)$ and $s_0=s_0(d,p,\Omega)$ such that  for every $s\in(s_0 ,1)$ and  every $u\in C_c^\infty(\Omega)$ 
\begin{align}\label{eq:robust-pcare-fract-fried}
\|u\|^p_{L^p(\Omega)}\leq Cs(1-s)\iint\limits_{\R^d\R^d} \frac{|u(x)-u(y)|^p}{|x-y|^{d+sp}}\, dy\, dx. 
\end{align}
\end{corollary}

\begin{proof}
It suffices to put $\nu_\eps(h)=\frac{p\eps(1-\eps)}{|\mathbb{S}^{d-1}|}|h|^{-d-(1-\eps)p}$ with $\eps=1-s$ and apply Theorem \ref{thm:robust-pcare-one-direct} and Theorem \ref{thm:robust-pcare-finite}.
\end{proof}

\section{Convergence of weak solutions}\label{sec:conv-solution}
In this section we establish the convergence in $L^p(\Omega)$ of weak solutions of nonlocal Dirichlet and Neumann problems to the corresponding local problems. We need the following Lemma involving the convergence from nonlocal to local of the Gauss-Green formula; see Theorem \ref{thm:gauss-green}. We point out that a new proof of the local divergence theorem which is derived from the nonlocal divergence theorem is established in  \cite{HK23}. 
\begin{lemma}\label{lem:colapsing-to-boundary}
Let $\Omega\subset \R^d$ be bounded Lipschitz domain.  
Let $v\in W^p_{\nu_\eps}(\Omega|\R^d)$ and $\varphi \in C^2_b(\R^d)$. Assume  that 
\begin{align*}
\kappa_\varphi:=\sup_{\eps>0}\|L_\eps \varphi\|_{L^\infty(\R^d)} <\infty\quad\text{ for $1<p<2$}.
\end{align*}
The following assertions hold true. 
\begin{enumerate}[$(i)$]
\item There is a constant $C_\varphi>0$ independent of $\eps$ such that 
\begin{align*}
\Big|\int_{\Omega^c} \mathcal{N}_\eps \varphi(y)v(y)\d y\Big|\leq C_\varphi\|v\|_{W^p_{\nu_\eps}(\Omega|\R^d)}.
\end{align*}
\item Assume $v\in W^{1,p}(\R^d)$, recall $\partial_{n,p}\varphi (x)=  |\nabla \varphi(x)|^{p-2}\nabla\varphi(x)\cdot n(x)$, then
\begin{align*}
\lim_{\eps\to0}\int_{\Omega^c} \mathcal{N}_\eps\varphi(y)v(y)\d y 
=K_{d,p}\, \int_{\partial\Omega}\partial_{n,p}\varphi(x) v(x)\d\sigma(x)\,.
\end{align*}
\end{enumerate}
\end{lemma}

\begin{proof}
(i) In view of the estimate \eqref{eq:second-difference} for $p\geq 2$ we have 
\begin{align*}
\Big|\psi(u(x+h)-u(x))+ \psi(u(x-h)-u(x)) \Big|
\leq C\|u\|^{p-1}_{ C_b^{2}(\R^d)}(1 \land |h|^p),
\end{align*}
with $\psi(t)= |t|^{p-2}t$, which implies 
\begin{align*}
|L_\eps\varphi|\leq C\|\varphi\|^{p-1}_{C^2_b(\R^d)}\quad\text{for $p\geq2$}.
\end{align*}
Therefore taking into account the assumption, in case either  we have 
\begin{align}\label{eq:xuniform-bdb-in-eps}
\kappa_\varphi:=\sup_{\eps>0}\|L_\eps \varphi\|_{L^\infty(\R^d)} <\infty\quad\text{ for $1<p<\infty$}.
\end{align}
Since $|\varphi(x+h)-\varphi(x)|\leq 2\|u\|_{ C_b^1(\R^d)} (1\land |h|)$ we find that 
\begin{align}\label{eq:xuniform-varphi}
\mathcal{E}^\eps(\varphi, \varphi)\leq 2^{p+1}|\Omega|\|\varphi\|^p_{C^1_b(\R^d)} \quad\text{for all }~~\eps>0\,.
\end{align}
Now, let $v\in C^\infty_c(\R^d)$. The nonlocal Gauss-Green formula \eqref{eq:gauss-green-nonlocal}  yields 
\begin{align*}
\Big|\int_{\Omega^c} \mathcal{N}_\eps \varphi(y)v(y)&\d y\Big|
= \Big| \mathcal{E}^\eps(\varphi, v) - \int_{\Omega} L_{\eps}\varphi(x)v(x)\d x\Big|\\
&\leq \mathcal{E}^\eps(\varphi, \varphi)^{1/p'} \mathcal{E}^\eps(v,v)^{1/p} +\|L_\eps\varphi\|_{L^{\infty}(\R^d)}|\Omega|^{1/p'} \|v\|_{L^p(\Omega)}\\
&\leq C_\varphi\|v\|_{W^p_{\nu_\eps}(\Omega|\R^d)},
\end{align*}
with $C_\varphi= |\Omega|^{1/p'}\big( 2^{(p+1)/p'}\|\varphi\|^{p-1}_{C^1_b(\R^d)}+\kappa_\varphi\big).$ Note that $ C^\infty_c(\R^d)$  is dense in $W^p_{\nu_\eps}(\Omega|\R^d)$ (see \cite[Theorem 3.70]{guy-thesis}).  
\noindent By the continuity of the  linear mapping $ v\mapsto \mathcal{E}^\eps(\varphi, v) - \int_{\Omega} L_{\eps}\varphi(x)v(x)\d x$, the Gauss-Green formula \eqref{eq:gauss-green-nonlocal} is applicable for $\varphi\in C_b^2(\R^d)$ and $v\in W^p_{\nu_\eps}(\Omega|\R^d)$.  Therefore, the above estimate yields $(i)$. 

\smallskip 
\noindent (ii) By Theorem \ref{thm:asymp-plevy-to-plaplace},  $L_{\eps}\varphi(x)\xrightarrow{\eps\to0} -K_{d,p}\Delta_p\varphi(x)$  (a.e for $1<p<2$). Together with \eqref{eq:xuniform-bdb-in-eps} and the fact that $v\in L^p(\Omega)\subset L^1(\Omega)$, the dominated convergence theorem yields
\begin{align*}
\int_{\Omega}L_{\eps}\varphi(x)v(x)\d x \xrightarrow[]{\eps\to 0} -K_{d,p}\int_{\Omega}\Delta_p \varphi(x)v(x)\d x\,.
\end{align*}
If $v\in W^{1,p}(\R^d)$, then $v|_\Omega, \varphi|_\Omega\in W^{1,p}(\Omega)$. Since $\partial \Omega$ is Lipschitz,  combining Theorem \ref{thm:BBM-dual-limit}, Lemma \ref{lem:colapsing-to-boundary} and the fact that $\cE^\eps(\varphi, \varphi)^{1/p'}\leq C_\varphi$ (by estimate \eqref{eq:xuniform-varphi}),  we get 
\begin{align*}
\mathcal{E}^\eps(\varphi, v) \xrightarrow{\eps\to 0}K_{d,p}\, \mathcal{E}^0(\varphi, v).
\end{align*}
\noindent Finally from the foregoing and the (non)local Gauss-Green formula we obtain $(ii)$ as follows
\begin{align*}
\lim_{\eps\to0}\int_{\Omega^c} &\mathcal{N}_\eps \varphi(y)v(y)\d y 
=
\lim_{\eps\to0} \mathcal{E}^\eps(\varphi, v)-\lim_{\eps\to0} \int_{\Omega}L_{\eps}\varphi(x)v(x)\d x\\
&= K_{d,p} \int_{\Omega}|\nabla\varphi(x)|^{p-2}\nabla\varphi(x)\cdot \nabla v(x)\d x -  K_{d,p}\int_{\Omega}\Delta_p\varphi(x) v(x)\d x\\
&= K_{d,p}\int_{\partial\Omega} |\nabla\varphi(x)|^{p-2}\nabla\varphi(x)\cdot n(x) v(x)\,\d\sigma(x)\,.
\end{align*}

\end{proof}

\begin{theorem}[\textbf{Convergence of Neumann problem I}]\label{thm:convergence-neumann} Assume $\Omega\subset \R^d$ is open bounded and connected with Lipschitz boundary. Let
$f, f_\eps\in  L^{p'}(\Omega)$  be such that $(f_\eps)_\eps$ converges weakly sense to  $f$. Let $g_\eps =\mathcal{N}_\eps\varphi$ and $g=\partial_{n,p}\varphi
$ where  $\varphi \in C^2_b(\R^d)$. In addition we assume  
\begin{align*}
\kappa_\varphi:=\sup_{\eps>0}\|L_\eps \varphi\|_{L^\infty(\R^d)} <\infty\quad\text{ for $1<p<2$}.
\end{align*}
Assume $w_\eps \in  \WnuOmRa$ is a weak solution of Neumann problem $L_\eps u= f_\eps$ in $\Omega$ and $\mathcal{N}_\eps u= g_\eps$ on $\Omega^c$ that is, 
\begin{align*}
\mathcal{E}^\eps(w_\eps, v) = \int_{\Omega}f_\eps(x) v(x)\d x + \int_{\Omega^c}g_\eps(y) v(y)\d y\quad\text{for all }~~v\in \WnuOmRa\,.
\end{align*}
Assume $w\in W^{1,p}(\Omega)$ is a weak solution to the  problem  $-K_{d,p}\Delta_p u=f$ in $\Omega$ and $K_{d,p}\partial_{n,p} u=g$ on $\partial\Omega$, i.e., 
\begin{align*}
\mathcal{E}^0(w, v) = K_{d,p}^{-1}\int_{\Omega } f(x) v(x)\d x + \int_{\partial\Omega } g(x)v(x)\d\sigma(x)\,\,\,\text{for all $u\in W^{1,p}(\Omega)$}. 
\end{align*}

\noindent Put $u_\eps= w_\eps-\fint_{\Omega}w_\eps$  and $u= w-\fint_{\Omega}w$.  Then $(u_\eps)_\eps$ strongly converges to $u$ in $L^p(\Omega)$, i.e.,  $\|u_\eps-u\|_{L^p(\Omega)}\xrightarrow[]{\eps\to0}0$. Moreover, the following weak convergence of the energies forms holds true
\begin{align*}
\mathcal{E}^\eps(u_\eps, v) \xrightarrow{\eps\to 0} K_{d,p}\,\mathcal{E}^0(u, v) 
\quad\text{for all $v\in W^{1,p}(\R^d)$}.
\end{align*}
\end{theorem}
\begin{remark}
(i) It is worth mentioning that, as proved in Section \ref{sec:existence-weak-sol} $u$ and $u_\eps$, $\eps<\eps_0$, always exist thanks to the Poincar\'{e} type inequalities. However $w$ (resp. $w_\eps$) exists if and only if $f$ and $K_{d,p}g$ (resp. $f_\eps$ and $g_\eps$) are compatible. Furthermore, the weak convergence of $(f_\eps)_\eps$ and Lemma \ref{lem:colapsing-to-boundary}  implies that 
\begin{align*}
\lim_{\eps\to 0}\langle f_\eps, 1\rangle+\langle g_\eps,  1\rangle=
\langle f, 1\rangle+\langle K_{d,p}g, 1\rangle.
\end{align*}
(ii) To prove the convergence of solutions here as well as  in the upcoming convergence results, we implicitly use the Gamma convergence of the nonlocal forms $(\mathcal{E}^{\varepsilon} (\cdot,\cdot))_\varepsilon$ to the  local $\mathcal{E}^0(\cdot,\cdot)$ (which follows from results established in Section \ref{sec:convergence-form}) and hence 
the Gamma convergence of the respective associated  functionals $(\mathcal{J}^{\varepsilon})_\varepsilon$ to the local  functional $\mathcal{J}$. 
\end{remark}

\begin{proof}
The robust Poincar\'{e} inequality (see Corollary \ref{cor:robust-poincare}) implies the existence of $\eps_0\in (0,1)$ and $C= C(d,p,\Omega)>0$ such that for all $\eps\in (0, \eps_0)$ and all  $v \in\WnuOmRa$ we have 
\begin{align}\label{eq:uniform-coercivity}
\|v-\mbox{$\fint_{\Omega} v$}\|^p_{\WnuOmRa}\leq C\mathcal{E}^\eps(v,v). 
\end{align}

\noindent By the weak convergence  of $(f_\eps)_\eps$, up to relabeling $\eps_0>0$, we can assume 
$M:=\sup\limits_{\eps\in (0,\eps_0)}\|f_\eps\|_{L^{p'}(\Omega)}<\infty$, so that 
\begin{align*}
\Big|\int_{\Omega^c} f_\eps(x) u_\eps(x)\d x\Big|
\leq M\|u_\eps\|_{\WnuOmRa}, 
\end{align*}
whereas,  Lemma \ref{lem:colapsing-to-boundary} $(i)$ yields 
\begin{align*}
\Big|\int_{\Omega^c} g_\eps( y) u_\eps(y)\d y\Big|
=  \Big|\int_{\Omega^c} \cN_\eps \varphi(y) u_\eps (y)\d y\Big|\leq C_\varphi\|u_\eps\|_{\WnuOmRa}. 
\end{align*}
Since $ u_\eps\in \WnuOmRa^\perp$, by definition  of $u_\eps$, we have 
\begin{align*}
\mathcal{E}^\eps(u_\eps,u_\eps) 
&= \int_{\Omega} f_\eps(x) u_\eps (x)\d x + \int_{\Omega^c} g_\eps(y) u_\eps (y)\d y\\
&\leq \|u_\eps\|_{\WnuOmRa}(\|f_\eps\|_{L^{p'}(\Omega)} + C_\varphi)\leq C \|u_\eps\|_{\WnuOmRa}. 
\end{align*}
\noindent Combining this with \eqref{eq:uniform-coercivity} yields the following uniform boundedness for a generic constant $C>0$, 
\begin{align}\label{eq:uniform-boundedness}
\|u_\eps\|^{p-1}_{\WnuOma}\leq \|u_\eps\|^{p-1}_{\WnuOmRa}\leq C\qquad\text{for all } \eps\in (0,\eps_0)\,.
\end{align}
Accordingly, by the asymptotic compactness Theorem \ref{thm:asymp-compactness}, there is $u\in W^{1,p}(\Omega)$ and subsequence $\eps_n\to 0$ such that $(u_{\eps_n})_n$ converges to $u$ in $L^p(\Omega)$ and we have
\begin{align*}
K_{d,p}\cE^0(u,u)\leq \liminf_{\eps\to 0}\cE^\eps(u_\eps, u_\eps). 
\end{align*}
In particular we get $u\in W^{1,p}(\Omega)^\perp$ since  each $u_\eps\in L^p(\Omega)^\perp$.  Next, we show that $u$ is, in fact, the unique weak solution to the Neumann problem $-K_{d,p}\Delta_p u=f$ in $\Omega$ and $K_{d,p}\partial_{n,p} u=g$ on $\partial\Omega$. 
To this end, it is sufficient to show that 
\begin{align*}
\cJ(u)&= \min\{\cJ(v):\,v\in W^{1,p}(\Omega)^\perp\},\\
\mathcal{J}(v) &= \frac{1}{p} K_{d,p}\mathcal{E}^0(v,v) -\int_\Omega f(x)v (x)\d x  
-K_{d,p} \int_{\partial\Omega} g (x)v (x)\d \sigma(x). 
\end{align*}
Recall that, by Proposition \ref{prop:min-equiv-var}, each $u_\eps$ satisfies
\begin{align*}
\cJ^\eps(u_\eps)&= \min\{\cJ^\eps(v):\,v\in \WnuOmRa^\perp\},\\
\mathcal{J}^\eps(v) &= \frac{1}{p} \mathcal{E}^\eps(v,v) -\int_\Omega f_\eps (x)v (x)\d x  -\int_{\Omega^c}g_\eps (x)v (x)\d x. 
\end{align*}
Now we consider $v\in W^{1,p}(\Omega)^\perp$.  Given that $\partial \Omega$ is Lipschitz, i.e.,  $\Omega$ is an extension domain, we let $\overline{v}\in W^{1,p}(\R^d)$ be an extension of $v$. In view of Lemma \ref{lem:colapsing-to-boundary}, Theorem \ref{thm:BBM-dual-limit} and the weak convergence we have 
\begin{align*}
&\lim_{\eps\to0}\mathcal{J}^\eps(\overline{v}) 
= \lim_{\eps\to0}\big(\frac{1}{p} \mathcal{E}^\eps(\overline{v},\overline{v}) -\int_\Omega f_\eps (x)v (x)\d x  -\int_{\Omega^c}g_\eps (y)v (y)\d y\big) \\
&=\frac{1}{p}K_{d,p}\mathcal{E}^0(v,v) -\int_\Omega f(x)v (x)\d x  -K_{d,p}\int_{\partial\Omega}g (x)v (x)\d\sigma(x)= \cJ(v). 
\end{align*}
The strong convergence of $(u_{\eps_n})_n$ and the weak convergence of $(f_{\eps_n} )_n $ in $L^p(\Omega)$ yield 
\begin{align*}
\lim_{n\to\infty}\int_\Omega f_{\eps_n} (x) u_{\eps_n} (x) \d x= \int_\Omega f(x) u(x) \d x.
\end{align*}
Analogously,  by further  taking into account the uniform boundedness of $ (L_{\eps_n}\varphi)_n$ (see \eqref{eq:xuniform-bdb-in-eps}), the pointwise convergence $L_{\eps_n}\varphi(x)\to -K_{d,p} \Delta_p \varphi(x)$ (see Theorem \ref{thm:asymp-plevy-to-plaplace}) and the Gauss-Green formula (see Appendix \ref{sec:appendix-gauss-green})  we get 
\begin{align*}
&\lim_{n\to\infty}\int_{\Omega^c}g_{\eps_n} (y) u_{\eps_n} (y) \d y
=\lim_{n\to\infty}\int_{\Omega^c}\cN_{\eps_n}\varphi (y)
u_{\eps_n} (y) \d y\\
&=\lim_{n\to\infty} \cE^{\eps_n} (\varphi, u_{\eps_n})
- \lim_{n\to\infty}\int_{\Omega}L_{\eps_n}\varphi (y) u_{\eps_n} (y) \d y\\
&=K_{d,p}\cE^0(\varphi, u) - K_{d,p}\int_{\Omega}\Delta_p\varphi (x) u (x) \d x\\
&= K_{d,p}\int_{\partial \Omega}\partial_{n,p} \varphi(x) u(x) \d\sigma(x)= K_{d,p}\int_{\partial\Omega} g(x) u(x) \d\sigma(x).  
\end{align*}
Altogether with the lower estimate $K_{d,p}\cE^0(u,u)\leq\liminf_{n\to\infty}\cE^{\eps_n}(u_{\eps_n}, u_{\eps_n})$, we obtain 
\begin{align*}
\cJ(u)
\leq\liminf_{n \to \infty} \cJ^{\eps_n}(u_{\eps_n}). 
\end{align*}

Recall that each $u_{\eps_n}$  minimizes the functional $\cJ^{\eps_n}$. Since,$\int_\Omega \overline{v}(x)\d x=0$ and   $\overline{v}\in W^{1,p}(\R^d)\subset \WnuOmRa$, we find that
\begin{align*}
\cJ(u)\leq \liminf_{n\to \infty}\cJ^{\eps_n}(u_{\eps_n})\leq 
\liminf_{n\to \infty} \cJ^{\eps_n}(\overline{v})= \cJ(v).
\end{align*}
It turns out that $\|u_{\eps_n}-u\|_{L^p(\Omega)}\to 0$ as $n\to\infty$ and  $u$ minimizes $\cJ$ that is, 
\begin{align*}
\cJ(u)&= \min_{v\in W^{1,p}(\Omega)^\perp } \cJ(v). 
\end{align*}
Whence $u$ is the unique weak solution to the Neumann on $W^{1,p}(\Omega)^\perp$ that is 
\begin{align*}
\cE^0(u, v) = \int_{\Omega}f(x) v (x)\d x+ \int_{\partial \Omega}g(x)v(x)\d \sigma(x)\,\qquad\text{for all $v\in W^{1,p}(\Omega)^\perp$}. 
\end{align*}
The uniqueness of $u$ implies  that $\|u_{\eps}-u\|_{L^p(\Omega)}\to 0$ as $\eps\to0$. 
Moreover for $v\in W^{1,p}(\R^d)$ we have $v-c\in W^{1,p}(\Omega)^\perp\cap \WnuOmRa^\perp$ with $c=\fint_\Omega v$ and hence 
\begin{align*}
\lim_{\eps\to0}\cE^\eps(u_\eps, v) &= 
\lim_{\eps\to0}\cE^\eps(u_\eps, v-c) 
\\&= \lim_{\eps\to0}\int_\Omega f_\eps (x)
(v (x)-c)\d x  +\lim_{\eps\to0} \int_{\Omega^c}g_\eps (y)(v (y)-c)\d y \\
&=\int_\Omega f(x)(v (x)-c)\d x + K_{d,p}\int_{\partial\Omega}g (x)(v (x)-c)\d x\\
&= K_{d,p}\cE^0(u, v-c) =K_{d,p}\cE^0(u, v). 
\end{align*}
\end{proof}
The above  convergence remains for weak solutions associated with the regional operators; 
\begin{align*}
L_{\Omega,\eps}u(x)= 2\pv \int_{\Omega} \psi(u(x)- u(y))\nu_\eps(x-y)\d y.
\end{align*}

\begin{theorem}[\textbf{Convergence of Neumann problem II}]\label{thm:convergence-neumann regional} 
Let the assumptions of Theorem \ref{thm:convergence-neumann} be in force. Let  $u_\eps \in  \WnuOma^\perp$ be  a weak solution to the regional Neumann problem $L_{\Omega,\eps} u= f_\eps$ on $\Omega$ and $\int_\Omega u=0$, i.e., 
\begin{align*}
\mathcal{E}^\eps_\Omega(u_\eps, v) = \int_{\Omega}f_\eps(x) v(x)\d x \quad\text{for all }~~v\in \WnuOma^\perp\,.
\end{align*}
Let $u\in W^{1,p}(\Omega)^\perp$ be the  weak solution of $-K_{d,p}\Delta_p u=f$ in $\Omega$ and $\partial_{n,p} u=0$ on $\partial\Omega$ i.e. 
\begin{align*}
K_{d,p}\,\mathcal{E}^0(u, v) = \int_{\Omega } f(x) v(x)\d x \quad\text{for all $u\in W^{1,p}(\Omega)^\perp$}. 
\end{align*}

\noindent Then $(u_\eps)_\eps$ strongly converges to $u$ in $L^p(\Omega)$, i.e.,  $\|u_\eps-u\|_{L^p(\Omega)}\xrightarrow[]{\eps\to0}0$. Moreover, the following weak convergence of the energies forms holds true
\begin{align*}
\mathcal{E}^\eps(u_\eps, v) \xrightarrow{\eps\to 0} K_{d,p}\,\mathcal{E}^0(u, v) 
\quad\text{for all $v\in W^{1,p}(\Omega)$}.
\end{align*}
\end{theorem}
\begin{proof}
The proof is analogous to that of Theorem \ref{thm:convergence-neumann}. 
\end{proof}

\begin{theorem}[\textbf{Convergence of Dirichlet problem}]\label{thm:convergence-dirichlet} Assume $\Omega\subset \R^d$ is open with a continuous boundary and in addition that $|\Omega|<\infty$ or that $\Omega$ is bounded in one direction.  Let
$g\in W^{1,p}(\R^d)$ and $f, f_\eps\in L^{p'}(\Omega)$ such that $(f_\eps)_\eps$ converges weakly to $f$. Let $u_\eps \in  \WnuOmRao$ be the  weak solution to the  Dirichlet problem $L_\eps u= f_\eps$ on $\Omega$ and $u= g$ on $\Omega^c$, i.e., $u-g\in \WnuOmRao$ and 
\begin{align*}
\mathcal{E}^\eps(u_\eps, v) = \int_{\Omega } f_\eps(x) v(x)\d x\quad\text{for all }~~v\in \WnuOmRao\,.
\end{align*}
Let $u\in W^{1,p}(\Omega)$ be the weak solution to the problem $-K_{d,p}\Delta_p u=f$ in $\Omega$ and $u=g$ on $\partial\Omega$ i.e., $u-g\in W^{1,p}_0(\Omega)$ and 
\begin{align*}
	 K_{d,p}\mathcal{E}^0(u, v) = \int_{\Omega } f(x) v(x)\d x \quad\text{for all $u\in W^{1,p}_0(\Omega)$}. 
\end{align*}

\noindent Then $(u_\eps)_\eps$ strongly converges to $u$ in $L^p_{\loc}(\R^d)$, where we put $u=g$ on $\Omega^c$, i.e., we have   $\|u_\eps-u\|_{L^p(B)}\xrightarrow{\eps\to 0}0$ for any bounded set $B\subset \R^d$. If, in addition,  $|\Omega|<\infty$ then we have $\|u_\eps-u\|_{L^p(\R^d)}\xrightarrow{\eps\to 0}0$.  Moreover, the weak convergence of the energies forms holds, i.e., 
\begin{align*}
\mathcal{E}^\eps(u_\eps, v) \xrightarrow{\eps\to 0} K_{d,p}\mathcal{E}^0(u, v) 
\qquad\text{for all $v\in W^{1,p}_0(\Omega)$}.
\end{align*}
\end{theorem}

\begin{proof}
By the robust Poincar\'{e}-Friedrichs inequality (see Theorem \ref{thm:robust-pcare-one-direct} and Theorem \ref{thm:robust-pcare-finite}) there exist $\eps_0\in (0,1)$ and $C= C(d,p,\Omega)>0$ such that for all $\eps\in (0, \eps_0)$ and all  $v \in\WnuOmRao$ we have 
\begin{align}\label{eq:uniform-coercivity-D}
\|v\|^p_{\WnuOmRa}\leq C\mathcal{E}^\eps(v,v). 
\end{align}
\noindent By virtue of the weak convergence  of $(f_\eps)_\eps$, up to relabeling $\eps_0>0$, we can assume that  
$M:=\sup_{\eps\in (0,\eps_0)}\|f_\eps\|_{L^{p'}(\Omega)}<\infty$. Furthermore,  since $g\in W^{1,p}(\R^d)$ and $\int_{\R^d}(1\land|h|^p)\nu_\eps(h)\d h=1$, by the estimate \eqref{eq:levy-p-estimate} we have 
\begin{align*}
\int_{\R^d}|g(x)|^p\d x+ \iil_{\R^d\R^d} |g(x)-g(y)|^p\nu_\eps(x-y)\d y\d x
\leq 2^{p+1} \|g\|^p_{W^{1,p}(\R^d)}.
\end{align*}
By virtue of the coercivity estimate \eqref{eq:uniform-coercivity-D}, by proceeding similarly as for the proof of \eqref{eq:weak-sol-bounded-D}, one finds a  constant  $C>0$ independent of $\eps$ such that, for all $\eps\in (0,\eps_0)$ we have 
\begin{align}\label{eq:uniform-boundedness-D}
\begin{split}
\|u_\eps\|_{\WnuOmRa}
&\leq C\big(\|f_\eps\|_{L^{p'}(\Omega)}+\|g\|_{W^p_{\nu_\eps}(\R^d)}
\big) \\
&\leq C (M+2^{p+1}\|g\|^p_{W^{1,p}(\R^d)}). 
\end{split}
\end{align}
Therefore taking into account  that   $\|g\|_{W^p_{\nu_\eps}(\Omega^c)} \leq2^{p+1}  \|g\|_{W^{1,p}(\R^d)}$ and $u_\eps=g$ on $\Omega^c$ and, we deduce from the estimate  \eqref{eq:uniform-boundedness-D} that
\begin{align*}
\|u_\eps\|_{W^p_{\nu_\eps}(\R^d)}= \big( \|g\|^p_{W^p_{\nu_\eps}(\Omega^c)}+ \|u_\eps\|^p_{\WnuOmRa}\big)^{1/p}
\leq C\qquad\text{for all } \eps\in (0,\eps_0),
\end{align*}
 for some $C= C(d,p,\Omega, M, g)>0$. Accordingly, by the asymptotic compactness Theorem \ref{thm:asymp-compactness}, there is $u\in W^{1,p}(\R^d)$ and subsequence $\eps_n\to 0$ such that $(u_{\eps_n})_n$ converges to $u$ in $L^p_{\loc}(\R^d)$. Moreover,  there holds 
\begin{align*}
K_{d,p}\int_{\R^d}|\nabla u(x)|^p\d x&\leq \liminf_{\eps\to 0}\cE^\eps_{\R^d}(u_\eps, u_\eps),\\
K_{d,p}\cE^0(u,u)&\leq \liminf_{\eps\to 0}\cE^\eps(u_\eps, u_\eps). 
\end{align*}
In particular we have $u=g$ on $\Omega^c$ since $u_{\eps}=g$ on $\Omega^c$ . Therefore, since $\partial\Omega$ is continuous, we get $u-g\in W^{1,p}_0(\Omega)$. Next, we show that $u$ is, in fact, the unique weak solution to the Dirichlet problem $-K_{d,p}\Delta_p u=f$ in $\Omega$ and $u=g$ on $\partial\Omega$. 
To this end, it is sufficient to show that $\cJ(u)= \min\{\cJ(v):\,v-g\in W^{1,p}_0(\Omega)\}$, 
\begin{align*}
\mathcal{J}(v) &= \frac{1}{p} K_{d,p}\mathcal{E}^0(v,v) - \int_\Omega f(x)v (x)\d x. 
\end{align*}
Recall that, by Proposition \ref{prop:min-equiv-var-D}, each $u_\eps$ satisfies $\cJ_0^\eps(u_\eps)= \min\{\cJ_0^\eps(v):\,v-g\in \WnuOmRao\}$, 
\begin{align*}
\cJ_0^\eps(v) &= \frac{1}{p} \cE^\eps(v,v) 
-\int_\Omega f_\eps (x)v (x)\d x. 
\end{align*} 
Now we consider $v\in W^{1,p}(\Omega)$ such that $v-g\in W^{1,p}_0(\Omega)$. Thus, we can consider $\overline{v}$ be the extension of $v$ by $\overline{v}=g$ on $\Omega^c$ so that $\overline{v}\in W^{1,p}(\R^d$ (since $v$ and $g$ have the same trace on $\partial\Omega$) and $v-g\in  \WnuOmRao$. Since,  each $u_{\eps_n}$ minimizes $\cJ^{\eps_n}_0$ and  $\overline{v}-g\in\WnuOmRao$, we have $\cJ_0^{\eps_n}(u_{\eps_n})\leq 
\cJ_0^{\eps_n}(\overline{v}).$ In view of Theorem \ref{thm:BBM-dual-limit} and the weak convergence we have 
\begin{align*}
\lim_{\eps\to0}\mathcal{J}^\eps_0(\overline{v}) 
&= \lim_{\eps\to0}\big(\frac{1}{p} \mathcal{E}^\eps(\overline{v},\overline{v}) -\int_\Omega f_\eps (x)v (x)\d x \big)\\
&=\frac{1}{p}K_{d,p}\mathcal{E}^0(v,v) -\int_\Omega f(x)v (x)\d x  = \cJ_0(v). 
\end{align*}
The strong convergence of $(u_{\eps_n})_n$ and the weak convergence of $(f_{\eps_n} )_n $ in $L^p(\Omega)$ yield 
\begin{align*}
\lim_{n\to\infty}\int_\Omega f_{\eps_n} (x) u_{\eps_n} (x) \d x= \int_\Omega f(x) u(x) \d x.
\end{align*}
Whence, we deduce that  
\begin{align*}
\cJ_0(u)\leq \liminf_{n\to \infty}\cJ_0^{\eps_n}(u_{\eps_n})\leq 
\liminf_{n\to \infty} \cJ_0^{\eps_n}(\overline{v})= \cJ_0(v).
\end{align*}
It turns out that, $\|u_{\eps_n}-u\|_{L^p(\Omega)}\to 0$ as $n\to\infty$,  $u-g\in W^{1,p}_0(\Omega)$ and 
\begin{align*}
\cJ_0(u)&= \min_{v-g\in W^{1,p}_0(\Omega)} \cJ_0(v).
\end{align*}
In other words, $u$ is the unique weak solution to the Dirichlet problem on $W^{1,p}(\Omega)$ that is $u-g\in W^{1,p}_0(\Omega)$ and 
\begin{align*}
 \cE^0(u, v) = \int_{\Omega}f(x) v (x)\d x\quad\text{for all $v\in W^{1,p}_0(\Omega)$}. 
\end{align*}
The uniqueness of $u$ implies  that $\|u_{\eps}-u\|_{L^p(\Omega)}\to 0$ as $\eps\to0$. 
Moreover, for $v\in W^{1,p}_0(\Omega)$, if $v=0$ on $\Omega^c$ then  we have $v\in\WnuOmRao$ and hence 
\begin{align*}
\lim_{\eps\to0}\cE^\eps(u_\eps, v) 
&= \lim_{\eps\to0}\int_\Omega f_\eps (x)
v (x)\d x =\int_\Omega f(x)v (x)\d x =K_{d,p}\cE^0(u, v). 
\end{align*}
\end{proof}
Let us recall (see Section \ref{sec:norming-cste-pfrac}) the fractional $p$-Laplacian and the corresponding normal derivative 
\begin{align*}
	(-\Delta)^s_pu(x)&:=C_{d,p,s}\pv \int_{\R^d}\frac{\psi(u(x)-u(y))}{|x-y|^{d+sp}}\d y, \\
	\cN_su(y)&:=C_{d,p,s}\int_{\Omega}\frac{\psi(u(y)-u(x))}{|x-y|^{d+sp}}\d x.
\end{align*}
\begin{theorem}\label{thm:convergence-neumann-frac} 
	Assume the conditions of Theorem \ref{thm:convergence-neumann} hold.  Let $f, f_s\in  L^{p'}(\Omega)$  be such that $(f_s)_s$ converges weakly sense to  $f$ as $s\to 1$ and we put $g_s=\cN_s \varphi$ and $\partial_{n,p}\varphi$. 
	Let  $w_s \in  W^{s,p}(\Omega|\R^d)$, $s\in (0,1)$  be a weak solution to the Neumann problem
	\begin{align*}
		\text{$(-\Delta)^s_pu= f_{s}$ on $\Omega$ and $\cN_su=g_{s}$ on $\Omega^c$.}
	\end{align*}
	Let  $w\in W^{1,p}(\Omega)$ be a   weak solution to the Neumann problem
	\begin{align*}
		\text{$-\Delta_p u=f$ in $\Omega$ and $\partial_{n,p} u=g$ on $\partial\Omega$.}
	\end{align*}
	\noindent Put $u_s= w_s-\fint_{\Omega}w_s$  and $u= w-\fint_{\Omega}w$. Then $(u_s)_s$ strongly converges to $u$ in $L^p(\Omega)$, i.e.,  $\|u_s-u\|_{L^p(\Omega)}\xrightarrow[]{s\to1}0$. 
\end{theorem}

\begin{proof}
	It is sufficient to consider $\nu_\eps(h)= a_{d,p,\eps}|h|^{-d-(1-\eps)p}$, $a_{d,p,\eps}= \frac{p\eps(1-\eps)}{|\mathbb{S}^{d-1}|}$ in Theorem \ref{thm:convergence-neumann}, accounting the fact that the asymptotic of the normalizing constant $C_{d,p,s}$ yields
	\begin{align*}
		\lim_{\eps\to0}\frac{C_{d,p,1-\eps}}{a_{d,p,\eps}} 
		= \lim_{s\to1 }\frac{C_{d,p,s}|\mathbb{S}^{d-1}|}{ps(1-s)}=\frac{2}{K_{d,p}}. 
	\end{align*}
\end{proof}

\begin{theorem}\label{thm:convergence-dirichlet-frac} 
Assume the assumptions of Theorem \ref{thm:convergence-dirichlet} hold.  Let $f, f_s\in  L^{p'}(\Omega)$  be such that $(f_s)_s$ converges weakly sense to  $f$ as $s\to 1$. Let  $u_s \in  W^{s,p}_0(\Omega|\R^d)$, $s\in (0,1)$  be the weak solution to the Dirichlet problem
\begin{align*}
\text{$(-\Delta)^s_pu= f_{s}$ on $\Omega$ and $u=g_{s}$ on $\Omega^c$.}
\end{align*}
Let $u\in W^{1,p}(\Omega)$ be the  weak solution of  the Dirichlet problem
\begin{align*}
\text{$-\Delta_p u=f$ in $\Omega$ and $u=g$ on $\partial\Omega$.}
\end{align*}
\noindent Then $(u_s)_s$ strongly converges to $u$ in $L^p_{\loc}(\R^d)$, where we put $u=g$ on $\Omega^c$.  If in addition $|\Omega|<\infty$ then we have $\|u_s-u\|_{L^p(\R^d)}\xrightarrow{s\to 1}0$. 
\end{theorem}

\begin{proof}
It is sufficient to consider $\nu_\eps(h)= a_{d,p,\eps}|h|^{-d-(1-\eps)p}$, $a_{d,p,\eps}= \frac{p\eps(1-\eps)}{|\mathbb{S}^{d-1}|}$ in Theorem \ref{thm:convergence-dirichlet}, accounting the fact that the asymptotic of the normalizing constant $C_{d,p,s}$ yields
\begin{align*}
\lim_{\eps\to0}\frac{C_{d,p,1-\eps}}{a_{d,p,\eps}} 
= \lim_{s\to1 }\frac{C_{d,p,s}|\mathbb{S}^{d-1}|}{ps(1-s)}=\frac{2}{K_{d,p}}. 
\end{align*}
\end{proof}

\appendix
\section{} \label{sec:appendix-estim}
\subsection{Elementary estimates for the $p$-Laplacian and $p$-L\'{e}vy operators}
We establish elementary estimates involving the mapping $x\mapsto |x|^{p-2}x$, $x\in \R^d$, $p\geq1$, useful in the study of the $p$-Laplacian and $p$-L\'{e}vy operators. We adopt the convention $|x|^{p-2}x=0$ if $x=0$. Although these estimates are elementary the case $1<p<2$ seems to be seldom in the literature.  
\begin{lemma}\label{lem:elm-est} 
For $x,y\in \R^d,$ there hold the following inequalities  
\begin{align*}
& \big||x|^{p-2}x-|y|^{p-2}y\big|\leq A_p|x-y|(|x|+|y|)^{p-2},\\
&\big(|x|^{p-2}x-|y|^{p-2}y\big)\cdot (x-y)\geq A'_p|x-y|^2(|x|+ |y|)^{p-2} \,.
\end{align*}
\end{lemma}
The constants are given by $A_p= p-1$, $A'_p=\min(2^{-1}, 2^{2-p})$ if $p\geq2$ and $A_p=2^{2-p}(3-p)\leq 2^{3-p}$, $A'_p=p-1$ if $1\leq p<2$.
One easily verifies that $A_p\leq 2^{1+|p-2|}$ and $A'_p\geq \min (p-1, 2^{1-p})$.

\begin{proof}
If $p\geq2$ then by monotonicity, $(|y|^{p-2}-|x|^{p-2}) (|y|^2-|x|^2)\geq0$ and hence  using $(a+b)^q\leq \max(1, 2^{q-1})(a^q+ b^q)$ for all $q, a, b\in (0, \infty)$ we get 
\begin{align*}
(|x|^{p-2}x-|y|^{p-2} y)\cdot (x-y)
&= \frac{1}{2} (|x|^{p-2}+|y|^{p-2}) |x-y|^2\\
&+ \frac{1}{2} (|x|^{p-2}-|y|^{p-2}) (|x|^2-|y|^2)\\
&\geq\min(2^{-1}, 2^{2-p}) (|x|+|y|)^{p-2} |x-y|^2.
\end{align*}
On the other hand, the fundamental theorem of calculus implies 
\begin{align*}
&|x|^{p-2}x-|y|^{p-2} y
=(x-y)\int_0^1 |tx+ (1-t) y|^{p-2} \d t \\
&+
(p-2)\int_0^1 |tx+ (1-t)y|^{p-4}\big[(tx+ (1-t) y)\cdot (x-y) \big] (tx+ (1-t) y)\d t\,. 
\end{align*}
Thus, if $-1<p-2<0$, by Cauchy-Schwartz inequality and $|tx+(1-t)y|^{2-p}\leq (|x|+|y|)^{2-p}$, 
\begin{align*}
(||x|^{p-2}x-|y|^{p-2} y|)\cdot (x-y) 
&\geq  (p-1)|x-y|^2\int_0^1 |tx+ (1-t) y|^{p-2} \d t\\
&\geq (p-1) |x-y|^2(|x|+|y|)^{p-2}.
\end{align*}
This completes the proof of the second inequality. Analogously, if $p\geq2$ then 
\begin{align*}
||x|^{p-2}x-|y|^{p-2} y|| &\leq (p-1)|x-y|\int_0^1 (t|x|+ (1-t) |y|)^{p-2} \d t\\
&\leq (p-1)|x-y|(|x|+|y|)^{p-2} \,. 
\end{align*} 
It remains to prove the first inequality when $1<p<2$.  Without loss of generality, assume 
$x\neq y$, $|y|\leq |x|$ and $y= |y| e$ with $|e|=1$. Put $z=x|y|^{-1}$, it suffices to bound $F(z)$ where 

\begin{align*}
F(z)= \frac{||z|^{p-2}z-e|}{|z-e|(1+|z|)^{p-2}},\quad \text{for}\,\, z\in \R^d\setminus\{e\}, \,\, |z|\geq 1.
\end{align*}
Put $r=|z-e|$ so that $|z|\leq 1+r$. Since $|z|^{p-2}z-e= |z|^{p-2}((z-e)+ (1-|z|^{2-p})e)$ we get
\begin{align*}
\frac{||z|^{p-2}z-e|}{|z-e|(1+|z|)^{p-2}}
&\leq (1+|z|^{-1})^{2-p}
\Big(1+ \frac{ (|z|^{2-p}-1)}{|z-e|}\Big)\\ 
&\leq  2^{2-p}
\big(1+ \frac{1} {r} ((1+r)^{2-p} -1) \big) \\
&= 2^{2-p}
\big(1+\frac{ (2-p)}{r}\int_0^r\frac{\d \tau}{(1+\tau)^{p-1}}\big)\leq 2^{2-p}(3-p).
\end{align*}
\end{proof}

A direct consequence of Lemma \ref{lem:elm-est} is the following; see also \cite[Lemma 2.2]{BL94}. 
\begin{corollary}\label{cor:elm-est-beta} Let $x,y\in \R^d$ and $\beta\geq0$. There hold the following inequalities 
\begin{align*}
& \big||x|^{p-2}x-|y|^{p-2}y\big|\leq A_p|x-y|^{1-\beta} (|x|+|y|)^{p-2+\beta},\\
&\big(|x|^{p-2}x-|y|^{p-2}y\big)\cdot (x-y)\geq A'_p|x-y|^{2+\beta}(|x|+ |y|)^{p-2-\beta} \,.
\end{align*}
\end{corollary}
\begin{proof}
It suffices to see  that $|x-y|(|x|+|y|)^{-1}\leq 1$ and use Lemma \ref{lem:elm-est}.
\end{proof}

\begin{lemma}[Simon inequalities]\label{lem:elm-est-simon} The following inequalities hold true for all $x,y\in \R^d.$ 
\begin{numcases}{\hspace{-2ex}\big||x|^{p-2}x-|y|^{p-2}y\big|\leq} 
A_p |x-y|(|x|+|y|)^{p-2} & $p\geq2,$ \label{eq:upper-elem-degen}
\\
A_p |x-y|^{p-1}&  $1\leq p<2$. \label{eq:upper-elem-sing} 
\end{numcases}
\begin{numcases}{\hspace{-2ex}\big(|x|^{p-2}x-|y|^{p-2}y\big)\cdot (x-y)\geq \hspace{-1ex}} 
\hspace{-1ex} A'_p |x-y|^p&  \hspace{-1ex} $p\geq2$, \label{eq:under-elem-degen} 
\\
\hspace{-2ex}A'_p |x-y|^2(|x|+|y|)^{p-2} &  \hspace{-3ex}$1\leq p<2$.\label{eq:under-elem-sing}
\end{numcases}
\end{lemma}

\begin{proof}
The claims follow from Corollary \ref{cor:elm-est-beta} by taking $\beta=0$ for \eqref{eq:upper-elem-degen} and \eqref{eq:under-elem-sing} and $\beta=|p-2|$ for 
\eqref{eq:upper-elem-sing} and \eqref{eq:under-elem-degen}. Note however that, \eqref{eq:upper-elem-sing} also follows from \eqref{eq:under-elem-degen} by duality. Indeed, if $1<p<2$ then $p'>2$  where $p+p'=pp'$.  Keeping in mind that by duality, $z'=|z|^{p-2}z$ if and only if $z=|z'|^{p'-2}z'$ the estimate \eqref{eq:under-elem-degen} implies
\begin{align*}
|x'-y'|^{p'}&\leq\max(2, 2^{p'-2}) \big(|x'|^{p'-2}x'-|y'|^{p'-2}y'\big)\cdot (x'-y')\\
&\leq \max(2, 2^{p'-2})|x-y||x'-y'|.
\end{align*}
Since $(p-1)(p'-1)=1$  and $\max(2^{p-1}, 2^{2-p})< 2<2^{3-p}$ it follows that 
\begin{align}\label{eq:better-upper-sing}
\big||x|^{p-2}x-|y|^{p-2}y\big|
&\leq \max(2^{p-1}, 2^{2-p})|x-y|^{p-1}\leq 2^{3-p}|x-y|^{p-1}.
\end{align}
\end{proof}             
The constant $2^{2-p}$ is expected in \eqref{eq:upper-elem-sing} as this is the case in one dimension.  

\begin{lemma}\label{lem:elem-est}{\ }
If $1\leq  p<2$ there holds  following inequality 
\begin{align}\label{eq:upper-elem-sing-sc}
&||b|^{p-2}b-|a|^{p-2} a|\leq 2^{2-p}|b-a|^{p-1}\quad\text{for all $a,b\in \R$}. 
\end{align}
\end{lemma}

\begin{proof}
 The case $p=1$ is obvious. We only prove for $1< p<2$, i.e., $\sigma= p-1\in (0,1)$.   For $t\in\R $,  we have 
\begin{align*}
||t|^{\sigma -1}t-1|= 
\begin{cases} 
t^{\sigma}-1& \text{if}\,\,\, t\geq 1,\\
1-t^{\sigma}& \text{if}\, \,\, 0\leq t<1,\\
(-t)^{\sigma}+1& \text{if}\,\,\, t<0.
\end{cases}
\end{align*}
Since $t\mapsto t^{\sigma-1}$ is decreasing on $(0,\infty)$ it follows that $t^{\sigma}=(t-1)t^{\sigma-1}+t^{\sigma-1}\leq (t-1)^{\sigma}+1$  for  $t\geq 1$, that is $t^{\sigma}-1\leq (t-1)^{\sigma}$  and $1= (1-t)\times 1^{\sigma-1}+ t\times 1^{\sigma-1}\leq (1-t)^{\sigma}+ t^{\sigma}$ for $0<t< 1$, that is $1-t^{\sigma}\leq (1-t)^{\sigma}.$ The concavity of $t\mapsto t^{\sigma}$ implies $(-t)^{\sigma}+1\leq 2^{1-\sigma}(-t+1)^{\sigma}$ for  $t<0.$ Altogether, we obtain the following inequality 
\begin{align*}
||t|^{\sigma -1}t-1|\leq 2^{1-\sigma}|t-1|^{\sigma}\quad\text{for all $t\in \R$ and $\sigma\in [0,1]$}. 
\end{align*}  
Thus \eqref{eq:upper-elem-sing-sc} is inherited from the one  above by taking $t=\frac{b}{a}$, $a\neq 0$.  
\end{proof}

Let us see some useful consequences of Lemma \ref{lem:elm-est-simon}. 
\begin{corollary}\label{cor:elm-est-taylor} The following inequalities hold true for all $x,y\in \R^d.$ 
\begin{align*}
|x|^p- |y|^p- p|y|^{p-2}\cdot y(x-y)
&\leq \hspace*{-1ex}
\begin{cases}
A_p |x-y|^{p}& \hspace*{-4ex} 1\leq p< 2,\\
\frac{p}{2}A_p|x-y|^{2}(|x-y|+2|y|)^{p-2}&p\geq 2.
\end{cases}
\end{align*}
\begin{align*}
|x|^p- |y|^p- p|y|^{p-2}y\cdot(x-y)
&\geq \hspace*{-1ex} 
\begin{cases}
A'_p |x-y|^{p}& p\geq 2,\\
\frac{p}{2}A'_p|x-y|^{2}(|x-y|+2|y|)^{p-2}& \hspace*{-2ex}1\leq p<2.
\end{cases}
\end{align*}

\end{corollary}

\begin{proof}
Applying the fundamental theorem of calculus on $t\mapsto|z_t|,$ $z_t= y + t(x-y)$ implies 
\begin{align*}
|x|^p-|y|^p- p|y|^{p-2}y\cdot(x-y)
&=p\int_0^1\big(|z_t|^{p-2}z_t -|y|^{p-2} y\big) \cdot (z_t-y)\, \frac{\d t}{t}.
\end{align*}
The result is inherited from Lemma \ref{lem:elm-est-simon} and  $|y|+|z_t|\leq (2|y|+|x-y|)$.
\end{proof}
\smallskip

\begin{corollary} For $R>0$, $q>0$ there is $c_{q,R}>0$ such that
$$|a|^{q-1}a- |a-b|^{q-1} (a-b)\leq c_{q,R}\max(b, b^q)\quad\text{for all $|a|\leq R$, $b\geq0$.}$$  
\end{corollary}
\begin{proof}
If $b>R$ then since $t\mapsto |t|^{q-1}t$ is increasing,
and $|a|\leq R$ we have  
$$-|a-b|^{q-1} (a-b)\leq  -|a-R|^{q-1} (a-R)\leq |a-R|^q\leq  2^qR^q.$$
Therefore, since $R^q\leq \max(b,b^q)$, we get
\begin{align*}
a^{q-1}a- |a-b|^{q-1} (a-b)\leq |a|^q+2^{q}R^q
\leq 2^{q+1}R^q\max(b,b^q).
\end{align*}
Now if $q\in (0, 1]$ then since $||x|^{q-1}x-|y|^{q-1}y|\leq 2^{2-q}|x-y|^q$ (see \eqref{eq:upper-elem-sing}) it follows that  
\begin{align*}
|a|^{q-1}a- |a-b|^{q-1} (a-b)\leq   2^{2-q} \max(b, b^q).
\end{align*}
Last, if $q>1$  and  $b\leq R$ then $|a-b|^{q-1}\leq 2^{q-1}R^{q-1}$. 
The estimate \eqref{eq:upper-elem-degen} for $p=q+1$ yields
\begin{align*}
|a|^{q-1}a- |a-b|^{q-1} (a-b)|
\leq qb(|a|^{q-1}+ (|a|+b)^{q-1}) 
\leq q2^qR^{q-1} \max(b, b^q).
\end{align*}
\end{proof}

\begin{corollary}\label{cor:elm-est-cutplus} For  $a_1, a_2, b_1, b_2\in \R$, let  $b=b_1-b_2, $ and $a=a_1-a_2$, there holds the following inequality
\begin{align*}
(\psi(b)- \psi(a)) &((b_1-a_1)_{+}-(b_2-a_2)_{+})\\
&\geq
\begin{cases} 
A'_p |(b_1-a_1)_{+}-(b_2-a_2)_{+}|^p&  p\geq2, 
\\
A'_p |(b_1-a_1)_{+}-(b_2-a_2)_{+}|^2(|b|+|a|)^{p-2} &  1< p<2,
\end{cases}
\end{align*}
where $\psi(t)=|t|^{p-2}t$, $t_{+}=\max(t, 0)$ and $t_{-}=\max(-t, 0)$ so that $t=t_{+}-t_{-}$.
\end{corollary}
\begin{proof}
First of all, sincxe $t\mapsto\psi(t)$ is increasing ($\psi'(t)= (p-1)|t|^{p-2}$) we get  
\begin{align}\label{eq:xxlouwerplus}
\begin{split}
(\psi(b)- \psi(a)) &((b_1-a_1)_{+}-(b_2-a_2)_{+}))
\\&= |\psi(b)-\psi(a)||(b_1-a_1)_{+}-(b_2-a_2)_{+}|. 
\end{split}
\end{align}
For instance, if  $(b_2-a_2)\leq 0$ and $(b_1-a_1)\geq0$ then we have 
\begin{align*}
b-a= (b_1-a_1)-(b_2-a_2)\geq (b_1-a_1)_{+}=(b_1-a_1)_{+}-(b_2-a_2)_{+}\geq0.
\end{align*}
In particular, $\psi(a)\leq \psi(b)$ and hence the relation \eqref{eq:xxlouwerplus} follows. The other cases can be derived analogously. 
On the other hand, the estimates \eqref{eq:under-elem-degen} and \eqref{eq:under-elem-sing}  yield
\begin{align*}
|\psi(b)- \psi(a)|\geq
\begin{cases} 
A'_p |b-a|^{p-1}&  p\geq2,
\\
A'_p  |b-a| (|b|+|a|)^{p-2} &  1<p<2,
\end{cases}
\end{align*}
Using the fact that  $|t_{+}-s_{+}|\leq |t-s|$, the desired estimates follow from the relation \eqref{eq:xxlouwerplus}. 
\end{proof}

Another important consequence of Lemma \ref{lem:elm-est} is the following.
\begin{theorem} Let $p,q\in [1,\infty) $. Under the notations of Lemma \ref{lem:elm-est},  we get  for all $x,y\in \R^d$ that 
\begin{align*}
& \big||x|^{p-2}x-|y|^{p-2}y\big|\leq \hspace*{-0.5ex} A_p\big[A'_{\frac{p-2}{q}+2}\big]^{-q}|x-y|^{1-q}
\big||x|^{\frac{p-2}{q}}x-|y|^{\frac{p-2}{q}} y\big|^q,\\
&\big(|x|^{p-2}x-|y|^{p-2}y\big)\cdot (x-y) \geq  \hspace*{-0.5ex} A'_p \big[A_{\frac{p-2}{q}+2}\big]^{-q}|x-y|^{2-q}
\big||x|^{\frac{p-2}{q}}x-|y|^{\frac{p-2}{q}} y\big|^q\,.
\end{align*}
\end{theorem}

\begin{proof}
Since  $\frac{p-2}{q}+2\geq1$, it suffices to apply Lemma \ref{lem:elm-est}  for $p$ and $\frac{p-2}{q}+2$.
\begin{align*}
&\big[(|x|+|y|)^{\frac{p-2}{q}}\big]^q
\leq \big[A'_{\frac{p-2}{q}+2}\big]^{-q}|x-y|^{-q}\big||x|^{\frac{p-2}{q}}x-|y|^{\frac{p-2}{q}} y\big|^q,\\
&\big[(|x|+|y|)^{\frac{p-2}{q}}\big]^q
\geq \big[A_{\frac{p-2}{q}+2}\big]^{-q}|x-y|^{-q}\big||x|^{\frac{p-2}{q}}x-|y|^{\frac{p-2}{q}} y\big|^q.
\end{align*} 
\end{proof}
\section{}
\label{sec:appendix-elment} 
\subsection{Pointwise evaluation of the operator $L$ and $\cN$} \label{sec:appendix-pointwise} 
We aim in this appendix to provide ancillary results about the $p$-L\'{e}vy operator that enters the scope at hand.
As usual we assume $\nu$ is symmetric and  $p$-L\'{e}vy integrable, i.e., $\nu(h)=\nu(-h)$  and $\nu\in L^1(\R^d, 1\land|h|^p)$. 
\begin{align*}
Lu (x)&= 2\pv \int_{\R^d}\psi(u(x)-u(y))\nuxminy\d y = \lim_{\varepsilon\to 0} L_{\varepsilon} u (x),\\
L_{\varepsilon} u (x)&= 2\int_{\R^d\setminus B_{\varepsilon}(x) }\!\! \psi(u(x)-u(y))\nuxminy\d y \qquad\hbox{$(x\in \R^d; \eps >0)$}.
\end{align*}
The pointwise definition of $Lu(x)$ and  $-\Delta_p u(x)$ for  $u\in C^2_b(\R^d)$ is warranted in the degenerate case (also often called the superquadratic case), i.e. $p\geq 2$. However, in the situation singular case (also often called the subquadratic case), i.e.,  $1<p<2$, 
the pointwise definition of $Lu(x)$ and  $-\Delta_p u(x) $ might not exist even for a bona fide function  $u\in C_c^\infty(\R^d)$.
In general, it is difficult to characterize a set of all functions on which the operators $L$ and  $-\Delta_p$  act in a reasonable pointwise sense. A reasonable alternative in both cases, i.e., $1<p<\infty$, is rather to  evaluate  $Lu$ and $\Delta_pu$ in the generalized sense, i.e., in the weak sense or via their respective associated energies forms. For instance, by duality the following identifications are well-defined; 
\begin{align*}
\langle Lu, \varphi\rangle&=
\mathcal{E}_{\R^d}(u,\varphi)\qquad u,\varphi\in W^p_\nu(\R^d),\\
\langle \Delta_pu, \varphi\rangle&=\int_{\R^d} |\nabla u(x)|^{p-2}\nabla  u(x)\cdot\nabla\varphi(x)\d x \qquad u,\varphi\in W^{1,p}(\R^d). 
\end{align*}
From the duality point of view, one obtains that $Lu\in (W^p_\nu(\R^d))' $ and $ \Delta_p u\in (W^{1,p}(\R^d))'$. 
It turns out that the operators $L: W^p_\nu(\R^d)\to (W^p_\nu(\R^d))'$ and $\Delta_p:W^{1,p}(\R^d)\to (W^{1,p}(\R^d))'$ are well-defined. We emphasize that 
$\langle Lu, \varphi\rangle$ and  $\langle \Delta_pu, \varphi\rangle$  are given as above. 
Morally, the  nonlocal operator $L$ may be as good or as bad as the local operator $-\Delta_p$.  The operator $L$ can be seen as a prototype of a nonlinear nonlocal operator of divergence just as the $-\Delta_p$ is a prototype of a nonlinear local operator of divergence.

\par Next, to investigate the pointwise evaluation of $Lu$, we need to introduce consider H\"{o}lder spaces.  Denote the space $C_b^{k+\gamma}(\R^d)=\big\{u\in C^k(\R^d)\,:\, \|u\|_{C_b^{k,\gamma}(\R^d)}<\infty\big\}$, $k\in \mathbb{N}$ and $0\leq \gamma\leq 1$ with the norm
$$\|u\|_{C_b^{k,\gamma}(\R^d)}=  \sum_{|\beta|\leq k} \sup_{x\in \R^d} |\partial^\beta u(x)|+  \sum_{|\beta|=k} \sup_{ x\neq y\in \R^d} \frac{|\partial^\beta u(x)- \partial^\beta u(y)|}{|x-y|^\gamma}.$$

For $0\leq \gamma\leq 1$, we define
\begin{align}\label{eq:exponent-levy}
\widehat{p}_\gamma=
\begin{cases}
p+\gamma-1& \quad \text{if}\quad p\geq2,\\
(\gamma+1)(p-1)& \quad \text{if}\quad 1<p<2.
\end{cases}
\end{align}

\begin{remark}
Note that $\widehat{p}_\gamma\leq p$ and hence we find that $L^1(\R^d,1\land|h|^{\widehat{p}_\gamma})\subset L^1(\R^d,1\land|h|^p).$ For the particular case $\gamma=0$ we have $p_0=p-1$ and  the case $\gamma=1$ we have
\begin{align*}
\widehat{p}_1=
\begin{cases}
p& \text{if } p\geq2,\\
2(p-1) &\text{if } 1<p<2. 
\end{cases}
\end{align*}

\end{remark}
For the subclass of kernels  $\nu\in L^1(\R^d, 1\land |h|^{\widehat{p}_\gamma})$, it is possible to evaluate $Lu$ pointwise when $u$ is sufficiently smooth.  

\begin{proposition}\label{prop:uniform-cont}
Assume $ u\in C^{1+\gamma}_b(\R^d) $, $0< \gamma\leq 1$ and $\nu\in L^1(\R^d, 1\land|h|^{\widehat{p}_\gamma})$.
\begin{enumerate}[(i)]
\item The map $x\mapsto L u(x)$ is bounded and uniformly continuous. Moreover,  
\begin{align*}
Lu(x) 
& =-\int_{\R^d} \big[\psi(u(x+h)-u(x))+\psi(u(x-h)-u(x))\big]\,\nu(h)\,\d h\\
&= -2\int_{\R^d} \big[\psi(u(x+h)-u(x))-\mathds{1}_{B_1(0)}(h) \psi(\nabla u(x)\cdot h)\big]\,\nu(h)\,\d h .
\end{align*} 
\item The map $x\mapsto L_\varepsilon u(x),$ $0<\varepsilon<1,$ is uniformly continuous.
\item The family $(L_\varepsilon u(x))_\varepsilon$ is uniformly bounded and uniformly converges to $Lu,$ i.e.
\begin{align*}
\|L_\eps u-Lu\|_{L^\infty(\R^d)}\xrightarrow[]{\eps \to 0}0.
\end{align*}
\end{enumerate}
\end{proposition}

\begin{proof} 
Fix $x\in \R^d$, note that $h \mapsto \psi(\nabla u(x)\cdot h)\nu(h) $ has vanishing integral over $B_1\setminus B_\varepsilon $, i.e., 
\begin{align*}
\int_{B_1(0)\setminus B_\varepsilon(0)} \psi(\nabla u(x)\cdot h) \nu(h)\d h= 0\qquad\text{for all}~~ 0<\varepsilon<1.
\end{align*}
This observation readily implies that 
\begin{align*}
L_{\varepsilon} u(x) &= -2\int_{B^c_\varepsilon (0)} \hspace*{-1ex}\big[\psi(u(x+h)-u(x))-\mathds{1}_{B_1(0)}(h) \psi(\nabla u(x)\cdot h)\big]\nu(h)\d h,\\
\intertext{ whereas, the simple  change of variables $y=x\pm h$  gives}
L_{\varepsilon} u(x) 
&=-\int_{ B^c_\varepsilon(0) } \big[\psi(u(x+h)-u(x))+\psi(u(x-h)-u(x))\big]\nu(h)\d h.
\end{align*}
We emphasize that  $L_{\varepsilon} u(x) $ exists, since $\nu\in L^1(\R^d\setminus B_\varepsilon(0))$. 
For $h\in \R^d$, by the mean value Theorem,  $u(x+ h)- u(x)= \nabla u(x+ \tau h)\cdot h$ for some $\tau\in (0,1)$.  Thus the estimates \eqref{eq:upper-elem-degen} and \eqref{eq:upper-elem-sing} with $a= \nabla u(x- \tau_1 h)\cdot h$ and $b= \nabla u(x+ \tau_2 h)\cdot h$ for suitable $\tau_1,\tau_2\in [0,1]$ yield
\begin{align}\label{eq:second-difference}
&\Big|\psi(u(x+h)-u(x))\hspace*{-0.5ex}+ \hspace*{-0.5ex}\psi(u(x-h)-u(x)) \Big|
\hspace*{-0.5ex}\leq\hspace*{-0.5ex} C\|u\|^{p-1}_{ C_b^{1+\gamma}(\R^d)}(1 \land |h|^{\widehat{p}_\gamma}),\\
&\Big|\psi(u(x+h)-u(x))-\psi(\nabla u(x)\cdot h)\Big|
\leq C\|u\|^{p-1}_{ C_b^{1+\gamma}(\R^d)}(1 \land |h|^{\widehat{p}_\gamma}). \label{eq:gradient-difference}
\end{align}
Here, $C>0$ is a generic constant only depending on $p$ and $d$. In view of the estimates \eqref{eq:second-difference} and \eqref{eq:gradient-difference}, since $h\mapsto (1 \land |h|^{\widehat{p}_\gamma})\nu(h)$ is integrable, the boundedness of $x\mapsto L u(x)$ and the uniform boundedness of $x\mapsto L_\varepsilon u(x)$ follow and one also gets rid of the principal value.  In addition, the uniform convergence of $(L_\varepsilon u)_\varepsilon$ follows since
\begin{align*}
\|L_\eps u-Lu\|_{L^\infty(\R^d)} \leq C\|u\|^{p-1}_{ C^{1+\gamma}_b(\R^d)} \int_{B_\varepsilon(0)}(1 \land |h|^{\widehat{p}_\gamma})\nu(h)\d h \xrightarrow[]{\varepsilon\to0}0. 
\end{align*}
Turning to the uniform continuity, we fix $x,z\in \R^d $ such that $|x-z|\leq \delta$ with $0<\delta<1$.
For every h $\in \R^d$, $h \ne 0$,  the estimates \eqref{eq:upper-elem-degen} and \eqref{eq:upper-elem-sing} imply
\begin{align}\label{eq:kappa-ineq}
|\psi(u(x) -u(x\pm h))-\psi(u(z) -u(z\pm h))|
\leq C\delta^{\gamma\kappa_p}\|u\|^{p-1}_{ C^{1+\gamma}_b(\R^d)},
\end{align}
where $\kappa_p=1$ if $p\geq2$ and $\kappa_p={p-1}$ if $1<p<2$. This combined with \eqref{eq:second-difference} yields 
\begin{align*}
|\psi(u(x) -u(x\pm h))-\psi(u(z) -u(z\pm h))|
\leq C\|u\|^{p-1}_{ C^{1+\gamma}_b(\R^d)} (\delta^{\gamma\kappa_p}\land|h|^{\widehat{p}_\gamma}).
\end{align*}
Therefore, the  integrability of $h\mapsto (1 \land |h|^{\widehat{p}_\gamma})\nu(h)$ implies the uniform continuity as follows 
\begin{align*}
\|L u(x)-L u(z) \|_{L^\infty(\R^d)} \leq C\|u\|^{p-1}_{ C^{1+\gamma}_b(\R^d)}  \int_{\R^d}(\delta^{\gamma\kappa_p}\land|h|^{\widehat{p}_\gamma})\nu(h)\d h \xrightarrow{\delta \to0}0. 
\end{align*}
The uniform continuity of $x\mapsto L_\varepsilon u(x)$ follows analogously. 
\end{proof}

It is natural to seek  for a larger  functional space on which $Lu$ is defined.
In an attempt to answer this question, assume in addition that $\nu:\R^d  \setminus\{0\}\to [0,\infty) $ is unimodal, i.e. $\nu$ is radial and almost decreasing, i.e.,  there is a constant $c$ such that $\nu(y)\leq c \nu(x)$ whenever $|y|\geq |x|$. Let us define the function 
$$\widehat{\nu}(x)=\nu(\tfrac{1}{2}(1+|x|)).$$ 

\begin{remark}{\, }  Assume  $\nu\in L^1(\R^d\setminus B_\varepsilon (0))$, $\varepsilon>0$ and $\nu$ is unimodal.
\begin{enumerate}[$(i)$]
\item It is straightforward to show that $\widehat{\nu}\in L^1(\R^d )\cap L^\infty(\R^d )$. 
\item  The space $L^{p-1}(\R^d ,\widehat{\nu})$ contains the spaces $ C^{1+\gamma}_{\loc}(\R^d )\cap L^\infty(\R^d )$, $L^\infty(\R^d )$ and $C_b^{1+\gamma}(\R^d )$.
\item 	If $\nu(h) = |h|^{-d-sp}$,  $s\in (0,1)$, then $\widehat{\nu}(h) \asymp (1+|h|)^{-d-sp}$.
\end{enumerate}
\end{remark}
\begin{proposition}\label{prop:L1-def-L}
Assume  that $\nu\in L^1(\R^d, 1\land|h|^{\widehat{p}_\gamma}) $ with $0< \gamma\leq 1$, and that $\nu$ is almost decreasing, i.e., there is $c>0$ such that $\nu(x)\geq c\nu(y)$ whenever $|x|\leq |y|$. 
\begin{enumerate}[$(i)$]
\item If $u\in C^{1+\gamma}_{\loc}(\R^d )\cap L^{p-1}(\R^d ,\widehat{\nu})$ then $Lu (x)$ is well defined for all $x\in \R^d$. 
\item  If $u\in C^{1+\gamma}_b(\R^d )$ and $\supp u\subset B_R(0)$, $R\geq 1$ then there is $C=C(R,d,\nu)>0$
\begin{align}\label{eq:estimate-test}
|Lu(x)|\leq C\|u\|^{p-1}_{ C^2_b(\R^d)} \widehat{\nu}(x)\qquad \text{for all}~x\in \R^d.
\end{align}
In particular, it holds that  $Lu\in L^1(\R^d)\cap L^\infty(\R^d )$.
\end{enumerate}
\end{proposition}

\begin{proof} $(i)$ Set $R=2|x|+1$, $x\in \R^d$, then for $y\in B^c_R(0)$ we have 
$|x-y|\geq
\frac{1}{2}(1+|y|)$ so that  $\nu(x-y)\leq c\widehat{\nu}(y)$. Thus, for $u \in L^{p-1}(\R^d, \widehat{\nu})$ we have 
\begin{align*}
\Big|\int_{B^c_R(0)} \psi(u(x)-u(y))\nu(x-y)\d y\Big|
&\leq C|u(x)|^{p-1}\|\widehat{\nu}\|_{L^1(\R^d)}\\
&+ C\int_{B^c_R(0)} |u(y)|^{p-1}\widehat{\nu}(y)\d y<\infty. 
\end{align*}
We conclude that $Lu(x)$ exists since by exploiting \eqref{eq:second-difference} we get
\begin{align*}
&\Big|\int_{B_R(0)} \psi(u(x)-u(y))\nu(x-y)\d y\Big|
 \\&= \frac12\Big| \int_{B_R(0)} \psi( u(x)-u(x-h)) +\psi( u(x)-u(x+h))\nu(h)\d h\Big|\\
&\leq C\|u\|^{p-1}_{C^{1+\gamma}(B_R(0))}\int_{\R^d} (1\land |h|^{\widehat{p}_\gamma})\nu(h)\d h<\infty. 
\end{align*}

$(ii)$ Assume $\supp u\subset B_R(0)$ for some $R\geq 1$. If $|x|\geq 4R $ then  $u(x)=0$. For $y \in B_R(0)$, we get  $\nu(x-y) \leq c \widehat{\nu}(x)$ since $|x-y|\geq \frac{|x|}{2}+R\geq \frac{1}{2}(1+|x|).$  Accordingly, 
\begin{align*}
|Lu(x) |\leq \int_{B_R(0)} |u(y)|^{p-1}\nu(x-y) \d y \leq c|B_R(0)|\|u\|^{p-1}_{C^{1+\gamma}_b(\R^d)} \widehat{\nu}(x).
\end{align*}
Whereas, if $|x|\leq 4R$ then $\frac{1}{2}(1+ |x|)\leq 4R$ and hence $\widehat{\nu}(x)\geq c_1$ with $c_1= c\nu(4R)>0$. The proof of \eqref{eq:estimate-test} is complete as follows using \eqref{eq:second-difference}, 
\begin{align*}
|Lu(x)|\leq c_1^{-1}C\|\nu\|_{L^1(\R^d,1\land|h|^{\widehat{p}_\gamma})} \|u\|^{p-1}_{C^{1+\gamma}_b(\R^d)} \widehat{\nu}(x).
\end{align*} 
\end{proof}
Next we show that $Lu$ remains continuous under less regularity on $u$ provided that $\nu$ is in the subclass of kernels $ L^1(\R^d, 1\land|h|^{\gamma(p-1)})\subset L^1(\R^d, 1\land|h|^p)$. 
\begin{theorem}
Let $u\in C^{\gamma}_b(\R^d), $ $0<\gamma\leq 1$. Assume that 
$\nu\in L^1(\R^d, 1\land|h|^{\gamma(p-1)})$ then $Lu$ bounded and uniform continuous. Moreover,  for all $x, z\in \R^d$,  we have 
\begin{align*}
|Lu(x)-Lu(z)|
\leq C\|u\|^{p-1}_{C^{\gamma}(\R^d)}\omega(|x-z|), 
\end{align*}
where $\omega: (0, \infty)\to (0, \infty)$ is an increasing modulus of continuity such that $\omega(r)\xrightarrow{r\to 0}0. $
\end{theorem}

\begin{proof}
Recall that $\kappa_p=1$ if $p\geq2$ and $\kappa_p={p-1}$ if $1<p<2$. For $r>0$, the estimate \eqref{eq:kappa-ineq} gives  
\begin{align*}
\Big|\int_{B^c_r(0)} \psi(u(x)-u(x+h))- \psi(u(z)-u(z+h))\nu(h)\d h\Big| 
\\\leq C\|u\|^{p-1}_{C^{\gamma}(\R^d)} |x-z|^{\gamma\kappa_p}\int_{B^c_r(0)}  \nu(h)\d h.  
\end{align*}
Using $|\psi(u(\cdot)-u(\cdot+h))|\leq |h|^{\gamma(p-1)}$ we get
\begin{align*}
\Big|\int_{B_r(0)} \psi(u(x)-u(x+h))- \psi(u(z)-u(z+h))\nu(h)\d h\Big| \\\leq C\|u\|^{p-1}_{C^{\gamma}(\R^d)}\int_{B_r(0)}  |h|^{\gamma(p-1)}\nu(h)\d h. 
\end{align*}
Summing the previous inequalities yields 
\begin{align*}
|Lu(x)-Lu(z)|
\leq C\|u\|^{p-1}_{C^{\gamma}(\R^d)}\big[ g(r)|x-z|^{\gamma\kappa_p} + h(r)\big].
\end{align*}
where $g$ and $h$ are the monotone functions,
\begin{align*}
g(r)=\int_{B^c_r(0)}  \nu(h)\d h,\quad\quad  h(r)=\int_{B_r(0)}  |h|^{\gamma(p-1)}\nu(h)\d h.  
\end{align*}
Note that $\rho(r)=\big(\frac{h(r)}{g(r)}\big)^{1/\gamma\kappa_p}$ is increasing and $ \rho(r)\xrightarrow{r\to 0}0$. Taking $r=\rho^{-1}(|x-z|)$  we obtain
\begin{align*}
|Lu(x)-Lu(z)|
\leq C\|u\|^{p-1}_{C^{\gamma}(\R^d)}\omega(|x-z|), \qquad \text{$\omega(|x-z|)= h\circ\rho^{-1}(|x-z|)$}.
\end{align*}
\end{proof}

\noindent The nonlocal normal derivative $\mathcal{N}u$ of a measurable function $u:\R^d\to \R$ can be thought of as the restriction on $\R^d\setminus\Omega$ of the regional operator on $\Omega$, namely, 
\begin{align*}
L_\Omega u(x)=-\cN u(x) = \int_{\Omega} \psi(u(x)-u(y))\nu(x-y)\d y\qquad\text{$x\in \R^d\setminus\Omega$}. 
\end{align*}
It is often desirable to know when the pointwise definition $\mathcal{N}u(x)$ makes sense at least almost everywhere. 

\begin{proposition}
Assume $\Omega\subset \R^d$ is open. The following assertions are true. 
\begin{enumerate}[$(i)$]
\item  If $u|_\Omega\in L^\infty(\Omega)$ then $\mathcal{N}u(x)$ exists for almost all  $x\in \R^d\setminus \overline{\Omega}$. 
\item If $u\in \WnuOmR$ then $\cN u\in L^q_{\loc}(\R^d\setminus \overline{\Omega})$ for any $1\leq q\leq p'$. 
\item If $u\in \WnuOmR$ then $\cN u\in L^{p'}(\R^d, w^{-1}(x))$,  $w(x)=\int_{\Omega}\nu(x-y)\d y$. 
\end{enumerate}
\end{proposition}
\begin{proof}
$(i)$ Observing that $\delta_x=\dist(x,\partial\Omega)>0$ we get 
\begin{align*}
|\cN u(x)|\leq
(\|u\|_{L^\infty(\Omega)}+|u(x)|)^{p-1} \int_{|h|>\delta_x}\nu(h)\d h<\infty. 
\end{align*}
$(ii)$ Let $K\subset \R^d\setminus \overline{\Omega}$ be compact and put $\delta=\dist(K, \partial\Omega)>0$ so
that   $|x-y|>\delta$ whenever $x\in K$ and $y\in \Omega$. The H\"{o}lder inequality implies 
\begin{align*}
&\int_K|\cN u(x)|^q\ dx\leq 
\int_K\Big(\int_{\Omega}|u(x)-u(y)|^{p-1}\nu(x-y)\d y\Big)^q\d x\\
&\leq \int_K\Big(\int_{\Omega}|u(x)-u(y)|^p\nu(x-y)\d y\Big)^{\frac{q}{p'}}\Big(\int_{\Omega}\nu(x-y)\d y\Big)^{\frac{q}{p}}\ dx \\
&\leq |K|^{1-\frac{q}{p'}}\Big(\int_{|h|>\delta}\nu(h)\d h\Big)^{\frac{q}{p}}
\Big(\int_K\int_{\Omega}|u(x)-u(y)|^p\nu(x-y)\d y\d x\Big)^{\frac{q}{p'}} \\
&\leq|K|^{1-\frac{q}{p'}} \Big(\int_{|h|>\delta}\nu(h)\d h\Big)^{\frac{q}{p}}|u|^{q(p-1)}_{\WnuOmR}. 
\end{align*}
$(iii)$ The H\"{o}lder inequality implies 
\begin{align*}
\int_{\R^d}|\cN u(x)|^{p'}w(x)^{-1} &\d x
\leq 
\int_{\R^d}\Big(\int_{\Omega}|u(x)-u(y)|^{p-1}\nu(x-y)\d y\Big)^{p'}\d x\\
&\leq \int_{\R^d}\int_{\Omega}|u(x)-u(y)|^p\nu(x-y)\d y\ dx \leq |u|^{p}_{\WnuOmR}. 
\end{align*}
\end{proof}

\subsection{Gauss-Green type formula}\label{sec:appendix-gauss-green}
In this section, we establish the nonlocal Gauss-Green type formula associated with the operator $L$. We start with the following general formula. 

\begin{theorem}[General nonlocal integration by parts] \label{thm:gene-gauss-green}
Assume $\Omega\subset \R^d$ is open bounded. Let $k:\R^d\times \R^d\setminus\diag \to [0,\infty)$ be measurable, anti-symmetric, i.e., $k(x, y)= -k(y,x)$,  satisfying  
\begin{align*}
\sup_{x\in \R^d}	\int_{\R^d}\vspace*{-2ex}|k(x,y)|\d y<\infty. 
\end{align*}
For every $v\in L^\infty(\R^d)$ the following identity holds 
\begin{align*}
\frac12\hspace*{-1ex}\iil_{(\Omega^c\times \Omega^c)^c}\hspace*{-1ex} (v(x)-v(y)) k(x,y)\d y\d x=\hspace*{-1ex} \iil_{\Omega\R^d}\hspace*{-1ex} v(x) k(x,y)\d y\d x- \hspace*{-1ex}\iil_{\Omega^c\Omega} \hspace*{-1ex}v(y) k(x,y)\d x\d y.
\end{align*}
\end{theorem}
\begin{proof}
First, observe that $(x,y)\mapsto v(x)k(x, y)$ belongs to $L^1(\Omega\times \R^d)$ since 
\begin{align*}
\iil_{\Omega\R^d} |v(x)||k(x,y)|\d y\,\d x\leq \|v\|_{L^1(\Omega)}\sup_{x\in \R^d}	\int_{\R^d}|k(x,y)|\d y<\infty. 
\end{align*}
Thus Fubini's theorem and the anti-symmetry $k(x,y)= -k(y,x)$ imply 
\begin{align*}
2\iil_{\Omega\Omega} \hspace{-0.7ex}v(x) k(x,y)\d y\d x
=  \iil_{\Omega\Omega}\hspace{-0.7ex} (v(x)- v(y)) k(x,y)\d y\d x. 
\end{align*}
Likewise, it follows that 
\begin{align*}
\iil_{\Omega\Omega^c} v(x) k(x,y)\d y\,\d x&= \iil_{\Omega\Omega^c} (v(x)- v(y)) k(x,y)\d y\,\d x+  \iil_{\Omega\Omega^c} v(y) k(x,y)\d y\,\d x\\
&=  \iil_{\Omega^c\Omega} v(y) k(x,y){\d x}{\d y} + \frac{1}{2}\iil_{\Omega\Omega^c} (v(x)- v(y)) k(x,y)\d y\,\d x\\
&+  \frac{1}{2}\iil_{\Omega^c\Omega} (v(x)- v(y)) k(x,y)\d y\,\d x .
\end{align*}
Summing up altogether yields the claim, since  $\Omega\times \R^d=\Omega\times \Omega\cup\Omega\times \Omega^c$.
\end{proof}
Recall  that  if $p\geq 2$ then   $\widehat{p}_1=p$ and hence $L^1(\R^d,1\land|h|^{\widehat{p}_1})= L^1(\R^d,1\land|h|^p)$ and if $1<p<2$ then $\widehat{p}_1=2(p-1)$  and $L^1(\R^d,1\land|h|^{\widehat{p}_1})\subset  L^1(\R^d,1\land|h|^p)$.  
\begin{theorem}[\textbf{Gauss-Green formula}]\label{thm:gauss-green} 
Let $\Omega\subset \R^d$ be open bounded, $u\in C^2_b(\R^d)$ and $v\in C^1_b(\R^d) $. Assume either $(i)$ $p\geq 2$, $(ii)$ $1<p<2$ and $(L_\varepsilon u)_\varepsilon$ is uniformly bounded or $(iii)$ $\nu\in L^1(\R^d, 1\land|h|^{\widehat{p}_1})$. There holds that
\begin{align}\label{eq:gauss-green-nonlocal}
\int_{\Omega} Lu(x)v(x)\d x= \mathcal{E}(u,v) -\int_{\Omega^c}\mathcal{N}u(y)v(y)\d y, 
\end{align}
In particular, for  $v=1$, one gets the integration by part formula
\begin{align}\label{eq:integration-by-part-nonlocal}
\int_{\Omega} Lu(x)\d x=-\int_{\Omega^c}\mathcal{N}u(y)\d y.
\end{align}
\end{theorem}

\begin{proof} Since $\nu\in L^1(\R^d\setminus B_\varepsilon(0))$, $\varepsilon>0$ and $u ,v$ are bounded, Theorem \ref{thm:gene-gauss-green} implies that 
\begin{align}\label{eq:kepsilon-formula}
\begin{split}
&\frac12\iil_{(\Omega^c\times \Omega^c)^c} (v(x)-v(y)) k_\varepsilon(x,y)\d y\,\d x\\&= \iil_{\Omega\R^d} v(x) k_\varepsilon(x,y)\d y\,\d x
-\iil_{\Omega^c\Omega} v(y) k_\varepsilon(x,y){\d x}{\d y}
\end{split}
\end{align}
where $k_\varepsilon(x,y)$ is the  anti-symmetric kernel defined by  
\begin{align*}
k_\varepsilon(x,y)= 2|u(x)-u(y)|^{p-2}(u(x)-u(y))\nu(x-y)\mathds{1}_{B^c_\varepsilon}(x-y).
\end{align*} 
Note that $\widehat{p}_1=p$, $p\geq2$ and  $\widehat{p}_1=2(p-1)$, $1<p<2$ so that  $\nu\in L^1(\R^d, 1\land|h|^{\widehat{p}_1}) \subset L^1(\R^d, 1\land|h|^p)$. In any case, by Proposition \ref{prop:uniform-cont},  $(L_\varepsilon u)_\varepsilon$ is uniformly bounded. Whence, 
\begin{align*}
\int_{\Omega} Lu(x) v(x)\d x= \lim_{\varepsilon\to 0}\int_{\Omega} L_\varepsilon u(x) v(x)\d x
=\lim_{\varepsilon\to 0}\iil_{\Omega\R^d} v(x) k_\varepsilon(x,y)\d y\,\d x.
\end{align*} 
On the other hand, since $\Omega$ is bounded and $u,v\in C_b^1(\R^d)$  we have 
\begin{align*}
|(v(x)-v(y))k_\varepsilon(x,y)|\leq C (1\land |x-y|^p)\nu(x-y)\in L^1(\Omega\times \R^d). 
\end{align*}
By  the convergence dominated theorem we have 
\begin{align*}
\cE(u,v) =\lim_{\varepsilon\to 0}\frac12\iil_{(\Omega^c\times \Omega^c)^c} (v(x)-v(y))k_\varepsilon(x,y) \d y\,\d x.
\end{align*}
The desired formula is obtained by letting $\varepsilon\to 0$ in \eqref{eq:kepsilon-formula}  as follows 
\begin{align*}
\int_{\Omega} Lu(x) v(x)\d x&- \cE(u,v) = 
\lim_{\varepsilon\to 0}\iil_{\Omega\R^d} v(x) k_\varepsilon(x,y)\d y\,\d x\\
&-
\lim_{\varepsilon\to 0}\frac12\iil_{(\Omega^c\times \Omega^c)^c} (v(x)-v(y))k_\varepsilon(x,y) \d y\,\d x\\
&=\lim_{\varepsilon\to 0} 
\iil_{\Omega^c\Omega} v(y) k_\varepsilon(x,y){\d x}{\d y}=-\int_{ \Omega^c } \cN u(y) v(y)\d y.
\end{align*}
\end{proof}

\smallskip
\subsection*{Acknowledgment}
 This work incorporates improvements of several partial results from the author's PhD thesis, which were obtained during his studies at Bielefeld University, in the framework of the International Research Training Group 2235 ``Searching for the regular in the irregular: Analysis of random and singular systems''. The author thanks  the host institution and the DFG for the financial support. 

\vspace{2mm}
\noindent \textbf{Data Availability Statement (DAS)}: Data sharing not applicable, no datasets were generated or analyzed during the current study.

\bibliographystyle{plain}

\end{document}